\documentclass[11pt,leqno]{article}
\usepackage{amsmath,amssymb,mathrsfs,xy,amsthm,makeidx}
\makeindex
\input{xy}
\xyoption{all}

\makeatletter
\newcommand{\nsection}{\@startsection{section}{1}{\z@}%
     {-3.5ex plus-1ex minus-.2ex}%
     {2.3ex plus.2ex}%
     {\reset@font\center\large\sc}}

\renewenvironment{thebibliography}[1]
{\nsection*{\refname\@mkboth{\refname}{\refname}}%
   \list{\@biblabel{\@arabic\c@enumiv}}%
        {\settowidth
	\labelwidth{\@biblabel{#1}}%
         \leftmargin
	 \labelwidth
         \advance
	 \leftmargin
	 \labelsep
         \@openbib@code
         \usecounter{enumiv}%
         \let\p@enumiv\@empty
	 \parskip=0pt
	 \itemsep=1pt
	 \parsep=1pt
	 \itemindent=\z@
         \renewcommand\theenumiv{\@arabic\c@enumiv}}%
   \sloppy
   \clubpenalty4000
   \@clubpenalty\clubpenalty
   \widowpenalty4000%
   \footnotesize
   \sfcode`\.\@m}
  {\def\@noitemerr
    {\@latex@warning{Empty `thebibliography' environment}}%
   \endlist}

 \def\@seccntformat#1{\csname the#1\endcsname.\hspace{2ex}}

 \renewcommand{\subsubsection}%
  {\@startsection{subsubsection}%
  {3}%
  {\z@}%
  {-3.5ex plus -1ex minus -.2ex}%
  {-0ex}%
  {\reset@font\normalsize\bfseries}
  }%

 \@addtoreset{subsubsection}{section}
 \@addtoreset{equation}{subsubsection}

\makeatother

\newtheoremstyle{thm}
 {1em}% Space above
 {3pt}% Space below
 {\itshape}% Body font
 {}% Indent amount
 {\bf}% Theorem head font
 {. ---}% Punctuation after theorem head
 {0.5em}% Space after theorem head
 {}% Theorem head spec (can be left empty, meaning `normal')

\newtheoremstyle{dfn}
 {1em}% Space above
 {3pt}% Space below
 {}% Body font
 {}% Indent amount
 {\bf}% Theorem head font
 {. {---}}% Punctuation after theorem head
 {0.5em}% Space after theorem head
 {}% Theorem head spec (can be left empty, meaning `normal')

\swapnumbers
\theoremstyle{thm}
\newtheorem{thm}[subsubsection]{Theorem}
\newtheorem{lem}[subsubsection]{Lemma}
\newtheorem*{lem*}{Lemma}
\newtheorem{cor}[subsubsection]{Corollary}
\newtheorem*{cor*}{Corollary}
\newtheorem{prop}[subsubsection]{Proposition}
\newtheorem*{prop*}{Proposition}
\newtheorem*{thm*}{Theorem}

\theoremstyle{dfn}
\newtheorem{dfn}[subsubsection]{Definition}
\newtheorem*{dfn*}{Definition}
\newtheorem{ex}[subsubsection]{Example}
\newtheorem*{ex*}{Example}
\newtheorem{rem}[subsubsection]{Remark}
\newtheorem*{rem*}{Remark}
\newtheorem*{cl}{Claim}

\renewcommand{\qedsymbol}{$\blacksquare$}
\newtheorem*{notation*}{Notation}

\allowdisplaybreaks

\usepackage{color}

\newcommand{\Ddag}[1]{\mathscr{D}^\dag_{{#1}}}
\newcommand{\DdagQ}[1]{\mathscr{D}^\dag_{{#1},\mathbb{Q}}}
\newcommand{\Dcomp}[2]{\widehat{\mathscr{D}}^{(#1)}_{#2}}
\newcommand{\DcompQ}[2]{\widehat{\mathscr{D}}^{(#1)}_{#2,\mathbb{Q}}}
\newcommand{\Dmod}[2]{{\mathscr{D}}^{(#1)}_{#2}}

\newcommand{\Emod}[2]{{\mathscr{E}}^{(#1)}_{#2}}
\newcommand{\Edag}[1]{\mathscr{E}^\dag_{#1}}
\newcommand{\EdagQ}[1]{\mathscr{E}^\dag_{#1,\mathbb{Q}}}
\newcommand{\Ecomp}[2]{\widehat{\mathscr{E}}^{(#1)}_{#2}}
\newcommand{\EcompQ}[2]{\widehat{\mathscr{E}}^{(#1)}_{#2,\mathbb{Q}}}

\newcommand{\Ecompb}[2]{\widehat{{E}}^{(#1)}_{#2}}
\newcommand{\EcompQb}[2]{\widehat{{E}}^{(#1)}_{#2,\mathbb{Q}}}
\newcommand{\Emodb}[2]{E^{(#1)}_{#2}}

\newcommand{\Dcompb}[2]{\widehat{{D}}^{(#1)}_{{#2}}}

\newcommand{\Dmodb}[2]{{{D}}^{(#1)}_{#2}}

\newcommand{\invlim}{\mathop{\underleftarrow{\mathrm{lim}}}}
\newcommand{\indlim}{\mathop{\underrightarrow{\mathrm{lim}}}}

\newcommand{\mr}[1]{\mathrm{#1}}
\newcommand{\ms}[1]{\mathscr{#1}}
\newcommand{\mc}[1]{\mathcal{#1}}
\newcommand{\mb}[1]{\mathbb{#1}}
\newcommand{\mf}[1]{\mathfrak{#1}}

\newcommand{\angles}[2]{\langle{#2}\rangle_{(#1)}}

\newcommand{\CR}[1]{R_{#1}}
\newcommand{\CK}[1]{K_{#1}}

\newcommand{\cc}[1]{[\![{#1}]\!]}
\newcommand{\pp}[1]{(\!({#1})\!)}
\newcommand{\shom}{\mathop{\mc{H}om}\nolimits}
\newcommand{\Dan}[1]{\ms{D}^{\mr{an}}_{#1}}
\newcommand{\LD}{\underrightarrow{LD}^b_{\mb{Q},\mr{qc}}}
\newcommand{\otimesdag}{\mathop{\otimes}\limits^{\mb{L}}}
\newcommand{\ku}[1]{\mathop{\mc{F}u}\nolimits_{#1}}
\newcommand{\fr}{h}
\newcommand{\emar}{\varepsilon_0^{\mr{rig}}}
\newcommand{\emars}{\varepsilon^{\mr{rig}}}

\newsavebox{\circlebox}
\savebox{\circlebox}{\fontencoding{OMS}\selectfont\char13}
\newlength{\circleboxwdht}
\newcommand{\ccirc}[1]{
  \setlength{\circleboxwdht}{\wd\circlebox}
  \addtolength{\circleboxwdht}{\dp\circlebox}
  \raisebox{0.2\dp\circlebox}{
    \parbox[][\circleboxwdht][c]{\wd\circlebox}
    {\centering\scriptsize #1}}
  \llap{\usebox{\circlebox}}
}

\newcommand{\Aut}[2]{\mathrm{Aut}_{#1}\bigl(#2\bigr)}
\DeclareMathOperator{\Frac}{Fr}
\DeclareMathOperator{\Res}{Res}
\DeclareMathOperator{\Tr}{Tr}
\DeclareMathOperator{\Spec}{Spec}
\DeclareMathOperator{\Spf}{Spf}
\DeclareMathOperator{\rg}{rk}

\newcommand{\ZZ}{\ensuremath{\mathbb{Z}}} 
\newcommand{\QQ}{\ensuremath{\mathbb{Q}}} 
\newcommand{\Fp}{\ensuremath{\mathbb{F}_{p}}}
\newcommand{\FF}{\ensuremath{\mathbb{F}}}
\newcommand{\Qp}{\ensuremath{\mathbb{Q}_p}}

\newcommand{\Rep}[2]{\ensuremath{\mathrm{Rep}_{#1}({#2})}}
\newcommand{\Del}[2]{\ensuremath{\mathrm{Del}_{#1}(#2)}}

\newcommand{\Ct}{K}
\newcommand{\Ctf}{\Lambda}

\newcommand{\Rob}{\ensuremath{\mc{R}}}
\newcommand{\An}[1][]{\ensuremath{\mc{A}_{#1}}}

\newcommand{\Fisosur}[1][]{F\text{-}\mathrm{Isoc}^{\dagger}_{#1}}
\newcommand{\Fiso}{F\text{-}\mathrm{Isoc}}

\DeclareMathOperator{\WD}{WD}
\newcommand{\val}[1][]{\ensuremath{v_{#1}}}
\newcommand{\id}[1][]{\ensuremath{\mathrm{id}_{#1} }}
\DeclareMathOperator{\tr}{tr}
\newcommand{\F}{F}
\renewcommand{\det}{\operatorname{det}}
\newcommand{\Fro}[1][]{\ensuremath{\varphi_{#1}} }
\newcommand{\Wt}[2][]{\ensuremath{W_{#1}(#2)}}
\newcommand{\nr}{\ensuremath{^\mathrm{ur}}}

\newcommand{\rig}{\mathrm{rig}}
\newcommand{\bCt}{\overline{\mathbb{Q}}_p} %macro
\newcommand{\bFt}{\overline{\mathbb{F}}} %macro

\newcommand{\one}%
{{\rm 1\hspace*{-0.4ex}\rule{0.1ex}{1.52ex}\hspace*{0.2ex}}}%{\bf 1}
\newcommand{\End}[2]{\mathrm{End}_{#1}\bigl(#2\bigr)}
\DeclareMathOperator{\Coker}{Coker}
\DeclareMathOperator{\Ker}{Ker}
\newcommand{\Gr}[1]{\ensuremath{{\bf Gr}\left(#1\right)}}
\newcommand{\ur}{\ensuremath{^u}}

\newcommand{\Aone}{\widehat{\mathbb{A}}}    %{\widehat{\mathbb{A}}^1_x} 
\newcommand{\Aoned}{\widehat{\mathbb{A}}'}   %{\widehat{\mathbb{A}}^1_{x'}} %
\newcommand{\Pone}{\widehat{\mathbb{P}}}    % {\widehat{\mathbb{P}}^{1}_{x}} 
\newcommand{\Poned}{\widehat{\mathbb{P}}'}    % {\widehat{\mathbb{P}}^{1}_{x'}} %
\newcommand{\Poneoned}{\widehat{\mathbb{P}}''}  %{\widetilde{\mathbb{P}}}
\newcommand{\nF}[1]{\ms{F}_{\mr{naive,\pi}}(#1)}
\newcommand{\nFnp}[1]{\ms{F}_{\mr{naive}}(#1)}
\newcommand{\DwL}{L}
\begin{document}
\title{Product formula for $p$-adic epsilon factors}
\author{Tomoyuki Abe, Adriano Marmora}
\date{}
\maketitle

\begin{abstract}
 Let $X$ be a smooth proper curve over a finite field of characteristic
 $p$. We prove a product formula for $p$-adic epsilon factors of
 arithmetic $\ms{D}$-modules on $X$. In particular we deduce the
 analogous formula for overconvergent $F$-isocrystals, which was
 conjectured previously. The $p$-adic product formula is the equivalent
 in rigid cohomology of the Deligne-Laumon formula for epsilon factors
 in  $\ell$-adic \'etale cohomology (for $\ell\not = p$). One of the
 main tools in the proof of this $p$-adic formula is a theorem of
 regular stationary phase for arithmetic $\ms{D}$-modules that we prove
 by microlocal techniques.
\end{abstract}

\setcounter{tocdepth}{2}
\tableofcontents

\section*{Introduction}
Inspired by the Langlands program, Deligne suggested that the constant
appearing in the functional equation of the $L$-function of an
$\ell$-adic sheaf, on a smooth proper curve over a finite field of
characteristic $p\neq\ell$, should factor as product
of local contributions (later called \emph{epsilon factors}) at each
closed point of the curve. He conjectured a \emph{product formula} and
showed  some particular cases of it, cf.\ \cite{Del:sem}. This formula
was proven by Laumon in the outstanding paper \cite{Lau}.

The goal of this article is to prove a product formula
for $p$-adic epsilon factors of arithmetic $\ms{D}$-modules on a curve.
This formula generalizes the conjecture formulated in
\cite{Marmora:Facteurs_epsilon} for epsilon factors of overconvergent
$F$-isocrystals, and it is an analog in rigid cohomology of the
Deligne-Laumon formula.

Let us give some notation. In this introduction we simplify the
exposition by assuming more hypotheses than necessary, and we
refer to the article for the general statements. Let $k$ be a
finite field of characteristic $p$, and let $q=p^f$ be its
cardinality. Let $X$ be a smooth, proper and geometrically connected
curve over $k$.

We are interested in  rigid cohomology \cite{Berthelot:prepub} on $X$,
which is a good $p$-adic theory in the sense that it is a Weil
cohomology. The coefficients for this theory are the overconvergent
$F$-isocrystals: they play the  role of the smooth sheaves in
$\ell$-adic cohomology, or vector bundles with (flat) connection in
complex analytic geometry. These coefficients are also known
in the literature as $p$-adic differential equations. As
their $\ell$-adic and complex analogs, the overconvergent
$F$-isocrystals form a category which is not stable under push-forward
in general. Berthelot \cite{Ber:rigD}, inspired by algebraic
analysis, proposed a framework to remedy this problem by introducing
arithmetic $F$-$\ms{D}$-modules (shortly $F$-$\ms{D}^{\dagger}$-modules)
and in particular the subcategory of holonomic modules, see
\cite{BerInt} for a survey. Thanks to works of many people ({\it e.g.\
}\cite{Caro:courbes}, \cite{Cr3},...), we have a satisfactory theory, at
least in the curve case. We note that another approach to $p$-adic
cohomologies has been initiated by Mebkhout and Narvaez-Macarro
\cite{MeNa}, and  is giving interesting developments, for example see
\cite{ArMe}. Although it would certainly be interesting to
transpose our calculations into this theory,  we place ourselves
exclusively in the context of Berthelot's arithmetic $\ms{D}$-modules
throughout this paper. Nevertheless, we point out that Christol-Mebkhout's
results in the local theory of $p$-adic differential equations are
indispensable both explicitly and implicitly in this article.
The local theory of arithmetic $\ms{D}$-modules has been developed by
Crew  (cf.\ \cite{Cr3}, \cite{Cr}) and we will use it extensively in our
work.

To state the $p$-adic product formula let us review the definitions of
local and global epsilon factors for holonomic
$\ms{D}^{\dagger}$-modules. The Poincar\'{e} duality was established for
overconvergent $F$-isocrystals in the works of Berthelot \cite{BePoinc},
Crew \cite{Crew:ens} and Kedlaya \cite{KeDuke},
and for the theory of $\ms{D}^{\dagger}$-modules by the first author
\cite{Abe3} based on the results of Virrion \cite{Vir}. This gives a
functional equation for the $L$-function (Caro \cite{Caro:courbes},
Etesse-LeStum \cite{LE}) of $\ms{M}$, see \S\ref{sb:main}. The constant
appearing in this functional equation is $\varepsilon(\ms{M}):=
\prod_{r\in\ZZ} \det(-F; H^rf_+\ms{M})^{(-1)^{r+1}}$, where
$f\colon X\rightarrow \Spec(k)$ is the structural morphism, and it is
called the \emph{global} epsilon factor of $\ms{M}$.

The \emph{local} epsilon factor of an arithmetic $\ms{D}$-module $\ms{M}$ at a
closed point $x$ of $X$ is defined up to the choice of a meromorphic
differential form $\omega\neq0$ on $X$. To define it, we restrict
$\ms{M}$ to the complete {\it trait} $S_x$ of $X$ at
$x$. To define the local factors
$\varepsilon(\ms{M}|_{S_x},\omega)$,  we consider a
localizing triangle, cf.\ (\ref{anotherloctriag}); hence, by linearity,
it remains to define the epsilon factors for punctual modules and
for free differential modules on the Robba ring with Frobenius
structure. The former case is explicit; the latter was done in
\cite{Marmora:Facteurs_epsilon} via the Weil-Deligne representation
attached to free differential modules by the $p$-adic monodromy
theorem.

The product formula (Theorem \ref{Product formula}) states that for any
holonomic $F$-$\ms{D}^{\dagger}$-module $\ms{M}$ on $X$, we have
\begin{equation}
 \label{intro:PF}\tag{PF}
  \varepsilon (\ms{M}) = q^{ r(\ms{M})(1-g)}
  \prod_{x\in |X| } \varepsilon(\ms{M}|_{S_x}, \omega),
\end{equation}
where $g$ is the genus of $X$, $r(\ms{M})$ denotes the generic rank of
$\ms{M}$, $|X|$ is the set of closed points of $X$, and $\omega\neq0$ is
a meromorphic differential form on $X$. This formula can be seen as a
multiplicative generalization of Grothendieck-Ogg-Shafarevich formulas in
rigid cohomology.

The proof of (\ref{intro:PF}) starts by following the track  of
Laumon: a geometric argument (see \cite[proof of 3.3.2]{Lau}) reduces to
prove the fundamental case where $X=\mathbb{P}^1_k$ and $\ms{M}$ is an
$F$-isocrystal overconvergent along a closed set $S$ of rational points
of $X$ (by refining the argument  we can even take $S=\{0,\infty\}$,
cf.\ \cite[p.121]{Katz:TravauxLau}). By saying that $\ms{M}$ is an
$F$-isocrystal, we mean that it is an arithmetic $\ms{D}$-module
corresponding to an $F$-isocrystal via the specialization map, see the
convention \S\ref{conventionocisoc}. In order to conclude, we need four
components: (1) a canonical extension functor $\ms{M}\mapsto
\ms{M}^{\mr{can}}$ from the category of holonomic
$F$-$\ms{D}^{\dag}$-modules on the formal disk to that of
holonomic $F$-$\ms{D}^{\dag}$-modules on the projective line,
overconvergent at $\infty$; (2) the proof of \eqref{intro:PF} for
$\ms{D}^{\dagger}$-modules in the essential image of this functor; (3)
the ``principle'' of stationary phase (for modules whose differential slopes
at infinity are less than $1$); (4) an exact sequence in the style
``nearby-vanishing cycles'' for
certain kinds of $\ms{D}^{\dagger}$-modules.

The first component is provided by the work of Crew \cite{Cr}, extending
the canonical extension of Matsuda for overconvergent
$F$-isocrystals. The second is technical but not difficult to
achieve. The third is the deepest among these four, and a large part of
this paper is devoted to it. This ``principle'' can  roughly be
described by saying that it provides a description of the behavior at
infinity of the  Fourier-Huyghe transform  of $\ms{M}$, in terms of
local contributions at closed points $s$ in $\mathbb{A}^1_k$
where $\ms{M}$ is singular ({\it i.e.\ }the characteristic cycle of
$\ms{M}$ does contain a vertical component at $s$,
cf.\ paragraph \ref{stabledef}).
These ``local contributions''  are called  local Fourier transforms
(LFT) of $\ms{M}$, and  one of the key points of our work is to provide a
good construction of them. Here, we differentiate from the work of
Laumon, who used vanishing cycles to construct the local Fourier
transform of an $\ell$-adic sheaf.

A definition of local Fourier transform has been given by Crew
\cite[8.3]{Cr} following the classical path: take the canonical
extension of a holonomic $F$-$\ms{D}^{\dagger}$-module at $0$, then
apply the Fourier-Huyghe transform, and finally restrict around
$\infty$. However, we need more information on the internal structure of
LFT, and therefore, we redefine it. Our approach is based on
microlocalization inspired by the classical works of Malgrange
\cite{Malgrange:Eq_diff} and Sabbah \cite{Sabbah:isomono}. Yet,
there are many more technical difficulties in our case because we need to
deal with differential operators of infinite order. We note that
the definition is still not completely local in the sense that it uses
the canonical extension and the Frobenius is constructed by global
methods.
Once we have established some fundamental 
properties of LFT, the proof of the regular stationary phase is analogous to
that of Sabbah \cite{Sabbah:isomono} in the classical case (see
also \cite{Lopez:Microloc} for its generalizations).

The fourth component is proved using an exact sequence of Crew
\cite{Cr}, Noot-Huyghe's results on Fourier transform, and the
properties of cohomological operations proven in \cite{Abe3}.

The end of the proof of \eqref{intro:PF} is classical and it follows
again Laumon, although there are still some differences from
the $\ell$-adic case that we have carefully pointed out in
\S\ref{sb:proof_gen_case}. In particular, in {\it loc.\ cit.\ }we
detail the proof of a determinant formula for the $p$-adic epsilon
factor. This $p$-adic formula gives a differential interpretation of the
local epsilon factors and promises to have new applications. Indeed, in
\S\ref{explicitcalcfrob} we give an explicit description of the
Frobenius acting on the Fourier-Huyghe transform. This might provide
explicit information on the
$p$-adic epsilon factors, and moreover have  arithmetic spin-offs. For
example, in the case of a Kummer
isocrystal, by carrying  out this calculation  and applying the product formula we can re-prove the
Gross-Koblitz formula. This and related questions will be addressed in a
future paper.

Concerning $\ell$-adic theory, we point out that  Abbes and
Saito \cite{AS} have recently given an interesting new local description of LFT
as well as an alternative proof of Laumon's determinant formula for
$\ell$-adic representations satisfying a certain ramification
condition.

Finally, regarding Langlands program
for $p$-adic coefficients, we mention the recent preprint \cite{Abe4}, where the $p$-adic product formula is used to show the equivalence between the (conjectural) \emph{$p$-adic}
Langlands correspondence for $\mathrm{GL}_n$ over function fields and 
Deligne's hope for \emph{petits camarades cristallins}, cf.\
\cite[Conj. (1.2.10-vi)]{WeilII} and also \cite[Conj. 4.13]{Crew:ens_mon}.

\bigskip
After this introduction, this article is divided into seven sections.
Here, we briefly describe their content; more information can be found
in the text at the beginning of each section and subsection.

The aim of the first section is to define the characteristic cycles of
holonomic $\ms{D}$-modules on curves over the field $k$ (which is
supposed here only of characteristic $p>0$), and prove some relations
with the microlocalizations. For this, we prove a level
stability theorem using microlocal techniques of \cite{Abe}. The
section starts with a short survey of microdifferential operators of
{\it loc.\ cit}.

The second section begins the study of local Fourier transforms for
holonomic $\ms{D}$-modules. This section is the technical core of the
paper. We start in \S\ref{Crewreview} by a review of Crew's theory of
 arithmetic $\ms{D}$-modules on a formal disk; then we study in
\S\ref{analye} the relations between microlocalization and
analytification of $\ms{D}^{\dagger}$-modules. This gives several
applications: namely the equality between Garnier's and Christol-Mebkhout's
definitions of irregularity (\S\ref{applithem}). We finish the section
by giving an alternative definition of local Fourier
transform (except for the Frobenius structure) in
\S\ref{secdefforloc}. We will see in \S\ref{section4} that this LFT
coincides with that of Crew and we will complete the definition in
\S\ref{section5} by endowing it with the Frobenius structure.
 
The third section reviews the cohomological operations on arithmetic
$\ms{D}$-modules. In particular, in \S\ref{cohoprev} we recall the
results of \cite{Abe3} which are used in this paper, and in
\S\ref{secgeforev} we review the global Fourier transform of
Noot-Huyghe.

The fourth section is devoted to the regular stationary phase. In
\S\ref{secgeomcalc} we establish some numerical results analogous to
those of Laumon for perverse $\ell$-adic sheaves. In \S\ref{secregstfo}
we prove the stationary phase for regular holonomic modules on the
projective line.
 
It is in the fifth section that we finally implement the Frobenius in
the theory. In \S\ref{secfroblocfou} we endow the local Fourier
transform with the Frobenius induced by that of the global Fourier
transform via the stationary phase isomorphism. In section
\S\ref{explicitcalcfrob} we explicitly describe the Frobenius on the
naive Fourier transform.

The sixth section provides a key exact sequence for the proof of the
product formula. This sequence should be seen as an analog of the
exact sequence of vanishing cycles appearing in Laumon's proof of the
$\ell$-adic product formula. The section begins with a result on
commutation of the Frobenius in \S\ref{secommfrob} and we then prove
the exactness of the sequence in \S\ref{secanexse}.

Finally, in the last section, we state and prove the $p$-adic product
formula. We begin in \S\ref{sb:Loc_con} with the definition of local
factors of holonomic modules; then, in \S\ref{sb:main} we recall the
definition of the $L$-function attached to a holonomic module and define
the global epsilon factor. We  state the product formula and we show
that it is in fact equivalent to the product formula for overconvergent
$F$-isocrystals conjectured in \cite{Marmora:Facteurs_epsilon}. 
The section continues with the proof of 
the product formula: some preliminary
particular cases in \S\ref{proof:geo_mod}, and  the general case, as well
as the determinant formula for local epsilon factors, in
\S\ref{sb:proof_gen_case}.

\subsection*{Acknowledgments}
The first author (T.A.) would like to express his gratitude to Prof.\
S. Matsuda for letting him know the reference \cite{Pu}, and to
C. Noot-Huyghe for answering various questions on her paper
\cite{NH}. He also like to thank professor \ T. Saito for his interest in the
work, and  teaching him the relation with {\it petits camarades} conjecture of Deligne.
Finally some parts of the work was done when the
he was visiting to IRMA of {\it Universit\'{e} de Strasbourg} in
2010. He would like to thank the second author, A. Marmora and the
institute for the hospitality. He was supported by Grant-in-Aid for JSPS
Fellows 20-1070, and partially, by {\it l'agence nationale de la
recherche} ANR-09-JCJC-0048-01.

The second author would like to thank P. Berthelot, B. Chiarellotto,
A. Iovita, B. Le Stum, and C. Noot-Huyghe, for their long  interest
in this work, advice, discussions and encouragement. He would like to
thank professors T. Saito and N. Tsuzuki, for several instructive
conversations about the $p$-adic product formula, especially during his
first visit to Japan in Fall 2006 (JSPS post-doctoral fellowship number
PE06005). He thanks X. Caruso and the members of the project
ANR-09-JCJC-0048-01, which sponsored part of this collaboration.

Finally, both authors would like to thank A. Abbes for his strong
interest in the work, and his effort to make them write the paper
quickly. Without him, we believe the completion of the paper would
have been  very much delayed.

\subsection*{Conventions and Notation}

\subsubsection{}
Unless otherwise stated, the filtration of a filtered ring (resp.\
module) is assumed to be increasing.
Let $(A,F_iA)_{i\in\mb{Z}}$ be a filtered ring or module. For $i\in
\mb{Z}$, we will often denote $F_iA$ by $A_i$. Recall that the filtered
ring $A$ is said to be a {\em noetherian filtered ring}\index{noetherian filtered ring} if its
associated Rees ring $\bigoplus_{i\in\mb{Z}}F_iA$ is noetherian.

\subsubsection{}
\label{nottop}
Let $A$ be a topological ring, and let $M$ be a finitely generated
$A$-module. We consider the product topology on $A^{\oplus n}$ for
any positive integer $n$. Let $\phi\colon A^{\oplus n}\rightarrow M$ be
a surjection, and we denote by $\ms{T}_{\phi}$ the quotient topology on
$M$ induced by $\phi$. Then $\ms{T}_\phi$ does not depend on the choice
of $\phi$ up to equivalence of topologies. We call this topology the
{\em $A$-module topology on $M$}.

\subsubsection{}
Let $K$ be a field, and $\sigma\colon K\xrightarrow{\sim}K$ be an
automorphism. A $\sigma$-$K$-vector space\index{.@miscellaneous!sigma@$\sigma$-$K$-vector space} is a $K$-vector space $V$
equipped with a $\sigma$-semi-linear endomorphism $\varphi\colon
V\rightarrow V$ such that the induced homomorphism
$K\otimes_{\sigma,K}V\rightarrow V$ is an isomorphism.

\subsubsection{}
\label{defRmsX}
Let $R$ be a complete discrete valuation ring of mixed characteristic
$(0,p)$, $k$ be its residue field, and $K$ be its field of fractions. \index{.@miscellaneous!k@$k$, $R$}\index{.@miscellaneous!K@$K$}
We
denote a uniformizer of $R$ by $\varpi$\index{.@miscellaneous!pi@$\varpi$}. For any integer $i\geq0$, we
put $R_i:=R/\varpi^{i+1}R$. The residue field $k$ is not assumed to be
perfect in general; we assume $k$ to be perfect from the middle of \S
\ref{section2}, and in the last section (\S\ref{section7}), we assume
moreover $k$ to be finite. We denote by $|\cdot|$  the
$p$-adic norm on $R$ or $K$ normalized as $|p|=p^{-1}$.

In principle, we use Roman fonts ({\it e.g.\ }$X$)
for schemes and script fonts ({\it e.g.\ }$\ms{X}$) for formal schemes.
For a smooth formal scheme $\ms{X}$ over
$\mr{Spf}(R)$, we denote by $X_i$ the reduction $\ms{X}\otimes_RR_i$
over $\mr{Spec}(R_i)$. We denote $X_0$ by $X$ unless otherwise
stated. In this paper, {\em curve} (resp.\ {\em formal curve}) means
dimension $1$ smooth separated connected scheme (resp.\ formal scheme)
of finite type over its basis.

When $X$ (resp.\ $\ms{X}$) is an affine scheme (resp.\ formal scheme),
we sometimes denote $\Gamma(X,\mc{O}_X)$ (resp.\
$\Gamma(\ms{X},\mc{O}_{\ms{X}})$) simply by $\mc{O}_X$ (resp.\
$\mc{O}_{\ms{X}}$) if this is unlikely to cause any confusion.

\subsubsection{}\label{coordinates}
%%% OLD CONVENTION
%Let $\ms{X}$ be a smooth formal scheme over $\mr{Spf}(R)$ of dimension
%$d$. A system of local coordinates is a subset
%$\{x_1,\dots,x_d\}$ of $\Gamma(\ms{X},\mc{O}_{\ms{X}})$ such that the
%morphism $\ms{X}\rightarrow\widehat{\mb{A}}^d_R$ defined by these
%functions is \'{e}tale. Let $s\in\ms{X}$ be a closed point. A system of
%local parameters at $s$ is a subset $\{y_1,\dots,y_d\}$ of
%$\Gamma(\ms{X},\mc{O}_{\ms{X}})$ such that its image in
%$\mc{O}_{X,s}$ forms a system of regular local parameters in the sense
%of [EGA $0_{\mr{IV}}$, 17.1.6]. When $d=1$, we say ``a local
%coordinate'' instead of saying ``a system of local coordinates'', and
%the same for ``a local parameter''.

Let $\ms{X}$ be a smooth formal scheme over $\mr{Spf}(R)$ of dimension
$d$. 
A system of (global) coordinates\index{coordinates, local coordinates} on $\ms{X}$ is a subset
$\{x_1,\dots,x_d\}$ of $\Gamma(\ms{X},\mc{O}_{\ms{X}})$ such that the
morphism $\ms{X}\rightarrow\widehat{\mb{A}}^d_R$ defined by these
functions is \'{e}tale. 
A system of local coordinates is a system of coordinates on an open subscheme 
$\ms{U}$ of $\ms{X}$.  

Let $s\in\ms{X}$ be a closed point. A system of
local parameters\index{local parameters} at $s$ is a subset $\{y_1,\dots,y_d\}$ of
$\Gamma(\ms{U},\mc{O}_{\ms{U}})$, for some open neighborhood $\ms{U}$ of $s$, such that its image in
$\mc{O}_{X,s}$ forms a system of regular local parameters in the sense
of [EGA $0_{\mr{IV}}$, 17.1.6]. When $d=1$, we say ``a (local)
coordinate'' instead of saying ``a system of (local) coordinates'', and
the same for ``a local parameter''.

\subsubsection{}
We freely use the language of arithmetic $\ms{D}$-modules. For details
see \cite{BerInt}, \cite{Ber1}, \cite{Ber2}. In particular, we use the
notation $\Dmod{m}{X}$, $\Dcomp{m}{\ms{X}}$, $\Ddag{\ms{X}}$.\index{differential operators!.@$\Dmod{m}{X}, \Dcomp{m}{\ms{X}}, \Ddag{\ms{X}},\DdagQ{\ms{X}}$, $\DdagQ{\ms{X}}(Z)$,\ldots|(} An index
$\mb{Q}$ means tensor with $\mb{Q}$.

\subsubsection{}
\label{conventionocisoc}
Let $X$ be a scheme of finite type over $k$, and $Z$ be a closed
subscheme of $X$. We put $U:=X\setminus Z$. We denote by
($F$-)$\mr{Isoc}(U,X/\Ct)$\index{categories!Isoc@$\mr{Isoc}(U,X/\Ct)$, $F$-$\mr{Isoc}(U,X/\Ct)$} the category of convergent
($F$-)isocrystal on $U$ over $K$ overconvergent along $Z$. If $X$ is
proper, we say, for sake of brevity, \emph{overconvergent \linebreak%%%%%%%%%%%%%%%%%%%%%%%%%%%%
($F$-)isocrystal \index{Fi@$F$-isocrystal, overconvergent isocrystal} on $U$ over $\Ct$}, instead of convergent
($F$-)isocrystal on $U$ over $\Ct$ overconvergent along $Z$, and we
denote the category by ($F$-)$\mr{Isoc}^\dag(U/K)$\index{categories!Isoc@$\mr{Isoc}^\dag(U/K)$, $F$-$\mr{Isoc}^\dag(U/K)$}.

Now, let $\ms{X}$ be a smooth formal scheme, and $Z$ be a divisor of its
special fiber. Let
$\ms{U}:=\ms{X}\setminus Z$, $X$ and $U$ be the special fibers of
$\ms{X}$ and $\ms{U}$ respectively. In this paper, we denote
$\DdagQ{\ms{X}}(^\dag Z)$ by $\DdagQ{\ms{X}}(Z)$\index{differential operators!.@$\Dmod{m}{X}, \Dcomp{m}{\ms{X}}, \Ddag{\ms{X}},\DdagQ{\ms{X}}$, $\DdagQ{\ms{X}}(Z)$,\ldots|)} for short. In the same
way, we denote $\mc{O}_{\ms{X},\mb{Q}}(^\dag Z)$\index{.@miscellaneous!Ox1@$\mc{O}_{\ms{X},\mb{Q}}(Z)$} by
$\mc{O}_{\ms{X},\mb{Q}}(Z)$. Let
$\ms{M}$ be a coherent ($F$-)$\DdagQ{\ms{X}}(Z)$-module such that
$\ms{M}|_{\ms{U}}$ is coherent as an
$\mc{O}_{\ms{U},\mb{Q}}$-module. Then we know that $\ms{M}$ is a
coherent $\mc{O}_{\ms{X},\mb{Q}}(Z)$-module by \cite{BerL}. Let
$\mc{C}$ be the full subcategory of the category of coherent
($F$-)$\DdagQ{\ms{X}}(Z)$-modules consisting of such $\ms{M}$. By \cite[4.4.12]{Ber1}
and \cite[4.6.7]{Ber2}
 the specialization functors $\mr{sp}_*$ and $\mr{sp}^*$ induce an
equivalence between $\mc{C}$ and the category
($F$-)$\mr{Isoc}^\dag(U,X/K)$. We will say
that $\ms{M}$ is a convergent ($F$-)isocrystal on $\ms{U}$
overconvergent along $Z$ by abuse of language.

\subsubsection{}
\label{convention-Tatetwist-shifts} 
The shift of a complex $\ms{C}$ will be denoted always by brackets $\ms{C}[d]$.
When parenthesis appear, like $\ms{C}(i)[d]$, it means that all the terms of the complex are Tate twisted $i$ times; cf.~\eqref{sb:dec_lambda} for a definition of Tate twist.   

%\subsubsection{} % NOT VALID ANY MORE
%\label{conventioncohomologycomplex} 
%We denote the cohomology in degree $i$ of a complex of sheaves $\ms{C}$ on a topological space by $\ms{H}^i(\ms{C})$. 
%Sometimes we prefer to write ${H}^i \ms{C}$, especially when the topological space is a point.
%This is an abuse of notations and it should not be  confused with the derived functor  $R^i\varGamma$ of global sections of a sheaf.    

\section{Stability theorem for characteristic cycles on curves}
\label{section1}
The definition of stable holonomic module (\ref{stabledef}) and 
the stability theorem (\ref{stabilitytheoremcy}) are the goal of this
section. Theorem \ref{stabilitytheoremcy} is needed to prove
the product formula for holonomic $\ms{D}^{\dagger}$-modules 
with Frobenius structure on a curve over a finite field.
Nevertheless, in this section, we tried to state the theorems in the
more possible generality: in particular, we \emph{do not} require $k$
to be perfect, neither we assume the existence of a Frobenius structure
on $\ms{D}^{\dagger}$-modules. Even if we put Frobenius structures, we
do not know if the proof of the stability theorem could be simplified.

%The statements about stable holonomic
%modules should be seen as general results about $p$-adic
%differential equations.
%For example, as consequence of Theorem \ref{stabilitytheoremcy}, we can
%define the characteristic cycle (\ref{defCycl}) and we can prove the
%finiteness (\emph{i.e.}\ objects have finite length) of the category of
%holonomic $\ms{D}^{\dagger}$-modules, cf.\ \eqref{cor_noeth_art}.
%Although these two corollaries are known for modules endowed with
%Frobenius structure, we do not know if the proof of the stability
%theorem could be simplified by this assumption.

\subsection{Review of microdifferential operators}
We review the definitions and properties of the arithmetic
microdifferential sheaves on curves, which are going to be used
extensively in this paper. For the general definitions in higher
dimensional settings and more details, see \cite{Abe}.

\subsubsection{}\label{microrecall}
Let $\ms{X}$ be a formal curve over $R$. We denote its special
fiber by $X$. Let $T^*X$ be the cotangent bundle of $X$ and
$\pi\colon T^*X\rightarrow X$ be the canonical
projection. We put $\mathring{T}^*X:=T^*X\backslash s(X)$ where
$s\colon X\rightarrow T^*X$ denotes the zero section. \index{.@miscellaneous!pi@$\pi\colon T^*X\rightarrow X$, $\mathring{T}^*X$}
Let $m\geq0$ be an
integer and $\ms{M}$ be a coherent $\DcompQ{m}{\ms{X}}$-module. One
of the basic ideas of microlocalization is to ``localize'' $\ms{M}$ over
$T^*X$ to make possible a more detailed analysis on $\ms{M}$. For this,
we define step by step the sheaves of rings $\Emod{m}{X_i}$,
$\Ecomp{m}{\ms{X}}$, $\EcompQ{m}{\ms{X}}$ on the cotangent bundle.
\index{microdifferential operators!naive $\Emod{m}{X_i}$, $\Ecomp{m}{\ms{X}}$, $\Emod{m}{\ms{X}}$, $(\Ecomp{m}{\ms{X}})_n$,\ldots|(}

Let $i$ be a non-negative integer.
Let us define $\Emod{m}{X_i}$ first. For the detail of this
construction, see \cite[2.2, Remark 2.14]{Abe}. There are mainly two
types of rings of sections of $\Emod{m}{X_i}$. Let ${U}$ be an open
subset of $T^*X$. If ${U}\cap s(X)$ is non empty, we get
$\Gamma({U},\Emod{m}{X_i})\cong\Gamma(\pi({U}\cap
s(X)),\Dmod{m}{X_i})$. Suppose that the intersection is empty. Let
${U}':=\pi^{-1}(\pi({U}))\cap\mathring{T}^*X$. Then the ring of
sections of $\Emod{m}{X_i}$ on ${U}$ is equal to that on ${U}'$, and
the ring of these sections is the ``microlocalization'' of
$\Gamma(\pi({U}),\Dmod{m}{X_i})$. Let us describe
locally the sections $\Gamma(U',\Emod{m}{X_i})$. Shrink $\ms{X}$ so that
it possesses a local coordinate denoted by $x$. We denote the
corresponding differential operator by $\partial$. There exists an
integer $N$ such that $\partial^{\angles{m}{Np^m}}$ is in the center of
$\Dmod{m}{X_i}$. Let $S$ be the multiplicative system generated by
$\partial^{\angles{m}{Np^m}}$ in $\Dmod{m}{X_i}$. 
The positive filtration on $\Gamma(\pi({U}'),\Dmod{m}{X_i})$, given 
by the order of differential operators, induces a ring filtration, necessarily indexed by $\ZZ$, on the localization $S^{-1}\Gamma(\pi({U}'),\Dmod{m}{X_i})$. Thus the elements  in $S^{-1}\Gamma(\pi({U}'),\Dmod{m}{X_i})$ 
  can have negative order, but they are ``finite'' in the sense 
  that  only finitely many negative powers of 
  $\partial^{\angles{m}{Np^m}}$ can appear in (the total symbol of) each of them.  
We define
\begin{equation*}
 \Gamma({U}',\Emod{m}{X_i}):=\bigl(S^{-1}\Gamma(\pi({U}'),
  \Dmod{m}{X_i})\bigr)^\wedge,
\end{equation*}
as the completion of this filtered ring 
with respect to negative order 
(see \cite[1.1.5]{Abe} for our conventions on the completion of a filtered ring).

Taking an inverse limit over $i$, we define $\Ecomp{m}{\ms{X}}$. The
sections of $\Ecomp{m}{\ms{X}}$ can be described concretely as
follows. Let ${U}$ be an open subset of $\mathring{T}^*X$, and assume
${V}:=\pi({U})$ to be affine. Let $\ms{V}$ be the open formal subscheme
sitting over $V$. Suppose moreover that $\ms{V}$ possesses a local
coordinate $x$, and we denote the corresponding differential operator
by $\partial$. For an integer $k\geq 0$, take the minimal integer $i$
such that $k\leq ip^m$, and let $l:=ip^m-k$. By the construction of
$\Ecomp{m}{\ms{X}}$, the operator $\partial^{\angles{m}{ip^m}}$ in
$\Dcomp{m}{\ms{X}}$ considered as a section of $\Ecomp{m}{\ms{X}}$ is
invertible, and the inverse is denoted by
$\partial^{\angles{m}{-ip^m}}$. Then we put
$\partial^{\angles{m}{-k}}:=\partial^{\angles{m}{l}}\cdot\partial
^{\angles{m}{-ip^m}}$. By using \cite[2.10]{Abe}, we get
\begin{equation*}
 \Gamma({U},\Ecomp{m}{\ms{X}})\cong\Bigl\{\sum_{k\in\mb{Z}}a_k
  \partial^{\angles{m}{k}}\Big\arrowvert a_k\in\mc{O}_{\ms{V}},\lim
  _{k\rightarrow+\infty}a_k=0 \Bigr\}.
\end{equation*}
Finally, by tensoring with $\mb{Q}$, we define $\EcompQ{m}{\ms{X}}$. One
of the most important properties of $\EcompQ{m}{\ms{X}}$ is that we get
an equality for any coherent $\DcompQ{m}{\ms{X}}$-module $\ms{M}$
(cf.\ \cite[2.13]{Abe})
\begin{equation*}
 \mr{Char}^{(m)}(\ms{M})=\mr{Supp}(\EcompQ{m}{\ms{X}}\otimes
  _{\pi^{-1}\DcompQ{m}{\ms{X}}}\pi^{-1}\ms{M}),
\end{equation*}
where $\mr{Char}^{(m)}$\index{characteristic variety and cycle!Char@$\mr{Char}^{(m)}$} denotes the characteristic variety of level
$m$ (cf.\ \cite[5.2.5]{BerInt}). The module
$\EcompQ{m}{\ms{X}}\otimes_{\pi^{-1}\DcompQ{m}{\ms{X}}}
\pi^{-1}\ms{M}$ is called the (naive) microlocalization of $\ms{M}$ of
level $m$. The ring $\EcompQ{m}{\ms{X}}$ is called the (naive)
microdifferential operators of level $m$.

\subsubsection{}\label{def_intermediate_microdiff}
In the last paragraph, we fixed the level $m$ to construct the ring of
microdifferential operators. However, to deal with
microlocalizations of $\DdagQ{\ms{X}}$-modules, we need to change levels
and see the asymptotic behavior. The problem is that there are no
reasonable transition homomorphism
$\EcompQ{m}{\ms{X}}\rightarrow\EcompQ{m'}{\ms{X}}$ for non-negative
integers $m'>m$. To remedy this, we
need to take an ``intersection''. Let
$\Emod{m}{\ms{X}}:=\bigcup_{n\in\mb{Z}}(\Ecomp{m}{\ms{X}})_n$, where
$(\Ecomp{m}{\ms{X}})_n$\index{microdifferential operators!naive $\Emod{m}{X_i}$, $\Ecomp{m}{\ms{X}}$, $\Emod{m}{\ms{X}}$, $(\Ecomp{m}{\ms{X}})_n$,\ldots|)} denotes the sub-$\pi^{-1}\mc{O}_{\ms{X}}$-module
of $\Ecomp{m}{\ms{X}}$ consisting of microdifferential operators of
order less than or equal to $n$, and put
$\Emod{m}{\ms{X},\mb{Q}}:=\Emod{m}{\ms{X}}\otimes\mb{Q}$.
%\index{microdifferential operators!naive!$\Emod{m}{\ms{X},\mb{Q}}$} 
Then
there exists a canonical homomorphism of
$\pi^{-1}\mc{O}_{\ms{X},\mb{Q}}$-algebras
$\psi_{m,m'}\colon\Emod{m'}{\ms{X},\mb{Q}}
\rightarrow\Emod{m}{\ms{X},\mb{Q}}$ sending $\partial^{\angles{m'}{1}}$
to $\partial^{\angles{m}{1}}$. We
define $\Emod{m,m'}{\ms{X}}:=\psi_{m,m'}^{-1}(\Emod{m}{\ms{X}})\cap
\Emod{m'}{\ms{X}}$.\index{microdifferential operators!intermediate!.1@$\Ecomp{m,m'}{\ms{X}}$, $(\EcompQ{m,m'}{\ms{X}})_k$, $\EcompQ{m,m'}{x}$,\ldots|(} By definition, $\Ecomp{m,m'}{\ms{X}}$ is the
$p$-adic completion of $\Emod{m,m'}{\ms{X}}$, and
$\EcompQ{m,m'}{\ms{X}}$ is $\Ecomp{m,m'}{\ms{X}}\otimes\mb{Q}$. When
$\ms{X}$ possesses a local coordinate $x$, we may write
\begin{equation*}
 \Gamma(\mathring{T}^*{X},\EcompQ{m,m'}{\ms{X}})\cong\Biggl
  \{\sum_{k\in\mb{Z}}a_k\partial^{k}\Bigg\arrowvert
  a_k\in\mc{O}_{\ms{X},\mb{Q}},\sum_{k<0}a_k\partial^{k}
  \in\Ecomp{m'}{\ms{X},\mb{Q}},~
  \sum_{k\geq0}a_k\partial^{k}\in\Ecomp{m}{\ms{X},\mb{Q}}\Biggr\}.
\end{equation*}
We note that the last condition
$\sum_{k\geq0}a_k\partial^{k}\in\Ecomp{m}{\ms{X},\mb{Q}}$ is equivalent
to $\sum_{k\geq0}a_k\partial^{k}\in\Dcomp{m}{\ms{X},\mb{Q}}$. 
Moreover, by construction of  $\Gamma(\mathring{T}^*{X},\EcompQ{m,m'}{\ms{X}})$, 
the expansion $\sum_{k\in\mb{Z}}a_k\partial^{k}$ satisfying the above conditions is unique, once the local coordinate $x$ is fixed (and so is $\partial$).
For an
integer $k$, we put $(\EcompQ{m,m'}{\ms{X}})_k:=\EcompQ{m,m'}{\ms{X}}
\cap(\EcompQ{m}{\ms{X}})_k$. 
We have canonical homomorphisms
$\Ecomp{m-1,m'}{\ms{X}}\rightarrow\Ecomp{m,m'}{\ms{X}}$ and
$\Ecomp{m,m'+1}{\ms{X}}\rightarrow\Ecomp{m,m'}{\ms{X}}$ (cf.\
\cite[4.6]{Abe}). We call these sheaves the intermediate rings of
microdifferential operators.

We also recall the following definitions, 
 $$\Emod{m,\dag}{\ms{X},\mb{Q}}
:=\varprojlim_{m'\rightarrow+\infty} \EcompQ{m,m'}{\ms{X}}
\quad\text{  and  }\quad \ms{E}^{\dag}_{\ms{X},\mb{Q}}:= \varinjlim_{m\rightarrow+\infty}\Emod{m,\dag}{\ms{X},\mb{Q}}, $$
\index{microdifferential operators!$\Emod{m,\dag}{\ms{X},\mb{Q}}$, $\ms{E}^{\dag}_{\ms{X},\mb{Q}}$}
where the transition maps are induced by the canonical homomorphisms above, 
cf. \cite[4.11]{Abe} for more details.

\subsubsection{}
\label{mainabe}
Now, let us explain the relation between the supports of the
microlocalizations of a $\DcompQ{m}{\ms{X}}$-module with respect to
intermediate rings and the characteristic variety. Let $\ms{X}$ be a
formal curve as in the last paragraph, and let $\ms{M}$ be a
coherent $\DcompQ{m}{\ms{X}}$-module. One might expect that, for
integers $m''\geq m'\geq m$,
\begin{equation}
 \label{charintemicro}
 \mr{Char}^{(m')}(\DcompQ{m'}{\ms{X}}\otimes_{\DcompQ{m}{\ms{X}}}\ms{M})=
  \mr{Supp}(\EcompQ{m',m''}{\ms{X}}\otimes_{\pi^{-1}\DcompQ{m}{\ms{X}}}
  \pi^{-1}\ms{M}).
\end{equation}
This does not hold in general as we can see by the counter-example
\cite[7.1]{Abe}. However, the statement holds for $m'$ large enough. The
following is one of main results of \cite{Abe}.
\begin{thm*}[{\cite[7.2]{Abe}}]
 There exists an integer $N$ such that
 {\normalfont(\ref{charintemicro})} holds for $m'\geq N$.
\end{thm*}

\subsection{Setup and preliminaries}

\subsubsection{}
\label{setup}
In this paragraph, we introduce some situations and notation.
In this paper, especially in the first two sections, we often consider
the following setting, which is called {\em Situation (L)}.
\begin{quote}
 Let $\ms{X}$ be an affine formal curve over
 $R$. Recall the convention \ref{defRmsX}, especially $X_i$ and $X$.
 Suppose that there exists a local coordinate $x$ in
 $\Gamma(\ms{X},\mc{O}_{\ms{X}})$ and fix it. We denote the
 corresponding differential operator by $\partial$.
\end{quote}
If moreover we assume the following, we say we are in {\em Situation
(Ls)}.
\begin{quote}
 Let $s$ be a closed point in $\ms{X}$, and we fix it. We suppose that
 there exists a local parameter at $s$ denoted by $y_s$ on $\ms{X}$.
\end{quote}
We use the following notation.

\begin{notation*}
 Let $\ms{X}$ be a formal curve over $R$, and $\ms{U}$ be an open affine
 formal subscheme of $\ms{X}$. We denote $\ms{U}\otimes R_i$ by $U_i$ as
 usual. Let $m'\geq m$ be non-negative integers.

\newcounter{enumipp}
 \begin{enumerate}
  \item We put
	$E^{\dag}_{\ms{U},\mb{Q}}:=\Gamma(\mathring{T}^*U,\ms{E}^{\dag}
	_{\ms{X},\mb{Q}})$, $\Emodb{m,\dag}{\ms{U},\mb{Q}}:=\Gamma(
	\mathring{T}^*U,\Emod{m,\dag}{\ms{X},\mb{Q}})$,
	$\EcompQb{m,m'}{\ms{U}}:=\Gamma(\mathring{T}^*U,\EcompQ{m,m'}
	{\ms{X}})$, $\Ecompb{m,m'}{\ms{U}}:=\Gamma(\mathring{T}^*U,
	\Ecomp{m,m'}{\ms{X}})$,	$(\Ecompb{m,m'}{\ms{U}})_k:=
	\Gamma(\mathring{T}^*U,(\Ecomp{m,m'}{\ms{X}})_k)$,
	$\Emodb{m,m'}{U_i}:=\Gamma(\mathring{T}^*U,
	\Emod{m,m'}{X_i})$.% 
\index{microdifferential operators!.2@$E^{\dag}_{\ms{U},\mb{Q}}$, $\Emodb{m,\dag}{\ms{U},\mb{Q}}$}%
\index{microdifferential operators!intermediate!.2@$\EcompQb{m,m'}{\ms{U}}$, $\Ecompb{m,m'}{\ms{U}}$, $(\Ecompb{m,m'}{\ms{U}})_k$,\ldots}%
  \item Let $E$ be one of $E^\dag_{\ms{U},\mb{Q}}$, $\Emodb{m,\dag}
	{\ms{U},\mb{Q}}$, $\EcompQb{m,m'}{\ms{U}}$. For a coherent
	$\DcompQ{m}{\ms{X}}$-module $\ms{M}$, we denote by
	$E\otimes_{\Dcomp{m}{}}\ms{M}$ or $E\otimes\ms{M}$ the $E$-module
	$E\otimes_{\Gamma(\ms{U},\DcompQ{m}{\ms{X}})}\Gamma(\ms{U},\ms{M})$.

  \item\label{not3} Let $x$ be a closed point in $\ms{X}$. Take a point $\xi_x$ in
	$\pi^{-1}(x)$ which is not in the zero
	section. Let $\ms{E}$ be one of sheaves of rings $\Ecomp{m,m'}{\ms{X}}$,
	$\EcompQ{m,m'}{\ms{X}}$, $\Emod{m,\dag}{\ms{X},\mb{Q}}$ or $\EdagQ{\ms{X}}$.  By the construction of $\ms{E}$, the fiber 
	$\ms{E}_{\xi_x}$ does not depend on the choice of $\xi_x$ and
	we denote it by $\ms{E}_{x}$.  
	Let a pair $(\ms{E}',\ms{E})$ be one of the four pairs
	$(\Ecomp{m,m'}{x},\Ecomp{m,m'}{\ms{X}})$,
	$(\EcompQ{m,m'}{x},\EcompQ{m,m'}{\ms{X}})$,
	$(\Emod{m,\dag}{x,\mb{Q}},\Emod{m,\dag}{\ms{X},\mb{Q}})$,
	$(\EdagQ{x},\EdagQ{\ms{X}})$. For a coherent
	$\Dcomp{m}{\ms{X}}$ or $\DcompQ{m}{\ms{X}}$-module $\ms{M}$ we
	write $\ms{E}'\otimes_{\Dcomp{m}{}}\ms{M}$ for
	$(\ms{E}\otimes_{\pi^{-1}\Dcomp{m}{\ms{X}}}\pi^{-1}\ms{M})
	_{\xi_x}$. This does not depend either on the choice of
	$\xi_x$. %by the construction of the ring $\ms{E}$. 
	We warn the reader that even if $\ms{E}$ is complete with respect to some topology we do not 
complete when we take the fiber at $\xi_x$.%
\index{microdifferential operators!intermediate!.1@$\Ecomp{m,m'}{\ms{X}}$, $(\EcompQ{m,m'}{\ms{X}})_k$, $\EcompQ{m,m'}{x}$,\ldots|)}\index{microdifferential operators!.3@$\Emod{m,\dag}{x,\mb{Q}}$, $\EdagQ{x}$}
  \setcounter{enumipp}{\theenumi}
 \end{enumerate}
 Moreover, assume that we are in Situation (L).

 \begin{enumerate}
  \setcounter{enumi}{\theenumipp}
  \item We denote by $\CR{\ms{X}}\{\partial\}^{(m,m')}$ the
  subring of $\Ecompb{m,m'}{\ms{X}}$ whose elements are ``horizontal with
  respect to $x$''. More precisely, we define
  \begin{equation*}
   \CR{\ms{X}}\{\partial\}^{(m,m')}:=\Bigl\{P=\sum_{n\in\mb{Z}}
    a_n\partial^{n}\in\Ecompb{m,m'}{\ms{X}}\Big\arrowvert
    \mbox{$P\partial^k=\partial^kP$ for any $k\geq 0$}\Bigr\}.
  \end{equation*} 
  We put $\CK{\ms{X}}\{\partial\}^{(m,m')}:=\CR{\ms{X}}\{\partial\}^{(m,m')}
  \otimes\mb{Q}$ and $\CR{X_i}\{\partial\}^{(m,m')}:=\CR{\ms{X}}\{\partial\}^{(m,m')}
  /\varpi^{i+1}$.
We note that by the hypothesis of Situation (L), the coordinate $x$
is fixed, and so is $\partial$. 
 \end{enumerate}
\index{microdifferential operators!intermediate!.3@$\CR{\ms{X}}\{\partial\}^{(m,m')}$, $\CK{\ms{X}}\{\partial\}^{(m,m')}$,\ldots} 
\end{notation*}

\begin{rem*}\label{remark_setup}
 (i) Let $\ms{E}$ be a sheaf of rings on a topological space. Then an
 $\ms{E}$-module $\ms{M}$ is said to be {\em globally finitely
 presented} if there exist integers $a,b\geq0$ and an exact
 sequence $\ms{E}^a\rightarrow\ms{E}^b\rightarrow\ms{M}\rightarrow0$ on
 the topological space. By \cite[5.3]{Abe}, when $\ms{X}$ is
 affine, there exists an equivalence of
 categories between the category of globally finitely presented
 $\Ecomp{m,m'}{\ms{X}}$-modules (resp.\ $\EcompQ{m,m'}{\ms{X}}$-modules)
 on $\mathring{T}^*X$ and that of
 $\Ecompb{m,m'}{\ms{X}}$-modules (resp.\
 $\EcompQb{m,m'}{\ms{X}}$-modules). We
 remind here that if there exists a coherent $\Dcomp{m}{\ms{X}}$-module
 $\ms{N}$ such that $\ms{M}\cong\Ecomp{m,m'}{\ms{X}}\otimes\ms{N}$, then
 $\ms{M}$ is globally finitely presented, and the same for
 $\EcompQ{m,m'}{\ms{X}}$-modules.
 
 (ii) Let  $\mathring{\pi}\colon\mathring{T}^*X\rightarrow X$ be the natural projection.
A sheaf $\ms{E}$ on $\mathring{T}^*X$ is called \emph{conic} %\index{conic sheaf} 
if the natural morphism 
$\mathring{\pi}^{-1}\mathring{\pi}_*\ms{E}\rightarrow \ms{E}$ is an isomorphism.
In particular, the support $\mr{Supp}(\ms{E})$ of such a sheaf is a conic subset %\index{conic subset} 
of $\mathring{T}^*X$:
\emph{i.e.} if $z\in \mr{Supp}(\ms{E})\subset \mathring{T}^*X$, then $\mathring{\pi}^{-1}(\mathring{\pi}(z))\subset \mr{Supp}(\ms{E})$.

By construction the sheaves $\ms{E}$  of microdifferentials are conic (where $\ms{E}$ stands for any of the sheaves 
$\EcompQ{m,m'}{\ms{X}}$, $\Emod{m,\dag}{\ms{X},\mb{Q}}$, \emph{etc.}, introduced before).
Since this is clearly a local property, it follows that any coherent $\ms{E}$-module is also conic.

 (iii)  Every element $Q$ in the ring $\CR{\ms{X}}\{\partial\}^{(m,m')}$ 
 can be written uniquely as $Q=\sum_{l<0} a_l \partial^{\angles{m'}{l}}
 + \sum_{k\geq 0} a_k \partial^{\angles{m}{k}}$, with $a_k$ in a finite
 \'{e}tale extension $R(\ms{X})$ of $R$ (cf.\
 \ref{K_Xnotdef}, \ref{fieldconstmicdif}), which coincides with $R$ if
 and only if $\ms{X}$ is geo\-metrically connected. The rings
 $K\{\partial\}^{(m,m')}$ and $\CR{X_i}\{\partial\}^{(m,m')}$ are
 subrings of $\EcompQb{m,m'}{\ms{X}}$ and $\Emodb{m,m'}{X_i}$
 respectively. We remind that they {\em do} depend
 on the choice of local coordinate $x$, and in particular, we are not
 able to globalize the construction.
\end{rem*}

\subsubsection{}
\label{topologyonring}
Let $\ms{X}$ be an affine formal curve over $R$
and let $E$ be one of the rings $\Ecompb{m,m'}{\ms{X}}$,
$\EcompQb{m,m'}{\ms{X}}$ or $\CK{\ms{X}}\{\partial\}^{(m,m')}$. We
finish this subsection by introducing some useful topologies on $E$ and
on finitely generated $E$-modules. Let's start with
$\Ecompb{m,m'}{\ms{X}}$.

For integers $k,l\geq 0$, we define a sub-$R$-module of
$\Ecompb{m,m'}{\ms{X}}$ by
\begin{equation*}
 U_{k,l}:=(\Ecompb{m,m'}{\ms{X}})_{-k}+\varpi^l\Ecompb{m,m'}{\ms{X}}.
\end{equation*}
We endow $\Ecompb{m,m'}{\ms{X}}$ with the topology denoted by $\ms{T}_0$
where the base of neighborhoods of zero is given by the system
$\{U_{k,l}\}_{k,l\geq0}$. With this topology, $\Ecompb{m,m'}{\ms{X}}$ is
a complete topological ring. Let us
see that it is complete with respect to this topology. Let $\{P_i\}$ be
a Cauchy sequence in $\Ecompb{m,m'}{\ms{X}}$. Then we may write
$P_i=Q_i+R_i$ such that: for any integers $k$ and $l$, there exists
an integer $N$ such that $Q_n-Q_N\in(\Ecompb{m,m'}{\ms{X}})_{-k}$
and $R_n-R_N\in\varpi^l\Ecompb{m,m'}{\ms{X}}$ for any $n\geq N$. By the
definition of the topology and the construction of the ring
$\Ecompb{m,m'}{\ms{X}}$, the limits $\lim_{i\rightarrow\infty}Q_i$ and
$\lim_{i\rightarrow\infty}R_i$ exist, and they
are denoted by $Q$ and $R$ respectively. Then we see that
$\lim_{i\rightarrow\infty}P_i=Q+R$ by definition.

Now, let us define topologies $\ms{T}$ and $\ms{T}_n$ for any integer
$n\geq0$ on $\EcompQb{m,m'}{\ms{X}}$. Let $n\geq0$ be an integer. For
integers $k,l\geq 0$, we can consider $\varpi^{-n}U_{k,l}$ as a
sub-$R$-module of $\EcompQb{m,m'}{\ms{X}}$, and we denote by $\ms{T}_n$
the topology on $\EcompQb{m,m'}{\ms{X}}$ generated by the open basis
$\{\varpi^{-n}U_{k,l}\}_{k,l\geq 0}$.
This topology makes it a locally convex topological space, and moreover
a Fr\'{e}chet space by \cite[Th\'{e}or\`{e}me 3.12]{Ti}. The identity
map $(\EcompQb{m,m'}{\ms{X}},\ms{T}_{n})\rightarrow(\EcompQb{m,m'}
{\ms{X}},\ms{T}_{n+1})$ is continuous by construction. By taking the
inductive limit (of locally convex spaces), we define a topology,
denoted by $\ms{T}$, on $\EcompQb{m,m'}{\ms{X}}$. It makes
$(\EcompQb{m,m'}{\ms{X}},\ms{T})$ an LF-space in the sense of
\cite[3.1]{Crew:ens}. The separateness can be seen from the fact that
the convex subset
$(\EcompQb{m,m'}{\ms{X}})_{-k}+\varpi^l\Ecompb{m,m'}{\ms{X}}$
in $\EcompQb{m,m'}{\ms{X}}$ is open in the $\ms{T}_n$-topology for any
$n$ and thus in the $\ms{T}$-topology.

Let $M$ be a finitely generate $\EcompQb{m,m'}{\ms{X}}$-module. We denote
the $(\EcompQb{m,m'}{\ms{X}},\ms{T}_n)$-module topology on $M$
(cf.\ \ref{nottop}) by $\ms{T}'_n$. Let us prove that the
topology $\ms{T}'_n$ is separated.
Let $\varphi\colon(\EcompQb{m,m'}{\ms{X}})^{\oplus a}\rightarrow M$ be a
surjective homomorphism, and put
$M':=\varphi((\Ecompb{m,m'}{\ms{X}})^{\oplus a})$. Consider the quotient
topology on $M'$ using the topology $\ms{T}_0$ on
$\Ecompb{m,m'}{\ms{X}}$. It suffices to show that $M'$ is
separated. Indeed, take $\alpha,\alpha'\in M$ such that $\alpha\neq
\alpha'$. There exists an integer $i\geq n$ such that
$\varpi^{i}\alpha,\varpi^{i}\alpha'\in M'$. If $M'$
is separated, there exists $U_{k,l}$, $U_{k',l'}$ such that
$(\varpi^{i}\alpha+\varphi(U_{k,l}))\cap(\varpi^{i}\alpha'+
\varphi(U_{k',l'}))=\emptyset$. Since
$\varpi^{-i}U_{k,l}\supset\varpi^{-n}U_{k,l}$, we get the claim. Let us
show that $M'$ is separated.
Since $\Ecompb{m,m'}{\ms{X}}$ is a noetherian complete $p$-adic ring
(cf.\ \cite[4.12]{Abe}), $M'$ is also $p$-adically complete, and in
particular, $p$-adically separated. Thus, it suffices to show that
$M'\otimes R_i$ is separated for any $i\geq0$ using the quotient
topology $\ms{Q}$ from $M'$. Consider the topology defined by
the filtration by order on $\Emodb{m,m'}{X_i}$. The topology
$\ms{Q}$ coincides with the quotient topology via
$(\Emodb{m,m'}{X_i})^{\oplus n}\rightarrow M'\otimes R_i$ induced by
$\varphi$. Since  $\Emodb{m,m'}{X_i}$ is a noetherian complete filtered
ring by \cite[4.8]{Abe}, we get that $M'\otimes R_i$ is separated,
and thus the topology $\ms{T}'_n$ on $M$ is separated.

This shows that, $\mr{Ker}(\varphi)$ is a
closed sub-$(\EcompQb{m,m'}{\ms{X}},\ms{T}_n)$-module. Thus
the topological vector space $(M,\ms{T}'_n)$ is a Fr\'{e}chet space. Of
course, the identity map $(M,\ms{T}'_n)\rightarrow(M,\ms{T}'_{n+1})$ is
continuous. We define the inductive limit topology (of locally convex
spaces) $\ms{T}'$ on $M$, which is called the {\em natural topology} on
$M$. If the natural
topology is separated, then $(M,\ms{T}')$ is an LF-space. When it is
separated, the open mapping theorem \cite[3.4]{Crew:ens} implies that
$(M,\ms{T}')$ coincides with the
$(\EcompQb{m,m'}{\ms{X}},\ms{T})$-module topology.

In the same way, we define topologies $\ms{S}$ and $\ms{S}_n$ on
$\CK{\ms{X}}\{\partial\}^{(m,m')}$ and on finitely generated
$\CK{\ms{X}}\{\partial\}^{(m,m')}$-modules when we are in Situation (L)
of \ref{setup}.

\begin{lem}
 \label{topology}
 Suppose we are in Situation {\normalfont(L)} of
 {\normalfont\ref{setup}}.
 Let $M$ be a finitely generated $\EcompQb{m,m'}{\ms{X}}$-module. We
 assume that it is also finite as
 $\CK{\ms{X}}\{\partial\}^{(m,m')}$-module. Then the natural topology as
 $\EcompQb{m,m'}{\ms{X}}$-module and the
 natural topology as $\CK{\ms{X}}\{\partial\}^{(m,m')}$-module are
 equivalent. In particular, if moreover $M$ is a free
 $\CK{\ms{X}}\{\partial\}^{(m,m')}$-module, then the topologies are
 separated, and $M$ becomes an LF-space.
\end{lem}
\begin{proof}
 Let us see the equivalence. Let
 $\phi\colon\CK{\ms{X}}\{\partial\}^{(m,m')\oplus a}\rightarrow M$ be a
 surjection. This surjection induces the surjection
 $(\EcompQb{m,m'}{\ms{X}})^{\oplus a}\rightarrow M$, and the quotient
 topology $(M,\ms{T}'_n)$ is defined. Let $(M,\ms{S}'_n)$ be the
 Fr\'{e}chet topology defined using the surjection $\phi$ and the
 $(\CK{\ms{X}}\{\partial\}^{(m,m')},\ms{S}_n)$-module structure, as done
 above in \ref{topologyonring} for $\ms{T}'_n$.
 Since $(M,\ms{T}'_n)$ is
 a topological $(\CK{\ms{X}}\{\partial\}^{(m,m')},\ms{S}_n)$-module by the
 definition, the homomorphism $\phi$ defines a continuous surjective
 homomorphism of topological modules
 $(\CK{\ms{X}}\{\partial\}^{(m,m')},\ms{S}_n)^{\oplus
 a}\rightarrow(M,\ms{T}'_n)$. By the open mapping theorem of
 Fr\'{e}chet spaces, we see that this homomorphism is strict, which
 implies that $\ms{T}'_n$ and $\ms{S}'_n$ are equivalent. The
 first claim follows by taking the inductive limit over $n$. When $M$ is
 free as a $\CK{\ms{X}}\{\partial\}^{(m,m')}$-module, then it is obvious
 that it is separated.
\end{proof}

\subsection{Relations between microlocalizations at different levels}
In this subsection, we investigate the behavior of
microlocalizations when we raise levels. In general, this is very
difficult. However, once we know that the supports of the
microlocalizations are stable (cf.\ \ref{stabledef}), the behavior is
very simple at least in the case of a curve.

\begin{lem}
 \label{lemmaexspope}
 Suppose we are in Situation {\normalfont(Ls)} of
 {\normalfont\ref{setup}}. Let $m'\geq m$ be
 non-negative integers, and $I$ be a left ideal
 of $\Ecompb{m,m'}{\ms{X}}$; we put $M:=\Ecompb{m,m'}{\ms{X}}/I$. Let
 $\ms{M}$ be the $\Ecomp{m,m'}{\ms{X}}$-module associated to $M$ (cf.\
 Remark {\normalfont\ref{setup} (i)}) on $\mathring{T}^*X$, and assume
 \begin{equation*}
  \mr{Supp}(\ms{M})  =\pi^{-1}(s)\cap\mathring{T}^*X.
 \end{equation*}
 Then for any integer $k$, there exists a positive integer $N$,
 $R\in\Ecompb{m,m'}{\ms{X}}$ and $S\in(\Ecompb{m,m'}{\ms{X}})_k$ such
 that
 \begin{equation*}
  y_s^N-\varpi R-S\in I.
 \end{equation*}
\end{lem}
\begin{proof}
 Since $\Ecompb{m,m'}{\ms{X}}$ is a noetherian ring (cf.\
 \cite[4.12]{Abe}), there exist $n$ operators
 $P_i\in\Ecompb{m,m'}{\ms{X}}$ for $i=1,\dots,n$ such that $I$ is
 generated by $\{P_i\}_{1\leq i\leq n}$. Let
 $\ms{U}:=\ms{X}\setminus\{s\}$, and $U$ be its special fiber. Then by
 the assumption on the support, there exists
 $Q_i\in\Ecompb{m,m'}{\ms{U}}$ for each $1\leq i\leq n$, such that
 \begin{equation}
  \label{sumequ1qp}
  \sum_{1\leq i\leq n}Q_i\cdot P_i=1.
 \end{equation}
 For $P\in\Ecompb{m,m'}{\ms{U}}$, we denote by $\overline{P}$ the image
 of $P$ in $\Emodb{m,m'}{U}$, and define
 $\varpi\mbox{-}\mr{ord}(P):=\mr{ord}(\overline{P})$. We put
 $\mu:=\max_i\{\varpi\mbox{-}\mr{ord}(P_i),0\}$.

 Since for any $f\in\mc{O}_{\ms{U}}$, there exists an integer $n$
 such that $\overline{y_s}^n\overline{f}\in\mc{O}_X$ where the
 overlines denote the images in $\mc{O}_U$ or $\mc{O}_X$, thus
 $y_s^nf\in\mc{O}_{\ms{X}}+\varpi\mc{O}_{\ms{U}}$.
 This shows that there exists an integer $N$ such that for any
 $i=1,\dots,n$, we can write
 \begin{equation*}
  y_s^NQ_i=Q'_i+\varpi R_i+S_i,
 \end{equation*}
 where $Q'_i\in\Ecompb{m,m'}{\ms{X}}, R_i\in\Ecompb{m,m'}{\ms{U}}$, and
 $S_i\in(\Ecompb{m,m'}{\ms{U}})_{k-\mu}$. Then by (\ref{sumequ1qp}),
 there exist $R'\in\Ecompb{m,m'}{\ms{U}}$ and
 $S'\in(\Ecompb{m,m'}{\ms{U}})_k$ such that
 \begin{equation*}
  y_s^N=\sum Q'_i\cdot P_i+\varpi R'+S'.
 \end{equation*}
 
 Let us show that for any integer $k'$
 \begin{equation}
  \label{relationincl}
  \bigl(\varpi\Ecompb{m,m'}{\ms{U}}+(\Ecompb{m,m'}{\ms{U}})_{k'}\bigr)
  \cap\Ecompb{m,m'}{\ms{X}}=\varpi\Ecompb{m,m'}{\ms{X}}+
  (\Ecompb{m,m'}{\ms{X}})_{k'}.
 \end{equation}
 It is evident that the right hand side is included in the left one, let
 us prove the opposite inclusion. Take
 elements $P\in\varpi\Ecompb{m,m'}{\ms{U}}$ and
 $Q\in(\Ecompb{m,m'}{\ms{U}})_{k'}$ such that
 $P+Q\in\Ecompb{m,m'}{\ms{X}}$. Write $P=\sum_{n\in\mb{Z}}a_n\partial^n$
 with $a_n\in\mc{O}_{\ms{U},\mb{Q}}$, and put
 $P_{>k'}:=\sum_{n>k'}a_n\partial^n$ and $P_{\leq k'}:=\sum_{n\leq
 k'}a_n\partial^n$.
 Then we get $P_{>k'}\in\Ecompb{m}{\ms{X}}$ and in particular
 contained in $\Ecompb{m}{\ms{X}}\cap \varpi\Ecompb{m}{\ms{U}}$,
which is $\varpi\Ecompb{m}{\ms{X}}$ thanks to uniqueness of the expansion 
$\sum_{n\in\mb{Z}}a_n\partial^n$ considered in \ref{def_intermediate_microdiff}.
 By assumption we have $P_{\leq k'}+Q\in(\Ecompb{m,m'}{\ms{X}})_{k'}$, which
 implies the equality (\ref{relationincl}).

 Since $y_s^N-\sum Q'_i\cdot P_i\in\Ecompb{m,m'}{\ms{X}}$,
 we get, by using (\ref{relationincl}), that
 \begin{equation*}
  \varpi R'+S'\in\varpi\Ecompb{m,m'}{\ms{X}}+(\Ecompb{m,m'}
   {\ms{X}})_k.
 \end{equation*}
 Thus the lemma follows.
\end{proof}

\begin{lem}
 \label{finiteness}
 Suppose we are in Situation {\normalfont(Ls)} of
 {\normalfont\ref{setup}}. Let $m'\geq m$ be
 non-negative integers, and $\ms{M}$ be a globally finitely presented
 $\EcompQ{m,m'}{\ms{X}}$-module such that
 \begin{equation*}
  \mr{Supp}(\ms{M})\cap\mathring{T}^*X=\pi^{-1}(s)\cap\mathring{T}^*X.
 \end{equation*}
 We denote $\Gamma(\mathring{T}^*X,\ms{M})$ by $M$. Then  we have the
 following.

 (i) The module $M$ is finite over
 $\CK{\ms{X}}\{\partial\}^{(m,m')}$. Moreover,
 if it is monogenic as an $\EcompQb{m,m'}{\ms{X}}$-module,
 there exists a $p$-torsion free $\Ecompb{m,m'}{\ms{X}}$-module
 $M'$, such that ${M}'\otimes\mb{Q}\cong{M}$, and
 $M'$ is finitely generated over $\CR{\ms{X}}\{\partial\}^{(m,m')}$.

 (ii) There exists a finite set of elements in $M$ such that
 for any open affine neighborhood $W$ of $s$ in $X$,
 $\Gamma(\mathring{T}^*W,\ms{M})$ is generated by these elements over
 $\CK{\ms{X}}\{\partial\}^{(m,m')}$. If $M$ is generated by $\alpha\in
 M$, then there exists an integer $N>0$ such that we can take this set
 to be $\bigl\{(x^iy^j_s)\,\alpha\bigr\}_{0\leq i,j<N}$.
\end{lem}
\begin{proof}
 Let $\ms{N}$ be a coherent sub-$\EcompQ{m,m'}{\ms{X}}$-module of
 $\ms{M}$. Let $\ms{N}'$ be either $\ms{N}$ or $\ms{M}/\ms{N}$. Then by
 the additivity of supports we know that
 \begin{equation*}
  \mr{Supp}(\ms{N}')\cap\mathring{T}^*X=\pi^{-1}(s)\cap\mathring{T}^*X\mbox{ or }
   \emptyset.
 \end{equation*}
 If it is $\emptyset$, then $\ms{N}'|_{\mathring{T}^*X}=0$, and
 in particular, $N':=\Gamma(\mathring{T}^*X,\ms{N}')=0$ is finite over
 $\CK{\ms{X}}\{\partial\}^{(m,m')}$. By induction on the number of
 generators of $M$, we reduce the verification of both (i) and (ii) to
 the monogenic case.

 From now on, we assume that $M$ is a monogenic module. Fix a
 generator $\alpha\in M$. Let $M'$ be the
 sub-$\Ecompb{m,m'}{\ms{X}}$-module of $M$ generated by $\alpha$. Let
 $I$ be the kernel of the homomorphism $\Ecompb{m,m'}{\ms{X}}\rightarrow
 M'$ of left $\Ecompb{m,m'}{\ms{X}}$-modules sending $1$ to $\alpha$.
 We note that, by definition, $M'\otimes\mb{Q}\cong M$. Since $M'$
 is $p$-torsion free, we get that
 $\mr{Supp}(\EcompQ{m,m'}{\ms{X}}\otimes
 M)=\mr{Supp}(\Ecomp{m,m'}{\ms{X}}\otimes M')$ where
 $\EcompQ{m,m'}{\ms{X}}\otimes M$ denotes the
 $\EcompQ{m,m'}{\ms{X}}$-module associated to $M$ (which is equal to
 $\ms{M}$ by \ref{setup}), and the same for $\Ecomp{m,m'}{\ms{X}}\otimes
 M'$. Thus by Lemma \ref{lemmaexspope} for $k=-1$, there exists a
 positive integer $N'$ and
 $T\in\varpi\Ecompb{m,m'}{\ms{X}}+(\Ecompb{m,m'}{\ms{X}})_{-1}$ such
 that $y_s^{N'}\equiv T\bmod I$.

 To conclude, it suffices to show that
 $M'':=\Ecompb{m,m'}{\ms{X}}/(y_x^{N'}-T)$ is
 generated over $\CR{\ms{X}}\{\partial\}^{(m,m')}$ by
 $\mf{S}:=\bigl\{x^i\,y^j_s\bigr\}_{0\leq i,j<N}$ where $N:=N'+\deg(s)$
 since there is a surjection $M''\twoheadrightarrow M'$. Since $M''$ and
 $\CR{\ms{X}}\{\partial\}^{(m,m')}$
 are $\varpi$-adically complete and $p$-torsion free, the conditions of
 \cite[III.2.11, Prop 14]{Bou} are fulfilled, and thus, it suffices to
 see that $M''/\varpi$ is generated by $\mf{S}$ over
 $\CR{\ms{X}}\{\partial\}^{(m,m')}/\varpi$. It remains 
 to show that $\Emodb{m,m'}{X}/(\bar{y_s}^{N'}-\bar{T})$
 is generated over $\CR{X}\{\partial\}^{(m,m')}$ by $\mf{S}$, where
 $\bar{T}\in(\Emodb{m,m'}{X})_{-1}$. Since
 $\Emodb{m,m'}{X}/(\bar{y_s}^{N'}-\bar{T})$ and
 $\CR{X}\{\partial\}^{(m,m')}$ are complete with respect to the
 filtrations by order, it is enough to
 prove the claim after taking $\mr{gr}$ by \cite[III.2.9, Prop
 12]{Bou}. Since the order of $\bar{T}$ is less than $0$, this amounts
 to prove that the commutative algebra
 $\mr{gr}(\Emodb{m,m'}{X})/(\bar{y_s}^{N'})$ is generated over
 $\CR{X_0}\{\partial\}^{(m,m')}$ by $\mf{S}$. This is a straightforward
 verification which is left to the reader.
\end{proof}

\begin{lem}
 \label{tesorcalc}
 We assume that we are in Situation {\normalfont(L)}.

 (i) Let $m'>m$. Then we have a canonical isomorphism
 \begin{equation*}
  \EcompQb{m+1,m'}{\ms{X}}\cong \CK{\ms{X}}\{\partial\}^{(m+1,m')}
   \widehat{\otimes}_{\CK{\ms{X}}\{\partial\}^{(m,m')}}
   \EcompQb{m,m'}{\ms{X}}
 \end{equation*}
 of bi-$(\CK{\ms{X}}\{\partial\}^{(m+1,m')},\EcompQb{m,m'}
 {\ms{X}})$-modules. Here the complete tensor product is taken with
 respect to the $p$-adic topology.
 
 (ii) Let $m'\geq m$. Then we have a canonical isomorphism
 \begin{equation*}
  \Ecompb{m,m'}{\ms{X}}\cong\CR{\ms{X}}\{\partial\}^{(m,m')}
   \widehat{\otimes}^f_{\CR{\ms{X}}\{\partial\}^{(m,m'+1)}}
   \Ecompb{m,m'+1}{\ms{X}}
 \end{equation*}
 of bi-$(\CR{\ms{X}}\{\partial\}^{(m,m')},\Ecompb{m,m'+1}{\ms{X}})$-modules. Here
 the complete tensor product $\widehat{\otimes}^f$ is taken with respect
 to the filtration by order.

 (iii) Let $m'\geq m$. For any $i\geq 0$, we have a canonical
 isomorphism
 \begin{equation*}
  \Emodb{m,m'}{X_i}\cong \CR{X_i}\{\partial\}^{(m,m')}\widehat{\otimes}^f
   _{\CR{X_i}\{\partial\}^{(m,m'+1)}}\Emodb{m,m'+1}{X_i}
 \end{equation*}
 of bi-$(\CR{X_i}\{\partial\}^{(m,m')},\Emodb{m,m'+1}{X_i})$-modules. Here
 the complete tensor product is taken with respect to the filtration by
 order.
\end{lem}
\begin{proof}
 Let us see (i). There exists a canonical homomorphism
 \begin{equation*}
  \varphi\colon\EcompQb{m,m'}{\ms{X}}\rightarrow\CK{\ms{X}}\{\partial\}
   ^{(m+1,m')}\widehat{\otimes}_{\CK{\ms{X}}\{\partial\}^{(m,m')}}
   \EcompQb{m,m'}{\ms{X}},
 \end{equation*}
 sending $P$ to $1\otimes P$. For short, we denote
 $\Gamma(\ms{X},\Emod{m+1,m'}{\ms{X}})$ by $E$, which is considered to
 be a subring of $\EcompQb{m,m'}{\ms{X}}$ using the canonical
 inclusion. We know that
 $\widehat{E}\otimes\mb{Q}\cong\EcompQb{m+1,m'}{\ms{X}}$ where
 $^\wedge$ denotes the $p$-adic completion. The image
 $\varphi({E})$ is contained in the image of \linebreak
 $\CR{\ms{X}}\{\partial\}^{(m+1,m')}\widehat{\otimes}_{\CR{\ms{X}}
 \{\partial\}^{(m,m')}}\Ecompb{m,m'}{\ms{X}}$. Indeed, let $P\in E$. We
 denote $\Gamma(\ms{X},\Dmod{m+1}{\ms{X}})$ by
 $\Dmodb{m+1}{\ms{X}}$. Then we
 may write $P=P_{\geq0}+P_{<0}$ where $P_{\geq0}\in\Dmodb{m+1}{\ms{X}}$
 and $P_{<0}\in E_{-1}\subset\Ecompb{m,m'}{\ms{X}}$. For
 $P_{\geq0}\in\Dmodb{m+1}{\ms{X}}$, we can write
 $P_{\geq0}=\sum_{i\geq0}\partial^{\angles{m+1}{i}}a_i$ where
 $a_i\in\mc{O}_{\ms{X}}$. Since this is a finite sum, $1\otimes
 P_{\geq0}$ is the image of
 $\sum_{i\geq0}\partial^{\angles{m+1}{i}}\otimes a_i$,
 and the claim follows.

 This implies that the homomorphism $\varphi$ induces the
 canonical homomorphism
 \begin{equation*}
  \widehat{\varphi}\colon\EcompQb{m+1,m'}{\ms{X}}\rightarrow
   \CK{\ms{X}}\{\partial\}^{(m+1,m')}\widehat{\otimes}
   _{\CK{\ms{X}}\{\partial\}^{(m,m')}}\EcompQb{m,m'}{\ms{X}}.
 \end{equation*}
 On the other hand, we have the canonical homomorphism
 \begin{equation*}
  \widehat{\psi}\colon\CK{\ms{X}}\{\partial\}^{(m+1,m')}\widehat{\otimes}
   _{\CK{\ms{X}}\{\partial\}^{(m,m')}}\EcompQb{m,m'}{\ms{X}}\rightarrow
   \EcompQb{m+1,m'}{\ms{X}}.
 \end{equation*}
 To conclude the proof, it suffices to see that
 $\widehat{\psi}\circ\widehat{\varphi}=\mr{id}$, $\widehat{\varphi}
 \circ\widehat{\psi}=\mr{id}$. Since $\psi$ and $\varphi$ are
 continuous, to see the former equality, it suffices to show the
 identity on $E$, which is obvious. Let $A:=\CK{\ms{X}}\{\partial\}
 ^{(m+1,m')}\cap E\subset\Ecompb{m+1,m'}{\ms{X}}$. To see the latter
 equality, it suffices to see after restricting to
 $\mr{Im}(A\otimes_{\mb{Z}}\Ecompb{m,m'}{\ms{X}}\rightarrow
 \CK{\ms{X}}\{\partial\}^{(m+1,m')}\widehat{\otimes}
 \EcompQb{m,m'}{\ms{X}})$. Since this is straightforward, we leave the
 argument to the reader.

 Now let us prove (ii). We have a canonical homomorphism
 \begin{equation*}
  \widehat{\psi}\colon\CR{\ms{X}}\{\partial\}^{(m,m')}\widehat{\otimes}
   ^f_{\CR{\ms{X}}\{\partial\}^{(m,m'+1)}}\Ecompb{m,m'+1}{\ms{X}}
   \rightarrow\Ecompb{m,m'}{\ms{X}}.
 \end{equation*}
 On the other hand, we also have the homomorphism
 \begin{equation*}
  \iota\colon\Ecompb{m,m'+1}{\ms{X}}\rightarrow \CR{\ms{X}}\{\partial\}
   ^{(m,m')}{\otimes}_{\CR{\ms{X}}\{\partial\}^{(m,m'+1)}}
   \Ecompb{m,m'+1}{\ms{X}}.
 \end{equation*}
 sending $P$ to $1\otimes P$.
 Since $\widehat{\psi}\circ\iota$ is the canonical inclusion, we get
 that $\iota$ is injective.
 We define $B:=\CR{\ms{X}}\{\partial\}^{(m,m')}
 {\otimes}_{\CR{\ms{X}}\{\partial\}^{(m,m'+1)}}\Ecompb{m,m'+1}{\ms{X}}$,
which is filtered by the tensor product filtration \cite[p.57]{HO}, denoted by $B_n$. Let
 $n<0$, and take $S:=\sum_{i}P_i\otimes Q_i$ in $B_n$. Then there exists
 $f\in\mc{O}_{\ms{X}}$ such that
 \begin{equation*}
  \sum_{i}P_i\otimes Q_i\equiv\partial^{\angles{m'}{n}}\otimes f
   \bmod B_{n-1}.
 \end{equation*}
 Suppose $S\not\in B_{n-1}$. Then $f\neq 0$. There exists an integer $N$
 such that $p^N\partial^{\angles{m'}{n}}\in
 \CR{\ms{X}}\{\partial\}^{(m,m'+1)}$. Thus for $N'\geq N$, we get $p^{N'}
 S\equiv1\otimes(p^{N'}\partial^{\angles{m'}{n}})\cdot
 f\bmod B_{n-1}$. If $1\otimes(p^{N'}\partial^{\angles{m'}{n}})\cdot
 f\in B_{n-1}$, we would get $\widehat{\psi}\circ\iota((p^{N'}
 \partial^{\angles{m'}{n}})\cdot f)\in(\Ecompb{m,m'}{\ms{X}})_{n-1}$,
 which is impossible. Thus, we get
 $1\otimes(p^{N'}\partial^{\angles{m'}{n}})\cdot
 f\not\in B_{n-1}$. This shows that $p^{N'}S\not\in B_{n-1}$ for
 any large enough $N'$. Thus,
 $\mr{gr}(\CR{\ms{X}}\{\partial\}^{(m,m')}{\otimes}_{\CR{\ms{X}}
 \{\partial\}^{(m,m'+1)}}\Ecomp{m,m'+1}{\ms{X}})$ is $p$-torsion
 free. In particular, the canonical homomorphism
 \begin{equation*}
  i\colon\CR{\ms{X}}\{\partial\}^{(m,m')}\widehat{\otimes}^f
   _{\CR{\ms{X}}\{\partial\}^{(m,m'+1)}}\Ecomp{m,m'+1}{\ms{X}}
   \rightarrow(\CR{\ms{X}}\{\partial\}^{(m,m')}\widehat{\otimes}
   ^f_{\CR{\ms{X}}\{\partial\}^{(m,m'+1)}}\Ecomp{m,m'+1}{\ms{X}})
   \otimes\mb{Q}
 \end{equation*}
 is injective.

 Now, let $E^{[m,m']}_{\ms{X}}:=\rho^{-1}_{m,m'}(\Ecompb{m,m'}{\ms{X}})$
 where
 $\rho_{m,m'}\colon\EcompQb{m,m'+1}{\ms{X}}\rightarrow\EcompQb{m,m'}
 {\ms{X}}$ is the canonical inclusion (cf.\ \cite[5.5]{Abe}). Consider the following diagram.
 \begin{equation*}
  \xymatrix@C=60pt{E^{[m,m']}_{\ms{X}}\ar[r]^<>(.5)j\ar@{.>}[rd]&
   (\CR{\ms{X}}\{\partial\}^{(m,m')}\widehat{\otimes}^f
   _{\CR{\ms{X}}\{\partial\}
   ^{(m,m'+1)}}\Ecompb{m,m'+1}{\ms{X}})\otimes\mb{Q}\\&
   \CR{\ms{X}}\{\partial\}^{(m,m')}\widehat{\otimes}^f
   _{\CR{\ms{X}}\{\partial\}
   ^{(m,m'+1)}}\Ecompb{m,m'+1}{\ms{X}}\ar[u]_i.}
 \end{equation*}
 Let us construct the dotted arrow making the diagram commutative.
 It suffices to see that $\mr{Im}(j)\subset\mr{Im}(i)$. Let
 $P:=\sum_{k\in\mb{Z}}\partial^ka_k\in E^{[m,m']}_{\ms{X}}$ where
 $a_k\in\mc{O}_{\ms{X},\mb{Q}}$. Since there exists an
 integer $N$ such that $j(p^N\cdot P)\in\mr{Im}(i)$, it suffices to show
 that $j(\partial^{k}a_k)\in\mr{Im}(i)$ for any integer $k$, which is
 easy. Thus $j$ induces
 the canonical homomorphism
 \begin{equation*}
  E^{[m,m']}_{\ms{X}}\rightarrow \CR{\ms{X}}\{\partial\}^{(m,m')}\widehat
   {\otimes}^f_{\CR{\ms{X}}\{\partial\}^{(m,m'+1)}}\Ecompb{m,m'+1}{\ms{X}}
 \end{equation*}
 of filtered rings since $i$ is injective. By taking the completion with
 respect to the filtration by order, we get the canonical homomorphism
 \begin{equation*}
  \widehat{\varphi}\colon\Ecompb{m,m'}{\ms{X}}\rightarrow
   \CR{\ms{X}}\{\partial\}^{(m,m')}\widehat{\otimes}^f
   _{\CR{\ms{X}}\{\partial\}^{(m,m'+1)}}\Ecompb{m,m'+1}{\ms{X}}.
 \end{equation*}
 We see easily that $\widehat{\varphi}\circ\widehat{\psi}=\mr{id}$,
 $\widehat{\psi}\circ\widehat{\varphi}=\mr{id}$ as in the proof of (i),
 which concludes the proof.

 Let us see (iii).
 The above argument shows that $B/B_n$ is
 $p$-torsion free for any $n\in\mb{Z}$.
 Since $B/B_n$ is $p$-torsion free and the inverse system $\{B/B_n\}_n$
 satisfies the Mittag-Leffler condition, we get
 $(\invlim_{n}B/B_n)\otimes R_i\cong\invlim_n(B/B_n\otimes R_i)$ for any
 $i$. Let $\widehat{B}$ be the completion with respect to the filtration
 by order. We get
 \begin{align*}
  &\widehat{B}\otimes R_i\cong(\invlim_nB/B_n)
  \otimes R_i\cong\invlim_n(B/B_n\otimes R_i)\\
  &\qquad\cong\invlim_n((B\otimes R_i)/\mr{Im}(B_n))\cong
  \CR{X_i}\{\partial\}^{(m,m')}\widehat{\otimes}^f_{\CR{X_i}
  \{\partial\}^{(m,m'+1)}}\Emodb{m,m'+1}{X_i}.
 \end{align*}
 By using (ii), the claim follows.
\end{proof}

\begin{lem}
 \label{flatness}
 Let $\ms{X}$ be a smooth formal curve over $R$, and
 $\pi\colon T^*X\rightarrow X$ as usual. Then the algebra
 $\EcompQ{m,m'}{\ms{X}}$ is flat over $\pi^{-1}\DcompQ{m}{\ms{X}}$.
\end{lem}
\begin{proof}
 By \cite[Corollary 2.9]{Abe} and \cite[3.5.3]{Ber1}, we know that
 $\EcompQ{m'}{\ms{X}}$ is flat over $\pi^{-1}\DcompQ{m}{\ms{X}}$. It
 suffices to show that
 $\mr{Tor}_1^{\pi^{-1}\DcompQ{m}{\ms{X}}}(\EcompQ{m,m'}{\ms{X}}
 ,\bullet)=0$. This
 amounts to prove
 \begin{equation*}
  \mr{Tor}_2^{\pi^{-1}\DcompQ{m}{\ms{X}}}(\EcompQ{m'}{\ms{X}}/
   \EcompQ{m,m'}{\ms{X}},\bullet)=0
 \end{equation*} 
 by the flatness of $\EcompQ{m'}{\ms{X}}$. However, by \cite[5.11 and
 7.8]{Abe}, this is equivalent to showing that
 \begin{equation*}
  \mr{Tor}_2^{\pi^{-1}\DcompQ{m}{\ms{X}}}(\pi^{-1}\DcompQ{m'}{\ms{X}}
   /\DcompQ{m}{\ms{X}},\bullet)=0,
 \end{equation*}
 which follows from the flatness of $\DcompQ{m'}{\ms{X}}$ over
 $\DcompQ{m}{\ms{X}}$.
\end{proof}

\begin{rem*}
 We do not know if $\EcompQ{m,m'}{\ms{X}}$ is flat over
 $\DcompQ{m}{\ms{X}}$ when the dimension of $\ms{X}$ is greater than
 $1$ and $m'>m$.
\end{rem*}

\subsubsection{}
\label{K_Xnotdef}
Let $\ms{X}$ be a connected smooth formal scheme over $R$. Let $\eta$ be
the generic point of $\ms{X}$, and denote by $R(\ms{X})$ the integral
closure of $R$ in the field $\mc{O}_{\ms{X},\eta}$. The ring $R(\ms{X})$
is a discrete valuation ring as well since it is finite over $R$ and
connected. Thus by \cite[IV, 17.7.7]{EGA}, $R(\ms{X})$ is \'{e}tale over
$R$. Moreover $R(\ms{X})/\varpi$ is the separable closure
of $k$ in $\mc{O}_{X,\eta}$, and $\ms{X}$ is geometrically connected
over $R(\ms{X})$ by \cite[II, 4.5.15]{EGA}. Put
$K(\ms{X}):=R(\ms{X})\otimes\mb{Q}$. This is an unramified field
extension of $K$. The field $K(\ms{X})$ is called the field of local
constants of $\ms{X}$. The reason for this naming comes from the next
lemma.

\begin{lem*}
 Suppose $\ms{X}$ is affine and possesses a system of local coordinates
 $\{x_1,\dots,x_d\}$. We denote the corresponding differential operators
 by $\{\partial_1,\dots,\partial_d\}$. Then we get
 \begin{equation*}
  K(\ms{X})=\bigl\{f\in\mc{O}_{\ms{X},\mb{Q}}\mid \partial_i(f)=0
   \mbox{ for $1\leq i\leq d$}
   \bigr\}. \index{.@miscellaneous!KX@$K(\ms{X})$}
 \end{equation*}
\end{lem*}
\begin{proof}
 For this, we may assume that $R=R(\ms{X})$. There exists a finite
 \'{e}tale extension $R'$ of $R$ such that $R'$ is a discrete valuation
 ring and $\ms{X}':=\ms{X}\otimes_RR'$
 has a $R'$-rational point. Note that $\ms{X}'$ is also connected and
 $K(\ms{X}')=R'\otimes\mb{Q}$. By Galois descent, it suffices to show
 the lemma for $\ms{X}'$. In this case the lemma follows from \cite[IV,
 17.5.3]{EGA}.
\end{proof}

\subsubsection{}
\label{fieldconstmicdif}
Let $L$ be a finite field extension of $K$. Let $I$ be a connected
interval in $\mb{R}_{\geq0}$. We denote by $\An[x,L](I)$ the ring of
analytic functions
$\bigl\{\sum_{i\in\mb{Z}}\alpha_ix^i\mid \alpha_i\in L,
\mbox{$\lim_{i\rightarrow\pm\infty}|\alpha_i|\rho^i=0$ for any $\rho\in
I$}\bigr\}$. For short, we often denote this by
$\mc{A}(I)$. For real numbers $0<a\leq b<\infty$, we put
\begin{equation*}
 \mc{A}_{x,L}(\{a,b]):=\left\{\sum_{i\in\mb{Z}}\alpha_ix^i
			\Bigg\arrowvert \alpha_i\in L,~\sup_{i\in\mb{Z}}
			\bigl\{|\alpha_i|a^i\bigr\}
			<\infty,~\text{$\lim_{i\rightarrow+\infty}
			|\alpha_i|b^i=0$}\right\}.
\end{equation*}
Similarly, we define $\mc{A}_{x,L}([a,b\})$. \index{rings of analytical functions! $\An[x,L](I)$, $\mc{A}(I)$, $\mc{A}_{x,L}(\{a,b])$, $\mc{A}_{x,L}([a,b\})$}
For a series $f(x)=\sum_{i\in\mb{Z}}\alpha_ix^i$ in $\mc{A}_{x,L}(\{a,b])$ and a real number $c\in [a,b]$, we put $|f(x)|_c:= \sup_{i\in\mb{Z}}
			\bigl\{|\alpha_i|c^i\bigr\}$; clearly we have 
$|f(x)|_c\in\mb{R}$.

Let $\omega_m:=p^{-1/p^m(p-1)}<1$, and $\omega:=\omega_0$. Note that
if $m'\geq m$, we get $\omega<\omega/\omega_{m'}\leq\omega/
\omega_m\leq 1$.
Then by  the definition of $K\{\partial\}^{(m,m')}$ we get the following
explicit description.
\begin{lem*}
 Suppose we are in Situation {\normalfont(L)} of
 {\normalfont\ref{setup}}. Let $L$ be the
 field of constant of $\ms{X}$. Then for any non-negative integers
 $m'\geq m$, we have an isomorphism
 \begin{equation*}
  \CK{\ms{X}}\{\partial\}^{(m,m')}\xrightarrow{\sim}\mc{A}_{x,L}
   (\{\omega/\omega_{m'},\omega/\omega_m])
 \end{equation*}
 sending $\partial$ to $x$.
\end{lem*}

\begin{lem}
 \label{Kringpid}
 Suppose we are in Situation {\normalfont(L)}. For any non-negative
 integers $m'\geq m$, the commutative ring
 $\CK{\ms{X}}\{\partial\}^{(m,m')}$ is a principal ideal
 domain. Moreover, $\CK{\ms{X}}\{\partial\}^{(m)}$ is a field.
\end{lem}

%%%% old proof
%\begin{proof}
% Let us show that $\CK{\ms{X}}\{\partial\}^{(m,m')}$ is a PID.
% We use the notation of paragraph \ref{fieldconstmicdif}.
% Let $L$ be the field of constants of $\ms{X}$.
% It suffices to show that $\mc{A}_L(\{a,b])$ is a principal ideal domain
% for any $0<a\leq b<\infty$. Let $f\in\mc{A}(\{a,b])$. Since $L$ is
% a discrete valuation ring, there are only finitely many critical points
% of the function
% $\left]\log(a),\log(b)\right]\rightarrow\mb{R};\rho\mapsto
% \log(|f(x)|_{\exp(\rho)})$. Thus, there exists a polynomial
% $p\in L[x]$ and $g\in\mc{A}(\left]a,b\right])^*$ such that
% $f=pg$ (cf.\ \cite[A.4]{Dw}). Since there are no zero points of $g$ on
% $\left]a,b\right]$, there are no critical points of the function
% $\left]\log(a),\log(b)\right]\rightarrow\mb{R};\rho\mapsto\log
% (|g(x)|_{\exp(\rho)})$, which shows that
% $g\in\mc{A}(\left\{a,b\right])$. In the same way, we get that
% $g^{-1}\in\mc{A}(\{a,b])$, and thus, $g\in\mc{A}(\{a,b])^*$.
% This shows that any ideal is generated by polynomials. Since $L[x]$ is
% a principal ideal domain, the ideal of $L[x]$ generated by the
% polynomials is also principal, and thus
% $\CK{\ms{X}}\{\partial\}^{(m,m')}$ is a PID.
%
% To see the latter claim, it suffices to show that
% $\mc{A}(\left\{a,a\right])$ is a field for any $a>0$. The
% verification is left to the reader.
%\end{proof}
\begin{proof}
 Let us show that $\CK{\ms{X}}\{\partial\}^{(m,m')}$ is a PID.
 We use the notation of \S\ref{fieldconstmicdif}.  
 Let $L$ be the field of constants of $\ms{X}$.
 It suffices to show that $\mc{A}_L(\{a,b])$ is a PID
 for any $a\leq b$ in $|\overline{K}^{\times}|$. 
  Let $f:=\sum_{n\in\mathbb{Z}} \alpha_n x^n\in\mc{A}_L(\{a,b])$ and 
 set $I:=\left]\ln(a),\ln(b)\right]$, where $\ln$ is
 the usual (real) natural logarithm. %, note that $I$ is bounded.
The Newton function $\mathrm{New}(f)\colon I \rightarrow \mathbb{R}$ is defined by $\rho\mapsto \ln(|f(x)|_{\exp(\rho)})$. 
By construction we have $\mathrm{New}(f)(\rho)=
\sup_{n\in \ZZ} \big\{\ln|\alpha_n| +n\rho\big\}$. 
 For every compact interval $J\subset I$, the restriction of  $\mathrm{New}(f)$ to $J$ is continuous, convex and piecewise linear with slopes in $\mathbb{Z}$, cf. \cite[A.3]{Dw};
therefore the function $\mathrm{New}(f)$ enjoys the same properties on $I$, with, \emph{a priori}, an infinite number of slopes.
We denote by $\mathrm{Z}(f)$ the set  of the critical points of $\mathrm{New}(f)$ in the interval $I$,
\emph{i.e.}\ the set of points in $I$ where the slope changes.
 We claim that $\mathrm{Z}(f)$ is finite.
 Since $a,b\in|\overline{K}|$ and this claim is stable under 
finite extensions of $L$, we may assume $a,b\in|L|$. Take $\gamma\in L$ such
that $|\gamma|=a$. %There exists an isomorphism
%$\mc{A}_L(\left\{a,b\right])\xrightarrow{\sim}\mc{A}_L
%(\left\{1,b/a\right])$ sending $f(x)$ to $f(\gamma x)$, so 
By a coordinate change from $x$ to $\gamma x$, we may assume $a=1$.
The properties of   $\mathrm{New}(f)$ recalled above imply that $\mathrm{Z}(f)$ is discrete in $I$, so it is 
enough to prove that it is finite in a right neighborhood of $\ln(a)=0$.  
The valuation of $L$ is discrete and 
$|f|_1:=\sup_{n\in\ZZ}{|\alpha_n|}<\infty$ by hypothesis, so there exists an
integer $n_0$ such that $|\alpha_{n_0}|=|f|_1$. On the other hand, take
$c\in\left]a,b\right[$. There exists $\epsilon>0$ such that
$\mr{New}(f)(\rho)$ is affine on $[\ln(c),\ln(c)+\epsilon]$ of slope
$N$. By convexity of $\mr{New}(f)$, we have $n_0\leq N$. Put
$h:=\sum\limits_{n=n_0}^{N}\alpha_nx^n \in L[x]$. By construction, we have
$\mr{New}(f)(x)=\mr{New}(h)(x)$ for
$x\in\left]0,\ln(c)+\epsilon\right]$, and the claim follows.

Therefore, there exists a polynomial $\tilde{h}\in L[x]$ and $g\in\mc{A}(\left]a,b\right])^*$  such that
 $f=\tilde{h}g$. Since there are no zero points of $g$ on
 $\left]a,b\right]$, there are no critical points of the function $\mathrm{New}(g)$ on $I$, which implies
 $g\in\mc{A}(\left\{a,b\right])$. In the same way, we get 
 $g^{-1}\in\mc{A}(\{a,b])$, and thus $g\in\mc{A}(\{a,b])^*$.
 This shows that any ideal is generated by polynomials. Since $L[x]$ is
 a principal ideal domain, the ideal of $L[x]$ generated by the
 polynomials is also principal, hence
 $\CK{\ms{X}}\{\partial\}^{(m,m')}$ is a PID.

 To see the latter statement of the lemma, it suffices to show that
 $\mc{A}_L(\left\{a,a\right])$ is a field for any $a\in
 |\overline{K}^{\times}|$. The
 verification is left to the reader.
\end{proof}

By this lemma, we get that any finitely generated
$\CK{\ms{X}}\{\partial\}^{(m,m')}$-module with a connection is a free
$\CK{\ms{X}}\{\partial\}^{(m,m')}$-module by \cite[Corollaire 4.3]{Ch}.

\subsubsection{}
\label{stabledef}
Let $\ms{X}$ be a formal curve, and $\ms{M}$ be a coherent
$\DcompQ{m}{\ms{X}}$-module. We say that
$\ms{M}$ is {\em holonomic} \index{holonomic!$\DcompQ{m}{\ms{X}}$-module} if the dimension of
$\mr{Char}^{(m)}(\ms{M})$ is $1$ or if $\ms{M}=0$. For an integer $m'\geq m$, let
$\ms{M}^{(m')}:=\DcompQ{m'}{\ms{X}}\otimes\ms{M}$. We
say $\ms{M}$ is {\em stable} \index{stable $\DcompQ{m}{\ms{X}}$-module} if for any $m''\geq m'\geq m$, we have
\begin{equation*}
 \mr{Supp}(\EcompQ{m',m''}{\ms{X}}\otimes_{\pi^{-1}\DcompQ{m}
  {\ms{X}}}\pi^{-1}\ms{M})=\mr{Char}^{(m)}(\ms{M}).
\end{equation*}
In particular, we have
$\mr{Char}^{(m')}(\ms{M}^{(m')})=\mr{Char}^{(m)}(\ms{M})$.
By Theorem \ref{mainabe}, any coherent $\DcompQ{m}{\ms{X}}$-module
is stable after raising the level sufficiently.
We say that a point $s\in X$ is a {\em singular point of $\ms{M}$} if
$\pi^{-1}(s)\subset\mr{Char}^{(m)}(\ms{M})$. \index{singular point}
From now on, to avoid too heavy notation, we sometimes denote
$\mr{Char}^{(m)}$ by just $\mr{Char}$.\index{characteristic variety and cycle!Chara@$\mr{Char}$} For a coherent
$\DdagQ{\ms{X}}$-module $\ms{M}$, there exists a stable
coherent $\DcompQ{m}{\ms{X}}$-module $\ms{M}'$ for some $m$
such that $\DdagQ{\ms{X}}\otimes\ms{M}'\cong\ms{M}$. We define
$\mr{Char}(\ms{M}):=\mr{Char}^{(m)}(\ms{M}')$, and we say that {\it
$\ms{M}$ is holonomic} if $\ms{M}'$ is. We define the set of
singularities of $\ms{M}$ as that of $\ms{M}'$. Note that when $\ms{M}$
is a coherent $F$-$\DdagQ{\ms{X}}$-module (cf.\ \ref{setupFrob}), the
definition of holonomicity is equivalent to that of Berthelot as written
in \cite[7.5]{Abe}. For the later use, we remind here the
following lemma.

\begin{lem}
 \label{generators}
 Suppose we are in Situation {\normalfont(L)} of
 {\normalfont\ref{setup}}.
 Let $\ms{M}$ be a monogenic stable holonomic $\DcompQ{m}{\ms{X}}$-module,
 and $\alpha\in\Gamma(\ms{X},\ms{M})$ be a generator. Let $S$ be the set of
 singular points of $\ms{M}$. 
 Take $s\in S$, and let $y_s$ be a local parameter of
 $\mc{O}_{\ms{X},s}$. Then for any integers $m''\geq m'\geq m$, there
 exists an integer $N$ such that $\{x^iy^j_{s}\,\alpha\}_{0\leq i,j<N}$
 generate $(\EcompQ{m',m''}{\ms{X}}\otimes\ms{M})_{s}$ over
 $\CK{\ms{X}}\{\partial\}^{(m',m'')}$.
\end{lem}
\begin{proof}
 We may shrink $\ms{X}$ so that we are in Situation (Ls) and
 $S=\{s\}$. Then this is just a direct consequence of Lemma
 \ref{finiteness} (ii).
\end{proof}

\subsubsection{}
\label{localsetting}
In this paragraph, we will consider Situation (L) of \ref{setup} and we will follow the notation fixed there. Let
$\ms{M}$ be a stable holonomic $\DcompQ{m}{\ms{X}}$-module and $s$ be a
closed point of $\ms{X}$ such that
$\mr{Char}(\ms{M})\supset\pi^{-1}(s)$. We consider the $\EcompQ{m,m'}{s}$-module
$\EcompQ{m,m'}{s}\otimes_{\Dcomp{m}{}}\ms{M}$, cf.\ Notation \ref{not3} in \ref{setup}. It can be seen as a
$\CK{\ms{X}}\{\partial\}^{(m,m')}$-module. When we are especially
interested in this $\CK{\ms{X}}\{\partial\}^{(m,m')}$-module structure,
we denote this module by $\EcompQ{m,m'}{s}(\ms{M})$.
{\em We caution here that this definition is only for this section, and
from {\normalfont\eqref{changing_not}}, we use the same notation
for a slightly different object.}\index{microlocalisation of $\ms{M}$!$\EcompQ{m,m'}{s}(\ms{M})$} In the same way, for a
$\Dcomp{m}{\ms{X}}$-module $\ms{M}'$, we put
$\Ecomp{m,m'}{s}(\ms{M}'):=\Ecomp{m,m'}{s}\otimes_{\Dcomp{m}{}}\ms{M}'$
and the same for $\Dmod{m}{X}$-modules etc.

By the condition on the characteristic variety and  Lemma \ref{generators}, we get that
$\EcompQ{m,m'}{s}(\ms{M})$ is finitely gene\-rated over
$\CK{\ms{X}}\{\partial\}^{(m,m')}$.  Let
$L$ be the field of constants of $\ms{X}$. We have the isomorphism of
Lemma \ref{fieldconstmicdif}
\begin{equation*}
 \mc{A}_{x',L}^{(m,m')}:=\mc{A}_{x',L}(\{\omega/\omega_{m'},\omega/
  \omega_m])\xrightarrow{\sim}\CK{\ms{X}}\{\partial\}^{(m,m')}
\end{equation*}
sending $x'$ to $\partial$.
We consider $\EcompQ{m,m'}{s}(\ms{M})$ as a finitely generated
$\mc{A}_{x',L}^{(m,m')}$-module using this isomorphism, and equip it
with the following connection: for $\alpha\in\EcompQ{m,m'}{s}(\ms{M})$,
we put
\begin{equation*}\label{EcompQ{m,m'}{s}(ms{M})_is_free}
 \nabla(\alpha):=(-x\alpha)\otimes dx'.
\end{equation*}
\begin{prop*}
The $\mc{A}_{x',L}^{(m,m')}$-module $\EcompQ{m,m'}{s}(\ms{M})$ is finite free and $\nabla$ is a connection.
We denote its rank by
$\mr{rk}(\EcompQ{m,m'}{s}(\ms{M}))$.
\end{prop*}
\begin{proof}
Let us check that $\nabla$ defines a connection. The additivity is evident and Leibniz rule follows
from the relation $x\partial = \partial x +1$ in the ring of microdifferentials.
The $\mc{A}_{x',L}^{(m,m')}$-module $\EcompQ{m,m'}{s}(\ms{M})$ is  finitely generated by \ref{generators}, and it is endowed with a connection, hence it is torsion free by \cite[6.1]{Crew:ens}. 
Since the ring $\mc{A}_{x',L}^{(m,m')}\cong\CK{\ms{X}}\{\partial\}^{(m,m')}$ is principal by \ref{Kringpid}, it follows that $\EcompQ{m,m'}{s}(\ms{M})$ is finite free.
\end{proof}

\begin{prop}
 \label{rankcompat}
 Suppose we are in Situation {\normalfont(L)} of
 {\normalfont\ref{setup}}. Let $\ms{M}$ be a stable holonomic
 $\DcompQ{m}{\ms{X}}$-module. Let $S$ be the set of singular points of
 $\ms{M}$.

 (i) For an integer $m'>m$ and $s\in S$, we get an isomorphism
 \begin{equation*}
  \EcompQ{m+1,m'}{s}\otimes_{\EcompQ{m,m'}{s}}\EcompQ{m,m'}{s}(\ms{M})
   \cong \CK{\ms{X}}\{\partial\}^{(m+1,m')}{\otimes}_{\CK{\ms{X}}
   \{\partial\}^{(m,m')}}\EcompQ{m,m'}{s}(\ms{M}).
 \end{equation*}
 In particular, $\mr{rk}(\EcompQ{m,m'}{s}(\ms{M}))=\mr{rk}
 (\EcompQ{m+1,m'}{s}(\ms{M}))$.

 (ii) For $m'\geq m$ and $s\in S$, we get an isomorphism
 \begin{equation*}
  \EcompQ{m,m'}{s}\otimes_{\EcompQ{m,m'+1}{s}}\EcompQ{m,m'+1}{s}
   (\ms{M})\cong\CK{\ms{X}}\{\partial\}^{(m,m')}\otimes_{\CK{\ms{X}}
   \{\partial\}^{(m,m'+1)}}\EcompQ{m,m'+1}{s}(\ms{M}).
 \end{equation*}
 In particular, $\mr{rk}(\EcompQ{m,m'+1}{s}(\ms{M}))=\mr{rk}
 (\EcompQ{m,m'}{s}(\ms{M}))$.
\end{prop}

\begin{proof}
 Let us see (i).
 Let $\ms{U}$ be an open affine neighborhood of $s$ such that
 $S\cap\ms{U}=\{s\}$.
 We put $M':=\Gamma(\mathring{T}^*U,\EcompQ{m,m'}{\ms{X}}\otimes\ms{M})$,
 where $U$ denotes the special fiber of $\ms{U}$.
 Since tensor products commute with direct limits, it suffices to show
 that
 \begin{equation*}
  \EcompQb{m+1,m'}{\ms{U}}\otimes_{\EcompQb{m,m'}{\ms{U}}}M'
   \cong\CK{\ms{X}}\{\partial\}^{(m+1,m')}{\otimes}_{
   \CK{\ms{X}}\{\partial\}^{(m,m')}}M'.
 \end{equation*}
 By Lemma \ref{tesorcalc} (i), we get
 \begin{align*}
  \EcompQb{m+1,m'}{\ms{U}}\otimes_{\EcompQb{m,m'}{\ms{U}}}M'&\cong
  \EcompQb{m+1,m'}{\ms{U}}\widehat{\otimes}_{\EcompQb{m,m'}{\ms{U}}}M'\\
  &\cong(\CK{\ms{X}}\{\partial\}^{(m+1,m')}\widehat{\otimes}_{
  \CK{\ms{X}}\{\partial\}^{(m,m')}}
  \EcompQb{m,m'}{\ms{U}})\widehat{\otimes}_{\EcompQb{m,m'}{\ms{U}}}M'\\
  &\cong \CK{\ms{X}}\{\partial\}^{(m+1,m')}\widehat{\otimes}_{\CK{\ms{X}}
  \{\partial\}^{(m,m')}}M'\\
  &\cong \CK{\ms{X}}\{\partial\}^{(m+1,m')}{\otimes}_{\CK{\ms{X}}
  \{\partial\}^{(m,m')}}M'.
 \end{align*}
 The last isomorphism follows from the fact that
 $\EcompQ{m,m'}{s}(\ms{M})$ is finite over
 $\CK{\ms{X}}\{\partial\}^{(m,m')}$. Since moreover
 $\EcompQ{m,m'}{s}(\ms{M})$ is free over
 $\CK{\ms{X}}\{\partial\}^{(m,m')}$, the claim for the rank follows from
 the preceding isomorphism.

 Let us prove (ii). Since we know that $\EcompQ{m,m'}{s}$ is flat
 over $\EcompQ{m,m'+1}{s}$ by \cite[5.13]{Abe}, we may suppose
 that $\ms{M}$ is a monogenic module using an extension
 argument. Let $\ms{U}$ be an open affine neighborhood of $s$ such that
 $S\cap\ms{U}=\{s\}$, and $U$ be its
 special fiber. As in the proof of (i), it suffices
 to show the claim over $\ms{U}$.
 By using Lemma \ref{finiteness},
 there exists a $p$-torsion free $\Ecompb{m,m'+1}{\ms{U}}$-module
 $M'$ such that $\Gamma(\mathring{T}^*\ms{U},\EcompQ{m,m'+1}{\ms{X}}
 \otimes\ms{M})\cong M'\otimes\mb{Q}$ and which is finitely generated
 as an $R\{\partial\}^{(m,m'+1)}$-module. Now, we get
 \begin{align}
  \label{calctens}
  \Ecompb{m,m'}{\ms{U}}\otimes_{\Ecompb{m,m'+1}{\ms{U}}}{M}'&\cong
  \invlim_i\Emodb{m,m'}{U_i}\otimes_{\Emodb{m,m'+1}{U_i}}{M}'_i\\
  \notag&\cong\invlim\Emodb{m,m'}{U_i}\widehat{\otimes}^f
  _{\Emodb{m,m'+1}{U_i}}{M}'_i.
 \end{align}
 Indeed, the first isomorphism holds since $\Ecompb{m,m'}{\ms{U}}\otimes
 M'$ is $p$-adically
 complete. To see the second isomorphism, take a good filtration on
 $M'_i$. The tensor filtration is good by \cite[Ch.I, Lemma
 6.15]{HO}. Thus $\Emodb{m,m'}{U_i}\otimes_{\Emodb{m,m'+1}{U_i}}{M}'_i$
 is complete by \cite[Ch.II, Theorem 10]{HO} since $\Emodb{m,m'}{U_i}$
 is a noetherian filtered complete ring by \cite[Proposition 4.8]{Abe}. Now
 by Lemma \ref{tesorcalc} (iii), we get
 \begin{equation*}
  \Emodb{m,m'}{U_i}\widehat{\otimes}^f_{\Emodb{m,m'+1}{U_i}}
  {M}'_i\cong\CR{X_i}\{\partial\}^{(m,m')}\widehat{\otimes}^f
  _{\CR{X_i}\{\partial\}^{(m,m'+1)}}{M}_i
 \end{equation*}
 by the same calculation as in the proof of (i). (For careful readers,
 we note here that the same statement of \cite[2.1.7/6,7]{BGR} holds for
 filtered rings by exactly the same arguments. The detail is left to the
 reader.)
 Thus by the same calculation as (\ref{calctens}), we get
 \begin{equation*}
  \Ecompb{m,m'}{\ms{U}}\otimes_{\Ecompb{m,m'+1}{\ms{U}}}{M}'\cong
   \CR{\ms{X}}\{\partial\}^{(m,m')}\otimes_{\CR{\ms{X}}
   \{\partial\}^{(m,m'+1)}}{M}'.
 \end{equation*}
 By tensoring with $\mb{Q}$, we get what we wanted.
\end{proof}

\subsection{Characteristic cycles and microlocalizations}
We will see how we can compute the multiplicities of holonomic modules
from its microlocalizations. In general, it is very difficult to
calculate the characteristic cycles in terms of intermediate
microlocalizations. However, the construction of the rings of {\em
naive} microdifferential operators are simple and formal, therefore we
can calculate the multiplicities easier.

\subsubsection{}
\label{grlglg}
For a graded ring ${(A,F_iA)}_{i\in\mb{Z}}$ and a finite graded $A$-module ${(M,F_i M)}_{i\in\mb{Z}}$, the graded
length of $M$ is the length of $M$ in the category of graded
$A$-modules, and we denote it by $\mr{g.lg}_A(M)$. When
$\mr{g.lg}_A(M)=1$, we say that $M$ is gr-simple. We say that $A$ is
gr-Artinian if $\mr{g.lg}_A(A)<\infty$.\index{gr-Artinian, gr-simple, $\mr{g.lg}$}

Let $A$ be a positively graded commutative ring, and $M$ be a finite
graded $A$-module. Let $\mf{p}\in\mr{Proj}(A)$, and
\begin{equation*}
 S_{\mf{p}}:=\{1\}\cup\bigl\{f\in A\setminus A_0\mid\mbox{$f$ is a
  homogeneous element which is not contained in $\mf{p}$}\bigr\}.
\end{equation*}
We denote by
$A_{\widetilde{\mf{p}}}$ the localization $S^{-1}_{\mf{p}}A$, and
$S^{-1}_{\mf{p}}M$ by $M_{\widetilde{\mf{p}}}$. We note that since
$S_{\mf{p}}$ consists of homogeneous elements, these are respectively a
graded ring and a graded module. Let $X:=\mr{Spec}(A_0)$,
$V:=\mr{Spec}(A)$, $P:=\mr{Proj}(A)$. The schemes $V$ and $P$ are
schemes over $X$, and there exists a canonical section $s\colon
X\rightarrow V$. We put $\mathring{V}:=V\setminus s(X)$. 
Let us denote by $q\colon\mathring{V}\rightarrow P$ the canonical
surjection defined
in \cite[II, 8.3]{EGA}. Now we get the following.

\begin{lem*}
 Let $\widetilde{M}(*):=\bigoplus_{n\in\mb{Z}}\widetilde{M}(n)$ be the
 quasi-coherent $\mc{O}_{P}$-module associated to $M$.
 Let $\mf{p}$ be a generic point of $\mr{Supp}(\widetilde{M}(*))\subset
 P$. Then we get
 \begin{equation*}
  \mr{g.lg}_{A_{\widetilde{\mf{p}}}}(M_{\widetilde{\mf{p}}})=
   \mr{lg}_{A_{\mf{p}}}(M_{\mf{p}}).
 \end{equation*}
\end{lem*}
\begin{proof}
 By \cite[II, 8.3.6]{EGA}, we get that the fiber of $q$ at $\mf{p}$ is
 $f\colon\mr{Spec}(A_{\widetilde{\mf{p}}})\rightarrow
 \mr{Spec}(A_{(\mf{p})})$.
 Since $f_*((M_{\widetilde{\mf{p}}})^\sim)\cong\widetilde{M}(*)$,
 the support of the sheaf $(M_{\widetilde{\mf{p}}})^\sim$ in
 $\mr{Spec}(A_{\widetilde{\mf{p}}})$ is contained in
 $V(\mf{p})$, and there exists an integer $n$ such that
 $\mf{p}^nM_{\widetilde{\mf{p}}}=0$. Thus $M_{\widetilde{\mf{p}}}$ is a
 graded $A_{\widetilde{\mf{p}}}/\mf{p}^n$-module. Let $N$ be a
 graded $A_{\widetilde{\mf{p}}}/\mf{p}^n$-module such that
 $N\neq0$. Then $N_{\mf{p}}\neq0$ by the definition of
 $A_{\widetilde{\mf{p}}}$. This shows that given a chain $0\subsetneq
 N_1\subsetneq\dots\subsetneq N_l=M_{\widetilde{\mf{p}}}$ of graded
 sub-$A_{\widetilde{\mf{p}}}$-modules, we can attach a chain
 $0\subsetneq(N_1)_{\mf{p}}\subsetneq\dots\subsetneq
 (N_l)_{\mf{p}}=M_{\mf{p}}$ of sub-$A_{\mf{p}}$-modules. Thus we see
 that $\mr{g.lg}_{A_{\widetilde{\mf{p}}}}(M_{\widetilde{\mf{p}}})
 \leq\mr{lg}_{A_\mf{p}}(M_{\mf{p}})$. To see the opposite inequality, it
 suffices to show that given a $A_{\widetilde{\mf{p}}}/\mf{p}^n$-module
 $N$ such that $\mr{g.lg}_{A_{\widetilde{\mf{p}}}}(N)=1$
 then $\mr{lg}_{A_\mf{p}}(N_{\mf{p}})=1$. Since
 $A_{\widetilde{\mf{p}}}/\mf{p}$ is the only gr-simple
 $A_{\widetilde{\mf{p}}}/\mf{p}^n$-module, it suffices to show that
 $\mr{lg}(A_{\mf{p}}\otimes_{A_{\widetilde{\mf{p}}}}(A_{\widetilde{\mf{p}}}
 /\mf{p}))=1$, which is obvious.
\end{proof}

\subsubsection{}
Let $A$ be a noetherian filtered ring, and $M$ be a finite
$A$-module. Let $F$ be a good filtration on $M$. Then consider
$\mr{g.lg}_{\mr{gr}(A)}(\mr{gr}(M))$. Exactly as in the classical
way (e.g.\ \cite[A.III.3.23]{Bjo}), we are able to show that this does not
depend on the choice of good filtrations. Recall we say that an
increasingly filtered ring $(A,F_iA)_{i\in\mb{Z}}$ is Zariskian if it is
noetherian filtered and $F_{-1}A$ is contained in the Jacobson radical
$J(F_0A)$ of $F_0A$.

\begin{lem*}
 \label{Zarmularitw}
 Let $A$ be a Zariskian filtered ring such that $\mr{gr}(A)$ is a
 gr-Artinian ring. Suppose moreover that $\mr{lg}_A(A)$ is finite and 
 \begin{equation*}
  \mr{lg}_A(A)=\mr{g.lg}_{\mr{gr}(A)}(\mr{gr}(A)).
 \end{equation*} 
 Then for any good filtered $A$-module $(M,M_i)_{i\in\mb{Z}}$, we get
 \begin{equation*}
  \mr{lg}_A(M)=\mr{g.lg}_{\mr{gr}(A)}(\mr{gr}(M)).
 \end{equation*}
\end{lem*}
\begin{ex*}\label{exemple_gr-artinian}
The field
 $A=k\cc{t^{-1}}[t]$ of formal Laurent series in $t^{-1}$
  is Zariskian filtered by the filtration given by the degree in $t$.
 We have that $\mathrm{gr}(A)=k[t^{-1},t]$ is gr-Artinian, of graded length $1$, and $\mathrm{lg}_A(A)=\mathrm{g.lg}_{\mathrm{gr}(A)}(\mathrm{gr}(A))=1$. 
%; so %the hypotheses of the above lemma are fulfilled. 
\end{ex*}
\begin{proof}
 First, let $(M,M_i)_{i\in\mb{Z}}$ be a good filtered $A$-module and take a sequence
 of sub-$A$-modules
 \begin{equation*}
  M\supsetneq M^{(1)}\supsetneq\dots\supsetneq M^{(l)}=0.
 \end{equation*}
 We equip $M^{(k)}$ with the induced filtration. These filtrations are
 good by \cite[Ch.II, 2.1.2]{HO} since $A$ is a Zariskian filtered
 ring. Since the filtration on $M^{(k)}/M^{(k+1)}$ is good and $A$ is a
 Zariskian filtered ring, the filtration is separated, and
 $\mr{gr}(M^{(k)}/M^{(k+1)})\neq0$. Thus we get the strictly decreasing
 sequence
 \begin{equation*}
  \mr{gr}(M)\supsetneq\mr{gr}(M^{(1)})\supsetneq\dots\supsetneq
   \mr{gr}(M^{(l)})=0.
 \end{equation*}
 This shows that $\mr{lg}_A(M)\leq\mr{g.lg}_{\mr{gr}(A)}
 (\mr{gr}(M))$. It suffices to show that if $(N,N_i)$ is a good filtered
 $A$-module such that $N$ is a simple $A$-module, then $\mr{gr}(N)$ is a
 gr-simple $\mr{gr}(A)$-module.

 Let
 \begin{equation*}
  A=:I^{(0)}\supsetneq I^{(1)}\supsetneq\dots\supsetneq I^{(l)}=0
 \end{equation*}
 be a composition series. The hypothesis and the above observation
 imply that
 \begin{equation*}
  \mr{gr}(A)\supsetneq\mr{gr}(I^{(1)})\supsetneq\dots\supsetneq
   \mr{gr}(I^{(l)})=0
 \end{equation*}
 is also a composition series in the category of graded modules. Since
 any simple graded $A$-module is
 appearing in the series, we get that for any simple $A$-module $N$,
 there exists a good filtration on $N$ such that the $\mr{gr}(A)$-module
 $\mr{gr}(N)$ is gr-simple. Indeed, there exists $0\leq k<l$ such that $N\cong
 I^{(k)}/I^{(k+1)}$. We put the good filtration induced by that of
 $I^{(k)}/I^{(k+1)}$ on $N$. Then since $\mr{gr}(I^{(k)}/I^{(k+1)})$ is
 a gr-simple $\mr{gr}(A)$-module, we get the claim. Since
 for any good filtrations $F$ and $G$ on $N$, we know that
 $\mr{gr}_F(N)$ and $\mr{gr}_G(N)$ have the same gr-length, we get that
 $\mr{gr}(N)$ is gr-simple for any good filtration. This concludes the
 proof of the lemma.
\end{proof}

\subsubsection{}\label{def_cycl_Dm}
Let us consider Situation (L) of \ref{setup}. Recall the notation of
\ref{localsetting}. Let $\ms{M}$ be a holonomic
$\DcompQ{m}{\ms{X}}$-module (not necessarily stable).
Let $\mr{Cycl}^{(m)}(\ms{M})=\sum_{s\in S}m_s\cdot[\pi^{-1}(s)]+
r\cdot[X]$ be the characteristic cycle of $\ms{M}$ (cf.\ \cite[2.1.17]{Abe2}). 
We note that, by construction, $\mr{Cycl}^{(m)}$ is additive on short exact sequences
of $\DcompQ{m}{\ms{X}}$-modules.
The
integer $m_s$ is called the {\em vertical multiplicity} of
$\ms{M}$ at $s$ and the integer $r$ is called the \emph{generic rank} (or \emph{horizontal multiplicity}) of $\ms{M}$.
\index{characteristic variety and cycle!Cycl@$\mr{Cycl}^{(m)}$, $\mr{Cycl}$|(}
\index{characteristic variety and cycle!multiplicity!vertical multiplicity}
\index{characteristic variety and cycle!multiplicity!generic rank (horizontal multiplicity)}

\begin{prop*}
 \label{levmcharcycle}
 We get
 \begin{equation*}
  p^m\cdot\mr{rk}_{\CK{\ms{X}}\{\partial\}^{(m)}}(\EcompQ{m}{s}
   (\ms{M}))=\deg_L(s)\cdot m_s,
 \end{equation*}
 where $L:=K(\ms{X})$ and $\deg_L(s):=\deg(s)\cdot[L:K]^{-1}$.
\end{prop*}
\begin{proof}
 We may assume that $\ms{M}$ is a
 monogenic module by an extension argument using the flatness of $\EcompQ{m}{\ms{X}}$ over $\pi^{-1}\DcompQ{m}{\ms{X}}$,
cf.\ \ref{flatness}. Let $\ms{M}'$ be a monogenic
 $\Dcomp{m}{\ms{X}}$-module
 without $p$-torsion such that $\ms{M}'\otimes\mb{Q}\cong\ms{M}$. Since
 $\Ecomp{m}{\ms{X}}$ is flat over $\Dcomp{m}{\ms{X}}$ by
 \cite[2.8-(ii)]{Abe}, we note that $\Ecomp{m}{s}(\ms{M}')$ is also
 $p$-torsion free. Let
 $L\{\partial^{\angles{m}{p^m}}\}^{(0)}$ be the subring of
 $\CK{\ms{X}}\{\partial\}^{(m)}$ topologically
 generated by $\partial^{\angles{m}{\pm p^m}}$ over $L$. Then we get
 \begin{equation*}
  p^m\cdot\mr{rk}_{\CK{\ms{X}}\{\partial\}^{(m)}}(\EcompQ{m}{s}
   (\ms{M}))=
   \mr{rk}_{L\{\partial^{\angles{m}{p^m}}\}^{(0)}}
   (\EcompQ{m}{s}(\ms{M})).
 \end{equation*}
 We know that $\Ecomp{m}{s}(\ms{M}')$ is finite over
 $\CR{\ms{X}}\{\partial\}^{(m)}$ by Lemma \ref{finiteness}, and in
 particular, finite over
 $R_L\{\partial^{\angles{m}{p^m}}\}^{(0)}:=L\{\partial^{\angles{m}{p^m}}
 \}^{(0)}\cap\CR{\ms{X}}\{\partial\}^{(m)}$. Since
 $R_L\{\partial^{\angles{m}{p^m}}\}^{(0)}$ is a discrete valuation ring
 whose uniformizer is $\varpi$, and $\Ecomp{m}{s}(\ms{M}')$ is
 $p$-torsion free, we get that $\Ecomp{m}{s}(\ms{M}')$ is free over
 $R_L\{\partial^{\angles{m}{p^m}}\}^{(0)}$. Thus, we get
 \begin{align*}
  \mr{rk}_{L\{\partial^{\angles{m}{p^m}}\}^{(0)}}(\EcompQ{m}{s}
   (\ms{M}))&=\mr{rk}_{R_L\{\partial^{\angles{m}{p^m}}\}^{(0)}}
  (\Ecomp{m}{s}(\ms{M}'))\\&=\deg(L/K)^{-1}\cdot\mr{rk}_{k
  \cc{\partial^{\angles{m}{-p^m}}}[\partial^{\angles{m}{p^m}}]}
  (\Ecomp{m}{s}(\ms{M}/\varpi)).
 \end{align*}
 Here $k\cc{\partial^{\angles{m}{-p^m}}}[\partial^{\angles{m}{p^m}}]$ is
 considered as a subring of
 $R_L\{\partial^{\angles{m}{p^m}}\}^{(0)}/\varpi$. It is the ring given in Example \ref{exemple_gr-artinian}, 
with $t:=\partial^{\langle p^m\rangle_{(m)}}$, and so it satisfies the hypotheses of Lemma \ref{Zarmularitw}.
 Now, we put $\ms{N}:=\ms{M}'/\varpi$. Take a good filtration
 $F_{\bullet}$ of $\ms{N}$. Let $\xi_s$ be the generic point of the
 fiber $\pi^{-1}(s)$. Then $(\mr{gr}^F(\ms{N}))_{\xi_s}$ is
 an Artinian $(\mc{O}_{T^{(m)*}X})_{\xi_s}$-module. Note that
 $\Emod{m}{s}(\ms{N})$ possesses a natural filtration induced by the
 good filtration of $\ms{N}$ (cf.\ \cite[1.7]{Abe} or \cite[A.3.2.4]{Lau2}).
 Thus, we get
 \begin{align*}
  m_s=\mr{lg}_{(\mc{O}_{T^{(m)*}X})_{\xi_s}}(\mr{gr}^F
   (\ms{N})_{\xi_s})&=\mr{g.lg}_{\mr{gr}(\Emod{m}{X,\xi_s})}
  \bigr(\mr{gr}(\Emod{m}{s}(\ms{N}))\bigl)
 \end{align*}
 by using Lemma \ref{grlglg}. Since the injection
 $\mc{O}_{X,s}[\xi^{\angles{m}{\pm
 p^m}}]\hookrightarrow\mr{gr}(\Emod{m}{X,\xi_s})$
 induces the isomorphism $\mc{O}_{X,s}[\xi^{\angles{m}{\pm p^m}}]
 \xrightarrow{\sim}\bigl(\mr{gr}(\Emod{m}{X,\xi_s})\bigr)^{\mr{red}}$
 where $\xi^{\angles{m}{\pm p^m}}$ denotes the class of
 $\partial^{\angles{m}{\pm p^m}}$ in $\mr{gr}(\Emod{m}{X,\xi_s})$, we
 get
 \begin{align*}
  \mr{g.lg}_{\mr{gr}(\Emod{m}{X,\xi_s})}
  \bigr(\mr{gr}(\Emod{m}{s}(\ms{N}))\bigl)&=
  \mr{g.lg}_{\mc{O}_{X,s}[\partial^{\angles{m}{\pm
 p^m}}]}
  \bigr(\mr{gr}(\Emod{m}{s}(\ms{N}))\bigl)\\
  &=(\deg(s))^{-1}\cdot\mr{g.lg}_{k[\partial^{\angles{m}
   {\pm p^m}}]}\bigr(\mr{gr}(\Emod{m}{s}(\ms{N}))\bigl).
 \end{align*}
By using Lemma \ref{Zarmularitw}, we get the proposition.
\end{proof}

\subsection{Stability theorem}
We summarize what we have got, and get the following characteristic
cycle version of Theorem \ref{mainabe}, which is one of main
theorems of this paper. Recall the notation of paragraph
\ref{localsetting}.

\begin{thm}[Stability theorem for curves]
 \label{stabilitytheoremcy}
 Let $\ms{X}$ be an affine formal curve over $R$ in Situation
 {\normalfont(L)} of {\normalfont\ref{setup}}, and $\ms{M}$ be a stable
 holonomic $\DcompQ{m}{\ms{X}}$-module. Let $S$
 be the set of singular points of $\ms{M}$, and suppose that
 $\ms{M}|_{\ms{X}\setminus S}$ is a convergent isocrystal. Let $r$ be
 the generic rank of $\ms{M}$.

 (i) For any $m''\geq m'\geq m$, we get
 \begin{equation*}
  \mr{rk}_{\CK{\ms{X}}\{\partial\}^{(m)}}\bigl(\EcompQ{m}{s}
   (\ms{M})\bigr)=
   \mr{rk}_{\CK{\ms{X}}\{\partial\}^{(m',m'')}}\bigl
   (\EcompQ{m',m''}{s}(\ms{M})\bigr)=
   \mr{rk}_{\CK{\ms{X}}\{\partial\}^{(m')}}\bigl
   (\EcompQ{m'}{s}(\ms{M})\bigr).
 \end{equation*}
 This number is denoted by $r_s$.

 (ii) For any $m'\geq m$, we get
 \begin{equation*}
  \mr{Cycl}^{(m')}(\DcompQ{m'}{\ms{X}}\otimes_{\DcompQ{m}{\ms{X}}}
   \ms{M})=r\cdot[X]+p^{m'}\,[K(\ms{X}):K]\cdot\sum_{s\in
   S}\deg(s)^{-1}\cdot r_s
   \cdot[\pi^{-1}(s)].
 \end{equation*}
% where $L:=K(\ms{X})$ and $\mr{deg}_L(s):=\mr{deg}(s)\cdot[L:K]^{-1}$.
\end{thm}
\begin{proof}
 Since $\ms{M}$ is stable, we get for any $m''\geq m'\geq m$,
 \begin{equation*}
  \mr{rk}(\EcompQ{m''}{s}(\ms{M}))=\mr{rk}(\EcompQ{m',m''}{s}
   (\ms{M}))=\mr{rk}(\EcompQ{m'}{s}(\ms{M}))
 \end{equation*}
 by Proposition \ref{rankcompat}. Thus (i) follows.

 Let us see (ii). For this, the vertical multiplicities are the only
 problem. By Proposition \ref{levmcharcycle}, we get for any $m'\geq m$
 \begin{equation*}
  p^{m'}\cdot\mr{rk}_{\CK{\ms{X}}\{\partial\}^{(m')}}
   (\EcompQ{m'}{s}(\ms{M}))=\mr{deg}(s)\cdot[K(\ms{X}):K]^{-1}\cdot
   m_s(\ms{M}^{(m')}).
 \end{equation*}
 Thus combining with (i), we get (ii).
\end{proof}

%\begin{rem*}
%The statement of Corollary \ref{cor_noeth_art} was already know for the category of holonomic
%$F$-$\DdagQ{\ms{X}}$-modules, cf.\ \cite[5.4.3 (ii)]{BerInt}.
%The main point is to give, for a $\DdagQ{\ms{X}}$-module $\ms{M}$, a definition of the characteristic cycle independent on the level, 
%\emph{i.e.}\ independent on the choice of a $\DcompQ{m}{\ms{X}}$-module 
%$\ms{M}^{(m)}$ such that $\ms{M}= \DdagQ{\ms{X}}\otimes_{\DcompQ{m}{\ms{X}}} \ms{M}^{(m)}$.
% In \emph{loc.\ cit.}, Berthelot used his Frobenius descent theorem; whereas here we can take, thanks to Theorem \ref{stabilitytheoremcy},  an integer $m$ big enough (so that $\ms{M}^{(m)}$ is stable). 
%\end{rem*}

%\begin{rem*}
Thanks to Theorem \ref{stabilitytheoremcy} we can now define the characteristic cycle of holonomic 
$\DdagQ{\ms{X}}$-modules, cf.\ \ref{defCycl}, and prove Corollary \ref{cor_noeth_art}.
These  were already know for the category of holonomic
$F$-$\DdagQ{\ms{X}}$-modules, cf.\ \cite[5.4.1 et 5.4.3 (ii)]{BerInt}.
The main point is to give, for a $\DdagQ{\ms{X}}$-module $\ms{M}$, a definition of the characteristic cycle independent on the level, 
\emph{i.e.}\ independent on the choice of a $\DcompQ{m}{\ms{X}}$-module 
$\ms{M}^{(m)}$ such that $\ms{M}= \DdagQ{\ms{X}}\otimes_{\DcompQ{m}{\ms{X}}} \ms{M}^{(m)}$.
 In \emph{loc.\ cit.}, Berthelot used his Frobenius descent theorem; whereas here we can take, thanks to Theorem \ref{stabilitytheoremcy},  an integer $m$ big enough (so that $\ms{M}^{(m)}$ is stable). 
%\end{rem*}

\begin{dfn}\label{defCycl}\index{characteristic variety and cycle!Cycl@$\mr{Cycl}^{(m)}$, $\mr{Cycl}$|)}
 Let $\ms{X}$ be a formal curve, and let $\ms{M}$ be a holonomic
 $\DdagQ{\ms{X}}$-module. Let $\ms{M}'$ is a stable coherent
 $\DcompQ{m}{\ms{X}}$-module such that
 $\DdagQ{\ms{X}}\otimes\ms{M}'\cong\ms{M}$. Let $S$ be the set of
 singular points of $\ms{M}$. We define
 \begin{equation*}
  \mr{Cycl}(\ms{M})=r\cdot[X]+[K(\ms{X}):K]\cdot\sum_{s\in S}
   \mr{deg}(s)^{-1}\cdot r_s\cdot[\pi^{-1}(s)],
 \end{equation*}
 where $r$ is the generic rank of $\ms{M}'$, and
 $r_s:=\mr{rk}\bigl(\EcompQ{m}{s}(\ms{M}')\bigr)\in\mb{N}$, which do
 not depend on the choice of $\ms{M}'$ by Theorem
 \ref{stabilitytheoremcy}.
\end{dfn}

\begin{rem*} 
 (i) When $\ms{M}$ possesses a Frobenius structure, the characteristic
 cycle here coincides with that of Berthelot \cite[5.4.2]{BerInt} (or
 \cite[2.3.13]{Abe2}).
  
 (ii) By Theorem \ref{stabilitytheoremcy}-(ii) and \S\ref{def_cycl_Dm}, $\mr{Cycl}$ is additive on short exact sequences of holonomic
 $\DdagQ{\ms{X}}$-modules.

 (iii) The characteristic cycle has integral coefficients. 
 To prove this we may  first assume 
$\ms{X}$ geometrically connected, indeed $\mr{Cycl}(\ms{M})$ does not change
 if we consider $\ms{X}$ as a smooth formal scheme over
$\mr{Spf}(R(\ms{X}))$; secondly we note that if we base change $\ms{X}$ by a finite extension, then  the multiplicities 
$r_s$ will not change  (by construction); therefore we may assume $S$ rational over $K$ and the integrality of 
  $\mr{Cycl}(\ms{M})$ is evident. 
\end{rem*}

\begin{cor}\label{cor_noeth_art}
 Let $\ms{X}$ be a formal curve over $R$. The category of holonomic
 $\DdagQ{\ms{X}}$-modules is both noetherian and artinian.
\end{cor}

\begin{proof}
The argument of the proof is the same as that of \cite[5.4.3 (ii)]{BerInt}. We recall it for the convenience of the reader.
We can prove the ascending chain condition as follows:
let $(\ms{M}_n \subseteq \ms{M})_{n\in \mathbb{N}}$ be an ascending filtration by holonomic sub-modules.
We may assume $\ms{M} \not = 0$.  The support of $\mr{Cycl}(\ms{M})$ has dimension one
because $\ms{M}$  is holonomic.
By additivity of $\mr{Cycl}$, cf.\ Remark \ref{defCycl}-(ii), we have, for all $n$, 
$$ \mr{Cycl}(\ms{M})= \mr{Cycl}(\ms{M}_0) + \sum_{i=1}^{n}\mr{Cycl}(\ms{M}_i/\ms{M}_{i-1}) +\mr{Cycl}(\ms{M}/\ms{M}_n).$$
Since $T^*X$ is a noetherian space and the coefficients ($r$ and $r_s$) appearing in $\mr{Cycl}(\ms{M})$ belong to $\mathbb{N}$, we get,
for $n$ big enough,  
$\mr{Cycl}(\ms{M}_n/\ms{M}_{n-1})=0$; 
therefore  $\ms{M}_n=\ms{M}_{n-1}$. 
 We can prove the descending chain condition in a similar way. 
%OLDER PROOF
% This follows from the additivity of $\mr{Cycl}$, and the fact that
% $r_s\in\mb{N}$ using the notation of the definition above.
%OLD PROOF
%The argument of the proof is the same as that of \cite[5.4.3 (ii)]{BerInt}. We recall it for the convenience of the reader.
%We use the following proprieties:
%\begin{enumerate}
%\item The quotient of holonomic modules is holonomic;
%\item\label{1} $T^*X$ is a noetherian space;
%\item\label{2} $\mr{Cycl}(\ms{M})$ is zero if and only if $\ms{M}=0\,$;
%\item $\mr{Cycl}$ is additive on short exact sequences $0\rightarrow\ms{M}'\rightarrow\ms{M}\rightarrow\ms{M}''\rightarrow 0\,$;
%\item\label{3} The coefficients $r$ and $r_s$ appearing in $\mr{Cycl}(\ms{M})$ belong to $\mathbb{N}\,$;
%\item\label{4} If $\ms{M} \not = 0$ is holonomic then the support of $\mr{Cycl}(\ms{M})$ has dimension $1$.
%\end{enumerate}
%For example we can prove the ascending chain condition as follows:
%let $(\ms{M}_n \subseteq \ms{M})_{n\in \mathbb{N}}$ be an ascending filtration by holonomic sub-modules.
%By additivity we have, for all $n$, 
%$$ \mr{Cycl}(\ms{M})= \mr{Cycl}(\ms{M}_0) + \sum_{i=1}^{n}\mr{Cycl}(\ms{M}_i/\ms{M}_{i-1}) +\mr{Cycl}(\ms{M}/\ms{M}_n).$$
%For $n$ big enough, by using properties \eqref{1}, \eqref{3} and \eqref{4}, we get  
%$\mr{Cycl}(\ms{M}_n/\ms{M}_{n-1})=0$, therefore  $\ms{M}_n=\ms{M}_{n-1}$ by \eqref{2}.
 %We can prove the descending chain condition in a similar way. 
\end{proof}

\section{Local Fourier transform}
\label{section2}
The aim of this section is to define the local Fourier
transform. We note that the definition itself is not difficult anymore
thanks to works of Huyghe and Matsuda: we
can take the canonical extension, the geometric Fourier
transform, and take the differential module around $\infty$ as
presented in $\cite[8.3]{Cr}$. However, with this definition, we are not
able to prove the stationary phase formula in the way  used in the
complex case. In this section, we instead define the local Fourier
transform using microlocalizations following the classical techniques,
and prove some basic properties.

\subsection{Local theory of arithmetic $\ms{D}$-modules}
\label{Crewreview}
Since the main goal of this paper (Theorem \ref{Product formula}) is to
prove a theorem connecting
local and global invariants, it is indispensable to work in local
situations. In the $\ell$-adic case, this was the theory of \'{e}tale
sheaves on traits, in other words Galois representations of local
fields. In our setting, the theory of arithmetic $\ms{D}$-modules on a
formal disk by Crew \cite{Cr3,Cr}, which can be seen as a generalization of
the theory of solvable $p$-adic differential equations, should be the
corresponding theory. We briefly review the theory in this subsection.

\subsubsection{}
\label{Crewconstan} 
Let $R$, $k$, $K$ be as usual, cf.\ \ref{defRmsX}; we recall that $\varpi$ denotes a uniformizer of $R$. 
The field $k$ will be assumed perfect in all \S\ref{section2}, with exception of  
\ref{Crewconstan} -- \ref{diffmoddef} and \ref{analye}.
Denote by $W(k)$\index{.@miscellaneous!W@$W(k)$} a Cohen ring with residue field $k$ (ring of Witt vectors if $k$ is perfect). 
For any commutative local ring $A$, we will denote by $\mf{m}_A$ (or
$\mf{m}$) its maximal ideal. %\index{.@miscellaneous!$\mf{m}_A$, $\mf{m}$}
If $A$ is an $I$-adic ring, for an ideal $I$ of $A$,
we will denote it by $(A,I)$ when we want to specify the ideal of definition explicitly. 

Let $(A,\mf{m})$ be a $2$-dimensional formally smooth local noetherian
$R$-algebra complete with respect to the $\mf{m}$-adic topology, such
that $p\in\mf{m}$, whose residue field $k_A$ is finite over $k$. In this
situation, $A$ is complete with respect to the $p$-adic topology
by \cite[$0_{\mr{I}}$, 7.2.4]{EGA}. 
Let $R_A:=R\otimes_{W(k)}
W(k_A)$ and $K_A:=\mr{Fr}(R_A)$. \index{.@miscellaneous!K@$K_A$, $R_A$}
 Note that $R_A$ is a discrete valuation ring. Now we get the
following.
\begin{lem*}
 \label{existconcdesc}
 The $R$-algebra $A$ is isomorphic to $R_A\cc{t}$.
\end{lem*}
\begin{proof}
 By \cite[$0_{\mr{IV}}$, 19.6.5]{EGA}, we get $A/\varpi A\cong
 k_A\cc{t}$. Moreover, the ring $R_A\cc{t}$
 is a complete noetherian local ring, formally smooth over $R$, such
 that its reduction over $k$ is isomorphic to $A/\varpi A$. Thus by
 \cite[$0_{\mr{IV}}$, 19.7.2]{EGA}, we get the lemma.
\end{proof}
The situation we have in mind is the following: 
the $R$-algebra $A$ is the completion
$\widehat{\mc{O}}_{\ms{X},x}$ 
of the local ring $\mc{O}_{\ms{X},x}$ of  a formal curve $\ms{X}$ at a
closed point $x$, with respect to the filtration by the powers of its
maximal ideal, denoted by $\mathfrak{m}_{\ms{X},x}$. In this case we
will put $k_x:=k_{\widehat{\mc{O}}_{\ms{X},x}}$,
$R_x:=R_{\widehat{\mc{O}}_{\ms{X},x}}$, and denotes by $K_x:=\Frac(R_x)$
the field of fractions of
$R_x$. \index{.@miscellaneous!Ox2@$\widehat{\mc{O}}_{\ms{X},x}$,
$\mc{O}_{\ms{X},x}$, $\mathfrak{m}_{\ms{X},x}$, $k_x$, $R_x$, $K_x$}

We simply denote $\mr{Spf}(A,\varpi A)$ by $\mr{Spf}(A)$.
The formal scheme $\ms{S}:=\mr{Spf}(A)$ \index{.@miscellaneous!S@$\ms{S}$}
is called a {\em formal disk} \index{formal disk}
and it consists of two points: an open point $\eta_{\ms{S}}$ and a
closed point $s$.  We put $\widetilde{\ms{S}}:=\mr{Spf}(A,\mf{m})$,
which consists of only one point. Note that Crew in \cite{Cr} used the
notation $\mr{Spf}(A)$ for
$\mr{Spf}(A,\mf{m})$. \index{.@miscellaneous!S@$\widetilde{\ms{S}}$,
$\eta_{\ms{S}}$} The reason why we introduced $\ms{S}$ and
$\widetilde{\ms{S}}$ will be clarified in Remark \ref{Rem_on_S_tilde}.

\subsubsection{}
\label{Cr_recall1}
In \cite{Cr3,Cr} Crew checked that on $\widetilde{\ms{S}}$, the theory
of arithmetic $\ms{D}$-modules can be constructed in the same manner.
He constructed the ring $\DdagQ{\widetilde{\ms{S}}}$ (resp.\
$\DdagQ{\widetilde{\ms{S}}}(0)$) of  Berthelot differential operators
(resp.\ overconvergent at $t=0$), cf.\
\cite[(3.1.5)]{Cr}. \index{differential
operators!D@$\DdagQ{\widetilde{\ms{S}}}$,
$\DdagQ{\widetilde{\ms{S}}}(0)$} He also constructed analytic variants
$\Dan{\widetilde{\ms{S}},\mb{Q}}$ and
$\ms{D}^{\mr{an}}_{\widetilde{\ms{S}},\mb{Q}}(0)$ \index{differential
operators!Dan@$\Dan{\widetilde{\ms{S}},\mb{Q}}$,
$\ms{D}^{\mr{an}}_{\widetilde{\ms{S}},\mb{Q}}(0)$,
${(\DcompQ{m}{x})}^{\mr{an}}$, $\ms{D}^{\mr{an}}_{x,\mb{Q}}$|(} of these
rings  by ``analytifying'' $\DdagQ{\ms{X}}$ and $\DdagQ{\ms{X}}(0)$,
where $\ms{X}:=\mr{Spf}(R\{t\})$. Here analytification roughly means to
tensor with $\mc{A}_{u,K}([0,1[)\subset K\cc{t}$ and take
the completion with respect to a suitable topology. 
Let us briefly
review the constructions of $\Dan{\widetilde{\ms{S}},\mb{Q}}$.
Choose an isomorphism
$\widehat{\mc{O}}_{\ms{X},x}\cong R_x\cc{t}$ by Lemma
\ref{existconcdesc}. For a positive integer $r$ and a non-negative integer $i$, we
define $\mc{O}_{r,i}$\index{.@miscellaneous!Ori@$\mc{O}_{r,i}$} to be $\widehat{\mc{O}}_{X_i,x}[T]/(pT-t^r)$ as
defined in \cite[4.1]{Cr}. When $r$ is divided by $p^{m+1}$,
$\mc{O}_{r,i}$ possesses a $\Dmod{m}{X_i}$-module structure, cf.\
\cite[Lemma 3.1.1]{Cr}.
We define a {\em ring}
\begin{equation*}
 (\DcompQ{m}{x})^{\mr{an}}:=\invlim_{n}\bigl((\invlim_i
  \mc{O}_{np^{m+1},i}\otimes_{\mc{O}_{X_i}}\Dmod{m}{X_i})
  \otimes\mb{Q}\bigr).
\end{equation*}
Although it is not defined explicitly in {\it loc.\ cit.}, this ring is
used to define
$\ms{D}^{\mr{an}}_{x,\mb{Q}}$ as  inductive limit of
$(\DcompQ{m}{x})^{\mr{an}}$ over $m$. \index{differential operators!Dan@$\Dan{\widetilde{\ms{S}},\mb{Q}}$,
$\ms{D}^{\mr{an}}_{\widetilde{\ms{S}},\mb{Q}}(0)$, ${(\DcompQ{m}{x})}^{\mr{an}}$, $\ms{D}^{\mr{an}}_{x,\mb{Q}}$|)}
By the remark below this ring
depends only on $\widetilde{\ms{S}}$ and not on the parameter $t$ used to define it, so we denote it 
by $\Dan{\widetilde{\ms{S}},\mb{Q}}$.
The construction of $\Dan{\widetilde{\ms{S}},\mb{Q}}(0)$ is analogous, by using $\Dmod{m}{X_i}(0)$ in place of
$\Dmod{m}{X_i}$.
\begin{rem*}\label{rem_indipendec_par}
 In the construction of the rings $\Dan{\widetilde{\ms{S}},\mb{Q}}$ and $\Dan{\widetilde{\ms{S}},\mb{Q}}(0)$, 
one uses the parameter $t$. Thus,
 {\itshape a priori} the construction depends on this choice. However,
 if we use another $t'$ such that its image in $k\cc{t}$ is a uniformizer
 to construct the rings, the resulting rings are canonically isomorphic
 to those constructed using the uniformizer $t$. Indeed, using the
 notation of \cite[4.1]{Cr}, let $\mc{O}_r(t)$ be the ring $\mc{O}_r$
 using the uniformizer $t$. Let $t'$ be another uniformizer. Then there
 exists a canonical isomorphism
 $\mc{O}_{r,\mb{Q}}:=\mc{O}_r(t)\otimes\mb{Q}\cong\mc{O}_r(t')
 \otimes\mb{Q}$. Moreover, in $\mc{O}_{r,\mb{Q}}$,
 there exists an inclusion $\mc{O}_r(t)\subset p^{-r}\mc{O}_r(t')$ and
 $\mc{O}_r(t')\subset p^{-r}\mc{O}_r(t)$. Thus the claim follows from the
 definition of $\Dan{}$ and $\Dan{}(0)$.
\end{rem*}

Moreover Crew generalize these constructions to define
analytification functors, cf.\ \cite[4.1]{Cr},
\begin{equation*}
 (-)^{\mr{an}}\colon\mr{Coh}(\DdagQ{\widetilde{\ms{S}}})\rightarrow
  \mr{Mod}(\Dan{\widetilde{\ms{S}},\mb{Q}}), \qquad
  (-)^{\mr{an}}\colon\mr{Coh}(\DdagQ{\widetilde{\ms{S}}}(0))\rightarrow
  \mr{Mod}(\Dan{\widetilde{\ms{S}},\mb{Q}}(0)),
  \index{functors!.1@$(-)^{\mr{an}}$}
\end{equation*}
where $\mr{Coh}(-)$ denotes the category of coherent modules and
$\mr{Mod}(-)$ denotes the category of modules. They send
$\DdagQ{\widetilde{\ms{S}}}$ to $\Dan{\widetilde{\ms{S}},\mb{Q}}$ and 
$\DdagQ{\widetilde{\ms{S}}}(0)$ to $\Dan{\widetilde{\ms{S}},\mb{Q}}(0)$
respectively, and for every coherent $\DdagQ{\widetilde{\ms{S}}}$-module
(resp.\ $\DdagQ{\widetilde{\ms{S}}}(0)$-module) $\ms{M}$  there exists a
natural monomorphism $\ms{M}\hookrightarrow {\ms{M}}^{\mr{an}}$.  
Finally, the morphisms
$\DdagQ{\widetilde{\ms{S}}}\hookrightarrow\Dan{\widetilde{\ms{S}},\mb{Q}}$
and
$\DdagQ{\widetilde{\ms{S}}}(0)\hookrightarrow\Dan{\widetilde{\ms{S}},\mb{Q}}(0)$
are flat (both left and right); the analytification functors are exact;
and we have $\ms{M}^{\mr{an}} \cong \Dan{\widetilde{\ms{S}},\mb{Q}}
\otimes_{\DdagQ{\widetilde{\ms{S}}}} \ms{M}$ (resp.\ $\ms{M}^{\mr{an}}
\cong \Dan{\widetilde{\ms{S}},\mb{Q}}(0)
\otimes_{\DdagQ{\widetilde{\ms{S}}}(0)} \ms{M}$), for any coherent
${\DdagQ{\widetilde{\ms{S}}}}$-module (resp.\
${\DdagQ{\widetilde{\ms{S}}}}(0)$-module)
$\ms{M}$, cf.\ \cite[Th.~4.1.1 and 4.1.2]{Cr}.

 We define the sheaves $\DdagQ{\ms{S}}$,
$\Dan{\ms{S},\mb{Q}}$ \index{differential operators!Ddag@$\DdagQ{\ms{S}}$,
$\Dan{\ms{S},\mb{Q}}$} by
\begin{alignat*}{2}
 \Gamma(\ms{S},\DdagQ{\ms{S}})&:=
 \DdagQ{\widetilde{\ms{S}}},&\qquad \Gamma(\eta_{\ms{S}},
 \DdagQ{\ms{S}})&:=
 \DdagQ{\widetilde{\ms{S}}}(0), \\
 \Gamma(\ms{S},\Dan{\ms{S},\mb{Q}})&:= \Dan{\widetilde{\ms{S}},\mb{Q}},&\qquad\Gamma(\eta_{\ms{S}},
 \Dan{\ms{S},\mb{Q}})&:=\Dan{
 \widetilde{\ms{S}},\mb{Q}}(0).
\end{alignat*}
By the Remark above, these sheaves are well-defined.
For any $\DdagQ{\widetilde{\ms{S}}}$-module $\ms{M}$, we can define a sheaf $\ms{M}^{\triangle}$ on $\ms{S}$
by 
\begin{alignat*}{2}
 \Gamma(\ms{S},\ms{M}^{\triangle})&:=\ms{M},
 &\qquad\Gamma(\eta_{\ms{S}}, \ms{M}^{\triangle}) &:=
 \DdagQ{\widetilde{\ms{S}}}(0)\otimes_{\DdagQ{\widetilde{\ms{S}}}}
 \ms{M}\,;
\end{alignat*}
and similarly  we can associate a sheaf on $\ms{S}$ to any
$\Dan{\widetilde{\ms{S}},\mb{Q}}$-module,  by replacing
$\DdagQ{\widetilde{\ms{S}}}(0)$ with
$\Dan{\widetilde{\ms{S}},\mb{Q}}(0)$ in the definition above. By
construction, the fibers of any sheaf of $\DdagQ{\ms{S}}$-modules
(resp.\ $\Dan{\ms{S},\mb{Q}}$-modules) $\ms{M}$ are given by
$\ms{M}_s = \Gamma(\ms{S}, \ms{M})$ and $\ms{M}_{\eta} = \Gamma(\eta, \ms{M})$. 
Hence the functors $\Gamma(\ms{S}, - )$ and $(-)^{\triangle}$ are
equivalence of categories, quasi-inverse each other, between the
category of coherent sheaves of $\DdagQ{\ms{S}}$-modules (resp.\
$\Dan{\ms{S},\mb{Q}}$-modules) and coherent
$\DdagQ{\widetilde{\ms{S}}}$-modules (resp.\
$\Dan{\widetilde{\ms{S}},\mb{Q}}$-modules).
We will often abusively use the same symbol $\ms{M}$ to denote a sheaf of coherent
$\DdagQ{\ms{S}}$-modules (resp.\ $\Dan{\ms{S},\mb{Q}}$-modules) and its
global sections on $\ms{S}$.

%the following proposition 
%(theorems A and B of Cartan) is evident in this context.
%One can check that coherent $\DdagQ{\ms{S}}$-modules correspond
%one-by-one to coherent $\DdagQ{\widetilde{\ms{S}}}$-modules.
%\begin{prop*} \label{equiv_Gamma}
%\end{prop*}

\subsubsection{}
\label{diffmoddef}
Put $\mc{O}^{\mr{an}}_{u,K}:=\mc{A}_{u,K}(\left[0,1\right[)$,
 define the bounded Robba ring (in fact a field) by
$\mc{R}^b_{u,K}:=\bigcup_{r<1}\mc{A}_{u,K}(\left[r,1\right\})$ and the
Robba ring by
$\mc{R}_{u,K}:=\bigcup_{r<1}\mc{A}_{u,K}(\left[r,1\right[)$\index{Robba's
rings!.1@$\mc{R}^b_{u,K}$, $\mc{R}^b_{\ms{S}}$, $\mc{R}^b$,
$\mc{R}_{u,K}$, $\mc{R}_{\ms{S}}$, $\mc{R}$}, cf.\
\ref{fieldconstmicdif}.

 By Remark \ref{rem_indipendec_par}, the rings $\mc{A}_{t,K_A}(\left[0,1\right[)$, $\mc{R}_{t,K_A}$ and $\mc{R}^b_{t,K_A}$ 
 depend on the choice of the coordinate $t$ of
$\ms{S}$ only up to a canonical isomorphism:
indeed they are the sub-rings of order zero differential operators in  $\Dan{\widetilde{\ms{S}},\mb{Q}}$, $\Dan{
 \widetilde{\ms{S}},\mb{Q}}(0)$ and $\DdagQ{\widetilde{\ms{S}}}(0)$ respectively; 
we denote these rings by
$\mc{O}^{\mr{an}}_{\ms{S}}$, $\mc{R}_{\ms{S}}$ and $\mc{R}^b_{\ms{S}}$. 
We often
omit the subscripts and write simply $\mc{O}^{\mr{an}}$,   $\mc{R}$ and $\mc{R}^b$.
\index{rings of analytical functions!$\mc{O}^{\mr{an}}_{\ms{S}}$, $\mc{O}^{\mr{an}}_{u,K}$, $\mc{O}^{\mr{an}}$}

Let us fix our conventions on the definition of differential
modules.\index{differential module}
In this paper, we will adopt the definition of Kedlaya
\cite[8.4.3]{Ke2}. Namely, we define a differential
$\mc{A}_K(\left[r,1\right[)$-module to be a locally free sheaf of finite
rank on the rigid analytic annulus $\mc{C}(\left[r,1\right[)$ over $K$
with a connection. In other words, it is a collection
$\{{M}_{r'}\}_{r<r'<1}$ where ${M}_{r'}$ is a finite differential
$\mc{A}_K([r,r'])$-module equipped with isomorphisms
$\mc{A}([r,r_1])\otimes {M}_{r_2}\cong {M}_{r_1}$, for $r_1<r_2$, of
differential modules, compatible with each other in the obvious sense.

Let us define the category $\mc{C}$ of differential
$\mc{R}$-modules. An object consists of a differential
$\mc{A}(\left[r,1\right[)$-modules for some $0<r<1$. Let ${M}$ be a
differential $\mc{A}(\left[r,1\right[)$-module, and ${N}$ be a
differential $\mc{A}(\left[r',1\right[)$-module where $0<r,r'<1$. Then
we define the set of homomorphisms by
\begin{equation*}
 \mr{Hom}_{\mc{C}}({M},{N}):=\indlim_{s\rightarrow 1^-}
  \mr{Hom}_{\partial}\bigl(\mc{A}(\left[s,1\right[)\otimes{M},
  \mc{A}(\left[s,1\right[)\otimes {N}\bigr),
\end{equation*}
where $s\geq\max\{r,r'\}$, and $\mr{Hom}_{\partial}$ denotes the
homomorphism of differential modules.

For a differential $\mc{A}(\left[r,1\right[)$-module
${M}$, we denote by $\Gamma({M})$ its global sections. 
The Fr\'{e}chet algebra
$\mc{A}(\left[r,1\right[)=\varprojlim\limits_{r<r'<1}\mc{A}([r,r'])$ is
Fr\'{e}chet-Stein in the sense of \cite[\S3]{ST} and any differential
$\mc{A}(\left[r,1\right[)$-module $M$ is in particular a coherent sheaf
for $\mc{A}(\left[r,1\right[)$, under the terminology of \emph{loc.\
cit.}; therefore  by  [\emph{loc.\ cit.\ }Cor.\ 3.1], the natural map
$\Gamma({M})\otimes_{\mc{A}(\left[r,1\right[)} \mc{A}([r,r'])
\rightarrow M_{r'}$  is an isomorphism for any $r'$, and the functor of
global sections is an equivalence on its essential image [\emph{loc.\
cit.\ }Cor.\ 3.3].

Let
${M}$ be a differential $\mc{R}$-module which is defined by a
differential $\mc{A}(\left[r,1\right[)$-module also denoted by
${M}$. We define
\begin{equation*}
 \Gamma({M}):=\indlim_{s\rightarrow1^-}\Gamma\bigl(\mc{A}(
  \left[s,1\right[)\otimes {M}\bigr).  \index{differential module!global sections $\Gamma(-)$ of}
\end{equation*}
We note that this is an $\mc{R}$-module with a connection.

\begin{dfn*}
\label{def_free_mod}
\label{def:Hol'}
 (i) We say that ${M}$ is a {\em free differential $\mc{R}$-module
 (resp.\ $\mc{A}(\left[r,1\right[)$-module)} \index{differential module!
 free differential module} if the module of global sections is a finite
 free $\mc{R}$-module (resp.\ $\mc{A}(\left[r,1\right[)$-module). This
 is equivalent to saying that $\Gamma({M})$ is finitely presented by
 \cite[4.8, 6.1]{Crew:ens}.

 (ii) We denote\footnote{The choice of notation $\mr{Hol}$ is motivated
 by Proposition \ref{characterisation_of_Dan0_hol_mod}} by
 $\mr{Hol}'(\eta_{\ms{S}})$ (resp. $\mr{Hol}(\eta_{\ms{S}})$) the
 category of differential (resp.\ free differential) modules on
 $\mc{R}_{\ms{S}}$. \index{categories!Hol@$\mr{Hol}'(\eta_{\ms{S}})$,
 $\mr{Hol}(\eta_{\ms{S}})$}
\end{dfn*}

\begin{lem*}\label{lemma_sgs}
Let ${M}$ be a {\em free} differential $\mc{R}$-module, and ${N}$
be a differential $\mc{R}$-module. Suppose there exists a
homomorphism $\varphi\colon\Gamma({M})\rightarrow\Gamma({N})$
which is compatible with the connections. Then there exists a
homomorphism of differential modules ${M}\rightarrow {N}$ inducing
$\varphi$. 
\end{lem*}
\begin{proof}
Indeed, there exists $0<r<1$ and a free differential
$\mc{A}(\left[r,1\right[)$-module ${M}'$ which induces ${M}$. Take
a finite basis $\{e_i\}_{i\in I}$ of $\Gamma({M}')$. Then there
exists $r<r'<1$ and a differential
$\mc{A}(\left[r',1\right[)$-module ${N}'$ inducing ${N}$ such that
$\varphi(e_i)\in\Gamma({N}')$ for any $i\in I$. By \cite[Cor.\ 3.3]{ST} this defines a
homomorphism $\mc{A}(\left[r',1\right[)\otimes {M}'\rightarrow {N}'$
inducing $\varphi$. Taking the inductive limit, we get what we want.
\end{proof}

\begin{rem*} \label{non-free-diff-mod}
%\begin{enumerate}
%\item
For any differential ideal $J$ of $\mc{A}(\left[r,1\right[)$ we can define a differential 
$\mc{A}(\left[r,1\right[)$-module by $J^{\triangle} := {\{J \otimes_{\mc{A}(\left[r,1\right[)} \mc{A}([r,r'])\}}_{r<r'<1}$.
However  $\Gamma(J^{\triangle})$ is not equal to $J$ in general. 
For example take a differential ideal $J$ of $\mc{A}(\left[r,1\right[)$ of infinite type (cf.\ \cite[Ex.\ 3.2]{CM}).
 Since $\mc{A}([r,r'])$ is differentially simple (cf.\ [\emph{loc.\ cit.\ }Rem.\ 3.1]), the ideal $J \otimes_{\mc{A}(\left[r,1\right[)} \mc{A}([r,r'])$  is equal to $\mc{A}([r,r'])$, and the differential module $J^{\triangle}$ is free of rank $1$.   
%\item  We do not know any example of a \emph{not free} differential $\mc{A}(\left[r,1\right[)$-module $M={\{ M_{r'}\}}_{r<r'<1}$  such that, for all $r'$, $M_{r'}$ is a free differential $\mc{A}([r,r'])$-module. 
%\end{enumerate}
\end{rem*}

\subsubsection{} \label{defsFhol}
Starting from here, with the exception of \ref{analye}, we assume
that $k$ is perfect.
Let $\sigma\colon\ms{S}\rightarrow\ms{S}$ be a lifting of the $h$-th
absolute Frobenius morphism. \index{.@miscellaneous!sigma@$\sigma$}
An $F$-$\DdagQ{\widetilde{\ms{S}}}$-module (resp.\
$F$-$\Dan{\widetilde{\ms{S}},\mb{Q}}$-module) is a
$\DdagQ{\widetilde{\ms{S}}}$-module (resp.\
$\Dan{\widetilde{\ms{S}},\mb{Q}}$-module) ${M}$ endowed with an
isomorphism ${M}\xrightarrow{\sim}\sigma^* {M}$.
%\index{F@$F$-$\DdagQ{\widetilde{\ms{S}}}$-module, $F$-$\Dan{\widetilde{\ms{S}},\mb{Q}}$-module}
For $F$-$\DdagQ{\widetilde{\ms{S}}}$-modules the definition of holonomy
is the same than that for formal smooth schemes, cf.\
\cite[5.3.5]{BerInt}, \cite[3.4]{Cr}.
An $F$-$\Dan{\widetilde{\ms{S}},\mb{Q}}$-module (resp.\
$F$-$\Dan{\widetilde{\ms{S}},\mb{Q}}(0)$-module) ${M}$ is said
\emph{holonomic} if there exists a holonomic
\index{holonomic!$F$-$\Dan{\widetilde{\ms{S}},\mb{Q}}$-module,
$F$-$\Dan{\widetilde{\ms{S}},\mb{Q}}(0)$-module}
$F$-$\DdagQ{\widetilde{\ms{S}}}$-module $\ms{N}$ such that
$\ms{N}^{\mr{an}}\cong {M}$, cf.\ {\cite[5.2]{Cr}}.
We denote by $F$-$\mr{Hol}(\ms{S})$ (resp.\
$F$-$\mr{Hol}(\eta_{\ms{S}})$, resp.\ $F$-$\mr{Hol}'(\eta_{\ms{S}})$)
the category of holonomic $F$-$\Dan{\widetilde{\ms{S}}}$-modules (resp.\
holonomic $F$-$\Dan{\widetilde{\ms{S}}}(0)$-modules, resp. the category
of differential module on $\mc{R}_{\ms{S}}$  with Frobenius structure).
\label{def:F-Hol'}
\index{categories!FHol@$F$-$\mr{Hol}(\ms{S})$,
$F$-$\mr{Hol}'(\eta_{\ms{S}})$, $F$-$\mr{Hol}'(\eta_{\ms{S}})$}

\begin{prop*}\label{characterisation_of_Dan0_hol_mod}
 The category of \emph{free} differential $\Rob_{\ms{S}}$-modules with Frobenius
structure is equivalent to the category of holonomic
$F$-$\Dan{\widetilde{\ms{S}},\mb{Q}}(0)$-modules. 
\end{prop*}
\begin{proof}
By
\cite[4.2.1]{Tsuzuki:Slope_fil}, we get that given a free differential
$\mc{R}$-module ${M}$ with Frobenius structure, there exists a
free differential $\mc{R}^b$-module ${M}'$ with Frobenius structure such
that $\mc{R}\otimes {M}'\cong {M}$. Thus, by \cite[Prop.~4.1.1]{Cr},
${M}$ can be viewed as a holonomic $\Dan{}(0)$-module. By
construction this induces the equivalence of categories.
\end{proof}

%\begin{dfn*}\end{dfn*}

The scalar extension defines a functor denoted by $(\cdot)^{\mr{an}}$
from the category of holonomic $F$-$\DdagQ{\ms{S}}$-modules to
$F$-$\mr{Hol}(\ms{S})$.

Let $\ms{S}:=\mr{Spf}(A)$ and $\ms{S}':=\mr{Spf}(B)$ be formal disks. If
we are given a finite \'{e}tale morphism
$\tau\colon\ms{S}\rightarrow\ms{S}'$, this induces a functor
$\tau_*\colon F\mbox{-}\mr{Hol}(\ms{S})\rightarrow
F\mbox{-}\mr{Hol}(\ms{S}')$, and also the pull-back $\tau^*$. In the
same way, if we are given a finite \'{e}tale morphism of generic points
$\tau'\colon\eta_\ms{S}\rightarrow\eta_{\ms{S}'}$, this defines a functor
$\tau'_*\colon F\mbox{-}\mr{Hol}(\eta_\ms{S})\rightarrow
F\mbox{-}\mr{Hol}(\eta_{\ms{S}'})$ and the pull-back $\tau'^*$, and the
same for $\mr{Hol}'(\eta_{\ms{S}})$ etc. %\index{functors!$\tau_*$, $\tau^*$, $\tau'_*$, $\tau'^*$}

\begin{rem*}\label{Rem_on_S_tilde}
 We note here that
 $\widetilde{\ms{S}}$ consists of a single point. There is no
 problem as long as we only consider finite \'{e}tale morphisms of
 $\widetilde{\ms{S}}$ like $\tau$ above, but in this paper, we need to
 use push-forwards and pull-backs in the situation where only morphisms
 on $\eta_{\ms{S}}$ like $\tau'$ are defined. Under this situation,
 adding the generic point $\eta_{\ms{S}}$ by considering $\ms{S}$
 instead of $\widetilde{\ms{S}}$ makes descriptions much simpler.
\end{rem*}

\subsubsection{}\label{coh_op_formal_disk}
Let us anticipate here some definitions for
$F$-$\ms{D}^{\dagger}$-modules on a formal disk, that are used in rest
of this section; for more details cf.\ \cite[3.4]{Cr} or
\ref{defvanishcycle}.

For any $\DdagQ{\widetilde{\ms{S}}}$-module $\ms{M}$ we put
$j^+\ms{M}:=\DdagQ{\widetilde{\ms{S}}}(0)\otimes_{\DdagQ{\widetilde{\ms{S}}}}
\ms{M}$; \index{functors!six operations!j@$j^+$}
conversely for any $\DdagQ{\widetilde{\ms{S}}}(0)$-module $\ms{N}$ we
define $j_+\ms{N}$ the $\DdagQ{\widetilde{\ms{S}}}$-module obtained from
$\ms{N}$ by restriction of scalars via
$\DdagQ{\widetilde{\ms{S}}}\hookrightarrow
\DdagQ{\widetilde{\ms{S}}}(0)$. \index{functors!six operations!j@$j_+$}
Similar definitions hold for 
$\Dan{\widetilde{\ms{S}},\QQ}$-modules and
$\Dan{\widetilde{\ms{S}},\QQ}(0)$-modules.
The $F$-$\DdagQ{\widetilde{\ms{S}}}$-module $\delta$\index{punctual
(type)!the $F$-$\Dan{}$-module $\delta$} is by definition the holonomic
$F$-$\DdagQ{\widetilde{\ms{S}}}$-module
\begin{equation*}
 \delta:=\DdagQ{\widetilde{\ms{S}}}/\DdagQ{\widetilde{\ms{S}}} \cdot
t=\mc{R}^b/\mc{O}_{\ms{S},\mb{Q}}=\Rob/\mc{O}^{\mr{an}}.
\end{equation*}
By construction $\delta$ is holonomic and $\delta=\delta^{\mr{an}}$.
We say that a $\DdagQ{\widetilde{\ms{S}}}$-module $\ms{M}$ is
\emph{punctual} (or \emph{punctual type})\index{punctual (type)} if
there exists a finite dimensional $K$-vector space such that
$\ms{M}\cong i_+ V :=\delta\otimes_K V$. \index{functors!six
operations!i@$i_+$} %, cf.\ \cite[\S3.4, p.\ 243]{Cr}.
By construction a punctual $F$-$\DdagQ{\widetilde{\ms{S}}}$-module is
automatically holonomic.

We say that a $\DdagQ{\widetilde{\ms{S}}}$-module $\ms{M}$ is
\emph{connection type}\index{connection type} if  there exists a
$\DdagQ{\widetilde{\ms{S}}}(0)$-module $\ms{N}$ such that $\ms{M}\cong
j_+\ms{N}$. For connection type $F$-$\DdagQ{\widetilde{\ms{S}}}$-modules
holonomicity is equivalent to coherence; and if $\ms{M}$ is a coherent
$F$-$\DdagQ{\widetilde{\ms{S}}}$-module, then $\ms{M}$ is connection
type if and only if the natural map $\ms{M}\rightarrow j_+j^+ \ms{M}$ of
$\DdagQ{\widetilde{\ms{S}}}$-modules  is an isomorphism.
We have similar definitions of punctual and connection type hold for
$\Dan{\widetilde{\ms{S}},\QQ}$-modules.

\subsubsection{}
\label{notationofCrint}
Let us recall more notation and definitions from 
\cite{Cr}, see [\emph{loc.\ cit.\ }\S1] for details. 
The bounded Robba ring $\mc{R}^b=\mc{R}_{u,K}^b$ is a discrete valuation
field with respect to the 1-Gauss norm.
Let us denote by $\mc{O}_{\mc{R}^b}$ \index{Robba's
rings!.2@$\mc{O}_{\mc{R}^b}$, $\mc{R}^b(\mc{L})$, $\mc{R}(\mc{L})$|(}
its integer ring and by $\mc{K}=k\pp{\bar{u}}$ its residue field, where
$\bar{u}$ is the class of $u$.
Choose a separable closure
$\mc{K}^{\mr{sep}}$\index{.@miscellaneous!K@$\mc{K}$,
$\mc{K}^{\mr{sep}}$}. We set
$G_{\mc{K}}:=\mr{Gal}(\mc{K}^{\mr{sep}}/\mc{K})$, and let
$I_{\mc{K}}$\index{.@miscellaneous!G@$G_{\mc{K}}$, $I_{\mc{K}}$}
be the inertia subgroup.
Since $\mc{O}_{\mc{R}^b}$ is henselian, given a finite separable
extension $\mc{L}$ of $\mc{K}$, there exists a unique finite unramified
extension $\mc{R}^b(\mc{L})$ \index{Robba's
rings!.2@$\mc{O}_{\mc{R}^b}$, $\mc{R}^b(\mc{L})$, $\mc{R}(\mc{L})$|)} of
$\mc{R}^b$ whose residue field (of its integer ring) is $\mc{L}$. Put
$\mc{R}(\mc{L}):=\mc{R}\otimes_{\mc{R}^b}\mc{R}^b(\mc{L})$. Let $h$ be a
positive integer, and we put $q:=p^h$. We fix a
lifting of $h$-th Frobenius $\sigma$ of $\mc{K}$ on $\mc{O}_{\mc{R}^b}$,
which induces the Frobenius homomorphism on $\mc{R}^b$ and $\mc{R}$,
also denoted by $\sigma$. This extends canonically to
$\mc{R}(\mc{L})$. Now, we put $\mc{B}_0:=\indlim_{\mc{L}}\mc{R}(\mc{L})$
\index{padic@$p$-adic hyperfunctions and microfunctions!$\mc{B}_0$,
$\mc{B}$, $\mc{C}$|(}
where $\mc{L}$ runs through finite separable extensions of $\mc{K}$
inside $\mc{K}^{\mr{sep}}$. Then this ring is
naturally equipped with a $G_{\mc{K}}$-action, and a Frobenius
homomorphism $\sigma$. We formally add ``$\log$'' to get the ring of
hyperfunctions:
we define $\mc{B}:=\mc{B}_0[\log(u)]$. 
The action of $G_{\mc{K}}$
extends canonically to $\mc{B}$;  so does $\sigma$; 
and we also have
the monodromy operator, which is the derivation by $\log(u)$; cf.\
[\emph{loc.\ cit.\ }\S1.4].
 We put\footnote{In \cite[6.1]{Cr}, he defined
$\mc{O}^{\mr{an}}_{K^{\mr{ur}}}$ to be
$\mc{O}^{\mr{an}}\widehat{\otimes}_KK^{\mr{ur}}$, but this should be a
typo.} $\mc{O}^{\mr{an}}_{K^{\mr{ur}}}:=\mc{O}^{\mr{an}}\otimes_K
K^{\mr{ur}}$, where $K^{\mr{ur}}$ denotes the maximal unramified field
extension of $K$ and $\mc{O}^{\mr{an}}$ is $\mc{A}_{u,K}(\left[0,1\right[)$, 
cf.\ \ref{diffmoddef}.
\index{rings of analytical functions!{$\mc{O}^{\mr{an}}_{K^{\mr{ur}}}$}} 
Crew defined an
$\mc{O}^{\mr{an}}_{K^{\mr{ur}}}$-module by
$\mc{C}:=\mc{B}/\mc{O}^{\mr{an}}_{K^{\mr{ur}}}$,\index{padic@$p$-adic
hyperfunctions and microfunctions!$\mc{B}_0$, $\mc{B}$, $\mc{C}$|)}
cf. [\emph{loc.\ cit.\ }(6.1.1)]. The
action of $G_{\mc{K}}$, the endomorphism $\sigma$, and the
nilpotent operator $N$, induce, by quotient, analogous structures on $\mc{C}$. 
We denote by $\mr{can}\colon\mc{B}\rightarrow\mc{C}$ the canonical
projection. By definition, the derivation
$N\colon\mc{B}\rightarrow\mc{B}$ factors through $\mr{can}$, and we get
$\mr{var}\colon\mc{C}\rightarrow\mc{B}$. These homomorphisms satisfy the
relations $N=\mr{can}\circ\mr{var}$ and $N=\mr{var}\circ\mr{can}$, cf.\
[\emph{loc.\ cit.\ }\S6.1].\index{padic@$p$-adic hyperfunctions and
microfunctions!$\mr{can}\colon\mc{B}\rightarrow\mc{C}$,
$\mr{van}\colon\mc{C}\rightarrow\mc{B}$}

\subsubsection{}
\label{mainresultscrewrev}
Let $\Del{\Ct\nr}{G_{\mc{K}}}$
\index{categories!Del@$\Del{\Ct\nr}{G_{\mc{K}}}$} denote the category of
Deligne\index{Deligne module} modules\footnote{
This terminology was first introduced by Fontaine in \cite[\S
1]{Fon}. These are also called $(\varphi,N,G_{\mc{K}})$-modules, and
this terminology is used more widely, especially in $p$-adic Hodge
theory. However, in our context, we
believe that ``Deligne module'' is more suitable.}:
{\itshape i.e.\ }finite dimensional $\Ct\nr$-vector spaces, endowed with
a semi-linear action of $G_{\mc{K}}$ (which acts on the
constants $\Ct\nr$ via its residual action), a  Frobenius isomorphism
$\Fro$, and a mo\-no\-dromy operator $N$, satisfying $N\Fro= q\,\Fro N$
where $q=p^h$ in \ref{notationofCrint}.
See \cite[\S3.1]{Marmora:Facteurs_epsilon} for more details.

In the following, for simplicity, we denote $\Dan{\ms{S},\mb{Q}}$ by
$\Dan{}$.\index{differential operators!Dan@$\Dan{}$} Crew classifies holonomic
$F$-$\Dan{}$-modules in terms of linear data (cf.\ \cite[6.1]{Cr}). To
do this, let ${M}$ be a holonomic $F$-$\Dan{}$-module. He defined
in {\it loc.\ cit.},
\begin{equation*}
 \mb{V}({M}):=\mr{Hom}_{\ms{D}^{\mr{an}}}({M},\mc{B}),\qquad
  \mb{W}({M}):=\mr{Hom}_{\ms{D}^{\mr{an}}}({M},\mc{C}). \index{functors!.V@$\mb{V}$,
  $\mb{W}$}
\end{equation*}
These are Deligne modules, and define (contravariant) functors
$\mb{V},\mb{W}\colon
F\mbox{-}\mr{Hol}(\Dan{})\rightarrow\Del{\Ct\nr}{G_{\mc{K}}}$. There are
the canonical homomorphism $\mb{V}({M})\rightarrow\mb{W}({M})$
induced by $\mr{can}$, and the variation homomorphism
$\mb{W}({M})\rightarrow\mb{V}({M})$ induced by $\mr{var}$. These
satisfy many compatibilities, they are endowed with an extra structure 
(the
``Galois variation'', which we do not recall here). These kind of
objects are called \emph{Solution data}, they form an artinian category,
denoted  $\mathrm{Sol}_K$,\index{categories!S@$\mathrm{Sol}_K$} and we
have an exact functor $\mb{S}\colon F$-$\mr{Hol}(\ms{S}) \rightarrow
\mathrm{Sol}_K$,
$\mb{S}(M)=(\mb{V}(M),\mb{W},\mathrm{can},\mathrm{var},\cdots)$; for the
details see {\it loc.\ cit}.
The
main point is that we can retrieve the original
module ${M}$ from these linear data: 
indeed Crew constructs a quasi-inverse functor $\mb{M}^{\mr{an}}\colon
\mathrm{Sol}_K \rightarrow F$-$\mr{Hol}(\ms{S})$ factorizing through the
category of holonomic $F$-$\DdagQ{\ms{S}}$-modules, cf.\
[{\it loc.\ cit}, 7.1]. \index{functors!.S@$\mb{S}$, $\mb{M}^{\mr{an}}$}

We can characterize
some properties of ${M}$ in terms of these linear data. For example
${M}$ is connection type
%is defined by a free differential $\mc{R}$-module with Frobenius %structure 
if and only if the canonical map
$\mb{V}({M})\rightarrow\mb{W}({M})$ is an isomorphism, cf.\ \cite[Cor.\ 6.1.2]{Cr}. The most
important property for us is the existence of the following exact
sequence of Deligne modules (cf.\ [{\it loc.\ cit.}, Cor.~6.1.1, Lemma~6.1.2]).
\begin{align}
 \label{vanishexact}
 &0\rightarrow\mr{Hom}_{\ms{D}^{\mr{an}}}({M},\mc{O}^{\mr{an}}
 _{K^{\mr{ur}}})\rightarrow\mb{V}({M})\rightarrow\mb{W}({M})
 \rightarrow\mr{Ext}^1_{\ms{D}^{\mr{an}}}({M},\mc{O}_{K^{\mr{ur}}}
 ^{\mr{an}})\rightarrow 0, \\
 &\forall i\geq 2,\quad
 \mr{Ext}^i_{\ms{D}^{\mr{an}}}({M},\mc{O}_{K^{\mr{ur}}}^{\mr{an}})=
\mr{Ext}^i_{\ms{D}^{\mr{an}}}({M},\mc{B})=
 \mr{Ext}^i_{\ms{D}^{\mr{an}}}({M},\mc{C})=0. \label{vanishexact2}
\end{align}
Let us mention another property which will be useful later.
\begin{lem*}\label{W=0}
 Let $M$ be an holonomic $F$-$\Dan{\ms{S},\mb{Q}}$-module. If
 $\mb{W}(M)=0$ then $M$ is free as differential 
 $\mc{O}^{\mr{an}}_{\ms{S}}$-module, cf.\ {\normalfont \ref{def_free_mod}}.
\end{lem*}
\begin{proof}
 By \eqref{vanishexact} it follows that $\mb{V}(M)$ is geometrically
 constant, which means that we have an isomorphism
 ${\mb{V}(M)}^{G_{\mc{K}}}\otimes_K K\nr \cong \mb{V}(M)$ of Deligne
 modules. By the construction of the functor  $\mb{M}^{\mr{an}}$ it
 follows immediately that natural evaluation map $M\rightarrow
 \mr{Hom}_{K\nr}({\mb{V}(M)},
 \mc{O}_{K^{\mr{ur}}}^{\mr{an}})^{^{G_{\mc{K}}}}$ is an isomorphism of
 $\Dan{\ms{S},\mb{Q}}$-modules (here $\mr{Hom}_{K\nr}$ denotes
 homomorphisms of $K\nr$-vector spaces). We have isomorphisms
 $\Dan{\ms{S},\mb{Q}}$-modules:
 \begin{equation*}
  \begin{split}
   \mr{Hom}_{K\nr}({\mb{V}(M)}, \mc{O}_{K^{\mr{ur}}}^{\mr{an}}
   )^{^{G_{\mc{K}}}}&\cong
   \mr{Hom}_{K\nr}({\mb{V}(M)}^{G_{\mc{K}}}\otimes_K K\nr,
   \mc{O}_{K^{\mr{ur}}}^{\mr{an}} )^{^{G_{\mc{K}}}}\cong
   \mr{Hom}_{K}({\mb{V}(M)}^{G_{\mc{K}}},
   \mc{O}_{K^{\mr{ur}}}^{\mr{an}})^{^{G_{\mc{K}}}}\\
   &\cong \mr{Hom}_{K}({\mb{V}(M)}^{G_{\mc{K}}},
   \mc{O}^{\mr{an}}_{\ms{S}} )\cong
   (\mc{O}^{\mr{an}}_{\ms{S}})^{\oplus\mr{rk}(M)},
  \end{split}
 \end{equation*}
 which concludes the proof.
\end{proof}
 
\subsubsection{}
\label{preisouadsf}
Let $\ms{X}$ be a formal curve over $R$, and let $\ms{M}$ be a holonomic
$\DdagQ{\ms{X}}$-module. Let $x\in\ms{X}$ be a closed point. We denote
by $\ms{S}_x:=\mr{Spf}(\widehat{\mc{O}}_{\ms{X},x})$ \index{.@miscellaneous!Si@$\ms{S}_x$}
where  $\widehat{\mc{O}}_{\ms{X},x}$ is the
completion of $\mc{O}_{\ms{X},x}$ for the $\mathfrak{m}_{\ms{X},x}$-topology.
 Let $\ms{M}$ be a coherent $\DdagQ{\ms{X}}$-module on
$\ms{X}$. Take an open affine neighborhood $\ms{U}$ of $x$, and we
denote by $\DdagQ{\ms{S}_x}\otimes\ms{M}$ the coherent
$\DdagQ{\ms{S}_x}$-module on $\ms{S}_x$ 
\begin{equation*}
 {\bigl(\DdagQ{\widetilde{\ms{S}_x}}\otimes_{\Gamma(\ms{U},
  \DdagQ{\ms{X}})}\Gamma(\ms{U},\ms{M})\bigr)}^{\triangle},
\end{equation*}
cf.\ \ref{Cr_recall1}.
This does not depend on the choice of $\ms{U}$, and we will also denote by
$\DdagQ{\ms{S}_x}\otimes\ms{M}$ its global sections on $\ms{S}_x$.
 For a holonomic
$F$-$\DdagQ{\ms{X}}$-module $\ms{M}$ (cf.\ \ref{setupFrob}), we put
\begin{equation*}
 \label{exactforloc}
 \ms{M}|_{S_x}:=(\DdagQ{\ms{S}_x}\otimes\ms{M})^{\mr{an}},\qquad
  \ms{M}|_{\eta_x}:=(\ms{D}^\dag_{\ms{S}_x,\mb{Q}}(0)
  \otimes\ms{M})^{\mr{an}}, \index{functors!.3@$\quad"|_{S_x}$, $\quad"|_{\eta_x}$}
\end{equation*}
which are defined in $F$-$\mr{Hol}(\ms{S}_x)$ and
$F$-$\mr{Hol}(\eta_{\ms{S}_x})$ respectively.
For example we have $\mc{O}_{\ms{X},\mb{Q}}|_{S_x} = \mc{O}^{\mr{an}}_{\ms{S}_x}$ and 
$\mc{O}_{\ms{X},\mb{Q}}|_{\eta_x}=\Rob_{\ms{S}_x}$.
The notation $\ms{M}|_{S_x}$ (resp. $\ms{M}|_{\eta_x}$)   is slightly abusive:  indeed  $\ms{M}|_{S_x}$ 
(resp. $\ms{M}|_{\eta_x}$) depends, up to an  isomorphism, 
  on the choice of a 
local parameter of $\ms{X}$ at $x$, cf.\ Remark
\ref{rem_indipendec_par}.

The following lemma combined with \cite[Th.~4.1.1]{Cr} shows that the
functors $|_{{S}_x}$ and $|_{\eta_x}$ are exact.

\begin{lem*}
The functor $\DdagQ{\ms{S}_x}\otimes(-)$ from
 the category of coherent $\DdagQ{\ms{X}}$-modules to that of coherent
 $\DdagQ{\ms{S}_x}$-modules is exact.
\end{lem*}
\begin{proof}
 Let $S_i$ be $\widetilde{\ms{S}}_x\otimes R_i$. In this proof, we
 denote $\Gamma(\ms{U}\otimes R_i,\Dmod{m}{X_i})$ by $\Dmod{m}{U_i}$.
 It suffices to show that the canonical homomorphism
 $\Dmod{m}{U_i}\rightarrow\Dmod{m}{S_i}$ is flat. For this, it suffices
 to show that $\mr{gr}(\Dmod{m}{U_i})\rightarrow\mr{gr}(\Dmod{m}{S_i})$
 is flat where the $\mr{gr}$ is taken with respect to the filtration by
 order. This follows from the flatness of
 $\mc{O}_{X_i}\rightarrow(\mc{O}_{X_i,x})^\wedge$ where the completion
 is taken with respect to the $\mf{m}_x$-topology.
\end{proof}

\subsubsection{}
\label{intcanext}
Let us recall the canonical extension,
which is one of key tools in this paper. For the detailed
argument, one can refer to \cite[\S 8]{Cr}. Let
$\Pone:=\widehat{\mb{P}}^1_R$ and $\ms{S}_0:=\mr{Spf}(\widehat{\mc{O}}_{\Pone,0})$. \index{.@miscellaneous!P@$\Pone$}
Then there exists a functor
\begin{equation*}
  F\mbox{-}\mr{Hol}(\ms{S}_0)\rightarrow
   F\mbox{-}\mr{Hol}(\DdagQ{\Pone}(\infty));\qquad
 {M}\mapsto{M}^{\mathrm{can}}\notag
  \index{functors!.2@${(-)}^{\mathrm{can}}$} 
\end{equation*}%\begin{eqnarray}
% \label{fonct:ext_can}\notag
%  F\mbox{-}\mr{Hol}(\ms{S}_0)& \rightarrow & 
%  F\mbox{-}\mr{Hol}(\DdagQ{\Pone}(\infty))\\
% {M} & \mapsto & {M}^{\mathrm{can}}\notag
%  \index{functors!.2@${(-)}^{\mathrm{can}}$}
%\end{eqnarray}
where $F$-$\mr{Hol}(\DdagQ{\Pone}(\infty))$
%\index{categories!$F$-$\mr{Hol}(\DdagQ{\Pone}(\infty))$}
denotes the category of holonomic $F$-$\DdagQ{\Pone}(\infty)$-modules.
By construction, this functor is fully faithful, exact, and it commutes
with tensor products and duals. Moreover, we have the following
properties (cf.\ \emph{loc.\ cit.\ }Th.\ 8.2.1 and paragraph after its
proof):

1.\ ${M}^{\mr{can}}|_{\Pone\setminus\{0\}}$ is a ``special'' convergent
isocrystal;

2.\ ${M}^{\mr{can}}|_{S_0}\cong {M}$; 

3.\
${M}^{\mr{can}}$ regular at $\infty$ (for the definition of regularity cf.\ \ref{def_regularity}). 

\noindent
This ${M}^{\mr{can}}$ is
called the canonical extension of ${M}$. By these properties, we
remind that when ${M}$ is a free differential module,
${M}^{\mr{can}}$ coincides with the canonical extension of Matsuda in
\cite[7.3]{Matsuda:Katz_corr}, and in this case, it sends unit-root
objects to unit-root objects.

\subsubsection{} 
In \ref{characterisation_of_Dan0_hol_mod} we have characterized holonomic $F$-$\Dan{\widetilde{\ms{S}},\mb{Q}}(0)$-modules. 
Let us conclude this subsection with a lemma characterizing holonomic
$F$-$\ms{D}^{\mr{an}}$-modules.

\begin{lem*}\label{characterisation_of_Dan_hol_mod}
 Let ${M}$ be an $F$-$\Dan{\widetilde{\ms{S}},\mb{Q}}$-module. Assume
 that ${M}(0):=\Dan{\widetilde{\ms{S}},\mb{Q}}(0)\otimes {M}$ is a
  free differential $\mc{R}$-module, and the kernel and cokernel
 of the canonical homomorphism $\alpha\colon {M}\rightarrow {M}(0)$
 are punctual $\DdagQ{\widetilde{\ms{S}}}$-modules. Then ${M}$ is a
 holonomic $F$-$\Dan{\widetilde{\ms{S}},\mb{Q}}$-module. In particular,
 if there exists a holonomic $F$-$\DdagQ{\widetilde{\ms{S}}}$-module
 ${\ms{M}}$ such that
 ${\ms{M}}^{\mr{an}}\cong {M}$ as
 $\Dan{\widetilde{\ms{S}},\mb{Q}}$-modules {\em without Frobenius
 structures}, then ${M}$ is a holonomic
 $F$-$\Dan{\widetilde{\ms{S}},\mb{Q}}$-module.
\end{lem*}
\begin{proof}
 We denote by $\mc{C}$ the full subcategory of the category of
 $\Dan{\widetilde{\ms{S}},\mb{Q}}$-modules with Frobenius structure
 consisting of objects considered in the statement of the lemma. We
 define functors
 $\mb{V}$ and $\mb{W}$ in the same way as \cite[6.1]{Cr} or especially
 [{\it loc.\ cit.\ }(6.1.9)] (cf.\ also \ref{mainresultscrewrev}). 
We first claim that $\mb{V}$ and $\mb{W}$ are exact
 functors. To see this, consider the following exact sequences
 \begin{gather*}
  0\rightarrow {N}_1\rightarrow {M}\rightarrow {M}'\rightarrow 0,\\
  0\rightarrow {M}'\rightarrow {M}(0)\rightarrow {N}_2\rightarrow0,
 \end{gather*}
 where ${N}_1$ and ${N}_2$ are the kernel and cokernel of $\alpha$
 respectively, and thus punctual $\DdagQ{\widetilde{\ms{S}}}$-modules by
 assumption. By hypothesis and Proposition \ref{characterisation_of_Dan0_hol_mod}, we know that ${M}(0)$ and ${N}_2$ are holonomic
 $F$-$\Dan{\widetilde{\ms{S}},\mb{Q}}$-modules, and thus by considering
 the canonical extensions (cf.\ \ref{intcanext}), we get that
 ${M}'$ is a holonomic $F$-$\Dan{\widetilde{\ms{S}},\mb{Q}}$-module
 as well. Thus, we get
 $\mr{Ext}^i({M}',\mc{B})=\mr{Ext}^i({M}',\mc{C})=0$ for
 $i>0$. Considering the long exact sequence induced by the first short
 exact sequence above, we get
 $\mr{Ext}^i({M},\mc{B})=\mr{Ext}^i({M},\mc{C})=0$ for $i>0$, cf.\ \cite[Lemma 6.1.2]{Cr}.

 Repeating the same construction of [\emph{loc.\ cit.\ }(6.1.10)] and
 arguments of [\emph{loc.\ cit.\ }Th.~6.1.1], we get an exact functor
 $\mb{S}'=(\mb{V},\mb{W},\ldots)$ from $\mc{C}$ to the category
 of solution data. It suffices to construct a canonical isomorphism
 ${M}\rightarrow\mb{M}^{\mr{an}}(\mb{S}'({M}))$. This can be shown
 in the same way as [\emph{loc.\ cit.\ }Th.~7.1.1].

 The last claim of the lemma follows from the first 
 because the analytification functor is exact, 
 cf. [\emph{loc.\ cit.\ }Prop.~5.1.1];  
 for any holonomic $F$-$\DdagQ{\widetilde{\ms{S}}}$-module $\ms{M}$,  we
 have
 $\ms{M}^{\mr{an}}\cong\Dan{\widetilde{\ms{S}},\mb{Q}}\otimes_
 {\DdagQ{\widetilde{\ms{S}}}}\ms{M} $, cf.\ [\emph{loc.\ cit.\ }Th.\ 4.1.1],
  $\ms{M}^{\mr{an}}(0) \cong (j_+j^+\ms{M})^{\mr{an}}$ 
 [\emph{loc.\ cit.\ }Lemma 4.1.4], and the analytification of a punctual
 module is punctual, cf.\ [\emph{loc.\ cit.\ }Prop.~5.1.1].
\end{proof}

\subsection{Analytification}
\label{analye}
In \S\ref{secdefforloc}, we will define a local Fourier
transform. The local Fourier transform should be local not only in the
scheme theoretic sense, but also ``rigid analytically''. In this
subsection, we will prove a crucial tool (cf.\ Proposition
\ref{analytifcisom}) which is indispensable to prove such properties.

\subsubsection{}
\label{defanalye}
In this subsection, we do not assume $k$ to be perfect.
Let $\ms{X}$ be a formal curve, and take a closed point
$x$ in $\ms{X}$. Choose an isomorphism
$\widehat{\mc{O}}_{\ms{X},x}\cong R_x\cc{t}$,
where $R_x:=R_{\widehat{\mc{O}}_{\ms{X},x}}$, cf.\
\ref{existconcdesc}. We
will define a ring $(\Ecomp{m,m'}{x})^{\mr{an}}$\index{microdifferential
operators!analytic!$(\EcompQ{m}{x})^{\mr{an}}$,$(\Ecomp{m,m'}{x})^{\mr{an}}$,
$(\Emod{m,\dag}{x,\mb{Q}})^{\mr{an}}$,$\ms{E}_{x,\mb{Q}}^{\mr{an}}$|(}
in the following way. Let us use the notation of \ref{Cr_recall1}.
%For a positive integer $r$ and a non-negative integer $i$, we
%define $\mc{O}_{r,i}$ to be $\widehat{\mc{O}}_{X_i,x}[T]/(pT-t^r)$ as
%defined in \cite[4.1]{Cr}. When $r$ is divided by $p^{m+1}$,
%$\mc{O}_{r,i}$ possesses a $\Dmod{m}{X_i}$-module structure (cf.\
%\cite[Lemma 3.1.1]{Cr}).
%We define a {\em ring}
%\begin{equation*}
 %(\DcompQ{m}{x})^{\mr{an}}:=\invlim_{n}\bigl((\invlim_i
  %\mc{O}_{np^{m+1},i}\otimes_{\mc{O}_{X_i}}\Dmod{m}{X_i})
  %\otimes\mb{Q}\bigr).
%\end{equation*}
%Although it is not defined explicitly in {\it loc.\ cit.}, this ring is
%used to define
%$\ms{D}^{\mr{an}}_{x,\mb{Q}}:=\Gamma(\ms{S}_x,\ms{D}^{\mr{an}}_
%{\ms{S}_x,\mb{Q}})$, which is the inductive limit of
%$(\DcompQ{m}{x})^{\mr{an}}$ over $m$. 
To construct the analytification
of $\EcompQ{m}{\ms{X}}$, we follow exactly the same way as in
\cite{Abe}; there are several steps. For the first step, we take the
microlocalization of
$\mc{O}_{np^{m+1},i}\otimes_{\mc{O}_{X_i}}\Dmod{m}{X_i}$ with respect to
the filtration by order (cf.\ \cite[2.1]{Abe}) and denote it
by $\Emod{m}{np^{m+1},x,i}$. Second step, take the inverse limit over
$i$, namely $\invlim_i\Emod{m}{np^{m+1},x,i}$ and denote it by
$\Ecomp{m}{np^{m+1},x}$. We put
$\EcompQ{m}{np^{m+1},x}:=\Ecomp{m}{np^{m+1},x}\otimes\mb{Q}$, and we
take  the inverse limit over $n$ to define
$(\EcompQ{m}{x})^{\mr{an}}$. 

Now, for an integer $m'\geq m$, we want to
define the analytification of $\EcompQ{m,m'}{\ms{X}}$. Also for this, we
follow the same way as {\it loc.\ cit}. Put $c:=p^{m'+1}$. Let $a$ be
either $m$ or $m'$. Then we define $\Emod{a}{nc,x}$ to be the subring of
$\Ecomp{a}{nc,x}$ consisting of the finite order operators. Then we may
prove in the same way as {\it loc.\ cit.\ }that there exists a canonical
homomorphism
\begin{equation*}
 \psi_{m,m'}\colon\Emod{m'}{nc,x}\otimes\mb{Q}\rightarrow\Emod{m}{nc,x}
  \otimes\mb{Q}.
\end{equation*}
We define $\Ecomp{m,m'}{nc,x}$ to be the $p$-adic completion of
$\psi_{m,m'}^{-1}(\Emod{m}{nc,x})\cap\Emod{m'}{nc,x}$. This ring is
noetherian by the same argument as \cite[4.12]{Abe}. We define
$\EcompQ{m,m'}{nc,x}:=\Ecomp{m,m'}{nc,x}\otimes\mb{Q}$.

Finally we define
\begin{equation*}
 (\EcompQ{m,m'}{x})^{\mr{an}}:=\invlim_{n}\Ecomp{m,m'}
  {nc,x,\mb{Q}}.
\end{equation*}
Obviously, there exists the canonical inclusion
$\EcompQ{m,m'}{x}\rightarrow(\EcompQ{m,m'}{x})^{\mr{an}}$.
In the same way as for $\EdagQ{\ms{X}}$, we define
\begin{equation*}
 (\Emod{m,\dag}{x,\mb{Q}})^{\mr{an}}:=\invlim_k~(\EcompQ{m,m+k}{x})
  ^{\mr{an}},\qquad\ms{E}_{x,\mb{Q}}^{\mr{an}}:=\indlim_m~(\ms{E}
  ^{(m,\dag)}_{x,\mb{Q}})^{\mr{an}}. 
\end{equation*}
We point out that the rings $(\EcompQ{m,m'}{x})^{\mr{an}}$,
$(\Emod{m,\dag}{x,\mb{Q}})^{\mr{an}}$, $\ms{E}_{x,\mb{Q}}^{\mr{an}}$
does not depend on the choice of the isomorphism 
$\widehat{\mc{O}}_{\ms{X},x}\cong R_x\cc{t}$ (cf.\ Remark
\ref{Crewconstan}). As an example, we have the following explicit
description whose verification is left to the reader.  We recall, cf.\ \ref{fieldconstmicdif}, that, for $0<r<1$,
 $|\cdot|_r$ denotes the $r$-Gauss norm on $\mc{O}^{\mr{an}}=\An[K_x,t]([0,1[)$, we
have
\begin{equation*}
 (\EcompQ{m,m'}{x})^{\mr{an}}=\left\{\sum_{k<0}a_k\partial
  ^{\angles{m'}{k}}+\sum_{k\geq0}b_k\partial^{\angles{m}{k}}
  \Bigg\arrowvert\parbox{7.5cm}{$a_k,b_k\in\mc{O}^\mr{an}$, and for any
  $0<r<1$, there exists $C_r>0$ such that $|a_k|_r<C_r$ for any $k$ and
  $\lim_{k\rightarrow\infty}|b_k|_r=0$.}\right\}. \index{microdifferential operators!analytic!$(\EcompQ{m}{x})^{\mr{an}}$,$(\Ecomp{m,m'}{x})^{\mr{an}}$,$(\Emod{m,\dag}{x,\mb{Q}})^{\mr{an}}$,$\ms{E}_{x,\mb{Q}}^{\mr{an}}$|)}
\end{equation*}

At last, let us fix a few notation. Let $\ms{M}$ be a coherent
$\DcompQ{m}{\ms{X}}$-module. Let $m''\geq m'\geq m$, and $\ms{E}$ be one
of $(\DcompQ{m'}{x})^{\mr{an}}=(\DcompQ{m'}{\widetilde{\ms{S}}})^{\mr{an}}$, $(\EcompQ{m',m''}{x})^{\mr{an}}$,
$(\Emod{m',\dag}{x,\mb{Q}})^{\mr{an}}$,
$\ms{E}^{\mr{an}}_{x,\mb{Q}}$. We denote
$\ms{E}\otimes_{(\DcompQ{m}{\ms{X}})_x}\ms{M}_x$ by
$\ms{E}\otimes_{\Dcomp{m}{}}\ms{M}$ or $\ms{E}\otimes\ms{M}$. This
notation goes together with Notation \ref{setup}.

\subsubsection{}
\label{deftopoean}
We put a topologies $\ms{T}_{n'}$ for $n'\geq0$ and $\ms{T}$ on
$\EcompQ{m,m'}{nc,x}$ in
exactly the same way as in \ref{topologyonring}. Precisely, we put
$U_{k,l}:=(\Ecomp{m,m'}{nc,x})_{-k}+\varpi^l\Ecomp{m,m'}{nc,x}$, and we
define a topology $\ms{T}_0$ on $\Ecomp{m,m'}{nc,x}$ as the topology
generated by $\{U_{k,l}\}$ as a base of neighborhoods of zero. The
topology $\ms{T}_{n'}$ on $\EcompQ{m,m'}{nc,x}$ is the locally convex
topology generated by $\{\varpi^{-n'}U_{k,l}\}$ as a base of
neighborhoods of zero, and $\ms{T}$ is the inductive limit topology.
We get that $\EcompQ{m,m'}{\ms{X}}\cap\varpi^{-n'}\Ecomp{m,m'}{nc,x}$
is dense in $(\varpi^{-n'}\EcompQ{m,m'}{nc,x},\ms{T}_{n'})$ where the
intersection is taken in $\EcompQ{m,m'}{nc,x}$. Indeed putting
$\mc{O}_r:=\invlim_i\mc{O}_{r,i}$, the intersection
$\mc{O}_{\ms{X},\mb{Q}}\cap\mc{O}_{nc}$ is dense in $\mc{O}_{nc}$.
In the same way as \ref{topologyonring},
for any finitely generated
$\EcompQ{m,m'}{nc,x}$-module, the
$(\EcompQ{m,m'}{nc,x},\ms{T}_{n'})$-module topology is separated.

\begin{prop}
 \label{analytifcisom}
 Suppose we are in Situation {\normalfont(Ls)} of
 {\normalfont\ref{setup}}. Moreover, we
 assume $x=y_s$. Let $m'\geq m$ be non-negative integers, and $\ms{M}$
 be a holonomic $\DcompQ{m}{\ms{X}}$-module (not necessarily stable). We
 assume that
 \begin{equation*}
  \mr{Supp}(\EcompQ{m,m'+1}{\ms{X}}\otimes_{\DcompQ{m}{\ms{X}}}\ms{M})
   \cap\mathring{T}^*X=\pi^{-1}(s)\cap\mathring{T}^*X.
 \end{equation*}
 Then the canonical homomorphism
 \begin{equation}
  \label{EtoEan}
  \EcompQb{m,m'}{\ms{X}}\otimes_{\Dcomp{m}{}}\ms{M}\rightarrow
   (\EcompQ{m,m'}{s})^{\mr{an}}\otimes_{\Dcomp{m}{}}\ms{M}
 \end{equation}
 is an isomorphism.
\end{prop}

\begin{proof}
 Suppose that $\mr{Supp}(\EcompQ{m,m'}{\ms{X}}\otimes_{\DcompQ{m}{\ms{X}}}
 \ms{M})\cap\mathring{T}^*X$ is empty.
 In this case, the source of the homomorphism
 (\ref{EtoEan}) is $0$. Thanks to the following isomorphism
 \begin{equation*}
  (\EcompQ{m,m'}{s})^{\mr{an}}\otimes_{\Dcomp{m}{}}\ms{M}\cong
   (\EcompQ{m,m'}{s})^{\mr{an}}\otimes_{\EcompQb{m,m'}{\ms{X}}}
   \EcompQb{m,m'}{\ms{X}}\otimes_{\Dcomp{m}{}}\ms{M},
 \end{equation*}
  the target of the homomorphism is also $0$, and we get the
 lemma. Thus we may assume that
 $\mr{Supp}(\EcompQ{m,m'}{\ms{X}}\otimes_{\DcompQ{m}{\ms{X}}}\ms{M})
 \cap\mathring{T}^*X=\pi^{-1}(s)$.

 First, we will see that the source of the homomorphism has an
 $\EcompQ{m,m'}{r,s}$-module structure for some integer $r$.
 For this, let us start by remarking that $\ms{M}$ is monogenic.
 Indeed, since $\ms{X}$ is affine, the $m$-th relative Frobenius
 homomorphism can be lifted. Let us denote this lifting by
 $\ms{X}'$ and $F\colon\ms{X}\rightarrow\ms{X}'$. Let the
 $\DcompQ{0}{\ms{X}'}$-module $\ms{N}$ be
 the Frobenius descent of $\ms{M}$ (cf.\ \cite[4.1.3]{Ber2}).
 Then by \cite[Proposition 5.3.1]{Gar}, there exists a surjection
 $\DcompQ{0}{\ms{X}'}\rightarrow\ms{N}$. Apply $F^*$
 to both sides. By composing with the canonical surjection
 $\DcompQ{m}{\ms{X}}\rightarrow F^*\DcompQ{0}{\ms{X}'}$ ,
 we get a surjective morphism $\DcompQ{m}{\ms{X}}\rightarrow\ms{M}$.

 Denote by $\varphi\colon\DcompQ{m}{\ms{X}}/I\cong\ms{M}$ the induced
 isomorphism, and put $I':=(\EcompQb{m,m'+1}{\ms{X}}\cdot
 I)\cap\Ecompb{m,m'+1}{\ms{X}}$. By Lemma
 \ref{lemmaexspope}, there exist $Q\in\Ecompb{m,m'+1}{\ms{X}}
 \subset\Ecompb{m,m'}{\ms{X}}$,
 $R\in(\Ecompb{m,m'+1}{\ms{X}})_{-p^{m'+1}}$, and a positive integer $d$
 such that $x^d-\varpi Q-R\in I'$. Since the order
 of $R$ is less than $-p^{m'+1}$, there
 exists $R'\in\Ecompb{m,m'}{\ms{X}}$ such that $R=\varpi R'$, and we get
 that $x^d-\varpi S\in\Ecompb{m,m'}{\ms{X}}\cdot I'$ where
 $S=Q+R'\in\Ecompb{m,m'}{\ms{X}}$. By increasing the integer $d$, we may
 assume that $d$ is divisible by $p^{m'+1}$.
 For any element $P$ in
 $\Dcompb{m'}{\ms{X}}:=\Gamma(\ms{X},\Dcomp{m'}{\ms{X}})$, we
 get $x^d\cdot P\in P\cdot x^d+p\Dcompb{m'}{\ms{X}}$. Thus for any
 operator $D\in\Ecompb{m,m'}{\ms{X}}$ we also get $x^d\cdot
 D\in D\cdot x^d+p\Ecompb{m,m'}{\ms{X}}$. This implies that for any
 integer $n>0$, there exists $S_n\in\Ecompb{m,m'}{\ms{X}}$ such that
 \begin{equation}
  \label{divreseachn}
  x^{nd}-\varpi^nS_n\in\Ecompb{m,m'}{\ms{X}}\cdot I'.
 \end{equation}

 Let $e$ be the absolute ramification index of $R$, and take
 $r\geq(e+1)d$, which is divisible by $p^{m'+1}$.
 Let $\alpha\in\Gamma(\ms{X},\ms{M})$. We claim the following.

 \begin{cl}
  For any sequence $\{P_i\}_{i\geq0}$ in
  $\EcompQb{m,m'}{\ms{X}}\cap\Ecomp{m,m'}{r,s}$ which converges to
  $0$ seen as a sequence in $(\Ecomp{m,m'}{r,s},\ms{T}_0)$, the sequence
  $\{P_i\cdot(1\otimes\alpha)\}$ in
  $\EcompQb{m,m'}{\ms{X}}\otimes\ms{M}$ converges to $0$ using the
  $(\EcompQb{m,m'}{\ms{X}},\ms{T}_0)$-module topology. In particular the
  sequence converges to $0$ using the natural topology which makes the
  module an LF-space by Lemma \ref{topology} and Lemma \ref{finiteness}.
 \end{cl}

 Let us admit this claim first, and see that there exists a canonical
 $\EcompQ{m,m'}{r,s}$-module structure on
 $\EcompQb{m,m'}{\ms{X}}\otimes\ms{M}$. For
 $P\in\Ecomp{m,m'}{r,s}$, we may write $P=\sum_{i\geq 0}P_i$ with
 $P_i\in\EcompQ{m,m'}{\ms{X}}$ and the
 sequence $\{P_i\}$ converges to $0$ in $\Ecomp{m,m'}{r,s}$. Then the
 claim says that $\{P_i\cdot(1\otimes\alpha)\}$  converges to $0$ in
 $\EcompQb{m,m'}{\ms{X}}\otimes\ms{M}$. So we may define
 $P\cdot(1\otimes\alpha)$ by $\sum_{i\geq 0}P_i\cdot(1\otimes\alpha)$,
 and we get the action of $\Ecomp{m,m'}{r,s}$ on
 $\EcompQb{m,m'}{\ms{X}}\otimes\ms{M}$ since the latter space is
 separated.

 Let us verify the claim. By (\ref{divreseachn}) and the choice of $r$,
 we get
 \begin{equation*}
  \Bigl(\frac{x^r}{p}\Bigr)^n\equiv\varpi^nT_n\bmod
   \EcompQb{m,m'}{\ms{X}}\cdot I',
 \end{equation*}
 where $T_n\in\Ecompb{m,m'}{\ms{X}}$
 ({\it e.g.\ }$T_n=S_{(e+1)n}\in\Ecompb{m,m'}{\ms{X}}$ when
 $r=(e+1)d$). We denote by $d_{i,n}$ the
 order of the image of $T_n$ in $\Emodb{m,m'}{X_i}$.
 Put $E_r:=\Ecompb{m,m'}{\ms{X},\mb{Q}}\cap\Ecomp{m,m'}{r,s}$.
 Let $Q$ be an element of $(E_r)_N+ \varpi^{N'}E_r$ for some integers
 $N$ and $N'\geq 0$. We may write
 \begin{equation*}
  Q=\sum_{n\geq 0}Q_n\Bigl(\frac{x^r}{p}\Bigr)^n,
 \end{equation*}
 where $Q_n\in
 A:=\bigoplus_{i=0}^{r-1}\CR{\ms{X}}\{\partial\}^{(m,m')}x^i$. Then
 $Q_n\in(A)_N+\varpi^{N'}A$ for any $n\geq0$. Thus, for any $n\geq0$, we
 get
 \begin{equation*}
  Q_n\Bigl(\frac{x^r}{p}\Bigr)^n\in(\Ecompb{m,m'}{\ms{X}})_{M+N}
   +\varpi^{N'}(\Ecompb{m,m'}{\ms{X}})+\EcompQb{m,m'}{\ms{X}}\cdot I',
 \end{equation*}
 where $M:=\max\{d_{i,N'-1}\mid i=1,\dots,N'-1\}=d_{N'-1,N'-1}$. Summing
 up, there exists an increasing sequence of integers
 $\{M_k\}_{k\geq 0}$ such that if $Q\in(E_r)_N+\varpi^{N'}E_r$, then
 $Q\in(\Ecompb{m,m'}{\ms{X}})_{M_{N'}+N}+\varpi^{N'}
 (\Ecompb{m,m'}{\ms{X}})+\EcompQb{m,m'}{\ms{X}}\cdot I'$. We can find
 a sequence of integers $\{N_k\}_{k\geq 0}$ such that the sequence
 $\{N_k+M_{k}\}_{k\geq0}$ is strictly decreasing. Back to the claim,
 for any integer $k\geq0$, there exists $n_k$ such that
 $P_j\in(E_r)_{N_k}+\varpi^{k}E_r$ for any $j\geq n_k$. Then, we have $P_j\cdot\alpha\in\varphi((\Ecompb{m,m'}{\ms{X}})_{M_k+N_k}
 +\varpi^{k}(\Ecompb{m,m'}{\ms{X}}))$, and we get the claim.

 There are two natural homomorphisms
 \begin{equation*}
  \alpha\colon\EcompQb{m,m'}{\ms{X}}\otimes_{\Dcomp{m}{}}\ms{M}
   \rightarrow\EcompQ{m,m'}{r,s}\otimes_{\Dcomp{m}{}}\ms{M},\qquad
    \beta\colon\EcompQ{m,m'}{r,s}\otimes\ms{M}\rightarrow
    \EcompQb{m,m'}{\ms{X}}\otimes\ms{M},
 \end{equation*}
 where $\alpha$ is induced by the inclusion
 $\EcompQb{m,m'}{\ms{X}}\rightarrow\EcompQ{m,m'}{r,s}$,
 and $\beta$ is defined by extending linearly the canonical homomorphism
 $\ms{M}\rightarrow\EcompQb{m,m'}{\ms{X}}\otimes\ms{M}$ using the
 $\EcompQ{m,m'}{r,s}$-module structure on
 $\EcompQb{m,m'}{\ms{X}}\otimes\ms{M}$ defined above. We can check
 easily that $\beta\circ\alpha=\mr{id}$. Thus
 $\alpha$ is injective. Let us see that $\alpha$ is surjective. It
 suffices to show that $\alpha$ is a homomorphism of
 $\EcompQ{m,m'}{r,s}$-modules. Take an element
 $P\in\EcompQ{m,m'}{r,s}$. It suffices to show that
 \begin{equation}
  \label{compatidisp}
  \alpha(P\cdot e)=P\cdot\alpha(e)
 \end{equation} 
 for any $e\in\EcompQb{m,m'}{\ms{X}}\otimes_{\Dcomp{m}{}}\ms{M}$. By
 density (cf.\ \ref{deftopoean}), there exists
 an integer $n\geq0$ such that $P$ is contained in the
 closure of $\EcompQb{m,m'}{\ms{X}}$ in $(\EcompQ{m,m'}{r,s},\ms{T}_n)$.
 Consider the $(\EcompQ{m,m'}{r,s},\ms{T}_n)$-module topologies on
 the both sides of $\alpha$. Since the both sides of $\alpha$ are
 finitely generated over $\EcompQ{m,m'}{r,s}$,  they are
 Fr\'{e}chet spaces by \ref{deftopoean}, and by the open mapping
 theorem,  the topology on the source of $\alpha$ is equivalent to the
 $(\EcompQb{m,m'}{\ms{X}},\ms{T}_n)$-module topology. Since $\alpha$ is $\EcompQb{m,m'}{\ms{X}}$-linear, it is
 continuous. Since the target of $\alpha$
 is separated, we get that (\ref{compatidisp}) holds by the continuity
 of $\alpha$. We conclude that $\alpha$ is an isomorphism.

 Let us finish the proof. It suffices to see that
 \begin{equation}
  \label{isomwantanalyif}
  (\EcompQ{m,m'}{s})^{\mr{an}}\otimes\ms{M}\rightarrow
   \invlim_{n}\EcompQ{m,m'}{nc,s}\otimes\ms{M}
 \end{equation}
 is an isomorphism where $c:=p^{m'+1}$. When $\ms{M}$ is a finite
 projective $\DcompQ{m}{\ms{X}}$-module, the equality is obvious.
 By the same argument as Lemma \ref{flatness}, $\EcompQ{m,m'}{nc,s}$
 is flat over $\DcompQ{m}{\ms{X}}$ for any positive integer $i$. Since
 $\DcompQ{m}{\ms{X}}$ is of finite homological dimension, there is a
 finite projective resolution $\ms{P}_\bullet\rightarrow\ms{M}$ whose
 length is $l$ by \cite[4.4.6]{Ber2}. (If fact we can take $l$ to be
 $2$.) We have the following diagram where the bottom sequence is exact.
 \begin{equation*}
  \xymatrix{
   0\ar[r]&(\EcompQ{m,m'}{s})^{\mr{an}}\otimes\ms{P}_l\ar[r]\ar[d]&
   \dots\ar[r]&(\EcompQ{m,m'}{s})^{\mr{an}}\otimes\ms{P}_0\ar[r]\ar[d]&
   (\EcompQ{m,m'}{s})^{\mr{an}}\otimes\ms{M}\ar[r]\ar[d]^{(*)}&0\\
   0\ar[r]&\EcompQ{m,m'}{nc,s}\otimes\ms{P}_l\ar[r]&
   \dots\ar[r]&\EcompQ{m,m'}{nc,s}\otimes\ms{P}_0\ar[r]&
   \EcompQ{m,m'}{nc,s}\otimes\ms{M}\ar[r]&0
   }
 \end{equation*}
 Let us show that
 $R^1\invlim_{n}(\EcompQ{m,m'}{nc,s}\otimes\ms{P}_j)=0$ for any
 $j$. This is equivalent to saying $R^1\invlim_{n}\EcompQ{m,m'}{nc,s}=0$
 since $\ms{P}_j$ is finite projective. Let $|\cdot|_{n}$ be the
 $p$-adic norm on $\EcompQ{m,m'}{nc,s}$. Let $E_n$
 be the closure of $\EcompQ{m,m'}{(n+1)c,s}$ in $\EcompQ{m,m'}{nc,s}$
 with respect to $|\cdot|_{n}$-norm. Then
 $R^1\invlim_nE_n\cong R^1\invlim_{n}\EcompQ{m,m'}{nc,s}$ (cf.\ the
 proof of \cite[5.9]{Abe}). Now, apply
 \cite[$0_{\mr{III}}$, 13.2.4 (i)]{EGA} to the system $\{E_n\}$ with
 respect to the $p$-adic norm, and the claim follows.

 By applying $\invlim_n$ to the bottom sequence of the diagram above, we
 get that the sequence
 \begin{equation*}
  \invlim_{n}\EcompQ{m,m'}{nc,s}\otimes\ms{P}_1\rightarrow
   \invlim_{n}\EcompQ{m,m'}{nc,s}\otimes\ms{P}_0\rightarrow
   \invlim_{n}\EcompQ{m,m'}{nc,s}\otimes\ms{M}\rightarrow0
 \end{equation*}
 is exact. Since (\ref{isomwantanalyif}) is an isomorphism for
 projective modules, by using the right exactness of tensor product,
 (\ref{isomwantanalyif}) is an isomorphism in general, and we
 conclude the proof of the proposition. Moreover, when we apply
 $\invlim_n$ to the bottom sequence of the above diagram, we get that
 the vertical homomorphisms are isomorphisms, and thus the top sequence
 is also exact. This implies the flatness of
 $(\EcompQ{m,m'}{s})^{\mr{an}}$ over $\DcompQ{m}{\ms{X}}$.
\end{proof}

\begin{cor}
 \label{ananddag}
 Under the hypothesis of Proposition
 {\normalfont\ref{analytifcisom}}, suppose moreover that $\ms{M}$ is
 stable. Then for $m'\geq m$ the canonical homomorphisms
 \begin{equation*}
  \Emodb{m',\dag}{\ms{X},\mb{Q}}\otimes_{\Dcomp{m}{}}\ms{M}
   \rightarrow(\Emod{m',\dag}{s,\mb{Q}})^{\mr{an}}
   \otimes_{\Dcomp{m}{}}\ms{M},\qquad
   E^\dag_{\ms{X},\mb{Q}}\otimes_{\Dcomp{m}{}}\ms{M}\rightarrow
   \ms{E}^{\mr{an}}_{s,\mb{Q}}\otimes_{\Dcomp{m}{}}\ms{M}
 \end{equation*}
 are isomorphisms.
\end{cor}
\begin{proof}
 Clearly the first equality implies the second one. To prove the first
 one, it suffices to show that $\invlim_{m''}(\EcompQ{m',m''}{s})
 ^{\mr{an}}\otimes\ms{M}\cong(\Emod{m',\dag}{s,\mb{Q}})^{\mr{an}}
 \otimes\ms{M}$ since $\Emodb{m'',\dag}{\ms{X},\mb{Q}}$ is a
 Fr\'{e}chet-Stein algebra (cf.\ \cite[5.9]{Abe}). For this we only
 need to repeat the argument of the last part of the proof of the
 previous proposition. Namely, we prove that
 $R^1\invlim_{m''}(\EcompQ{m',m''}{s})^{\mr{an}}=0$ using
 \cite[$0_{\mr{III}}$, 13.2.4 (i)]{EGA}. The detail is left to the
 reader.
\end{proof}

\subsubsection{}
 \label{anyelem}
 Let $\ms{X}$ be a formal curve over $R$.
 Let $\ms{M}$ be a stable holonomic $\DcompQ{m}{\ms{X}}$-module, and
 let $s$ be a singular point of $\ms{M}$. Then by the Proposition
 \ref{analytifcisom}, for any integers $m''\geq m'\geq m$, the module
 $\EcompQ{m',m''}{s}\otimes\ms{M}$ (cf.\
 \ref{localsetting}) possesses a canonical
 $(\DcompQ{m'}{s})^{\mr{an}}$-module structure, and in particular, we
 get a canonical homomorphism
 \begin{equation*}
  \psi\colon(\DcompQ{m'}{s})^{\mr{an}}\otimes\ms{M}
   \rightarrow\EcompQ{m',m''}{s}\otimes\ms{M}.
 \end{equation*}
 Taking the inductive limit over $m'$, we get a canonical homomorphism
 \begin{equation*}
  \ms{M}|_{S_s}\rightarrow\EdagQ{s}\otimes\ms{M}.
 \end{equation*}
 By an abuse of language, we sometimes denote the image of
 $\alpha\in M_s$ where $M_s$ is either
 $(\DcompQ{m'}{s})^{\mr{an}}\otimes\ms{M}$ or $\ms{M}|_{S_s}$ by
 $1\otimes\alpha$.

 Let $\ms{U}$ be an open affine neighborhood of $s$ such that there
 exists a local parameter at $s$ and $s$ is the unique singularity of
 $\ms{M}$ in $\ms{U}$. Then by the proposition, we get
 $\EcompQb{m',m''}{\ms{U}}\otimes\ms{M}\cong\EcompQ{m',m''}{s}\otimes
 \ms{M}$. Now we define topologies as follows.

 \begin{dfn*}
  (i) We equip $\EcompQ{m',m''}{s}\otimes\ms{M}$ with the natural
  topology as an $\EcompQb{m',m''}{\ms{U}}$-module which makes it an
  LF-space by Lemma \ref{topology} and Lemma \ref{finiteness}. Note that
  this topology does not depend on the choice of $\ms{U}$ by the same
  lemma.

  (ii)  We equip $(\DcompQ{m'}{s})^{\mr{an}}\otimes\ms{M}$ with
  the projective limit topology of the projective system of Banach spaces
  $\bigl\{\mc{O}_{np^{m+1}}\widehat{\otimes}(\DcompQ{m'}{\ms{U}}
  \otimes\ms{M})\bigr\}_{n\geq0}$. This makes
  $(\DcompQ{m'}{s})^{\mr{an}}\otimes\ms{M}$ a Fr\'{e}chet space.
 \end{dfn*}

 \begin{rem*}
 (i) The topology on $\EcompQ{m',m''}{s}\otimes\ms{M}$ is also
  equivalent to the $(\EcompQb{m',m''}{\ms{U}},\ms{T})$-module topology
  by Lemma \ref{topology}.

 (ii) The homomorphism $\psi$ is continuous by the claim in the proof of
  Proposition \ref{analytifcisom}. In particular, if a sequence
  $\{\alpha_i\}$ converges to $\alpha$ in
  $(\DcompQ{m'}{s})^{\mr{an}}\otimes\ms{M}$, we get that
  $\{1\otimes\alpha_i\}$ converges to $1\otimes\alpha$ in
  $\EcompQ{m',m''}{s}\otimes\ms{M}$.
 \end{rem*}

%\subsection{Applications.}
%\label{applithem}
%In this subsection, we will give two important corollaries of
%Proposition \ref{analytifcisom}.

\subsubsection{}
The following corollary of
Proposition \ref{analytifcisom}
plays an important role when we prove a
fundamental properties of local Fourier transforms (cf.\ Lemma
\ref{nongeometpt} and \ref{locFourrel}).
\begin{cor*}
 \label{analyticstructure}
 Let $\ms{X}$ and $\ms{X}'$ be two formal curves, and take points
 $x\in\ms{X}$ and $x'\in\ms{X}'$. Assume that there exists an
 isomorphism $\iota\colon\ms{S}_x\xrightarrow{\sim}\ms{S}_{x'}$ of
 formal disks over $R$.

 (i) Let
 $\ms{M}$ and $\ms{M}'$ be holonomic $\DcompQ{m}{\ms{X}}$ and
 $\DcompQ{m}{\ms{X}'}$-module respectively, and assume that
 \begin{equation*}
  \iota_*\bigl((\DcompQ{m}{x})^{\mr{an}}\otimes\ms{M}\bigr)
   \xrightarrow{\sim}(\DcompQ{m}{x'})^{\mr{an}}\otimes\ms{M}'
 \end{equation*}
 as $(\DcompQ{m}{x'})^{\mr{an}}$-modules. Then there
 exists canonical isomorphisms 
 \begin{equation*}
  \iota_*\bigl(\EcompQ{m,m'}{x}\otimes\ms{M}\bigr)\cong
   \EcompQ{m,m'}{x'}\otimes\ms{M}',\qquad
   \iota_*\bigl(\Emod{m,\dag}{x,\mb{Q}}\otimes\ms{M}\bigr)
   \cong\Emod{m,\dag}{x',\mb{Q}}
 \otimes\ms{M}'
 \end{equation*} 
 for $m'\geq m$.

 (ii) Let $\ms{M}$ and $\ms{M}'$ be holonomic $\DdagQ{\ms{X}}$ and
 $\DdagQ{\ms{X}'}$-module respectively, and assume that
 \begin{equation*}
  \iota_*\ms{M}|_{S_x}\cong\ms{M}'|_{S_{x'}}
 \end{equation*}
 as $\ms{D}^{\mr{an}}_{x',\mb{Q}}$-modules. Then there exists a
 canonical isomorphism $\iota_*\bigl(\EdagQ{x}\otimes\ms{M}\bigr)
 \cong\EdagQ{x'}\otimes\ms{M}'$.
\end{cor*}
\begin{proof}
 First, let us prove (i). We get
 \begin{align*}
  &\iota_*\EcompQ{m,m'}{x}\otimes\ms{M}
  \xrightarrow[(*)]{\sim}\iota_*(\EcompQ{m,m'}{x})
  ^{\mr{an}}\otimes\ms{M}\xleftarrow{\sim}\iota_*
  (\EcompQ{m,m'}{x})^{\mr{an}}\otimes((\DcompQ{m}{\ms{X}})^{\mr{an}}
  \otimes\ms{M})\\
  &\qquad\xrightarrow[\iota']{\sim}(\EcompQ{m,m'}{x'})
   ^{\mr{an}}\otimes((\DcompQ{m}{\ms{X}'})^{\mr{an}}\otimes\ms{M})
  \xrightarrow{\sim}(\EcompQ{m,m'}{x'})^{\mr{an}}\otimes\ms{M}
  \xleftarrow[(*)]{\sim}\EcompQ{m,m'}{x'}\otimes\ms{M}'.
 \end{align*}
 Here we used Proposition \ref{analytifcisom} two times at $(*)$, and
 $\iota'$ denotes the isomorphism induced by $\iota$. To show the
 equality for $\Emod{m,\dag}{\ms{X},\mb{Q}}$, it suffices to use
 Proposition \ref{analytifcisom}, Corollary \ref{ananddag}, and the
 Fr\'{e}chet-Stein property of $\Emodb{m,\dag}{\ms{U},\mb{Q}}$ where
 $\ms{U}$ is an affine neighborhood of $x$.

 Now, let us prove (ii). Let $\ms{M}$ be a $\DdagQ{\ms{X}}$-module. Let
 $\ms{M}^{(m)}$ be a coherent $\DcompQ{m}{\ms{X}}$-module such that
 $\DdagQ{\ms{X}}\otimes\ms{M}^{(m)}\cong\ms{M}$, and the same for
 $\ms{M}'^{(m)}$. Then since these are
 coherent, there exists $N$ such that
 \begin{equation*}
  \ms{S}_{x'}\otimes_{\ms{S}_x}(\DcompQ{N}{\ms{X}})^{\mr{an}}
   \otimes\ms{M}^{(m)}\xrightarrow{\sim}(\DcompQ{N}{\ms{X'}})
   ^{\mr{an}}\otimes\ms{M}'^{(m)}.
 \end{equation*}
 Thus (ii) follows from (i).
\end{proof}

\subsection{Equality between two definitions of irregularity. }
\label{applithem}

Another important corollary of Proposition \ref{analytifcisom} is a
comparison result of multiplicities of characteristic cycles
(irregularity of Garnier) and irregularity of Christol-Mebkhout, cf.\ Corollary~\ref{compCMGAR}.

\subsubsection{}\label{irr} 
 Let us
review the definitions first. 
We assume that $k$ is perfect and that there exists a lifting of $h$-th absolute Frobenius
$R\xrightarrow{\sim}R$.
Let ${M}$ be a solvable  differential
$\mc{R}_K$-module, cf.\ \cite[12.6.4]{Ke2} or \cite[8.7]{CM}. By a result of Christol and Mebkhout 
(cf.\ \cite[2.4-1]{CM4} for free differential modules, or in general 
\cite[12.6.4]{Ke2})
%\cite[12.6.4]{Ke2}\footnote{
%Here, Christol-Mebkhout's decomposition theorem
%\cite[2.4-1]{CM4} might be sufficient since we are only dealing with
%free differential modules in the following.},
there exists a
canonical decomposition ${M}\cong\bigoplus_{\beta\geq0}{M}_\beta$
where ${M}_\beta$ is a differential $\mc{R}_K$-module purely of differential slope\index{slope!differential slope}
$\beta$. The irregularity of $M$ is defined by $\mr{irr}({M}):=\sum_{\beta\geq0}\beta\cdot
\mathrm{rk}({M}_\beta)$. \index{irregularity $\mr{irr}$, $\mr{irr}^{\mr{CM}}_x$, $\mr{irr}^{\mr{Gar}}_x$|(}
We say that $M$ is regular\index{differential module! regular free
differential module} if $\mr{irr}({M})=0$, or equivalently $M=M_0$.

Let $\ms{X}$ be a formal curve over $R$, and $S$ be
a closed subset of $\ms{X}$ such that its complement is dense in
$\ms{X}$. Let $\ms{M}$ be a convergent isocrystal on
$\ms{U}:=\ms{X}\setminus S$ overconvergent along $S$. For $x$ in $S$, we
define the irregularity of Christol-Mebkhout
\begin{equation*}
 \mr{irr}^{\mr{CM}}_x(\ms{M}):=\mr{irr}(\ms{M}|_{\eta_x});
\end{equation*} 
when $\mr{irr}(\ms{M}|_{\eta_x})=0$ we say that $\ms{M}$ is regular at
$x$, or that $x$ is a regular singularity for
$\ms{M}$.\label{def_regularity}\index{Fi@$F$-isocrystal, overconvergent
isocrystal! regular singular point of}
We also have the irregularity of Garnier \cite[5.1.2]{Gar2}. 
%We assume that $k$ is perfect and that there exists a lifting of $h$-th absolute Frobenius
%$R\xrightarrow{\sim}R$.
For
simplicity, we assume moreover that $\ms{M}$ possesses a Frobenius
structure. Let us denote by $j_+\ms{M}$ the underlying
$\DdagQ{\ms{X}}$-module of $\ms{M}$, which is {\it a priori} a coherent
$\DdagQ{\ms{X}}(S)$-module. This is a holonomic module, and Garnier
defined\footnote{In {\it loc.\ cit.}, he defined only in the case where
$x$ is a $k$-rational point, but we do not think we need this
assumption here. In {\it loc.\ cit.}, $\mc{O}_{\ms{X},\mb{Q}}|_{S_x}$ is
denoted by $\mr{sp}_*\mc{O}_{\left]x\right[}$. } for $x$ in $S$
\begin{equation*}
 \mr{irr}^{\mr{Gar}}_x(\ms{M}):=\chi\bigl(R\shom_{\DdagQ{\ms{X}}}
  (j_+\ms{M}, \mc{O}_{\ms{X},\mb{Q}}|_{S_x}  %\mr{sp}_*\mc{O}_{\left]x\right[_{\ms{X}_K}}
	)\bigr)
  -(m_{\xi_x=0}(j_+\ms{M})-m_x(j_+\ms{M})),\index{irregularity $\mr{irr}$, $\mr{irr}^{\mr{CM}}_x$, $\mr{irr}^{\mr{Gar}}_x$|)}
\end{equation*}
where 
$m_x$ (resp.\ $m_{\xi_x=0}$) denotes the vertical multiplicity at $x$
(resp.\ the generic rank)  of $j_+\ms{M}$, cf.\ \ref{def_cycl_Dm} and
\ref{defCycl}. Since $\ms{M}$ is an isocrystal we have
$m_{\xi_x=0}(j_+\ms{M})=\mr{rk}(\ms{M})$, cf.\ \cite[2.2.5,
2.2.6]{Gar2}. The finiteness of
$\chi\bigl(R\shom_{\DdagQ{\ms{X}}}(j_+\ms{M},
\mc{O}_{\ms{X},\mb{Q}}|_{S_x}) \bigr)$ was not known at the time Garnier
defined the irregularity; now, using some results of Crew, based on the
local monodromy theorem, this finiteness is easy, and this is also
proven in the next corollary.

\subsubsection{}
The following corollary has already been announced\footnote{See
A. Marmora, {\em About p-adic Local Fourier
Transform}, Poster 2 at Journ\'{e}es de G\'{e}om\'{e}trie
Arithm\'{e}tique de Rennes, available at {\tt
http://perso.univ-rennes1.fr/ahmed.abbes/Conference/posters.html}.}
by the second author using a local computation, which is different from
our method here.

\begin{cor*}
 \label{compCMGAR}
 Let $\ms{M}$ be a convergent $F$-isocrystal on $\ms{X}\setminus S$
 overconvergent along $S$. Then we have
 \begin{equation}
  \label{garinv0}
  R\shom_{\DdagQ{\ms{X}}}(j_+\ms{M},
   \mc{O}_{\ms{X},\mb{Q}}|_{S_x}) =0.
 \end{equation}
 Moreover, we have
 \begin{equation}
  \label{equaliirrgarcm}
  \mr{irr}_x^{\mr{Gar}}(\ms{M})=\mr{irr}_x^{\mr{CM}}(\ms{M}).
 \end{equation}
\end{cor*}
\begin{proof}
 First of all, let us show the equality (\ref{garinv0}).
 The ring 
 $\mc{O}_{\ms{X},\mb{Q}}|_{S_x}$ has a canonical
 $\Dan{\ms{S}_x,\mb{Q}}$-module structure and identifies with
 $\mc{O}^{\mr{an}}_{\ms{S}_x}$, cf.\ \ref{preisouadsf}; let us denote it
 simply by $\mc{O}^{\mr{an}}$.
 Let $\mc{A}\rightarrow\mc{B}$ be a {\em flat} homomorphism of sheaves
 of rings on a topological space $T$. For $\mc{M}\in D^-(\mc{A})$ and
 $\mc{N}\in D^+(\mc{B})$, we have
 $R\mr{Hom}_{\mc{A}}(\mc{M},\mc{N})\cong R\mr{Hom}_{\mc{B}}
 (\mc{B}\otimes_{\mc{A}}\mc{M},\mc{N})$. Using this, we have
 \begin{multline*}
  R^n\shom_{\DdagQ{\ms{X}}}(j_+\ms{M},\mc{O}_{\ms{X},\mb{Q}}|_{S_x})
  \cong R^n\shom_{\DdagQ{\ms{S}_x}}
  (\DdagQ{\ms{S}_x}\otimes j_+\ms{M},\mc{O}^{\mr{an}}) 
  \\ \cong
  R^n\mr{Hom}_{\DdagQ{\ms{S}_x}}
  (\DdagQ{\ms{S}_x}\otimes j_+\ms{M},\mc{O}^{\mr{an}}) 
  \cong R^n\mr{Hom}_{\Dan{\ms{S}_x,\mb{Q}}}
  (\Dan{\ms{S}_x,\mb{Q}}\otimes j_+\ms{M},\mc{O}^{\mr{an}})
 \end{multline*}
 where the first and the last isomorphism follow by Lemma
 \ref{exactforloc} and \cite[Th.\ 4.1.1]{Cr}, and
 the second is obtained by taking global sections,
 cf.\ \ref{Cr_recall1} and also \ref{preisouadsf}.
% and the last one by
% adjunction again because
% $\DdagQ{\ms{S}_x}\hookrightarrow\Dan{\ms{S}_x,\mb{Q}}$ is flat, cf.\
% \cite[Th.\ 4.1.1]{Cr}.

 Since $\Dan{\ms{S}_x, \mb{Q}}\otimes j_+\ms{M}= (j_+\ms{M})|_{{S}_x}$ is
 clearly of connection type, the canonical map
 \begin{equation*}
  \mr{can}\colon
   \mb{V}((j_+\ms{M})|_{{S}_x})\rightarrow\mb{W}((j_+\ms{M})|_{{S}_x})
 \end{equation*}
 is an isomorphism, cf.\ \cite[Cor.\ 6.1.2]{Cr}.  We have
 $R^n\mr{Hom}_{\Dan{\ms{S}_x,\mb{Q}}}((j_+\ms{M})|_{{S}_x},
 \mc{O}^{\mr{an}}_{K^{\mr{ur}}})=0$ by \eqref{vanishexact} and
 \eqref{vanishexact2}, from which it follows
 $R^n\mr{Hom}_{\Dan{\ms{S}_x,\mb{Q}}}((j_+\ms{M})|_{{S}_x},\mc{O}^{\mr{an}})=0$. The
 first claim is proven.
%%% Alternative proof using cohomological operations
  %Let $i\colon x\hookrightarrow\ms{S}_x$.
  %We have
 %\begin{equation*}
  %R^n\shom_{\DdagQ{\ms{S}_x}}
   %(\DdagQ{\ms{S}_x}\otimes j_+\ms{M},\mc{O}^{\mr{an}})
   %\cong \ms{H}^{n-1}(i^+\mb{D}_{\ms{S}_x}
   %(\DdagQ{\ms{S}_x}\otimes j_+\ms{M}))\cong \ms{H}^{n-1}i^!(\DdagQ{\ms{S}_x}\otimes j_+\ms{M}).
 %\end{equation*}
 %Since $\DdagQ{\ms{S}_x}\otimes j_+\ms{M}$ is clearly 
 %of connection type, we have $i^!(\DdagQ{\ms{S}_x}\otimes j_+\ms{M})=0$
 %by \cite[2.2]{Cr3}. Thus the first claim follows.

 Now, let us start the proof of the equality of the irregularities.
 The irregularity $\mr{irr}_x^{\mr{CM}}$ only depends on
 $\ms{M}|_{\eta_x}$ by definition. By Corollary \ref{analyticstructure}
 combined with Theorem \ref{stabilitytheoremcy}, we get that
 $\mr{irr}_x^{\mr{Gar}}$ only depends on its analytification as
 well. This says that we may assume 
 $\ms{X}\cong\widehat{\mb{A}}^1$, $x=0$, and that $\ms{M}$ is the canonical
 extension of $\ms{M}|_{\eta_x}$.  Note that, thanks to (\ref{garinv0}),
 $\mr{irr}_x^{\mr{Gar}}$ satisfies Grothendieck-Ogg-Shafarevich type
 formula (GOS-type formula) by \cite[5.3.2]{Gar2}. By \cite[1.2]{CM}, we
 know that $\mr{irr}_x^{\mr{CM}}$ also satisfies GOS-type formula.

 The equality (\ref{equaliirrgarcm}) holds when $\ms{M}$ is regular
 singular at $x$. Indeed, by the definition of regularity we have $\mr{irr}^{\mr{CM}}_x(\ms{M})=0$, and  it suffices to show
 that $\mr{irr}^{\mr{Gar}}_x(\ms{M})=0$. Now, by using the
 structure theorem of regular $p$-adic differential equation
 \cite[12.3]{CM} and the additivity of $\mr{irr}^{\mr{Gar}}_x$ (cf.\
 \cite[5.1.3]{Gar2}), we may assume that
 $\ms{M}$ is of rank $1$. For this case we refer to
 \cite[5.3.1]{Gar2}.

 Finally, let us prove the general case. By GOS-type formulas, we get
 \begin{align*}
  \chi(\widehat{\mb{P}},\ms{M})&=\mr{rk}(\ms{M})
   \cdot\chi(\widehat{\mb{P}})-\mr{irr}_x^{\mr{CM}}(\ms{M})-
  \mr{irr}_{\infty}^{\mr{CM}}(\ms{M}),\\
    \chi(\widehat{\mb{P}},\ms{M})&=\mr{rk}(\ms{M})
   \cdot\chi(\widehat{\mb{P}})-\mr{irr}_x^{\mr{Gar}}(\ms{M})-
  \mr{irr}_{\infty}^{\mr{Gar}}(\ms{M}).
 \end{align*}
 By using the regular case we have proven, we get $\mr{irr}_{\infty}
 ^{\mr{CM}}(\ms{M})=\mr{irr}_{\infty}^{\mr{Gar}}(\ms{M})=0$ since
 $\ms{M}$ is regular at $\infty$. Thus comparing these two equalities,
 we get what we want.
\end{proof}

\subsection{Definition of local Fourier transform}
\label{secdefforloc}
In this subsection, we define the local Fourier transform. We only
define the so called $(0,\infty')$-local Fourier
transforms. In a later section, we will define an analog of
$(\infty,0')$-local Fourier transform in very
special cases, and we do not deal with $(\infty,\infty')$-local Fourier
transform in this paper.

\subsubsection{}
\label{setupFour}
Let us fix a situation under which we use the Fourier
transforms. Let $\fr$ be a positive integer, and put $q:=p^\fr$. We
assume that the residue field $k$ of $R$ is perfect. We moreover assume
that the absolute $\fr$-th Frobenius automorphism of $k$ lifts to an
automorphism $\sigma\colon R\xrightarrow{\sim}R$.

We assume that the field $K$ contains a root
$\pi$ of the equation $X^{p-1}+p=0$ (so that it contains all of
them). \index{.@miscellaneous!pi@$\pi$, Dwork's}
The Dwork exponential series $\theta_{\pi}(x)= \exp(\pi(x-x^p))$ %\index{.@miscellaneous!$\theta_{\pi}(x)$}
in $K\cc{x}$ has a radius of convergence strictly greater than $1$ and
it converges for $x=1$ to a $p$-th root of unity
$\zeta=\theta_{\pi}(1)$. The choice of $\pi$ determines a non-trivial
additive character $\psi_{\pi}\colon\mb{F}_p\rightarrow K^*$, \index{.@miscellaneous!psi@$\psi_{\pi}$} by sending
$x$ to $\zeta^{x}$. Conversely if $\psi\colon\mb{F}_p\rightarrow K^*$ is
a non-trivial additive character, then $\psi(1)$ is a $p$-th root of
unity in $K$ and the polynomial $X^{p-1}+p$ splits completely in
$K$. There exists a unique root $\pi_{\psi}$ of $X^{p-1}+p$ such that
$\pi_{\psi}\equiv\psi(1) -1$ modulo $(\psi(1) -1)^2$ and then we have
$\theta_{\pi_{\psi}}(1)=\psi(1)$. For the details see
\cite[1.3]{BerDw}.

We denote by $\widehat{\mb{A}}_{R,x}^1:=\mr{Spf}(R\{x\})$ (resp.\ $\widehat{\mb{P}}_{R,x}^1$)
the affine (resp.\ projective) \index{.@miscellaneous!A@$\Aone$, $\widehat{\mb{A}}_{R,x}^1$, $\Aoned$, $\widehat{\mb{A}}_{R,x'}^1$|(} \index{.@miscellaneous!P@$\Pone$}\index{.@miscellaneous!P@$\widehat{\mb{P}}_{R,x}^1$, $\Poned$, $\widehat{\mb{P}}^1_{R,x'}$, $\Poneoned$|(}
line over $R$ with the fixed coordinate $x$, and we denote
by  $\widehat{\mb{A}}_{R,x'}^1$ (resp.\ $\widehat{\mb{P}}^1_{R,x'}$) its dual line. 
To lighten the notation we often put 
 $\Aone:=\widehat{\mb{A}}_{R,x}^1$, $\Pone:=\widehat{\mb{P}}_{R,x}^1$, 
 $\Aoned:=\widehat{\mb{A}}_{R,x'}^1$ and
$\Poned:=\widehat{\mb{P}}^1_{R,x'}$.  \index{.@miscellaneous!A@$\Aone$, $\widehat{\mb{A}}_{R,x}^1$, $\Aoned$, $\widehat{\mb{A}}_{R,x'}^1$|)}
We denote by $\partial$ and
$\partial'$ (or $\partial_x$ and $\partial_{x'}$ if we want to clarify
the coordinates) the differential operators corresponding to $x$ and
$x'$. \index{differential operators!delta@$\partial$, $\partial_x$, $\partial'$, $\partial_{x'}$} We denote by $\infty$ (resp.\
$\infty'$) the point at infinity of $\Pone$ (resp.\ $\Poned$). Let
$\Poneoned:=\Pone\times\Poned$, and
$Z:=(\{\infty\}\times\Poned)\cup(\Pone
\times\{\infty'\})$. \index{.@miscellaneous!P@$\widehat{\mb{P}}_{R,x}^1$, $\Poned$, $\widehat{\mb{P}}^1_{R,x'}$, $\Poneoned$|)}

To summarize notation once and for all, we use terminologies of the
next section, and consider the following diagram of {\em couples}:
\begin{equation}
 \label{fourierdiagtwo}
 \xymatrix@R=10pt{
  &(\Poneoned,Z)\ar[dr]^{p'}\ar[dl]_{p}\\
 (\Pone,\{\infty\})&&(\Poned,\{\infty'\}),}
\end{equation}
where $p$ and $p'$ are the projections. These morphisms are not used
till the next section.

If we take $\ms{X}$ to be
$\Aone$ (resp.\ $\Aoned$), we are in Situation (L) of \ref{setup}
using the fixed coordinate. We use freely the notation of
\ref{setup}, especially $\CK{\Aone}\{\partial\}^{(m,m')}$.

For a smooth formal scheme $\ms{X}$ over $R$, we put
$\ms{X}^{(1)}:=\ms{X}\otimes_{R,\sigma}R$. We denote by $y$ (resp.\
$y'$) the coordinate of ${\Pone}^{(1)}$ (resp.\
$\Poned\phantom{}^{(1)}$) induced by $x$ (resp.\ $x'$). The relative Frobenius of
$\mb{P}^1_k$  lifts to the morphism
$F_{\Pone}\colon\Pone\rightarrow\Pone^{(1)}$ sending $y$ to $x^q$. %\index{.@miscellaneous!F@$F_{\Pone}$}
We have a similar morphism for $\Poned$, and denote it by $F_{\Poned}$.\\

Let us introduce some notation for a formal disk around a closed point
of $\Aone$. We put $\ms{S}:=\mr{Spf}(R\cc{u})$ \index{.@miscellaneous!S@$\ms{S}$} and
$\ms{S}':=\mr{Spf}(R\cc{u'})$. \index{.@miscellaneous!S@$\ms{S}'$}
We denote $\eta_{\ms{S}}$ and
$\eta_{\ms{S}'}$ by $\eta$ and $\eta'$ respectively. Let $E$ be a finite
unramified extension of $K$, and $R_E$  the ring of integers of $E$. We put $\ms{S}_E:=\mr{Spf}(R_E\cc{u})$ and
denote $\eta_{\ms{S}_E}$ by $\eta_{E}$. Let $s$ be a closed point of
$\Aone=\mr{Spf}(R\{x\})$, $\mathfrak{m}_s$ the maximal ideal  of $k[x]$ associated to $s$. We denote by $y_s$ the monic generator of $\mathfrak{m}_s$, and by $\widetilde{y}_s$ a lifting of $y_s$ in $R\{ x\}$, which is 
a local parameter of $\Aone$ at $s$, cf.\ \ref{coordinates};
we denote also by $\widetilde{y}_s$ its image in $\widehat{\mc{O}}_{\Aone,s}$, the completion 
of the local ring of $\Aone$ at $s$. 
We define $\tau_{\widetilde{y}_s}\colon\ms{S}_s:=\Spf(\widehat{\mc{O}}_{\Aone,s})\rightarrow\ms{S}$ \index{.@miscellaneous!tau@$\tau_{\widetilde{y}_s}$, $\sigma_{\widetilde{\mf{s}}}$, $\widetilde{\tau}_{\widetilde{\mf{s}}}$,
$\tau'$|(}
by
sending $u$ to $\widetilde{y}_s$. We note that if we take another
lifting $\widetilde{y}'_s$, then there exists a canonical equivalence of
functors $\tau^*_{\widetilde{y}_s}\cong\tau^*_{\widetilde{y}'_s}$
and $\tau_{\widetilde{y}_s*}\cong\tau_{\widetilde{y}'_s*}$, since
$\widetilde{y}_s$ and $\widetilde{y}'_s$ are congruent modulo
$\varpi$. We denote $\tau^*_{\widetilde{y}_s}$ and
$\tau_{\widetilde{y}_s*}$ by $\tau_s^*$ and $\tau_{s*}$ respectively.

For $\mf{s}\in\mb{A}^1_k(\overline{k})$, let $s$ be the closed point
of $\mb{A}_k^1$ defined by $\mf{s}$. Let $k_s$ be the residue
field of $s$, $R_s$ be the unique finite \'{e}tale extension of $R$
corresponding to $k_s$, namely
$W(k_s)\otimes_{W(k)}R$, and $K_s$ be its field of
fractions. Let $\widetilde{s}$ be the closed point of
$\Aone_s:=\Aone\otimes R_s$ defined by $\mf{s}$, and
$\Pone_s:=\Pone\otimes R_s$. The rational point $\mf{s}$
corresponds to an element of $k_s$ also denoted by $\mf{s}$, and take a
lifting $\widetilde{\mf{s}}$ of $\mf{s}$ in $R_s$. We have the
homomorphism
$\sigma_{\widetilde{\mf{s}}}\colon\ms{S}_{\widetilde{s}}
\rightarrow\ms{S}$
sending $u$ to $x-\widetilde{\mf{s}}$. The functors
$\sigma_{\widetilde{\mf{s}}}^*$ and $\sigma_{\widetilde{\mf{s}}*}$ do
not depend on the choice of $\widetilde{\mf{s}}$ up to a canonical
equivalence. We denote $\sigma_{\widetilde{\mf{s}}}^*$ and
$\sigma_{\widetilde{\mf{s}}*}$ by $\sigma_\mf{s}^*$ and
$\sigma_{\mf{s}*}$ respectively.

By the \'{e}taleness of $R_s$ over $R$, there is a canonical
inclusion $R_s\hookrightarrow\widehat{\mc{O}}_{\Aone,s}$. We define
$\widetilde{\tau}_{\widetilde{\mf{s}}}\colon\ms{S}_{s}\rightarrow
\ms{S}_{K_s}$ by sending $u$ to $x-\widetilde{\mf{s}}$.
Also in this case, $\widetilde{\tau}_{\widetilde{\mf{s}}}^*$ and
$\widetilde{\tau}_{\widetilde{\mf{s}}*}$ do not depend on the choice
of $\widetilde{\mf{s}}$ up to a canonical equivalence and we denote them
by $\widetilde{\tau}_{\mf{s}}^*$ and $\widetilde{\tau}_{\mf{s}*}$
respectively. Let $\tau'\colon\eta_{\infty'}\rightarrow\eta'$ be the
finite \'{e}tale morphism induced by sending $u'$ to $1/x'$. Summing up,
we defined the following morphisms.
\begin{equation*}
 \vcenter{
 \xymatrix@R=20pt@C=30pt{
  \widehat{\mb{A}}^1_{R_s}&
  \ms{S}_{\widetilde{s}}
  \ar[drrr]^<>(.8){\sigma_{\widetilde{\mf{s}}}}&&&
  \ms{S}_{K_s}\ar[d]^{\iota}\\
 \widehat{\mb{A}}^1_R&
  \ms{S}_{s}\ar[rrr]_{\tau_s}\ar[urrr]^<>(.2){\widetilde{\tau}
  _{\widetilde{\mf{s}}}}&&&\ms{S}
  }}
  \hspace{2cm}
  \tau'\colon\eta_{\infty'}\rightarrow\eta'
\end{equation*}
Here $\iota$ is the base change morphism. \index{.@miscellaneous!tau@$\tau_{\widetilde{y}_s}$, $\sigma_{\widetilde{\mf{s}}}$, $\widetilde{\tau}_{\widetilde{\mf{s}}}$,
$\tau'$|)}

Now, we recall that $|\pi|=\omega$, and we get the following isomorphism
\begin{equation}\label{iso_tau}
 \tau\colon\mc{A}_{K,u'}([\omega_m,\omega_{m'}\})\xrightarrow{\sim}
 \mc{A}_{K,x}(\{\omega/\omega_{m'},\omega/\omega_m])\xrightarrow{\sim}
 \CK{\Aone}\{\partial\}^{(m,m')},
\end{equation}
where the first isomorphism sends $u'$ to $\pi x^{-1}$, and the second
 is that of Lemma \ref{fieldconstmicdif}, and thus
$\tau(u')=\pi\partial^{-1}$.

\subsubsection{}\label{subsection_deflocFour}
Let $\ms{M}$ be a stable holonomic $\DcompQ{m}{\Aone}$-module and let $S\subset\Aone$ be the set of its singular points,
{\itshape i.e.\ }closed points $s$ such that $\pi^{-1}(s)\subset \mr{Char}(\ms{M})$, cf. \ref{stabledef}.
 
We follow the notation of \ref{setup}.
In particular, for  $m''\geq m'\geq m$ 
 non-negative integers and $s\in S$, we consider the microlocalization $$\EcompQ{m',m''}{s}\otimes\ms{M}
:= (\EcompQ{m',m''}{\ms{X}}\otimes_{\pi^{-1}\Dcomp{m}{\ms{X}}}\pi^{-1}\ms{M})
	_{\xi_x},$$ cf. Notation \ref{not3} in \ref{setup}. 
It is a finite
 $\CK{\Aone}\{\partial\}^{(m',m'')}$-module by Lemma \ref{finiteness},
 and using $\tau$, this can be seen as a finite
 $\mc{A}_{K,u'}([\omega_{m'},\omega_{m''}\})$-module.
We put 
\begin{equation}\label{changing_not}
\EcompQ{m',m''}{s}(\ms{M}):=\left(\EcompQ{m',m''}{s}\otimes\ms{M},\nabla\colon \alpha \mapsto (\pi^{-1}\partial^2x)\cdot \alpha\otimes du'\right),
\end{equation}\index{microlocalisation of $\ms{M}$!$\EcompQ{m,m'}{s}(\ms{M})$}
which is considered as a differential $\mc{A}_{K,u'}([\omega_{m'},\omega_{m''}\})$-module.
The fact that $\nabla$ defines a connection follows formally from the relation $\partial x = x\partial +1$ in  
$\EcompQ{m',m''}{s}$, and the verification is like that for formal microlocalisation, cf. \cite[p.750]{Lopez:Microloc}. 
  We warn the reader that the same notation $\EcompQ{m',m''}{s}(\ms{M})$ has been used in
 \ref{localsetting}, but from now on, we refer only to \eqref{changing_not}. 
The difference between the connections defined in \ref{localsetting} and \eqref{changing_not} is due to the  change of variable from 
$0$ to $\infty$ via the isomorphism $u'\mapsto \pi x^{-1}$. 

 For a fixed $m'$, the projective system
 $\bigl\{\EcompQ{m',m''}{s}(\ms{M})\bigr\}_{m''\geq m'}$ defines a
 differential module on $\mc{A}_{K,u'}(\left[\omega_{m'},1\right[)$ by
 Proposition \ref{rankcompat}, which is denoted by
 $\Emod{m',\dag}{s,\mb{Q}}(\ms{M})$\index{microlocalisation of $\ms{M}$!$\Emod{m',\dag}{s,\mb{Q}}(\ms{M})$}. This defines a differential
 $\mc{R}_{u',K}$-module which does not depend on the choice of $m'$ up
 to a canonical isomorphism by the same proposition.

\begin{dfn*}
 \label{deflocFour}
 (i) 
 Let $\ms{M}$ be a holonomic $\DdagQ{\Pone}(\infty)$-module.
 Let $s$ in $\Aone$ be a singular point of $\ms{M}$.
 Take a stable coherent $\DcompQ{m}{\Aone}$-module $\ms{M}^{(m)}$
 such that $\DdagQ{\Aone}\otimes\ms{M}^{(m)}\cong\ms{M}|_{\Aone}$.
 Then $\ms{M}^{(m)}$ defines the differential
 $\mc{R}_{u',K}$-module $\Emod{m',\dag}{s,\mb{Q}}(\ms{M}^{(m)})$ which
 does not depend on the choice of $m'$ up to a canonical
 isomorphism. We see easily that this does not depend on the choice of
 $\ms{M}^{(m)}$ as well in the category of differential
 $\mc{R}$-modules. We denote this differential $\mc{R}$-module by
 $\ms{F}_\pi^{(s,\infty')}(\ms{M})$, or
 $\ms{F}_{\pi,K}^{(s,\infty')}(\ms{M})$ if we want to indicate the base, \index{functors!local Fourier transform!$\ms{F}_\pi^{(s,\infty')}$, $\ms{F}_{\pi,K}^{(s,\infty')}$|(}
 and we call it {\em the local Fourier transform of $\ms{M}$ at
 $s$}. When $s$ is not a singular point of $\ms{M}$, we put
 $\ms{F}_\pi^{(s,\infty')}(\ms{M})=0$. This defines a functor
 \begin{equation*}
  \ms{F}_\pi^{(s,\infty')}\colon\mr{Hol}(\DdagQ{\Pone}(\infty)) 
   \rightarrow\mbox{Hol}'(\eta')
 \end{equation*} \index{functors!local Fourier transform!$\ms{F}_\pi^{(s,\infty')}$, $\ms{F}_{\pi,K}^{(s,\infty')}$|)}
 for any closed point $s\in\Aone$. Here $\mr{Hol}$ denotes the category
 of holonomic modules, and $\mbox{Hol}'(\eta')$
 denotes  the category of differential $\Rob_{\ms{S}'}$-modules (not necessarily free), cf.\ \ref{def:Hol'}. 
  If no confusion can arise, we omit the subscript
 $\pi$.

 (ii) Let ${M}$ be a holonomic $F$-$\ms{D}^{\mr{an}}_{\ms{S},\mb{Q}}$-module, and let
 $\mf{s}\in\mb{A}^1(\overline{k})$. Take
 the canonical extension ${M}^{\mr{can}}$ of ${M}$ at $\mf{s}$: this is an
 $F$-$\DdagQ{\Aone_s}$-module on
 $\Pone_{\mf{s}}\setminus\{\widetilde{s},\infty\}$
 overconvergent along $\{\widetilde{s},\infty\}$, we have
 ${M}^{\mr{can}}|_{S_{\widetilde{s}}}\cong\sigma_{\widetilde{\mf{s}}}^*{M}$, and
 $\infty$ is a regular singular point (cf.\ \ref{intcanext}). We define
 {\em the local Fourier transform of ${M}$ at $\mf{s}$} to be
 $\ms{F}^{(\widetilde{s},\infty')}({M}^{\mr{can}})$. We denote it by
 $\Phi_\pi^{(\mf{s},\infty')}({M})$, or
 $\Phi_{\pi,K}^{(\mf{s},\infty')}({M})$ if we want to indicate the
 base. This defines the functor
 \begin{equation*}
  \Phi_\pi^{(\mf{s},\infty')}\colon F\mbox{-Hol}(\ms{S})\rightarrow
   \mbox{Hol}'(\eta'_{K_s}).
 \end{equation*}\index{functors!local Fourier transform!$\Phi_\pi^{(\mf{s},\infty')}$, $\Phi_{\pi,K}^{(\mf{s},\infty')}$}
 If no confusion can arise, we omit the subscript $\pi$.

 We sometimes denote $m\in\Gamma(\ms{S},{M})$ by $m\in {M}$.
 As in \ref{anyelem}, any $m\in {M}$ defines an element
 in $\Phi^{(\mf{s},\infty')}({M})$. We denote this element by
 $\widehat{m}$.
\end{dfn*}

\begin{rem}
 \label{locFordfnrem}
 (i) By using the stationary phase theorem \ref{SPT}, we may prove that
 the local Fourier transforms are {\em free} differential
 $\mc{R}$-modules. Moreover, we can also prove that the local Fourier
 transform coincides with that
 defined by Crew (cf.\ Corollary \ref{coinccrew}).
 Using this, we will endow local Fourier transforms with a Frobenius
 structure later in \S\ref{endowFrob}.

 (ii) In Definition \ref{deflocFour} (ii) above, we can also construct
 $\Phi^{(0,\infty')}_\pi$ in a purely local way (i.e.\ without using
 canonical extensions). For a holonomic
 $\ms{D}^{\mr{an}}_{\ms{S},\mb{Q}}$-module ${M}$ with Frobenius
 structure, we define the local Fourier transform to
 be $\ms{E}^{\mr{an}}\otimes_{\Dan{}}{M}$, and put a connection in
 the same manner as in \eqref{changing_not} above.

 A problem of this construction is to see that this is a free differential
 $\mc{R}$-module. For this, we need to compare with Definition
 \ref{deflocFour} (ii), and this is why we did not adopt this
 definition. Namely, the module coincides with
 $\Gamma(\Phi^{(0,\infty')}({M}))$ by Lemma \ref{globalsectcalc}
 below. As written in (i), we will prove that
 $\Phi^{(0,\infty)}({M})$ is a free differential
 $\mc{R}$-module. This shows that $\ms{E}^{\mr{an}}\otimes {M}$
 defines a free differential $\mc{R}$-module, which is what we wanted.

 By this comparison, once we have the stationary phase formula, the
 functor $\ms{E}^{\mr{an}}\otimes_{\ms{D}^{\mr{an}}}(-)$ from the
 category of holonomic $\ms{D}^{\mr{an}}$-modules {\em with Frobenius
 structure} to the category of $\ms{E}^{\mr{an}}$-modules, is exact by
 Proposition \ref{exactness} below. However, we do not know if
 $\ms{E}^{\mr{an}}$ is flat over $\ms{D}^{\mr{an}}$ or not.
\end{rem}

\begin{ex}
\label{LFTex}
Here some basic examples to illustrate Definition \ref{deflocFour}, more properties will be proven in the sections \S\ref{section4} and \S\ref{section5}.
%\begin{enumerate}

(i) Consider the holonomic $F$-$\ms{D}^{\mr{an}}_{\ms{S},\mb{Q}}$-module $\mc{O}^{\mr{an}}$ (with the trivial Frobenius structure).
The point $0$ is not  singular, hence by definition  $\Phi^{(0,\infty')}(\mc{O}^{\mr{an}})=0$.
{\itshape A posteriori} this can be also computed by the point (ii) of Remark \ref{locFordfnrem}, tensoring the presentation $\ms{D}^{\mr{an}}\xrightarrow{\cdot \partial}\ms{D}^{\mr{an}}\rightarrow \mc{O}^{\mr{an}}\rightarrow 0$ with 
$\ms{E}^{\mr{an}}$. 

(ii)
\label{LFTex2} Consider the $F$-$\ms{D}^{\mr{an}}_{\ms{S},\mb{Q}}$-module  $\delta$, cf.\ \ref{coh_op_formal_disk}.
The canonical extension $\delta_0:= \delta^{\mr{can}}$ at $0$ has the following presentation on $\Aone$:
$$\DdagQ{\Aone}\xrightarrow{\cdot x}\DdagQ{\Aone}\rightarrow \delta_0|_{\Aone} \rightarrow 0.$$
The $\Rob$-module  $\Phi^{(0,\infty')}(\delta)$ is free of rank $1$ (the easiest way to prove this is to use Corollary \ref{irrandfotrans}); the connection is trivial (it follows immediately by the definition); 
and the Frobenius it is also trivial (it follows using global Fourier transform and the stationary phase theorem, cf.\ \ref{stationaryfrobcom}).

(iii) Consider the connection type holonomic $F$-$\ms{D}^{\mr{an}}_{\ms{S},\mb{Q}}$-module $\Rob$.
Again we can compute $\Phi^{(0,\infty')}(\Rob)$ globally, or locally (with exception of Frobenius).
To proceed locally,  we can consider the presentation of $\Rob$ 
$$\ms{D}^{\mr{an}}\xrightarrow{\cdot \partial x}\ms{D}^{\mr{an}}\rightarrow \Rob \rightarrow 0,$$
as $\ms{D}^{\mr{an}}$-module (Crew's  solution data functor $\mb{S}$  permits to show that this is a presentation, cf.\ \ref{mainresultscrewrev}); then tensoring with $\ms{E}^{\mr{an}}$ we reduce to the case \ref{LFTex2} above and so $\Phi^{(0,\infty')}(\Rob)=\Rob$, endowed with the trivial connection.   
To proceed globally we consider $\Rob^{\mr{can}}=\mc{O}_{\Pone,\mb{Q}}(0,\infty)$ and use the global Fourier transform
and the stationary phase. We obtain the trivial Frobenius structure on  $\Phi^{(0,\infty')}(\Rob)$. 
The details are left to the reader.

(iv) For any $m\geq 0$, the $\DcompQ{m}{\Aone}$-module $\DcompQ{m}{\Aone}$ is stable but not holonomic.
For any closed point $s$ of $\Aone$, and $m''\geq m'\geq m$, we have 
$\EcompQ{m',m''}{s}(\DcompQ{m}{\Aone})=\EcompQ{m',m''}{s}$, which is of infinite type as $\mc{A}_{K,u'}([\omega_{m'},\omega_{m''}\})$-module.
%\end{enumerate}
\end{ex}

\begin{lem}
 \label{globalsectcalc}
 Let $\ms{M}$ be a stable holonomic $\DcompQ{m}{\Aone}$-module, and
 let $s$ in $\Aone$ be a singular point of $\ms{M}$. We get an isomorphism
 $\Gamma(\Emod{m,\dag}{s,\mb{Q}}(\ms{M}))
 \cong\Emod{m,\dag}{s,\mb{Q}}\otimes\ms{M}$. In particular, for a
 holonomic $\DdagQ{\Pone}(\infty)$-module $\ms{N}$ and its singularity
 $s\in\Aone$, we get $\Gamma(\ms{F}^{(s,\infty')}(\ms{N}))\cong
 \ms{E}^{\dag}_{s,\mb{Q}}\otimes\ms{N}$.
\end{lem}
\begin{proof}
 Using Proposition \ref{analytifcisom} and Corollary \ref{ananddag},
 there exists an affine open
 neighborhood $\ms{U}$ of $s$ such that for $m'\geq m$
 \begin{equation*}
    \EcompQb{m,m'}{\ms{U}}\otimes\ms{M}\cong
   \EcompQ{m,m'}{s}\otimes_{\Dcomp{m}{}}\ms{M},\qquad
  \Emodb{m,\dag}{\ms{U},\mb{Q}}\otimes\ms{M}\cong
   \Emod{m,\dag}{s,\mb{Q}}\otimes_{\Dcomp{m}{}}\ms{M}.
 \end{equation*}
 Thus we get the lemma by the fact that
 $\Emodb{m,\dag}{\ms{U},\mb{Q}}$ is a Fr\'{e}chet-Stein algebra (cf.\
 \cite[5.9]{Abe}).
\end{proof}

\subsubsection{}
\label{baseextcom}
Let $E$ be an unramified finite extension of $K$. Then there exists a
finite \'{e}tale morphism $r\colon\ms{S}_E\rightarrow\ms{S}$ sending $u$
to $u$. This defines a functor $\mr{Res}^E_K:=r_*$. \index{functors!.Res@$\mr{Res}^E_K$}
Then the following diagram of functors
\begin{equation*}
 \xymatrix@C=50pt{
  F\mbox{-Hol}(\ms{S}_E)\ar[r]^{\mr{Res}^E_K}\ar[d]_{\Phi_E^{(0,\infty')}}
  &F\mbox{-Hol}(\ms{S}_K)\ar[d]^{\Phi^{(0,\infty')}_K}\\
 F\mbox{-Hol}'(\eta'_E)\ar[r]_{\mr{Res}^{E}_{K}}&
  F\mbox{-Hol}'(\eta'_K)
  }
\end{equation*}
is commutative up to a canonical isomorphism. The verification is
straightforward.

\begin{prop}
 \label{exactness}
 The functors $\ms{F}^{(s,\infty')}$ and $\Phi^{(\mf{s},\infty')}$ are
 exact.
\end{prop}
\begin{proof}
 For $\ms{F}^{(s,\infty')}$, it follows from Lemma \ref{flatness}. For
 $\Phi^{(\mf{s},\infty')}$, use the fact that taking canonical extension
 is exact.
\end{proof}

\begin{dfn}
 \label{locASdiffmod}
 For $\widetilde{s}\in R$, we define a free differential
 $\mc{R}_{u,K}$-module $\DwL(\widetilde{s})$ \index{differential module!$\DwL(\widetilde{s})$, $\DwL(\mf{s})$, Dwork's|(} in the following way; the
 underlying module is $\mc{R}_{u,K}$. The connection is defined as
 follows:
 \begin{equation*}
  \nabla(1)=\widetilde{s}\cdot\pi u^{-2}\otimes du.
 \end{equation*}
 Let $\widetilde{s}'$ be an element of $R$ whose class in $k$ is
 equal to that of $\widetilde{s}'$. Then there exists a canonical
 isomorphism $\DwL(\widetilde{s})\xrightarrow{\sim}\DwL
 (\widetilde{s}')$ sending $1$ to
 $\exp(\pi(\widetilde{s}'-\widetilde{s})u^{-1})$. This shows that the
 differential module $\DwL(\widetilde{s})$ only depends on the class
 $s$ of $\widetilde{s}$ in $k$.

 Now, take an element $\mf{s}\in\mb{A}^1(\overline{k})$. Let
 $\widetilde{\mf{s}}\in K_s$ be a lifting of $\mf{s}$ in $k_s$. Then we
 get a differential $\mc{R}_{K_s}$-module
 $\DwL(\widetilde{\mf{s}})$. As proven above, this does not depend on
 the choice of liftings up to an isomorphism. We denote this abusively
 by $\DwL(\mf{s})$. This is called the {\em Dwork differential
 module}. \index{differential module!$\DwL(\widetilde{s})$, $\DwL(\mf{s})$, Dwork's|)}
\end{dfn}

\subsubsection{}
Using the notation of \ref{setupFour}, we have two
lemmas. The definition of $F$-$\DdagQ{\Pone}(\infty)$-modules is
recalled in \ref{setupFrob}.
\begin{lem*}
 \label{nongeometpt}
 Let $\ms{M}$ be a holonomic $F$-$\DdagQ{\Pone}(\infty)$-module, and
 $\mf{s}\in\mb{A}^1_k(\overline{k})$. Then there exists a canonical
 isomorphism
 \begin{equation*}
  \ms{F}^{(\widetilde{s},\infty')}_{K_s}(\ms{M}\otimes K_s)
   \xrightarrow{\sim}\Phi^{(\mf{s},\infty')}_{K_s}
   (\widetilde{\tau}_{\mf{s}*}\ms{M}|_{S_{s}})
 \end{equation*}
 in the category $\mr{Hol}'(\eta'_{K_s})$.
 Here $\ms{M}\otimes K_s$ denotes the pull-back of $\ms{M}$ to
 $\Aone_s=\Aone\otimes_R R_s$.
\end{lem*}
\begin{proof}
 There exists a canonical isomorphism
 $\alpha\colon\ms{S}_{\widetilde{s}}\xrightarrow{\sim}\ms{S}_{s}$.
 Using Corollary \ref{analyticstructure}, it
 suffices to show that there exists a canonical isomorphism
 $\alpha_*(\ms{M}\otimes K_s)|_{S_{\widetilde{s}}}
 \xrightarrow{\sim}\ms{M}|_{S_{s}}$. The verification is
 straightforward.
\end{proof}

\begin{lem}
 \label{locFourrel}
 Let $\ms{M}$ be a holonomic $F$-$\DdagQ{\Pone}(\infty)$-module, and
 $\mf{s}\in\mb{A}^1_k(\overline{k})$. Then we have
 a canonical isomorphism
 \begin{equation*}
  \ms{F}^{(s,\infty')}(\ms{M})\cong\mr{Res}^{K_s}_K
   (\Phi^{(0,\infty')}(\widetilde{\tau}_{\mf{s}*}\ms{M}|_{S_s})
   \otimes_{\mc{R}_{K_s}}\DwL(\mf{s})).
 \end{equation*}
\end{lem}
\begin{proof}
 When $s$ is a $k$-rational point, the verification consists into checking the
 definition using Corollary \ref{analyticstructure}. When $s$ is not a
 rational point, we need to check that
 $\mr{Res}^{K_s}_K(\ms{F}^{(\widetilde{s},\infty')}(\ms{M}
 \otimes K_s))\cong\ms{F}^{(s,\infty')}(\ms{M})$ by using the
 notation and result of Lemma \ref{nongeometpt} above. The verification
 is easy.
\end{proof}

\section{Complements to cohomological operations}
\label{section3}
In this section, we review some known results of six functors which are
indispensable in this paper, and give complements to properties of
geometric Fourier transforms defined by Noot-Huyghe. The proofs for the
properties of geometric Fourier transforms are almost the same as that
of \cite{Lau}, so we content ourselves by pointing out the differences.

\subsection{Cohomological operations}
\label{cohoprev}
\subsubsection{}
\label{setupFrob}
In this section, we assume $k$ to be perfect.
Let $h>0$ be an integer, and we put $q:=p^h$ \index{.@miscellaneous!hq@$h$, $q$} as usual.
We assume that there exists an automorphism
$\sigma\colon R\xrightarrow{\sim}R$ \index{.@miscellaneous!sigma@$\sigma$} which is a lifting of the absolute
$\fr$-th Frobenius on $k$. Let $\ms{X}$ be a smooth formal scheme over
$R$. We define $\ms{X}'$ by the following cartesian diagram.
\begin{equation*}
 \xymatrix{
  \ms{X}'\ar[r]\ar[d]\ar@{}[rd]|{\square}&\ms{X}\ar[d]\\\mr{Spf}(R)
  \ar[r]_\sigma&\mr{Spf}(R)} %\index{.@miscellaneous II!$\ms{X}'$}
\end{equation*}
For a $\DdagQ{\ms{X}}$-module $\ms{M}$, we denote  by $\ms{M}^\sigma$ %\index{functors!$(-)^{\sigma}$}
the
$\DdagQ{\ms{X}'}$-module defined by changing base by $\sigma$.
Note that even when there are no lifting of the relative
Frobenius $F^{(\fr)}_{X/k}\colon X\rightarrow X'$ \index{.@miscellaneous!F@$F^*$, $F^{(\fr)}_{X/k}$|(}
%\begin{equation*}
 %\xymatrix{
  %X\ar[dr]\ar[r]^-{F^{(\fr)}_{X/k}} & X'\ar[r]\ar[d]\ar@{}[rd]|{\square}& X\ar[d]\\
  %& \mr{Spec}(k) \ar[r]_\sigma&\mr{Spec}(k)}
%\end{equation*}
to a morphism of formal schemes $\ms{X}\rightarrow\ms{X}'$,
we are able to define the pull-back functor $F^*$\index{.@miscellaneous!F@$F^*$, $F^{(\fr)}_{X/k}$|)}  from the
category of $\DdagQ{\ms{X}'}$-modules to that of
$\DdagQ{\ms{X}}$-modules  (cf.\ \cite[Remarques
4.2.4]{Ber2}). Recall that an
$F^{(h)}$-$\DdagQ{\ms{X}}$-module is a pair of a $\DdagQ{\ms{X}}$-module
$\ms{M}$ and an isomorphism
$\ms{M}\rightarrow F^{(h)*}\ms{M}^\sigma$. %\index{F@$F^{(h)}$-$\DdagQ{\ms{X}}$-module, $F$-$\DdagQ{\ms{X}}$-module}
We often abbreviate $F^{(h)}$ by $F$ if there is nothing to be
confused. Let $D^b_{\mr{coh}}(\DdagQ{\ms{X}})$
\index{categories!Dbc@$D^b_{\mr{coh}}(\DdagQ{\ms{X}})$}
be the derived category of  $\DdagQ{\ms{X}}$-modules with bounded
coherent cohomology. We define a complex of
$F^{(h)}$-$\DdagQ{\ms{X}}$-modules as a complex $\ms{M}^{\bullet}$ in
$D^b_{\mr{coh}}(\DdagQ{\ms{X}})$ endowed with an isomorphism
$\Phi\colon\ms{M}^{\bullet}\rightarrow
F^{(h)*}(\ms{M}^{\bullet})^{\sigma}$ in
$D^b_{\mr{coh}}(\DdagQ{\ms{X}})$. We say that such complex
$(\ms{M}^{\bullet},\Phi)$ is \emph{holonomic} if its cohomology sheaves
are holonomic $F^{(h)}$-$\DdagQ{\ms{X}}$-modules, and we will denote by
$F^{(h)}$-$D^b_{\mr{hol}}(\DdagQ{\ms{X}})$
\index{categories!FhDbh@$F^{(h)}$-$D^b_{\mr{hol}}(\DdagQ{\ms{X}})$}
the category of holonomic $F^{(h)}$-$\DdagQ{\ms{X}}$-complexes with
bounded  cohomology, cf.\ \cite[5.3.5]{BerInt}. Let  $Z$ be a divisor of
the of $X$. We denote by $F^{(h)}$-$D^b_{\mr{hol}}(\DdagQ{\ms{X}}(Z))$
\index{categories!FhDbh@$F^{(h)}$-$D^b_{\mr{hol}}(\DdagQ{\ms{X}}(Z))$}
the full subcategory of $F^{(h)}$-$D^b_{\mr{hol}}(\DdagQ{\ms{X}})$ of
complexes $\ms{M}$ such that the canonical homomorphism
$\ms{M}\rightarrow \DdagQ{\ms{X}}(Z)\otimes_{\DdagQ{\ms{X}}}\ms{M}$ is
an isomorphism, cf.\ \cite[2.1.2]{Caro:courbes}. We call them holonomic
(complex of) $F^{(h)}$-$\DdagQ{\ms{X}}(Z)$-modules.

\subsubsection{}
\label{frobstrcor}
Let $R'$ be a discrete valuation ring finite \'{e}tale over $R$, and let
$\ms{C}$ be an object of $F^{(h)}$-$D^b_{\mr{hol}}(\mr{Spf}(R'))$. Let
$\Phi\colon\ms{C}\xrightarrow{\sim}F^{(h)*}\ms{C}$ be the Frobenius
structure of $\ms{C}$. There exists a canonical $\sigma$-semi-linear
homomorphism $\ms{C}\rightarrow F^{(h)*}\ms{C}^{\sigma}$ sending $x$ to
$1\otimes x$. The composition $\ms{C}\rightarrow
F^{(h)*}\ms{C}^{\sigma}\xrightarrow[\Phi^{-1}]{\sim}\ms{C}$, where the
first homomorphism is the canonical homomorphism, makes the complex
$\ms{C}$ a $\sigma$-$K$-vector spaces. We note that this homomorphism is
in fact an isomorphism since $\sigma$ is.
This correspondence induces an
equivalence between $F^{(h)}$-$D^b_{\mr{hol}}(\mr{Spf}(R))$ and the
category of finite $\sigma$-$K$-complexes, and we identify them.

\subsubsection{}
\label{sb:dec_lambda}
Let us fix notation for Dieudonn\'{e}-Manin slopes\index{slope!Dieudonne-Manin slope} and Tate twists. 
Recall that we denote by $e$ the absolute ramification index of $\Ct$.
We denote by $\Ct_{\sigma}\langle t \rangle$ the ring of non-commutative
polynomials defined by the relation $t\alpha = \sigma(\alpha) t$,
for every $\alpha\in\Ct$. For any $\alpha\in\Ct$ and integer $s>1$, we
put
\begin{equation*}
 K^{(\alpha; seh)}:=\Ct\langle t \rangle/\Ct\langle t
  \rangle(t^s-\alpha), \index{Fi@$F$-isocrystal, overconvergent isocrystal!$K^{(\alpha; seh)}$}
\end{equation*}
endowed with the Frobenius action given by the multiplication on the
left by $t$. It is a $\sigma$-$\Ct$-module of rank $s$. When
$\sigma(\varpi)=\varpi$, we normalize the Dieudonne-Manin slope so that
it is purely of slope $\lambda:=\frac{-\val[K](\alpha)}{seh}$.
For any smooth formal scheme $\ms{X}$ over $R$, we denote by
$\mc{O}_{\ms{X},\mb{Q}}^{(\alpha; seh)}$ the pull-back of $K^{(\alpha;
seh)}$ by the structural morphism of $\ms{X}$.

Let us define Tate twists. For any $\DdagQ{\ms{X}}$-module $\ms{M}$ and
integer $n$, we put
\begin{equation*}
 \ms{M}(n):=\ms{M}\otimes_{\mc{O}_{\ms{X},\mb{Q}}}
  \mc{O}_{\ms{X},\mb{Q}}^{(p^{-hn}; eh)}. \index{slope!Dieudonne-Manin slope!$n$-th Tate twist $\ms{M}(n)$}
\end{equation*}
This is called the  \emph{$n$-th Tate twist} of $\ms{M}$.
Let us define the twist by a Dieudonne-Manin slope $\lambda\in\QQ$.
There is a unique way to write  $\lambda=\frac{r}{seh}$,  where $r$ and
$s$ are coprime integers and $s>0$ (if $\lambda=0$ by convention we put
$r=0$ and $s=1$). Let $\varpi\in\Ctf$ be a uniformizer; for any coherent
$\DdagQ{\ms{X}}$-module $\ms{M}$ and $\lambda\in\QQ$, we put
\begin{equation*}
 \ms{M}^{(\lambda)}:=\ms{M}\otimes_{\mc{O}_{\ms{X},\mb{Q}}}
  \mc{O}_{\ms{X},\mb{Q}}^{(\varpi^{-r}; seh)}, \index{slope!Dieudonne-Manin slope!twist by $\lambda$, $\ms{M}^{(\lambda)}$}
\end{equation*}
which is usually called the \emph{twist}\footnote{This is called
{\itshape d\'ecal\'e} in French.}
of $\ms{M}$ by the \emph{slope} $\lambda$. When $\sigma(\varpi)=\varpi$,
its Dieudonne-Manin slope is indeed shifted by $+\lambda$. The notation
$\ms{M}^{(\lambda)}$ is slightly abusive because it depends on the
choice of the uniformizer $\varpi$, whereas that for Tate twists is
intrinsic. We give analogous definitions for overconvergent
$F$-isocrystals\footnote{For an overconvergent $F$-isocrystal, in
\cite{Marmora:Facteurs_epsilon}, $M^{(\lambda)}$ was denoted by
$M(\lambda)$. We modify here the notation to avoid any possible
confusion with Tate twists.} and free differentials modules with
Frobenius structure over the the Robba ring.

\subsubsection{}
\label{defpullbackpush}
To fix notation, let us review the theory of arithmetic
$\ms{D}$-modules concerning this paper.

We say that $(\ms{X},Z)$ is a d-couple\index{d-couple} if $\ms{X}$ is a smooth formal
scheme and $Z$ is a divisor of its special fiber. Here, $Z$ can be
empty. Let $(\ms{Y},W)$ be another d-couple. A morphism of d-couples
$f\colon(\ms{X},Z)\rightarrow(\ms{Y},W)$ is a morphism
$\widetilde{f}\colon\ms{X}\rightarrow\ms{Y}$ such that
$\widetilde{f}(\ms{X}\setminus Z)\subset\ms{Y}\setminus W$ and
$\widetilde{f}^{-1}(W)$ is a divisor. The morphism
$\widetilde{f}$ is called the realization of $f$, and if it is
unlikely to be confused, we often denote $\widetilde{f}$ by $f$. Let
$\textsf{P}$ be a property of morphisms. We say that the morphism of
d-couples $f$ satisfies the property $\textsf{P}$ if $\widetilde{f}$
satisfies the property $\textsf{P}$.

Let $\ms{X}$ be a smooth formal scheme over $R$, $X$ be its special
fiber, and $\ms{X}_K$ be its Raynaud generic fiber. Then we have
the specialization map $\mr{sp}\colon\ms{X}_K\rightarrow\ms{X}$ of
topoi. Recall that, we say that a $\DdagQ{\ms{X}}$-module $\ms{M}$ is
a convergent ($F$-)isocrystal if $\mr{sp}^*(\ms{M})$ is a convergent
($F$-)isocrystal, and the same for overconvergent ($F$-)isocrystals
(cf.\ \ref{conventionocisoc}).

Let $(\ms{X},Z)$ be a d-couple, and we denote by $\mb{D}_{\ms{X},Z}$ the
dual functor with respect to $\DdagQ{\ms{X}}(Z)$-modules. If it is
unlikely to cause any confusion, we often denote this by $\mb{D}$. \index{functors!six operations!D@$\mb{D}_{\ms{X},Z}$, $\mb{D}$}
 Let
$f\colon(\ms{X},Z)\rightarrow(\ms{Y},W)$ be a morphism of d-couples. We
have the extraordinary pull-back functor $\widetilde{f}^!$ from the
category of coherent $\DdagQ{\ms{Y}}(W)$-modules to that of
$\DdagQ{\ms{X}}(f^{-1}(W))$-modules (cf.\ \cite[4.3.3]{BerInt}). Let
$\ms{M}$ be a bounded coherent ($F$-)$\DdagQ{\ms{Y}}(W)$-complex.
When $\widetilde{f}^!(\ms{M})$ is coherent, we define $f^!(\ms{M})$ \index{functors!six operations!f@$f^"!$, $f^+$|(}to
be $\DdagQ{\ms{X}}(Z)\otimes\widetilde{f}^!(\ms{M})$.
Suppose in turn that $\widetilde{f}^!\circ\mb{D}_{\ms{Y},W}(\ms{M})$ is
a bounded coherent ($F$-)$\DdagQ{\ms{X}}(f^{-1}(W))$-complex. In this
case, we put
\begin{equation*}
  f^+\ms{M}:=\bigl(\mb{D}_{\ms{X},Z}\circ f^!\circ \mb{D}_{\ms{Y},W}\bigr)
  (\ms{M}). \index{functors!six operations!f@$f^"!$, $f^+$|)}
% f^+\ms{M}:=\mb{D}_{\ms{X},Z}\bigl(\DdagQ{\ms{X}}(Z)\otimes
%  _{\DdagQ{\ms{X}}(f^{-1}(W))}(\widetilde{f}^!\circ \mb{D}_{\ms{Y},W}
%  (\ms{M}))\bigr).
\end{equation*}

Now, suppose that the morphism of d-couples $f$ is a
proper morphism. We have the push-forward functor $\widetilde{f}_+$ from
the category of coherent $\DdagQ{\ms{X}}(f^{-1}(W))$-modules to that of
coherent $\DdagQ{\ms{Y}}(W)$-modules. Let $\ms{N}$ be a
($F$-)$\DdagQ{\ms{X}}(Z)$-module. Suppose $\ms{N}$ is coherent as a
$\DdagQ{\ms{X}}(f^{-1}(W))$-module. Then we denote by $j_+\ms{N}$\index{functors!six operations!j@$j_+$} this
coherent module. We define $f_+(\ms{N})$ \index{functors!six operations!f@$f_+$, $f_"!$|(}
to be the coherent
$\DdagQ{\ms{Y}}(W)$-module $\widetilde{f}_+(j_+(\ms{N}))$.
Assume in turn that $\mb{D}_{\ms{X},Z}(\ms{N})$ is coherent as a
$\DdagQ{\ms{X}}(f^{-1}(W))$-module. We define
\begin{equation*}
 f_!\ms{N}:=\bigl(\mb{D}_{\ms{Y},W}\circ f_+\circ\mb{D}_{\ms{X},Z}\bigr)
  (\ms{N}). \index{functors!six operations!f@$f_+$, $f_"!$|)}
\end{equation*}

When we are given $\ms{M}$ and $\ms{N}$ in $\LD(\widehat{\ms{D}}
^{(\bullet)}_{\ms{X}}(Z))$ \index{categories!LD@$\LD(\widehat{\ms{D}}
^{(\bullet)}_{\ms{X}}(Z))$} (cf.\ \cite[4.2.2, 4.2.3]{BerInt} for the
notation when $Z$ is empty, but the construction is the same), we denote
the object $\ms{M}\otimesdag{}^\dag_{\mc{O}_{\ms{X},\mb{Q}}(Z)}\ms{N}$
in $\LD(\widehat{\ms{D}}^{(\bullet)}_{\ms{X}}(Z))$ by
$\ms{M}\otimes\ms{N}$. \index{functors!six operations!t@$\otimes$, $\widetilde{\otimes}$|(}

\subsubsection{}
\label{prop6func}
\newcounter{enumip}
In this paragraph, let us summarize some properties of the cohomological
functors defined in the previous paragraph which will be used later in
this paper. Let $f\colon(\ms{X},Z)\rightarrow(\ms{Y},W)$ be a morphism
of d-couples. We denote respectively by $d_{\ms{X}}$, $d_{\ms{Y}}$,
$d$ the dimension of $\ms{X}$, $\ms{Y}$, and
$d_{\ms{X}}-d_{\ms{Y}}$. Then we get the following.
\begin{enumerate}
 \item \label{smoothpoin}
       If $f$ is smooth and of constant relative dimension, we get
       $f^!\cong f^+(d)[2d]$ (cf.\ \cite[Theorem 5.5]{Abe3}). If $f$ is
       a closed immersion of connected formal schemes and $\ms{M}$ is an
       overconvergent $F$-isocrystal, we get $f^!(\ms{M})\cong
       f^+(\ms{M})(d)[2d]$ (cf.\ \cite[Theorem 5.6]{Abe3}).
       For conventions on brackets and parenthesis see \eqref{convention-Tatetwist-shifts}. 
 \item If $f$ is proper and $\widetilde{f}|_{\ms{X}\setminus Z}
       \colon\ms{X}\setminus Z\rightarrow\ms{Y}\setminus W$ is also
       proper, we get $f_!\xrightarrow{\sim}f_+$ (cf.\ \cite[Theorem
       5.4]{Abe3}).
 \item Let $\ms{M}$ be an object of $D^b_{\mr{coh}}(\DdagQ{\ms{X}}(Z))$
       and $\ms{N}$ be an object of
       $D^b_{\mr{coh}}(\DdagQ{\ms{Y}}(W))$. There exists a canonical
       isomorphism $f_+(\ms{M}\otimes f^!\ms{N})\cong
       f_+\ms{M}\otimes\ms{N}$ (cf.\ \cite[2.1.4]{Carocoh}. The
       compatibility with Frobenius pull-back can be seen easily from
       the definition of the homomorphism.).
 \setcounter{enumip}{\theenumi}
\end{enumerate}
By using \ref{smoothpoin} above, we get
$\mb{D}_{\ms{X},Z}(\mc{O}_{\ms{X},\mb{Q}}(Z))\cong\mc{O}_{\ms{X},\mb{Q}}
(Z)(-d_{\ms{X}})$. Now, we define the twisted tensor product
$\widetilde{\otimes}$ \index{functors!six operations!t@$\otimes$, $\widetilde{\otimes}$|)}
on $\ms{X}$ to be
$\mb{D}_{\ms{X},Z}(\mb{D}_{\ms{X},Z}(-)\otimes\mb{D}_{\ms{X},Z}(-))$. We
note that the definition is slightly different from that of
\cite{NH}. The main reason we introduce this new tensor product is the
following.
\begin{enumerate}
 \setcounter{enumi}{\theenumip}
 \item \label{raisondetretwist}
       Assume that $f$ is a morphism of connected formal schemes. There
       exists a canonical isomorphism $f^!((-)\otimes(-))[d]\cong
       f^!(-)\otimes f^!(-)$. Similarly, we also have
       $f^+((-)\widetilde{\otimes}(-))[-d]\cong f^+(-)\widetilde{\otimes} f^+(-)$ (cf.\
       \cite[5.8]{Abe3}).
 \setcounter{enumip}{\theenumi}
\end{enumerate}
The following result enables us to compare these two tensor products in
special cases.
\begin{enumerate}
 \setcounter{enumi}{\theenumip}
 \item \label{twisttensor}
       Assume $\ms{X}$ is connected.
       If $\ms{M}$ be an overconvergent $F$-isocrystal, and $\ms{N}$ be
       a coherent $F$-$\DdagQ{\ms{X}}(Z)$-module. Then we get
       $\ms{M}\widetilde{\otimes}\ms{N}\cong\ms{M}\otimes
       \ms{N}(d_{\ms{X}})$ (cf.\ \cite[Proposition 5.8]{Abe3}).
 \setcounter{enumip}{\theenumi}
\end{enumerate}

\subsubsection{}
\label{properbasech}
Now, let us see the proper base change theorem. Consider the following
cartesian diagrams of d-couples.
\begin{equation*}
 \xymatrix{
  (\ms{X}',Z')\ar[r]^{i'}\ar[d]_{f'}\ar@{}[rd]|{\square}&(\ms{X},Z)
  \ar[d]^{f}\\(\ms{Y}',W')\ar[r]_i&(\ms{Y},W)
  }
\end{equation*}
Here, we say that the diagram of d-couples is cartesian if it is
cartesian for the underlying smooth formal schemes, and
$\ms{X}'\setminus Z'=(\ms{X}\setminus Z)\times_{(\ms{Y}'\setminus
W')}(\ms{Y}\setminus W)$. Then we get $i^!\circ f_+\cong f'_+\circ
i'^!$ (cf.\ \cite[Theorem 5.7]{Abe3}). This isomorphism is compatible
with Frobenius structures by the same theorem. We call this the proper
base change isomorphism.

Assume that $f$ is proper. In this case, for a bounded coherent
($F$)-$\DdagQ{\ms{X}}$-complex $\ms{M}$, we get $i^+\circ
f_!(\ms{M})\cong f'_!\circ i'^+(\ms{M})$ if the both sides are
defined. This follows by using $\mb{D}\circ\mb{D}=\mr{id}$ \cite[II,
3.5]{Vir}. This is also called the proper base change isomorphism.

\subsubsection{}
\label{kunnethformual}
We also have the K\"{u}nneth formula. Namely, let
$f\colon\ms{X}\rightarrow\ms{X}'$ and $g\colon\ms{Y}\rightarrow\ms{Y}'$
be smooth morphisms between smooth formal schemes over a smooth formal
scheme $\ms{T}$. Let $D$ be a divisor of the special fiber of
$\ms{T}$. We denote by $D_{\ms{X}^{(')}}$ (resp.\ $D_{\ms{Y}^{(')}}$,
$D'$) be the divisor of the special fiber of $\ms{X}^{(')}$ (resp.\
$\ms{Y}^{(')}$, $\ms{X}'\times_{\ms{Y}}\ms{Y}'$) which is the pull-back
of $D$. Let $\ms{M}$ (resp.\ $\ms{N}$) be an element of
$\underrightarrow{LD}^b_{\mb{Q},\mr{qc}}(\widehat{\ms{D}}
^{(\bullet)}_\ms{X}(D_{\ms{X}}))$ (resp.\ $\underrightarrow{LD}^b
_{\mathbb{Q},\mathrm{qc}}(\widehat{\mathscr{D}}^{(\bullet)}
_\mathscr{Y}(D_{\ms{Y}}))$). Then we get
\begin{equation*}
 (f\times g)_+(\ms{M}\boxtimes_{\mc{O}_{\ms{T}}(D)}^{\mb{L}}\ms{N})
  \cong(f_+\ms{M})\boxtimes_{\mc{O}_{\ms{T}}(D)}^{\mb{L}}(g_+\ms{N})
\end{equation*}
in $\underrightarrow{LD}^b_{\mb{Q},\mr{qc}}(\widehat{\ms{D}}^{(\bullet)}
_{\ms{X}'\times_{\ms{T}}\ms{Y}'}(D'))$. To see this, we apply the
K\"{u}nneth formula \cite[Proposition 4.9]{Abe3} to the diagram
\begin{equation*}
 \xymatrix@R=10pt@C=30pt{
  &\ms{X}\times_{\ms{T}}\ms{Y}\ar[dr]^{f\times\mr{id}}\ar[dl]
  _{\mr{id}\times g}\ar[dd]_{f\times g}&\\\ms{X}\times_{\ms{T}}
  \ms{Y}'\ar[dr]_{f\times\mr{id}}
  &&\ms{X}'\times_{\ms{T}}\ms{Y}\ar[dl]^{\mr{id}\times g}\\&
  \ms{X}'\times_{\ms{T}}\ms{Y}'&}
\end{equation*}
and $\ms{M}\boxtimes^{\mb{L}}_{\mc{O}_{\ms{T}}(D)}
\mc{O}_{\ms{Y}'}(D_{\ms{Y}'})$ on $\ms{X}\times_{\ms{T}}\ms{Y}'$ and
$\mc{O}_{\ms{X}'}(D_{\ms{X}'})\boxtimes_{\mc{O}_{\ms{T}}(D)}^{\mb{L}}\ms{N}$
on $\ms{X}'\times_{\ms{T}}\ms{Y}$.

\subsubsection{}
\label{Dmodandrigcohcomp}
In this paragraph, let us see the relation between the rigid
cohomology and the push-forward of arithmetic $\ms{D}$-modules.
Let $\ms{X}$ be a {\em proper} smooth formal scheme of dimension $d$.
Let $Z$ be a divisor of the special fiber of $\ms{X}$, $\ms{U}$ be the
complement, and $U$ be its special fiber. We denote by
$\mr{sp}\colon\ms{X}_K\rightarrow\ms{X}$ the specialization map where
$\ms{X}_K$ denotes the Raynaud generic fiber of $\ms{X}$. Let
$f\colon(\ms{X},Z)\rightarrow\mr{Spf}(R)$ be the structural
morphism. Let $\ms{M}$ be a coherent $\DdagQ{\ms{X}}(Z)$-module which is
overconvergent along $Z$. Suppose that it is coherent as a
$\DdagQ{\ms{X}}$-module. In \ref{prop6func}, we noted
that $\mb{D}_{\ms{X},Z}(\mc{O}_{\ms{X},\mb{Q}}(Z))\cong
\mc{O}_{\ms{X},\mb{Q}}(Z)(-d)$. For an isocrystal $M$, we denote by
$M^\vee$ the dual as isocrystal. This isomorphism leads us to the
following comparison of dual functors (cf.\ \cite[Corollary
3.12]{Abe3}):
\begin{equation}
 \label{carodualfrob}
 \mr{sp}^*(\mb{D}_{\ms{X},Z}(\ms{M}))\cong(\mr{sp}^*(\ms{M}))^\vee(-d).
\end{equation}

We get the following relation with the rigid cohomology:
\begin{equation}
 \label{rigDmodpush}
 \ms{H}^i(f_+\ms{M})\cong H^{d+i}_{\mr{rig}}(U,\mr{sp}^*\ms{M})(d), \index{functors!.Hrig@$H^{d+i}_{\mr{rig}}$, $H^{d+i}_{\mr{rig},c}$|(}
\end{equation}
where $d=\dim(U)$. For the details of the proof, see \cite[3.14]{Abe3}.
To see the relation for cohomologies with compact support, we use the
Poincar\'{e} duality of rigid cohomology to get
\begin{equation}
 \label{rigDmodcompact}
  \ms{H}^i(f_!\ms{M})\cong H^{d+i}_{\mr{rig},c}(U,\mr{sp}^*\ms{M})(d). \index{functors!.Hrig@$H^{d+i}_{\mr{rig}}$, $H^{d+i}_{\mr{rig},c}$|)}
\end{equation}
For the detailed account, one can refer to \cite[5.9]{Abe3}.

When $\ms{X}$ is a curve, and for a holonomic $\DdagQ{\ms{X}}(
Z)$-module $\ms{M}$, we get that the following pairing
\begin{equation}
 \label{poincaredual}
 \ms{H}^i(f_+\ms{M})\times \ms{H}^{-i}(f_!\,\mb{D}_{\ms{X},Z}(\ms{M}))
  \rightarrow K
\end{equation}
is perfect. This can be seen from \cite[5.5]{Abe3}.

\subsubsection{}
\label{defvanishcycle}
For the later use, we review the cohomological functors $i_+$, $j^+$,
$j_!$, $i^!$, and $\mb{D}$ in the theory of formal disks. Let $\ms{S}$
be the formal disk over $K$. For an object ${M}$ in
$F$-$\mr{Hol}(\eta)$, $j_+({M})$ is by definition \index{functors!six operations!j@$j_+$}
the underlying $F$-$\Dan{\ms{S},\mb{Q}}$-module. For an object ${N}$
in $F$-$\mr{Hol}(\ms{S})$, we denote
$\Dan{\ms{S},\mb{Q}}(0)\otimes {N}$ in $F$-$\mr{Hol}(\eta)$ by
$j^+ {M}$. We denote by
$i\colon\{0\}\hookrightarrow\ms{S}$ the closed immersion. The
definitions of the functors $i_+$, $i^!$, $j^!\cong j^+$ are essentially the \index{functors!six operations!i@$i^"!$, $i^+$|(}
\index{functors!six operations!i@$i_+$}
same as in the global case, and are used frequently in \cite{Cr}, so for
the details see [{\it loc.\ cit}, 3.4, {\it etc.}]. 

Now, we denote by $\mb{D}_{\ms{S}}$ (resp.\ $\mb{D}_\eta$) the dual \index{functors!six operations!D@$\mb{D}_{\ms{S}}$, $\mb{D}_\eta$}
functor with respect to $\Dan{\ms{S},\mb{Q}}$ (resp.\
$\Dan{\ms{S},\mb{Q}}(0)$). These define functors from
$F$-$\mr{Hol}(\ms{S})$ (resp.\ $F$-$\mr{Hol}(\eta)$) to itself by
\cite[5.2]{Cr}. For an object ${M}$ in $F$-$\mr{Hol}(\eta)$, we
define $j_!{M}:=\mb{D}_{\ms{S}}j_+\mb{D}_{\eta}({M})$, \index{functors!six operations!j@$j_"!$}
and for an
object ${N}$ in $F$-$\mr{Hol}(\ms{S})$, we put
$i^+{N}:=(i^!\,\mb{D}_{\ms{S}}{N})^\vee$ where $^\vee$ denotes the \index{functors!six operations!i@$i^"!$, $i^+$|)}\index{functors!.4@$(-)^\vee$}
dual of $\sigma$-$K$-vector spaces. By using facts of
\ref{prop6func}, we get isomorphisms
$i_+\,\mb{D}_\eta\cong\mb{D}_{\ms{S}}\,i_+$ and
$j^+\,\mb{D}_{\ms{S}}\cong\mb{D}_{\eta}\,j^+$. By using these
isomorphisms, the localization triangle \cite[3.4.3]{Cr} induces the
following distinguished triangle:
\begin{equation}
 \label{anotherloctriag}
 j_!\,j^+\rightarrow\mr{id}\rightarrow i_+\,i^+\xrightarrow{+1}.
\end{equation}

\begin{dfn*}
 For a holonomic $F$-$\Dan{\ms{S},\mb{Q}}$-module ${M}$, we put
 $\Psi({M}):=\mb{V}(\mb{D}_{\ms{S}}{M})$ and
 $\Phi({M}):=\mb{W}(\mb{D}_{\ms{S}}{M})$, and call them the nearby
 cycles and the vanishing cycles respectively. These define functors
 \begin{equation*}
  \Psi,\Phi\colon F\mbox{-}\mr{Hol}(\Dan{\ms{S},\mb{Q}})
   \rightarrow\Del{\Ct\nr}{G_{\mc{K}}}. \index{functors!.Psi@$\Psi$, $\Phi$}
 \end{equation*}
\end{dfn*}

We note here that when ${M}$ is a free differential module on
$\mc{R}$ with Frobenius structure, we get
$\mb{D}_{\eta}({M})\cong {M}^\vee(-1)$, where $^\vee$ denotes the
dual as a $\mc{R}$-module by (\ref{carodualfrob}). For example
$\Psi(\mc{R})\cong\mb{V}(\mc{R})(1)\cong K^{\mr{ur}}(1)$. On the other
hand, if ${M}$ is of punctual type such
that ${M}=V\otimes_K\delta$ where $V$ is a $K$-vector space with
Frobenius structure, we get that $\Phi({M})\cong V_{K^{\mr{ur}}}$.

\begin{lem}
 \label{calcofdiff}
 Let ${M}$ be an object of $F$-$\mr{Hol}(\eta)$. Then we get the
 following exact sequence:
 \begin{equation*}
  0\rightarrow {M}^{\partial=0}(1)\otimes_K\delta
   \rightarrow j_!{M}\rightarrow j_+{M}\rightarrow
   {M}/\partial {M}(1)\otimes_K\delta\rightarrow0.
 \end{equation*}
\end{lem}
\begin{proof}
 Since $j_+j^+j_!\,{M}\cong j_+{M}$, we get
 $\Psi(j_!\,{M})\cong\Psi(j_+\,{M})$ by \cite[Prop.~6.1.1]{Cr}. We also get
 $\Psi(j_!{M})\xrightarrow{\sim}\Phi(j_!{M})$. Thus,
 (\ref{vanishexact}) induces the following exact sequence:
 \begin{equation*}
  0\rightarrow\mr{Hom}_{\ms{D}^{\mr{an}}}(\mb{D}(j_+{M}),
   \mc{O}^{\mr{an}}_{K^{\mr{ur}}})\rightarrow\Phi(j_!\,{M})
   \rightarrow\Phi(j_+\,{M})\rightarrow\mr{Ext}^1_{\ms{D}^{\mr{an}}}
   (\mb{D}(j_+{M}),\mc{O}^{\mr{an}}_{K^{\mr{ur}}})\rightarrow0.
 \end{equation*}
 We get isomorphisms
 \begin{equation*}
  R\mr{Hom}_{\ms{D}^{\mr{an}}}(\mb{D}_{\ms{S}}(j_+{M}),
   \mc{O}^{\mr{an}})\cong
   R\mr{Hom}_{\ms{D}^{\mr{an}}}(\mc{O}^{\mr{an}}(-1),j_+{M})\cong
   R\mr{Hom}_{\ms{D}^{\mr{an}}(0)}(\mc{R},{M})(1).
 \end{equation*}
 Here the first isomorphism follows from the fact that $\mb{D}_{\ms{S}}$
 gives an anti-equivalence of categories combined with the isomorphism
 $\mb{D}_{\ms{S}}(\mc{O}^{\mr{an}})\cong\mc{O}^{\mr{an}}(-1)$, and the
 second by adjoint. Thus the lemma follows.
\end{proof}

\begin{rem*}
 We note that the dimension of ${M}^{\partial=0}$ and
 ${M}/\partial{M}$ over $K$ are the same by the index theorem of
 Christol-Mebkhout \cite[14.13]{CM}.
\end{rem*}

\subsection{Geometric Fourier transforms}
\label{secgeforev}
\subsubsection{}
\label{defoflofdw}
We briefly review the definition of geometric Fourier transform due to
Noot-Huyghe \cite{NH}. For simplicity, we only review under the
situation of \ref{setupFour}.

To define Fourier transforms, we need to define an integral kernel
$\ms{L}_{\pi}$ of the transform. \index{Fi@$F$-isocrystal, overconvergent isocrystal!$\ms{L}_{\pi}$, $\ms{L}(s\cdot x')$, Dwork's|(}
We define a convergent
$F$-isocrystal on $\widehat{\mb{A}}^1$ overconvergent along $\infty$
denoted by $\ms{L}_\pi$ in the following way. Let $t$ be the
coordinate. As an
$\mc{O}_{\widehat{\mb{P}}_{R,t}^1,\mb{Q}}(\infty)$-module, it is
$\mc{O}_{\widehat{\mb{P}}_{R,t}^1,\mb{Q}}(\infty)$. We denote the element
corresponding to $1$ by $e$. We define its connection by
\begin{equation*}
 \nabla(e)=-\pi e\otimes dt.
\end{equation*}
This module is equipped with Frobenius structure. The
Frobenius structure $\Phi\colon
F^*\ms{L}_\pi\xrightarrow{\sim}\ms{L}_\pi$ is defined by
\begin{equation*}
 \Phi(1\otimes e):=\exp(\pi(t-t^q))e. \index{Fi@$F$-isocrystal, overconvergent isocrystal!$\ms{L}_{\pi}$, $\ms{L}(s\cdot x')$, Dwork's|)}
\end{equation*}

Now, let us consider the situation in \ref{setupFour}. There exists the
canonical coupling
$\mu\colon\Aone\times\Aoned\rightarrow\widehat{\mb{A}}^1_{R,t}$ sending $t$
to $x\otimes x'$. \index{.@miscellaneous!mu@$\mu$, canonical coupling}
By the general theory of overconvergent
$F$-isocrystals, the pull-back $\mu^*\ms{L}_\pi$ is a
convergent $F$-isocrystal on $\Aone\times\Aoned$ overconvergent along
$Z$ in $\Poneoned$. This is a coherent $\DdagQ{\Poneoned}(Z)$-module, and
its restriction to $\Aone\times\Aoned$ is nothing but
$\ms{H}^{-1}(\mu^!\ms{L}_\pi)$. By abuse of language, we denote this
$\DdagQ{\Poneoned}(Z)$-module by $\mu^!\ms{L}_\pi[-1]$, or sometimes by
$\ms{L}_{\pi,\mu}$. In the same way, there exists a
unique coherent complex of $\DdagQ{\Poneoned}(Z)$-modules whose
restriction to $\Aone\times\Aoned$ is $\ms{H}^1(\mu^+\ms{L}_{\pi})$. We
also denote this by $\mu^+\ms{L}_{\pi}[1]$.

\subsubsection{}
Now, let us recall the definition of the geometric Fourier transform. We
continuously use the notation of \ref{setupFour}.
Recall the diagram (\ref{fourierdiagtwo}). Let $\ms{M}$ be a coherent
$\DdagQ{\Pone}(\infty)$-module. Noot-Huyghe defined the geometric
Fourier transform of $\ms{M}$ to be
\begin{align}
 \label{geomfour}
 \ms{F}_{\pi}(\ms{M})&:=p'_+(\ms{L}_{\pi,\mu}\otimes_{\mc{O}
 _{\Poneoned,\mb{Q}}(\infty)}^{\mb{L}} p^!\ms{M}[-2])\\
 \notag&\bigl(=p'_+(\mu^!\ms{L}_{\pi}\otimes_{\mc{O}
 _{\Poneoned,\mb{Q}}(\infty)}^{\mb{L}} p^!\ms{M}[-3])
 \bigr). \index{functors!geometric (global) Fourier transform!$\ms{F}_{\pi}$, $\ms{F}$}
\end{align}
She also proved that $p'_!~(\ms{L}_{\pi,\mu}\otimes_{\mc{O}
_{\Poneoned,\mb{Q}}(\infty)}^{\mb{L}} p^!\ms{M})$ is well-defined, and
also showed an analog of the result of Katz and Laumon
\cite{Huy2} (see \cite[A.3.2]{Abe3} for a proof). Namely, the canonical
homomorphism
\begin{equation}
 \label{KatzLaumonisom}
 p'_!~(\ms{L}_{\pi,\mu}\otimes_{\mc{O}_{\Poneoned,\mb{Q}}(\infty)}
 ^{\mb{L}} p^!\ms{M}[-2])\rightarrow\ms{F}_\pi(\ms{M})
\end{equation}
is an isomorphism.
Since Fourier transform is defined using three cohomological functors
$p'_+$, $\otimes$, $p^!$, and there is a canonical Frobenius structure
on $\ms{L}_{\pi,\mu}$, Fourier transform commutes with Frobenius
pull-backs. In particular, if $\ms{M}$ is a coherent
$F$-$\DdagQ{\ms{X}}$-complex, there exists a canonical Frobenius
structure on the complex $\ms{F}_\pi(\ms{M})$.

\begin{lem}
 \label{Tsuzukicalc}
 We have a canonical isomorphism
 \begin{equation*}
  \mu^+\ms{L}_\pi(1)[2]\cong\mu^!\ms{L}_\pi.
 \end{equation*}
\end{lem}
\begin{proof}
 Note that $\mu|_{\Aone\times\Aoned\setminus\{(0,0')\}}$ is
 smooth. This shows that
 $\mu^!\ms{L}_{\pi}$ and $\mu^+\ms{L}_{\pi}(1)[2]$ are generically
 isomorphic by \ref{prop6func}.\ref{smoothpoin}. The module
 $\mu^!\ms{L}_\pi[-1]$ is concentrated at degree $0$, and it is finite
 over $\mc{O}_{\Pone,\mb{Q}}(\infty)$. This implies that it is an
 overconvergent isocrystal along $\infty$. Since the dual of an
 isocrystal is also an isocrystal, $\mu^+\ms{L}_\pi$ is also
 an overconvergent isocrystal along $\infty$. Moreover, both sides are
 overconvergent $F$-isocrystals. Thus, the two modules in the statement
 are isomorphic by using \cite{Kedful} and \cite[Theorem 4.1.1]{Tsu}.
\end{proof}

\subsubsection{}
Let $s$ be a closed point of $\Aone$. Let $k_s$ be the residue
field and $K_s$ be the corresponding unramified extension of $K$,
$R_s$ its valuation ring. Then there exists a closed immersion
$i_s\colon\Poned_{R_s}\hookrightarrow\Poneoned_{R_s}$ sending $x'$ to
$(s,x')$. We define $\ms{L}(s\cdot
x'):=i^!_s(\ms{L}_{\pi,\mu})$. \index{Fi@$F$-isocrystal, overconvergent
isocrystal!$\ms{L}_{\pi}$, $\ms{L}(s\cdot x')$, Dwork's}
In the same way, given a closed point $s'$ in $\Aoned$, we define
$\ms{L}(x\cdot s')$ on $\Pone$. When $s$ is a rational point, we can
check that $\tau'_*(\ms{L}(s\cdot x')|_{\eta_{\infty'}})\cong
\DwL'(s)$, where $\DwL'(s)$ is the Dwork differential module
\ref{locASdiffmod} on $\eta'$.

\subsubsection{}\label{NFT}
Finally, let us review a fundamental property of Fourier transform
shown by Noot-Huyghe. There exists an isomorphism of rings
\begin{equation*}
 \iota\colon\Gamma(\Poned,\DdagQ{\Poned}(\infty'))\xrightarrow{\sim}
  \Gamma(\Pone,\DdagQ{\Pone}(\infty))
\end{equation*}
sending $x'$ to $\pi^{-1}\partial$ and $\partial'$ to $-\pi x$. It is
also shown by Huyghe that coherent
$\DdagQ{\Pone}(\infty)$-modules corresponds to
$\Gamma(\Pone,\DdagQ{\Pone}(\infty))$-modules by taking global
sections. Given a coherent
$\Gamma(\Pone,\DdagQ{\Pone}(\infty))$-module $M$, we denote
by $\nF{M}$ \index{functors!naive (global) Fourier transform!$\ms{F}_{\mr{naive},\pi}$, $\ms{F}_{\mr{naive}}$|(}  
the coherent
$\Gamma(\Poned,\DdagQ{\Poned}(\infty'))$-module obtained from $M$ via
transport of structure by $\iota$. For $m\in M$ we denote by
$\widehat{m}$ the corresponding element of $\nF{M}$.

Let $\ms{M}$ be a coherent $\DdagQ{\Pone}(\infty)$-module. We denote by
$\nF{\ms{M}}$ 
the coherent $\DdagQ{\Poned}$-module corresponding to
$\nF{\Gamma(\Pone,\ms{M})}$. We call this the {\em naive Fourier
transform} of $\ms{M}$. Then Noot-Huyghe showed in \cite[5.3.1]{NH} that
there exists a canonical isomorphism
\begin{equation}
 \label{compgeomnaive}
 \ms{F}_\pi(\ms{M})\xrightarrow{\sim}\ms{F}_{\mr{naive},\pi}
  (\ms{M})[-1].
\end{equation}
We often denote by $m\in\ms{M}$ to mean $m\in\Gamma(\Pone,\ms{M})$.  
 For simplicity, we will often omit
the index $\pi$ in the notation $\nF{\ms{M}}$ or 
$\ms{F}_{\pi}(\ms{M})$. \index{functors!naive (global) Fourier transform!$\ms{F}_{\mr{naive},\pi}$, $\ms{F}_{\mr{naive}}$|)}

\subsubsection{}
Standard properties of $\ell$-adic geometric Fourier transform
explained in \cite[1.2, 1.3]{Lau} hold also for $p$-adic Fourier
transform with suitable changes. Since the proof works well with
little change, we leave the reader to formulate and verify these
properties. Although the most of these properties are not used in this
paper, we still need a few analogous results. We will write the
statements of these results with short comments of the proofs.

\begin{prop}[{\cite[1.2.2.2]{Lau}}]
 \label{calcfouriereasy}
 Let $V$ be a coherent $F$-$\DdagQ{\mr{Spf}(R)}$-module ({\it i.e.\
 }a finite dimensional $\sigma$-$K$-vector space). Then, we have a
 canonical isomorphism
 \begin{equation*}
  \ms{F}(q^!(V))\cong i_{0'+}(V)(1),
 \end{equation*}
 where $q\colon\Pone\rightarrow\mr{Spf}(R)$ denotes the structural
 morphisms, and $i_{0'}\colon\mr{Spf}(R)\hookrightarrow\Poned$ is the
 closed immersion defined by $0'$ in $\Poned$.
\end{prop}

\begin{rem*}
 When $V$ is trivial, another calculation for this Fourier transform was
 carried out by Baldassarri and Berthelot in \cite{BB}. In their
 calculation, there are no Tate twist contrary to our calculation
 here. This is because the definitions of the Frobenius structures on
 the geometric Fourier transform are slightly different. For the precise
 argument, see \cite[Remark 3.15 (ii)]{Abe3}.
\end{rem*}

\subsubsection{}
Let $\ms{F}'$ be the dual geometric Fourier transform: the functor
$\ms{F}'\colon D^b_{\mr{coh}}(\DdagQ{\Poned})\rightarrow
D^b_{\mr{coh}}(\DdagQ{\Pone})$ defined in the same way as $\ms{F}$
except for reversing the role of $p$ and $p'$. We get the following
inversion formula.

\begin{thm*}[{\cite[1.2.2.1]{Lau}}]
 \label{involForge}
 Let $\ms{M}$ be a coherent $F$-$\DdagQ{\Pone}(\infty)$-module. Then,
 there exists a functorial isomorphism
 \begin{equation*}
  \ms{F}'\circ\ms{F}(\ms{M})[2]\cong\ms{M}(1).
 \end{equation*}
\end{thm*}

\begin{rem*}
 We may also prove the theorem in more general cases: let $\ms{X}$ be a
 smooth formal scheme over $\mr{Spf}(R)$, and let $\ms{E}$ be a locally
 free sheaf of finite rank $r$ on $\ms{X}$. Consider the projective
 bundle $p\colon\ms{P}:=\widehat{\mb{P}}(\ms{E}\oplus\mc{O}_{\ms{X}})
 \rightarrow\ms{X}$ to define the geometric Fourier
 transform (cf.\ \cite[3.2.1]{NH}). Then the theorem is reformulated as
 $\ms{F}'\circ\ms{F}(\ms{M})[4-2r]\cong\ms{M}(r)$. Let $Z$ be the
 divisor at infinity of $\ms{P}$. For the proof, we need to show that
 there exists an isomorphism $\pi_+(\mc{O}_{\ms{P},\mb{Q}}(Z))\cong
 \mc{O}_{\ms{X},\mb{Q}}[r](r)$. This can be seen from \cite[Cor.\ 4.4]{Pet}
 and \cite[3.14 or 3.15 (i)]{Abe3}.
\end{rem*}

\subsubsection{}
Let $\alpha\in(q-1)^{-1}\mb{Z}$. We define a convergent $F$-isocrystal
$\ms{K}_\alpha$ on $\Pone\setminus\{0,\infty\}$ overconvergent along $\{0,\infty\}$ in
the following way. As an $\mc{O}_{\Pone,\mb{Q}}(0,\infty)$-module, it
is isomorphic to $\mc{O}_{\Pone,\mb{Q}}(0,\infty)$. We denote the global section
corresponding to $1$ by $e$. We define its connection by
\begin{equation*}
 \nabla(e)=(\alpha x^{-1})\cdot e\otimes dx.
\end{equation*}
The Frobenius structure $\Phi\colon F^*\ms{K}_\alpha\xrightarrow{\sim}
\ms{K}_\alpha$ is defined by
\begin{equation*}
 \Phi(1\otimes e):=x^{\alpha(q-1)}\cdot e.
\end{equation*}
We often use the same notation $\ms{K}_\alpha$ for the underlying
coherent $\DdagQ{\Pone}(\infty)$-module. This is called the {\em Kummer
isocrystal}. \index{Fi@$F$-isocrystal, overconvergent isocrystal!$\ms{K}_\alpha$, Kummer's}

\begin{prop*}[{\cite[1.4.3.2]{Lau}}]
 \label{calcKummergeom}
 Let $j\colon(\Pone,\{0,\infty\})\rightarrow(\Pone,\{\infty\})$ be the
 canonical morphism of couples. Assume $\alpha\not\in\mb{Z}$. Then we
 get that the canonical homomorphism
 \begin{equation*}
  j_!j^+\ms{K}_\alpha\rightarrow\ms{K}_\alpha
 \end{equation*}
 is an isomorphism. Moreover, let $G(\alpha,\pi)$ be the
 following $K$-vector space with Frobenius structure:
 \begin{equation*}
  G(\alpha,\pi):=H^1_{\mr{rig}}(\mb{A}_k^1\setminus\{0\},\ms{K}
   _\alpha\otimes\ms{L}_\pi). \index{Fi@$F$-isocrystal, overconvergent isocrystal!$G(\alpha,\pi)$}
 \end{equation*}
 Then, we have
 \begin{equation*}
  \ms{F}_\pi(\ms{K}_\alpha)[1]\cong\ms{K}_{1-\alpha}\otimes
   G(\alpha,\pi)(1).
 \end{equation*}
\end{prop*}
\begin{proof}
 The first statement follows from Lemma \ref{calcofdiff}.
 For the latter claim, the proof works essentially the same as in
 \cite{Lau} by replacing $m^*$ by $m^!$ and using the K\"{u}nneth
 formula \ref{kunnethformual}. The Tate twist appearing here comes from
 the isomorphism (\ref{rigDmodpush}).
\end{proof}

\section{Stationary Phase}
\label{section4}
In this section, we will prove the stationary phase formula when the differential
slope at infinity is less than or equal to $1$. However, in this
section, we do not consider Frobenius structures on the local Fourier
transforms, so the stationary phase in this section is still
temporary. This will be completed in the next section.

\begin{quote}
 Throughout this section, we continuously use the assumptions and
notation of paragraph \ref{setupFour}.
\end{quote}

\subsection{Geometric calculations}
\label{secgeomcalc}
\subsubsection{}
\label{geomnotlau}
Following \cite[2.2.1]{Lau}, we will define several invariants which will be used throughout this section. 
Let ${M}$ be a solvable differential module on
the Robba ring,  cf.\ \cite[12.6.4]{Ke2} or \cite[8.7]{CM}. We denote by $\mr{rk}({M})$ the rank, by
$\mr{irr}({M})$ the irregularity (cf.\ \ref{irr}), and by $\mr{pt}({M})$ the
greatest differential slope of ${M}$ as usual. By a result of Christol and Mebkhout (cf.\ \cite[2.4-1]{CM4} or in general 
\cite[12.6.4]{Ke2}), %\footnote{Actually, in what follows, we only use the
%notation ${M}_I$ for free differential modules,
%and the decomposition theorem of \cite[2.4-1]{CM4} is enough.}, 
we get
the differential slope decomposition ${M}=\bigoplus {M}_\beta$ where
${M}_\beta$ is purely of  slope $\beta$. For any
interval $I\subset\left[0,\infty\right[$, we put
${M}_I:=\bigoplus_{\beta\in I}{M}_\beta\subset {M}$. Let
$\ms{M}$ be a holonomic $F$-$\DdagQ{\Pone}(\infty)$-module. For any
closed point $x$ in $\Aone$, we put
\begin{alignat*}{2}
 r(\ms{M})&:=-\mr{rk}(\ms{M}|_{\eta_z})\leq 0,&\qquad
 s_x(\ms{M})&:=-\mr{irr}(\ms{M}|_{\eta_x})\leq 0,\\
 r_x(\ms{M})&:=\mr{dim}_{K_x}(i_x^!\ms{M}),&\qquad
 a_x(\ms{M})&:=r(\ms{M})+s_x(\ms{M})-r_x(\ms{M}), \index{local constants! $r(\ms{M})$, $s_x(\ms{M})$, 
$r_x(\ms{M})$, $a_x(\ms{M})$}
\end{alignat*}
where $z$ can be taken to be any closed point in $\Aone$, and
$i_x\colon\mr{Spf}(R_x)\hookrightarrow\Aone$ is the closed immersion
for $x$. The following lemma compare these invariants to the generic rank and the vertical multiplicity, cf.\ 
\ref{def_cycl_Dm} and \ref{defCycl}.

\begin{lem}
 \label{lemcyclecalc}
 We preserve the notation. Let
 \begin{equation*}\
 \mr{Cycl}(\ms{M}|_{\Aone})=r\cdot[\mb{A}^1]+\sum_{x\in{|\mb{A}^1|}}
  m_x\cdot[\pi^{-1}(x)].
 \end{equation*}
 Recall that $\pi\colon T^*\mb{A}^1\rightarrow\mb{A}^1$ is the canonical
 projection and $|\mb{A}^1|$ is the set of closed points. Then  we have
 $r(\ms{M})=-r$ and, for any closed point $x\in\Aone$, $a_x(\ms{M})=-m_x$.
\end{lem}
\begin{proof}
 This follows from Corollary \ref{compCMGAR}.
\end{proof}

Let $C$ be a complex of $D^b_{\mr{coh}}(\DdagQ{\mr{Spf}(R)})\cong
D^b_{\mr{fin}}(K\mbox{-mod})$, where the latter category is the derived
category of complexes of $K$-vector spaces whose cohomology is finite
dimensional. We put $\chi(C):=\sum_{i}(-1)^i\dim_K \ms{H}^i(C)$.
Let $q\colon\Pone\rightarrow\mr{Spf}(R)$ be the structural morphism.
Let $\ms{M}$ be a coherent $F$-$\DdagQ{\Pone}(\infty)$-module.
We define $\chi(\Aone,\ms{M}):=\chi(q_+\ms{M})$.
\index{.@miscellaneous!chi@$\chi(-)$, $\chi(\Aone,-)$}

We know that the Grothendieck-Ogg-Shafarevich type formula holds by
\cite[5.3.2]{Gar2} or \cite[2.4.7]{Abe2}. Using the above lemma, we can
write the formula as follows.
\begin{equation}
 \label{GOSformula}
 -\chi(\Aone,\ms{M})=r(\ms{M})-\sum_{x\in|\Aone|}\deg(x)\cdot
  a_x(\ms{M})+\mr{irr}(\ms{M}|_{\eta_\infty})
\end{equation}

\begin{cor}
 \label{irrandfotrans}
 Let $\ms{M}$ be a holonomic $\DdagQ{\Pone}(\infty)$-module, and
 let $x\in |\Aone|$ be a singular point of $\ms{M}$. Then we get
 \begin{equation*}
  \mr{rk}(\ms{F}^{(x,\infty')}(\ms{M}))=-\deg(x)\cdot a_x(\ms{M}).
 \end{equation*}
\end{cor}
\begin{proof}
 This follows from the definition of the local Fourier transform (cf.\ \ref{deflocFour}), using the stability theorem 
(cf.\ \ref{stabilitytheoremcy}-i), combined with
 Proposition \ref{levmcharcycle} and Lemma \ref{lemcyclecalc}.
\end{proof}

\begin{lem}
 \label{smalllemforGOSL}
 Using the notation of Lemma {\normalfont\ref{lemcyclecalc}}, we get
 $r_x(\ms{M})=r(\ms{M})$ if and only if $m_x=0$.
\end{lem}
\begin{proof}
 When $m_x=0$, we know that $\ms{M}$ is a convergent isocrystal on an
 open neighborhood of $x$, and the lemma follows easily. Let us see the
 ``only if'' part.
 We know that $\dim_{K^{\mr{ur}}}\mb{V}(\ms{M})=\mr{rk}(\ms{M})$
 by \cite[(6.1.11)]{Cr}. By \cite[2.2]{Cr3}, we get
 \begin{equation*}
  i_x^!\ms{M}\cong R\mr{Hom}(\ms{M}|_{S_x},\mc{O}^{\mr{an}})^*~[1],
 \end{equation*}
 where $^*$ denotes the dual in the derived category
 $D^b_{\mr{fin}}(K\mbox{-mod})$. The exact sequence
 (\ref{vanishexact}) implies
 $\dim_{K^{\mr{ur}}}\mb{W}(\ms{M}|_{S_x})=0$, and thus
 $\mb{W}(\ms{M}|_{S_x})=0$. 
By Lemma \ref{W=0}, we get
 that $\ms{M}|_{S_x}$ is a free differential $\mc{O}^{\mr{an}}$-module, and in
 particular $m_x=0$.

%By the construction of the functor $\mb{M}$
 %of \cite[\S 7]{Cr}, we get
 %that $\ms{M}|_{S_x}$ is a coherent $\mc{O}^{\mr{an}}$-module, and in
 %particular $m_x=0$.
\end{proof}

\begin{lem}
 \label{rkirralmost}
 Let ${M}$ be a solvable free differential module on the Robba ring
 $\mc{R}$ over $K$. %, solvable in the sense of Christol-Mebkhout, cf.\ {\normalfont{\cite[8.7]{CM}}}. 
We further assume that ${M}$ is purely of differential slope
 $1$. Let $\DwL(\mf{s})$ be the Dwork differential module, for $\mf{s}\in\mb{A}^1_k(\overline{k})$, cf.\
  {\normalfont\ref{locASdiffmod}}. We
 consider the tensor product ${M}\otimes_{\mc{R}}\DwL(\mf{s})$ as a
 differential $\mc{R}_{K_s}$-module. Then we get the following.

 (i) For almost all
 $\mf{s}\in\mb{A}^1_k(\overline{k})$, we get
 \begin{equation*}
  \mr{irr}_{K_s}({M}\otimes_{\mc{R}}\DwL(\mf{s}))=\mr{rk}({M}),
 \end{equation*}
 where $\mr{irr}_{K_s}$ denotes the irregularity as an
 $\mc{R}_{K_s}$-module.

 (ii) There exists an $\mf{s}\neq0$ in $\overline{k}$ such that the
 irregularity $\mr{irr}_{K_s}({M}\otimes\DwL(\mf{s}))$ is less than
 $\mr{rk}(\ms{M})$.
\end{lem}

\begin{proof}
 Let us prove (i). We use the induction on the rank of ${M}$ over
 $\mc{R}_K$. Suppose there exists a geometric point $\mf{s}$ such that
 \begin{equation*}
  \mr{irr}({M}\otimes\DwL(\mf{s}))<\mr{rk}({M}),\qquad
   \mr{pt}({M}\otimes\DwL(\mf{s}))=1.
 \end{equation*}
 These conditions show that ${M}\otimes\DwL(\mf{s})$ has at
 least two slopes including $1$. Thus, there exists the canonical
 decomposition ${M}\otimes\DwL(\mf{s})={M}'_1\oplus {M}'_{<1}$
 where ${M}'_1$ is purely of slope $1$, and ${M}'_{<1}$ is purely
 of slope less than $1$, and these modules are non-zero. Thus, we get
 the decomposition
 \begin{equation*}
  {M}\otimes_KK_s={M}'_1\otimes\DwL(-\mf{s})\oplus{M}'_{<1}
   \otimes\DwL(-\mf{s}).
 \end{equation*}
 Since ${M}$ is purely of slope $1$, both
 ${M}'_1\otimes\DwL(-\mf{s})$ and
 ${M}'_{<1}\otimes\DwL(-\mf{s})$ are purely of slope $1$ as
 well. Thus by the induction hypothesis, the lemma holds for
 these two modules. This implies that the lemma also holds for
 ${M}$.

 If $\mr{pt}({M}\otimes\DwL(\mf{s}))=1$ for any $\mf{s}$, we get
 the lemma by the above argument.
 Suppose $\mr{pt}({M}\otimes\DwL(\mf{s}))<1$ for some $\mf{s}$.
 Then for any $\mf{s}'\neq\mf{s}$, we get
 $\mr{pt}({M}\otimes\DwL(\mf{s}'))=1$. If there exists
 $\mf{s}'\neq\mf{s}$ such that
 $\mr{irr}({M}\otimes\DwL(\mf{s}'))<\mr{rk}({M})$, then we may
 use the above argument. Otherwise, the lemma is trivial.

 Now, let us move to (ii). By using \cite[2.0-1]{Meb}, there exists a
 number $a$ in $\overline{K}$ whose absolute value is $1$,
 and an integer $h$ such that the irregularity of
 ${M}\otimes\exp(\pi ax^{-p^h})$ is less than $\mr{rk}({M})$. Here
 we are using the notation of {\it loc.\ cit}. We remind that in {\it
 loc.\ cit.}, there is an assumption on the
 spherically completeness of $K$. However, as mentioned in [{\it loc.\
 cit.}, 2.0-4],
 this hypothesis is used only to use a result of Robba, and when
 $p\neq2$, the assumption was removed by Matsuda as written there. This
 result was extended also to the case $p=2$ by Pulita \cite[Theorem
 4.6]{Pu}, and we no longer need to assume the spherically completeness
 here. Arguing as the proof of \cite[4.2.3 (ii)]{Gar2} using
 \cite[1.5]{Mat}\footnote{In {\it loc.\ cit.}, $p\neq2$ is assumed
 extensively. However, the proof of Lemma 1.5 works
 also for $p=2$ without any change.}, there exists a number $a'$ in
 $\overline{K}$ whose absolute value is $1$ such that $\exp(\pi
 ax^{-p^h})$ is isomorphic to $\exp(\pi a'x^{-1})$ as differential
 $\mc{R}$-modules, and the latter is isomorphic to
 $\DwL(\overline{a'})$ where the overline denotes the residue class.
\end{proof}

\subsubsection{}
Let $\ms{E}$ be a coherent $F$-$\DdagQ{\Pone}(\infty)$-module.
We denote by $\ms{E}':=\ms{H}^1(\ms{F}_\pi(\ms{E}))$ the geometric
Fourier transform. We have the following analog of
\cite[2.3.1.1]{Lau}.
\begin{prop*}
 \label{GOScorlau}
 (i) $\displaystyle r(\ms{E}')=\sum_{s\in
 S}\mr{deg}(s)\cdot a_s(\ms{E})+\mr{rk}((\ms{E}|_{\eta_\infty})
 _{\left]1,\infty\right[})-\mr{irr}((\ms{E}|_{\eta_\infty})
 _{\left]1,\infty\right[})$.

 (i') $\displaystyle r(\ms{E})=\sum_{s'\in
 S'}\mr{deg}(s')\cdot a_{s'}(\ms{E}')+\mr{rk}((\ms{E}'|
 _{\eta_\infty})_{\left]1,\infty\right[})-\mr{irr}((\ms{E}'|
 _{\eta_\infty})_{\left]1,\infty\right[}).$

 (ii) For $s'\in\Aoned\setminus\{0'\}$, we get
 \begin{equation*}
  r_{s'}(\ms{E}')=r(\ms{E}')+\mr{rk}((\ms{E}|_{\eta_\infty})_{1})
   -\mr{irr}((\ms{E}|_{\eta_\infty})_{1}\otimes\ms{L}(x\cdot s')|
   _{\eta_\infty}).
 \end{equation*}

 (ii') For $s\in\Aone\setminus\{0\}$, we get
 \begin{equation*}
  r_{s}(\ms{E})=r(\ms{E})+\mr{rk}((\ms{E}'|_{\eta_\infty})_{1})
   -\mr{irr}((\ms{E}'|_{\eta_\infty})_{1}\otimes\ms{L}(s\cdot x')|
   _{\eta_\infty}).
 \end{equation*}
 
 (iii) $\displaystyle
 r_{0'}(\ms{E}')=r(\ms{E}')+\mr{rk}((\ms{E}|_{\eta_\infty})
 _{[0,1[})-\mr{irr}((\ms{E}|_{\eta_\infty})_{[0,1[}).$

 (iii') $\displaystyle
 r_{0}(\ms{E})=r(\ms{E})+\mr{rk}((\ms{E}'|_{\eta_{\infty'}})
 _{[0,1[})-\mr{irr}((\ms{E}'|_{\eta_{\infty'}})_{[0,1[}).$
 
\end{prop*}
\begin{proof}
 The idea of the proof is the same as that of {\it loc.\ cit}. We sketch
 the proof. By the proper base change theorem \ref{properbasech}, we get
 \begin{equation*}
  r_{s'}(\ms{E}')=\chi(\Aone,\ms{E}\otimes\ms{L}(x\cdot s')).
 \end{equation*}
 Let us prove (i). Since both sides of the equality are invariant under
 base extension, we may assume that $S$ consists of rational
 points.
 We have $a_s(\ms{E})=a_s(\ms{E}\otimes\ms{L}(x\cdot
 s'))$ for $s,s'\neq 0$ since $\ms{L}(x\cdot s')$ has no singular
 points in $\Aone$. Using (\ref{GOSformula}), it remains to show
 \begin{equation*}
  -r(\ms{E})-\mr{irr}(\ms{E}|_{\eta_\infty}\otimes\ms{L}(x\cdot s')|
   _{\eta_\infty})=\mr{rk}((\ms{E}|_{\eta_\infty})
   _{\left]1,\infty\right[})
 -\mr{irr}((\ms{E}|_{\eta_\infty})_{\left]1,\infty\right[})
 \end{equation*}
 for some $s'\in\overline{k}$. This follows from Lemma
 \ref{rkirralmost} (i). The claims (ii) and (iii) follow from (i) and
 (\ref{GOSformula}), and (i)', (ii)', (iii)' follow by the involutivity
 \ref{involForge} of the geometric Fourier transform.
\end{proof}

\begin{cor}
 \label{calcofirr}
 Let $\ms{E}$ be a holonomic $F$-$\DdagQ{\Pone}(\infty)$-module such
 that $0$ is the only singularity, and it is regular at
 infinity ({\it i.e.\ }$\mr{irr}(\ms{E}|_{\eta_\infty})=0$). Then the
 geometric Fourier transform $\ms{E}'$ is not singular except
 for $0'$, we have $s_{0'}(\ms{E}')=0$,
 $a_{0'}(\ms{E}')=r(\ms{E})$, and the differential slope at $\infty'$ is strictly
 less than $1$.
\end{cor}
\begin{proof}
 By (ii) of Proposition \ref{GOScorlau} and the hypothesis that
 $\ms{E}$ is regular at $\infty$, we get for $s'\neq0',\infty'$
 \begin{equation*}
  r_{s'}(\ms{E}')=r(\ms{E}'),
 \end{equation*}
 which shows that $m_{s'}=0$ for $s'\neq 0',\infty'$ by Lemma
 \ref{smalllemforGOSL}. Thus $\ms{E}'$ is not singular except for $0'$
 and $\infty'$.

 Now, by (ii') of the proposition, we get for $s\neq 0,\infty$
 \begin{equation}
  \label{2-1}
  \mr{rk}((\ms{E}'|_{\eta_\infty})_1)-\mr{irr}((\ms{E}'|_{\eta_\infty})_1
  \otimes\ms{L}(s\cdot x')|_{\eta_\infty})=0.
 \end{equation}
 Suppose $(\ms{E}'|_{\eta_\infty})_1\neq 0$. Then by Lemma
 \ref{rkirralmost} (ii), there exists $s\neq0$ such that the
 irregularity of $(\ms{E}'|_{\eta_\infty})_1\otimes\ms{L}(s\cdot
 x')|_{\eta_\infty}$ is less than $r((\ms{E}'|_{\eta_\infty})_1)$, which
 contradicts with (\ref{2-1}). This shows that
 $(\ms{E}'|_{\eta_\infty})_1=0$. Now, we get
 \begin{align}
  \label{aeconnectequ}
  a_{0'}(\ms{E}')-s_{0'}(\ms{E}')&=r(\ms{E}')-r_{0'}(\ms{E}')\\\notag
  &=-\left(\mr{rk}((\ms{E}|_{\eta_\infty})_{\left[0,1\right[})
  -\mr{irr}((\ms{E}|_{\eta_\infty})_{\left[0,1\right[})\right)\\\notag
  &=-\mr{rk}(\ms{E}|_{\eta_\infty})=r(\ms{E}),
 \end{align}
 where the second equality holds by (iii) and the third by the
 assumption that $\ms{E}$ is regular at infinity. Combining with (i'),
 we get
 \begin{equation*}
  r(\ms{E})=r(\ms{E})+s_{0'}(\ms{E}')+\mr{rk}((\ms{E}'|_{\eta_\infty})
   _{\left]1,\infty\right[})-\mr{irr}((\ms{E}'|_{\eta_\infty})
   _{\left]1,\infty\right[}).
 \end{equation*}
 Thus,
 \begin{equation*}
  s_{0'}(\ms{E}')=\mr{irr}((\ms{E}'|_{\eta_\infty})
   _{\left]1,\infty\right[})-\mr{rk}((\ms{E}'|_{\eta_\infty})
   _{\left]1,\infty\right[})\geq 0.
 \end{equation*}
 On the other hand, we have $s_{0'}(\ms{E}')\leq 0$ by
 definition. This shows that $s_{0'}(\ms{E}')=0$, and
 $(\ms{E}'|_{\eta_\infty})_{\left]1,\infty\right[}=0$. Thus
 $a_{0'}(\ms{E}')=r(\ms{E})$ by (\ref{aeconnectequ}).
\end{proof}

\subsection{Regular stationary phase formula}
\label{secregstfo}
\subsubsection{}
Let $\ms{M}$ be a holonomic $F$-$\DdagQ{\Pone}(\infty)$-module. Let
$s\in\Aone$ be a singular point of $\ms{M}$. Recall the notation of
\ref{setupFour} and Definition \ref{deflocFour}. We
have a canonical homomorphism
\begin{equation*}
 \ms{H}^{1}(\ms{F}_{\pi}(\ms{M}))|_{\eta_{\infty'}}\cong
  \nF{\ms{M}}|_{\eta_{\infty'}}\rightarrow\tau'^*\bigl(\EdagQ{s}
  \otimes\ms{M}\bigr)\cong\tau'^*\Gamma(\ms{F}_{\pi}^{(s,\infty')}
  (\ms{M})),
\end{equation*}
where the second homomorphism sends $\alpha\otimes\widehat{m}$
($\alpha\in\mc{R}_{\ms{S}_\infty'}$,
and $m\in\Gamma(\Pone,\ms{M})$) to $\alpha(1\otimes m)$. The first
isomorphism is the isomorphism by Noot-Huyghe (\ref{compgeomnaive}), and
the second one is that of Lemma \ref{globalsectcalc}. We see easily
that this homomorphism is compatible with the connections.
By construction (cf.\ \ref{preisouadsf}) the source is a free differential $\mc{R}$-module, thus we get a
canonical homomorphism of differential $\mc{R}$-modules
\begin{equation*}
 \alpha_s\colon\ms{H}^{1}(\ms{F}_{\pi}(\ms{M}))|_{\eta_{\infty'}}
  \rightarrow\tau'^*\ms{F}_\pi^{(s,\infty')}(\ms{M})
\end{equation*}
by Lemma \ref{lemma_sgs}.

\begin{thm}[Regular stationary phase]
 \label{SPT}
 Let $\ms{M}$ be a holonomic $F$-$\DdagQ{\Pone}(\infty)$-module whose
 {\em differential slopes at infinity are less than or equal to $1$}. Let
 $S$ be the set of singular points of $\ms{M}$ in $\Aone$.
 Then the canonical homomorphism
 \begin{equation}
  \label{stationaryhom}
  (\alpha_s)_{s\in S}\colon\ms{H}^{1}(\ms{F}_{\pi}(\ms{M}))|
  _{\eta_{\infty'}}\rightarrow\bigoplus_{s\in S}\tau'^*
  \ms{F}_\pi^{(s,\infty')}(\ms{M})
 \end{equation}
 is an isomorphism.
\end{thm}
\begin{proof}
 First of all, we will reduce to the case where $S$ consists of rational
 points. There exists an unramified Galois extension $E$ of $K$ such
 that $S$ consists of $k_E$-rational points where $k_E$ denotes the
 residue field of $E$ as usual. Note that
 \begin{equation*}
  \nF{\ms{M}}|_{\eta_{\infty'}}\otimes_K E\cong
   \nF{\ms{M}\otimes_K E}|_{\eta_{\infty'}},\qquad
   \ms{F}^{(s,\infty')}(\ms{M})\otimes_K E\xrightarrow{\sim}\bigoplus
   _{s'\mapsto s}\ms{F}^{(s',\infty)}(\ms{M}\otimes_K E),
 \end{equation*}
 where the direct sum in the second isomorphism runs over the set of
 closed points of $\widehat{\mb{A}}^1_{R_E}$ which map to $s$.
 Indeed, the first isomorphism follows since the cohomological operators
 used in the definition of geometric Fourier transform are compatible
 with base change. The second isomorphism follows from Lemma
 \ref{locFourrel} and \ref{nongeometpt}.
 Since the left hand sides of these two isomorphisms have the action of
 $G:=\mr{Gal}(E/K)$, we define a $G$-action on the right hands
 sides by transport of structure. Note that the $G$-invariant parts are
 isomorphic to $\nF{\ms{M}}|_{\eta_{\infty'}}$ and
 $\ms{F}^{(s,\infty')}(\ms{M})$ respectively. By definition, the
 following diagram is commutative
 \begin{equation*}
  \xymatrix@C=50pt{
   \nF{\ms{M}}|_{\eta_{\infty'}}\ar[r]^<>(.5){\alpha_s}\ar[d]&
   \tau'^*\ms{F}^{(s,\infty')}(\ms{M})\ar[d]\\
  \nF{\ms{M}\otimes_K E}|_{\eta_{\infty'}}\ar[r]_<>(.5)
   {\bigoplus\alpha_{s'}}&\bigoplus_{s'\mapsto s}\tau'^*_{E}\ms{F}
   ^{(s',\infty)}(\ms{M}\otimes_K E),
   }
 \end{equation*}
 where $\tau'_E$ denotes $\tau'$ over $E$.
 Thus, if the theorem holds for $\ms{M}\otimes_K E$, then it also holds
 for $\ms{M}$ by taking $G$-invariant parts, and the claim
 follows. From now on, we assume that $S$ consists of rational points.

 Now, we will see that $\alpha_s$
 is surjective for any $s\in S$. By the exactness of Fourier
 transforms and local Fourier transforms (cf.\ Lemma
 \ref{exactness} and (\ref{compgeomnaive})), it suffices to show the
 claim when $\ms{M}$ is monogenic. Suppose it is generated by
 $m_0\in\Gamma(\Pone,\ms{M})$. Let ${N}$ be a differential
 $\mc{A}_{t'}(\left[r,1\right[)$-module, where $t':=1/x'$, satisfying
 the following:
 \begin{enumerate}
  \item $\mc{R}_{\infty'}\otimes {N}\cong\nF{\ms{M}}
	~|_{\eta_{\infty'}}$ where
	$\mc{R}_{\infty'}:=\mc{R}_{\ms{S}_{\infty'}}$;

  \item there exists an element $m'_0\in {N}$ which is mapped to
	$1\otimes\widehat{m}_0$ in $\nF{\ms{M}}|_{\eta_{\infty'}}$
	by the canonical injective homomorphism
	${N}\rightarrow\nF{\ms{M}}|_{\eta_{\infty'}}$.
 \end{enumerate}

 Now, let $\ms{M}^{(m)}$ be a stable coherent
 $\DcompQ{m}{\Pone}(\infty)$-module generated by a single element
 $m''_0$ in $\Gamma(\Pone,\ms{M}^{(m)})$ such that
 $\DdagQ{\Pone}(\infty)\otimes\ms{M}^{(m)}\cong\ms{M}$
 and $1\otimes m''_0$ is sent to $m_0$ via this isomorphism. Then there
 exists an integer $m'$ such that $m'\geq m$, $\omega_{m'}\geq r$ and
 $\alpha_s$ is induced by a homomorphism
 $\mc{A}(\left[\omega_{m'},1\right[)\otimes {N}
 \rightarrow\tau'^*\Emod{m',\dag}{s}(\ms{M}^{(m)})$ sending $1\otimes
 m'_0$ to $1\otimes(m''_0)^\wedge$. To see that this is surjective,
 it suffices to show that the canonical homomorphism
 \begin{equation*}
  \theta\colon\mc{A}(\left[\omega_{m'},\omega_{m''}\right\})
   \otimes {N}\rightarrow\tau'^*\Ecomp{m',m''}{s}(\ms{M})
 \end{equation*}
 is surjective for any $m''\geq m'$.  We have the local coordinate
 $y_s=x-s$ around $s$. By the choice of ${N}$,
 $1\otimes(y_s^nm_0)^{\wedge}$ is also contained in $\theta({N})$ for
 any positive integer $n$. Indeed, we have
 \begin{equation*}
  1\otimes(y_s^nm''_0)^\wedge=(-\pi^{-1}\partial_{x'}-s)^n\cdot
   \bigl(1\otimes(m''_0)^\wedge\bigr)=
   \theta\bigl((\pi^{-1}t'^2\partial_{t'}-s)^n\cdot
   m'_0\bigr).
 \end{equation*}
 Thus, the surjectivity follows from Lemma \ref{generators}.

 Let ${M}$ be an object of $F$-$\mr{Hol}(\ms{S})$, and let $\ms{M}_0$
 be the canonical extension of ${M}$ at $0$. Then the homomorphism
 \begin{equation*}
  \nF{\ms{M}_0}|_{\eta_{\infty'}}\rightarrow\tau'^*
   \Phi^{(0,\infty')}({M})
 \end{equation*}
 is surjective by the argument above. By corollary \ref{calcofirr}, the
 greatest differential slope of $\nF{\ms{M}_0}|_{\eta_{\infty'}}$
 is strictly less than $1$. The surjectivity of the homomorphism
 implies the following:
 \begin{quote}
  (*)
  The greatest differential slope of $\Phi^{(0,\infty')}({M})$ is strictly
  less than $1$ for any object ${M}$ in $F$-$\mr{Hol}(\ms{S})$.
 \end{quote}
 For $s\in S$, we know that
 \begin{equation*}
  \ms{F}^{(s,\infty')}(\ms{M})\cong\Phi^{(0,\infty')}(\tau_{s*}
   \ms{M}|_{S_s})\otimes\DwL(s)
 \end{equation*}
 by Lemma \ref{locFourrel}. Thus we get the following:
 \begin{quote}
  The greatest slope of $\ms{F}^{(s,\infty')}(\ms{M})$ is equal to $1$
  for $s\in S\setminus\{0\}$.
 \end{quote}

 We put ${M}':=\nF{\ms{M}}|_{\eta_{\infty'}}$. For a
 $k$-rational point $s$ in $\widehat{\mb{A}}^1$, we define
 \begin{equation*}
  {M}'^{s}:=\DwL(s)\otimes(\DwL(-s)\otimes{M}')
   _{\left[0,1\right[}\subset {M}'.
 \end{equation*}
 By definition, ${M}'^s$ is a direct factor of ${M}'$, and in
 particular, there exists a canonical projection
 $p_s\colon {M}'\rightarrow {M}'^s$.
 Take two {\em distinct} points $s,t\in S$, and consider the
 canonical homomorphism
 \begin{equation*}
  \DwL(-s)\otimes {M}'^{t}\xrightarrow{\mr{id}\otimes\alpha_s}
   \DwL(-s)\otimes\ms{F}^{(s,\infty')}(\ms{M})\cong\Phi^{(0,\infty')}
   (\tau_{s*}\ms{M}|_{S_s}).
 \end{equation*}
 Since $\DwL(-s)\otimes {M}'^{t}\cong \DwL(-s+t)
 \otimes(\DwL(-t)\otimes\ms{M}')_{\left[0,1\right[}$, we get that the
 source is purely of slope $1$. We know that   the greatest  
 slope of the target is strictly less than $1$ by (*). This shows that the homomorphism
 is $0$. Thus the homomorphism
 \begin{equation*}
  {M}'^{t}\hookrightarrow {M}'\xrightarrow{\alpha_s}
   \tau'^*\ms{F}^{(s,\infty')}(\ms{M})
 \end{equation*}
 is $0$ if $t\neq s$.

 Let $s$ be a point in $S$ and let $S_s:=S\setminus\{s\}$.
 By using the same argument, the canonical homomorphism
 $\sum_{t\in S_s}{M}'^t\hookrightarrow {M}\xrightarrow{p_s}
 {M}'^s$ is $0$. This shows that $\sum_{t\in
 S_s}{M}'^t\cap {M}'^s=0$ in ${M}'$, and thus, the canonical
 homomorphism $\bigoplus_{s\in S}{M}'^s\rightarrow {M}'$
 is injective.

 Now, since ${M}'\rightarrow\ms{F}^{(s,\infty')}(\ms{M})$ is
 surjective, the homomorphism
 \begin{equation*}
  \DwL(-s)\otimes {M}'\rightarrow \DwL(-s)\otimes\ms{F}
   ^{(s,\infty')}(\ms{M})\cong\Phi^{(0,\infty')}(\ms{M}|_{S_s})
 \end{equation*}
 is also surjective. Since the target has slope $<1$, the homomorphism
 \begin{equation*}
  (\DwL(-s)\otimes {M}')_{\left[0,1\right[}\rightarrow \DwL(-s)\otimes\ms{F}^{(s,\infty')}(\ms{M})
 \end{equation*}
 is surjective, and we deduce that the composition
 \begin{equation*}
  \beta_s\colon {M}'^{s}\hookrightarrow {M}'\xrightarrow{\alpha_s}
   \ms{F}^{(s,\infty')}(\ms{M}),
 \end{equation*}
 is surjective as well. Thus, we get that the canonical homomorphism
 \begin{equation}
  \label{locFouraux}
  \bigoplus_{s\in S} {M}'^{s}\hookrightarrow {M}'
  \xrightarrow{(\alpha_s)}\bigoplus_{s\in S}
  \ms{F}^{(s,\infty')}(\ms{M}),
 \end{equation}
 is also surjective and equal to $\bigoplus_{s\in S}\beta_s$. This leads
 us to the following inequalities:
 \begin{equation*}
  \mr{rk}({M}')\geq\sum_{s\in S}\mr{rk}({M}'^{s})\geq
   \sum_{s\in S}\mr{rk}(\ms{F}^{(s,\infty')}(\ms{M}))
   =\mr{rk}({M}'),
 \end{equation*}
 where the last equation holds by Corollary \ref{irrandfotrans} combined
 with Proposition \ref{GOScorlau} (i) using the assumption that the
 greatest slope of $\ms{M}$ at infinity is less than or equal to
 $1$. Since a surjection of differential $\mc{R}$-modules with the same
 ranks is an isomorphism, we get that (\ref{locFouraux}) is an
 isomorphism. Since an injection of differential
 $\mc{R}$-modules with the same ranks is an isomorphism, we have
 $\bigoplus_{s\in S}{M}'^{s}={M}'$, and combining these, we obtain
 the theorem.
\end{proof}

\begin{cor}
 \label{coinccrew}
 Let ${M}$ be a holonomic $F$-$\Dan{\ms{S},\mb{Q}}$-module. Then the
 local Fourier transform $\Phi^{(0,\infty)}({M})$
 coincides with $\mc{F}^{0,\infty'}({M})$ of Crew
 {\normalfont \cite[(8.3.1)]{Cr}}. In particular, the local Fourier
 transform is a free $\mc{R}$-module.
\end{cor}
\begin{proof}
 Apply the stationary phase formula to the canonical extension of
 ${M}$.
\end{proof}

\begin{cor}
 \label{locfourstricless}
 For any holonomic $F$-$\Dan{\ms{S},\mb{Q}}$-module ${M}$, the differential slope
 of $\Phi^{(0,\infty')}({M})$ is strictly less than $1$. Moreover,
 when ${M}$ is an $\mc{R}$-module, we have
 \begin{gather*}
  \mr{rk}(\Phi^{(0,\infty')}(j_+{M}))=\mr{rk}(\Phi^{(0,\infty')}
   (j_!{M}))=\mr{rk}({M})+\mr{irr}
   ({M}),\\
  \mr{irr}(\Phi^{(0,\infty')}(j_+{M}))=\mr{irr}(\Phi^{(0,\infty')}
  (j_!{M}))=\mr{irr}({M}),
 \end{gather*}
 using the cohomological functors of {\normalfont\ref{defvanishcycle}}.
\end{cor}
\begin{proof}
 Let us see the first claim. Let $\ms{M}$ be the canonical extension of
 $\tau_{0}^*({M})$. By
 stationary phase formula, it suffices to see that the slope of
 $\nF{\ms{M}}|_{\eta_{\infty'}}$ is strictly less than $1$. For
 this, apply Corollary \ref{calcofirr}.

 Let us prove the latter claim. The first equalities hold by
 Lemma \ref{calcofdiff} and its remark. Let us calculate the rank and
 irregularity of $\Phi^{(0,\infty')}(j_+{M})$. Let $\ms{M}$ be the
 canonical extension of $\tau_0^*(j_+{M})$. For the rank, apply
 Corollary \ref{irrandfotrans}. It remains to calculate the
 irregularity. We put $\ms{M}':=\nF{\ms{M}}$. By Proposition
 \ref{GOScorlau} (i) and (iii'), we get
 \begin{align*}
  r(\ms{M}')&=a_0(\ms{M})=r(\ms{M})-\mr{irr}(\ms{M}|_{\eta_0})
  -r_0(\ms{M}),\\
  r_0(\ms{M})&=r(\ms{M})+\mr{rk}(\ms{M}'|_{\eta_{\infty'}})-
  \mr{irr}(\ms{M}'|_{\eta_{\infty'}}).
 \end{align*}
 Since $\mr{rk}(\ms{M}'|_{\eta_{\infty'}})=-r(\ms{M}')$, we get
 $\mr{irr}(\ms{M}|_{\eta_0})=\mr{irr}(\ms{M}'|_{\eta
 _{\infty'}})$. Thus the corollary follows.
\end{proof}

\section{Frobenius structures}
\label{section5}
\label{endowFrob}
In this section, we endow the local Fourier transform with
Frobenius structure. We define the Frobenius structure using that of
geometric Fourier transform. In the first subsection, we show
that, with this Frobenius structure, the stationary phase theorem
is compatible with Frobenius structures. In the second subsection, we
explicitly describe this Frobenius structure in terms of
differential operators.

\begin{quote}
 Throughout this section, we continuously use the assumptions and
notation of paragraph \ref{setupFour}.
\end{quote}

\subsection{Frobenius structures on local Fourier transforms}
\label{secfroblocfou}
\begin{dfn}
 \label{transportgeomnaiv}
 For a coherent $\DdagQ{\Pone}(\infty)$-module $\ms{M}$, recall that we
 have the canonical isomorphism
 $\ms{F}_{\mr{naive,\pi}}(\ms{M})\cong\ms{H}^{1}(\ms{F}_\pi(\ms{M}))$
 (cf.\ (\ref{compgeomnaive})). Since the geometric Fourier transform is
 defined by cohomological operators, this functor commutes with
 Frobenius pull-backs. By transporting structure,
% for a coherent $\DdagQ{\Pone^{(1)}}$-module $\ms{M}'$,
 we have a canonical isomorphism
  \begin{equation*}
  \epsilon_{\ms{M}}\colon F^*_{\Poned}(\ms{F}_{\mr{naive},\pi}
   (\ms{M})^{\sigma})
   \xrightarrow{\sim}\ms{F}_{\mr{naive},\pi}
   (F^*_{\Pone}\ms{M}^{\sigma}), \index{.@miscellaneous!eps@$\epsilon_{\ms{M}}$}
 \end{equation*}
 where $\ms{M}^{\sigma}$ is the pull-back by Frobenius morphism.
 When $\ms{M}$ is a coherent $F$-$\DdagQ{\Pone}(\infty)$-module, we
 define a Frobenius structure on $\ms{F}_{\mr{naive},\pi}(\ms{M})$ using
 this isomorphism.
\end{dfn}

\begin{lem}
 \label{globalfourier}
 There exists an operator $\Upsilon_\pi\in\DdagQ{\Poned}(\infty')$ such
 that the following condition holds: let $\ms{M}$ be any coherent
 $\DdagQ{\Pone^{(1)}}(\infty)$-module (which may not possess Frobenius
 structure). For any $m\in\Gamma(\Pone^{(1)},\ms{M})$, we have
 \begin{equation*}
  \epsilon_{\ms{M}}(1\otimes\widehat{m})=\Upsilon_\pi\cdot
   \widehat{1\otimes m}.
 \end{equation*}
 Moreover, the operator $\Upsilon_\pi$ is compatible with base
 changes. We omit $\pi$ in the notation $\Upsilon_{\pi}$ if there is
 nothing to be confused. Note that $\Upsilon$ is {\em not} unique in
 general.
\end{lem}
\begin{proof}
 Let us take an operator $\Upsilon_{\pi}$ such that
 \begin{equation*}
  \epsilon_{\DdagQ{\Pone}}(1\otimes\widehat{1})=\Upsilon_\pi\cdot(
   \widehat{1\otimes 1}).
 \end{equation*}
 The existence follows easily from the fact that the homomorphism
 $$\DdagQ{\Poned}(\infty')\xrightarrow{\sim}\nF{\DdagQ{\Pone}(\infty)}
 \rightarrow\nF{F^*\DdagQ{\Pone^{(1)}}(\infty)}$$ is
 surjective, in other words, $\widehat{1\otimes1}$ is a generator of
 $\nF{F^*\DdagQ{\Pone^{(1)}}(\infty)}$ over
 $\DdagQ{\Poned}(\infty')$. For $m\in\ms{M}$,
 let us denote by $\rho\colon\DdagQ{\Pone}(\infty)\rightarrow\ms{M}$
 the homomorphism sending $1$ to $m$. Then by the functoriality of
 geometric Fourier transform, we get the following commutative diagram:
 \begin{equation*}
  \xymatrix@C=70pt{
   F^*(\nF{\DdagQ{\Pone^{(1)}}(\infty)})
   \ar[r]^<>(.5){F^*\ms{F}_{\mr{naive}}(\rho)}
   \ar[d]_{\epsilon_{\Ddag{}}}&
   F^*(\nF{\ms{M}})\ar[d]^{\epsilon_{\ms{M}}}\\
   \nF{F^*\DdagQ{\Pone^{(1)}}(\infty)}\ar[r]_<>(.5)
    {\ms{F}_{\mr{naive}}F^*(\rho)}&\nF{F^*\ms{M}}.
   }
 \end{equation*}
 Thus the lemma follows.
\end{proof}

\begin{dfn}
 \label{dfnlocalfrobfou}
 Let ${M}$ be a holonomic $F$-$\Dan{\ms{S},\mb{Q}}$-module. Let
 $\mf{s}\in\mb{A}^1_k(\overline{k})$. We denote by
 $\ms{M}$ the canonical extension of $\sigma_{\mf{s}}^*{M}$
 at $\widetilde{s}$. Recall the notation of \ref{setupFour}. Let
 \begin{equation*}
\widetilde{\alpha}_s \colon \tau'_*\nF{\ms{M}}|_{\eta_\infty}\xrightarrow{\sim}
   \Phi^{(\mf{s},\infty')}({M})
 \end{equation*}
 be the isomorphism given by the stationary phase theorem.
 Here we used abusively the notation $\tau'$ and $\infty$ on
 $\Pone\otimes K_s$. Since $F^*$ and $\tau'_*$ commute, the isomorphisms
 $\widetilde{\alpha}_s$ and $\epsilon_{\ms{M}}$ induce an isomorphism
 $\epsilon_{{M},\mf{s}}$ as follows
 \index{.@miscellaneous!eps@$\epsilon_{{M},\mf{s}}$}
\begin{align*}
 \epsilon_{{M},\mf{s}}\colon
 F^*\bigl(\Phi^{(\mf{s},\infty')}(M)\bigr)
 \xrightarrow[\sim]{F^*(\widetilde{\alpha}_s^{-1})}&
 F^*\bigl(\tau'_*\nF{\ms{M}}|_{\eta_\infty}\bigr)
 \xrightarrow{\sim}
 \tau'_*\bigl(F^*\nF{\ms{M}}|_{\eta_\infty}\bigr)\\
 &\xrightarrow{\tau'_*(\epsilon_{\ms{M}})|_{\eta_\infty}}
 \tau'_*\bigl(\nF{\ms{F^*M}}|_{\eta_\infty}\bigr)
 \xrightarrow[\sim]{F^*(\widetilde{\alpha}_s)}
 \Phi^{(\mf{s},\infty')}(F^*{M}).
%%%%%%%%%%%%%%%%%%%%%%%%%%%%%%%%%%%%%%%%%
 % \xymatrix{
 %    F^*(\Phi^{(\mf{s},\infty')}
 %   ({M})) \ar[r]^-{F^*(\widetilde{\alpha}_s^{-1})}_-{\sim} &
 %   F^*(\tau'_*\nF{\ms{M}}|_{\eta_\infty})\ar[r]_{\sim} &
 %    \tau'_*(F^*\nF{\ms{M}}|_{\eta_\infty}) 
 %	\ar[d]^-{\tau'_*(\epsilon_{\ms{M}})|_{\eta_\infty}}  \\
 % && \tau'_*(\nF{\ms{F^*M}}|_{\eta_\infty})	
 %   \ar[r]_-{\sim}^-{F^*(\widetilde{\alpha}_s)} &
 %   \Phi^{(\mf{s},\infty')}(F^*{M}).
 %}
\end{align*}
We define the Frobenius structure on $\Phi^{(0,\infty')}({M})$ by
 composing this isomorphism with the isomorphism of functoriality
 $\Phi^{(\mf{s},\infty')}(F^*{M})\xrightarrow{\sim}
 \Phi^{(\mf{s},\infty')}({M})$ induced by the Frobenius structure of
 ${M}$.
\end{dfn}

\begin{lem}
 Let $\ms{M}$ be a holonomic $F$-$\DdagQ{\Pone}(\infty)$-module.
 Let $s$ be a $k$-rational point of $\mb{A}_k^1$, and
 $\gamma_s$ be the translation isomorphism of $\Pone$ sending $s$ to
 $0$. Then we have an isomorphism compatible with Frobenius structures
 \begin{equation*}
  \ms{F}_\pi(\gamma^!_s\ms{M})\cong\ms{F}_\pi(\ms{M})
   \otimes\ms{L}(-s\cdot x').
 \end{equation*}
\end{lem}
\begin{proof}
 The proof is formally the same as that of \cite[1.2.3.2]{Lau} using the
 K\"{u}nneth formula \ref{kunnethformual}, so we leave the details to the
 reader.
\end{proof}

\begin{lem}
 Let ${M}$ be a holonomic $F$-$\Dan{\ms{S},\mb{Q}}$-module, and $s$
 be a $k$-rational point of $\widehat{\mb{A}}^1_k$. Then we have the
 following isomorphism compatible with Frobenius structures:
 \begin{equation*}
  \Phi^{(s,\infty')}({M})\cong\Phi^{(0,\infty')}({M})
   \otimes\DwL(s).
 \end{equation*}
\end{lem}
\begin{proof}
 Let $\ms{M}_s$ denotes the canonical extension of $\tau_s^*({M})$
 at $s$. By the previous lemma, we know that
 $\ms{F}(\gamma_s^!\ms{M}_s)\otimes\ms{L}(s\cdot x')\cong\ms{F}
 (\ms{M}_s)$. By definition, the stationary phase isomorphism
 \begin{equation*}
  \ms{F}(\gamma_s^!\ms{M}_s)|_{\eta_{\infty'}}\xrightarrow{\sim}
   \ms{F}^{(0,\infty')}(\gamma_s^!\ms{M}_s)\cong
   \tau'^*\Phi^{(0,\infty')}({M})
 \end{equation*}
 is compatible with the Frobenius structures. Tensoring both sides with
 $\DwL(s)$ , we get the lemma.
\end{proof}

\begin{prop}
 \label{frobstruccalcabsw}
 Let ${M}$ be a holonomic $F$-$\Dan{\ms{S},\mb{Q}}$-module, and
 $\mf{s}\in\mb{A}^1_k(\overline{k})$. Recall the isomorphism
 \begin{equation*}
  \epsilon_{{M},\mf{s}}\colon F^*(\Phi^{(\mf{s},\infty')}
   ({M}))\xrightarrow{\sim}\Phi^{(\mf{s},\infty')}
   (F^*{M})
 \end{equation*}
 in Definition {\normalfont\ref{dfnlocalfrobfou}}. Then for any $m\in {M}$, we have
 \begin{equation*}
  \epsilon(1\otimes\widehat{m})=\tau'_*(\Upsilon)\cdot
   \widehat{1\otimes m}.
 \end{equation*}
 Here, $\Upsilon$ is the operator of Lemma
 {\normalfont\ref{globalfourier}}.
\end{prop}
\begin{proof}
 Since Frobenius structures are compatible with base changes and
 $\Upsilon_\pi$ does not depend on a base as well, we may
 suppose that $\mf{s}$ is a $k$-rational point by Lemma
 \ref{nongeometpt}. From now on, we assume that $s:=\mf{s}$ is a
 rational point. Let ${M}^\mr{can}$ be the canonical extension of
 $\tau^*_s {M}$ at $s$. Let $\epsilon_{{M}^{\mr{can}}}$  be the unique
  homomorphism
  making the
 following diagram commutative.
 \begin{equation*}
  \xymatrix@C=70pt{
   F^*(\ms{F}^{(s,\infty')}({M}^{\mr{can}}))\ar[r]^{\sim}
   \ar@{.>}[d]_{\epsilon_{{M}^{\mr{can}}}}&
   F^*(\Phi^{(s,\infty')}({M}))\ar[d]^{\epsilon_{{M},s}}\\
  \ms{F}^{(s,\infty')}(F^*{M}^{\mr{can}})
   \ar[r]^{\sim}_{a}&
   \Phi^{(s,\infty')}(F^*{M})
   }
 \end{equation*}
 Here the horizontal isomorphisms are induced by the canonical
 isomorphism ${M}\cong {M}^{\mr{can}}|_{\eta_s}$.
 For any $x'\in\Gamma(\Pone^{(1)},{M}^{\mr{can}})$, we get
 \begin{equation}
  \label{comingfromglob}
  \epsilon_{{M}^{\mr{can}}}(1\otimes\widehat{x'})=\tau'_*(\Upsilon)
  \cdot\widehat{1\otimes x'}
 \end{equation}
 by Lemma \ref{globalfourier}.
 Let ${N}$ be a differential $\mc{A}(\left]r,1\right[)$-module with
 Frobenius structure such
 that $\mc{R}\otimes {N}\cong\nF{{M}^{\mr{can}}}|
 _{\eta_\infty}$, and 
 let $\ms{M}^{(m)}$ be a stable $\DcompQ{m}{\Aone^{(1)}}$-module such
 that
 $\DdagQ{\Aone^{(1)}}\otimes\ms{M}^{(m)}\cong {M}^{\mr{can}}
 |_{\Aone^{(1)}}$. We note that
 ${N}\hookrightarrow\ms{F}^{(s,\infty')}({M}^{\mr{can}})$.
 There exist an integer $\underline{n}$ and isomorphisms
 \begin{equation*}
  \mc{A}(\left[\omega_{\underline{n}},1\right[)\otimes {N}
 \xrightarrow{\sim}\Emod{\underline{n},\dag}{s,\mb{Q}}(\ms{M}^{(m)}),\qquad
 \mc{A}(\left[\omega_{\underline{n}+\fr},1\right[)\otimes F^*{N}
 \xrightarrow{\sim}\Emod{\underline{n}+\fr,\dag}{s,\mb{Q}}(F^*\ms{M}^{(m)})
 \end{equation*}
 inducing the stationary phase isomorphisms. We note that the outer
 square of the diagram
 \begin{equation*}
  \xymatrix{
   F^*\mc{A}(\left[\omega_{\underline{n}},1\right[)\otimes {N}\ar[r]^{\sim}
   \ar[d]_{\sim}&
   F^*\Emod{\underline{n},\dag}{s}(\ms{M}^{(m)})\ar@{.>}[d]_{\epsilon^{(\underline{n})}}
   \ar@{^{(}->}[r]&
   F^*(\ms{F}^{(s,\infty')}({M}^{\mr{can}}))\ar[d]^{\epsilon
   _{{M}^{\mr{can}}}}\\
  \mc{A}(\left[\omega_{\underline{n}+\fr},1\right[)\otimes F^*{N}\ar[r]
   _{\sim}&
   \Emod{\underline{n}+\fr,\dag}{s}(F^*\ms{M}^{(m)})\ar@{^{(}->}[r]_{b}&
   \ms{F}^{(s,\infty')}(F^*{M}^{\mr{can}})
   }
 \end{equation*}
 is commutative. Let  $\epsilon^{(\underline{n})}$  be the unique isomorphism making the above diagram  commutative. 
We have $\epsilon^{(\underline{n})}(1\otimes\widehat{x})=\tau'_*(\Upsilon)\cdot
 \widehat{1\otimes x}$ for any $x\in\ms{M}^{(m)}$ by the injectivity and
 (\ref{comingfromglob}). We put $c=a\circ b$. By changing $m$, we may
 take $\underline{n}=m$.

 Let $x\in {M}$. We may assume that
 $x\in(\DcompQ{m}{s})^{\mr{an}}\otimes\ms{M}^{(m)}$ by increasing $m$ if
 necessary. This element can be
 seen as an element of $\EcompQ{m,m'}{s}\otimes\ms{M}^{(m)}$ with some
 integer $m'\geq m$. By Remark
 \ref{anyelem} (ii), there exists a sequence $\{1\otimes x_k\}$ in
 $\mr{Im}(\ms{M}^{(m)}\rightarrow\EcompQ{m,m'}{s}\otimes\ms{M}^{(m)})$
 with $x_k\in\ms{M}^{(m)}$ which converges to $x$ in
 $\EcompQ{m,m'}{s}\otimes\ms{M}^{(m)}$ using the topology induced by the
 $\EcompQ{m,m'}{s}$-module structure.
 Consider the topology induced by the finite
 $\CK{\Aone^{(1)}}\{\partial\}^{(m,m')}$-module structure. By Lemma
 \ref{topology}, these topologies are equivalent, and we get that the
 same sequence converges to $x$ also in this topology. Since
 \begin{equation*}
  F^*\EcompQ{m,m'}{s}(\ms{M}^{(m)})\cong\CK{\Aone}\{\partial\}
   ^{(m+\fr,m'+\fr)}\otimes_{\CK{\Aone^{(1)}}\{\partial\}^{(m,m')}}
   \EcompQ{m,m'}{s}(\ms{M}^{(m)})
 \end{equation*}
 by definition, the sequence $\{1\otimes(1\otimes x_k)\}$ in
 $F^*\EcompQ{m,m'}{s}(\ms{M}^{(m)})$ converges to the element
 $1\otimes x$ using the
 $\CK{\Aone}\{\partial\}^{(m+\fr,m'+\fr)}$-module topology.
 Since $\epsilon^{(m)}$ is a homomorphism of finite
 $\CK{\Aone}\{\partial\}^{(m+\fr,\dag)}$-modules, it is in particular a
 continuous homomorphism of topological modules over the noetherian
 Banach algebra $\CK{\Aone}\{\partial\}^{(m+\fr,\dag)}$. Since the
 topology is separated, we get
 \begin{align*}
  \epsilon^{(m)}(1\otimes\widehat{x})=\lim
  _{i\rightarrow\infty}\epsilon^{(m)}&(1\otimes\widehat{x_i})
  =\lim_{i\rightarrow\infty}\tau'_*(\Upsilon)\cdot(1\otimes
  x_i)^{\wedge}\\&=\tau'_*(\Upsilon)\cdot\lim_{i\rightarrow\infty}
  (1\otimes x_i)^{\wedge}=\tau'_*(\Upsilon)\cdot(1\otimes x)^{\wedge}.
 \end{align*}
 Now, we get
 \begin{equation*}
  \epsilon_{{M},s}(1\otimes\widehat{x})=
   c(\epsilon^{(m)}(1\otimes\widehat{x}))=c(\tau'_*(\Upsilon)(1\otimes
   x)^{\wedge})=\tau'_*(\Upsilon)((1\otimes x)^{\wedge})
 \end{equation*}
 and the proposition follows.
\end{proof}

\begin{dfn}
 \label{dfnlocFourFrob}
 Let $\ms{M}$ be a holonomic $F$-$\DdagQ{\Pone}(\infty)$-module. Let
 $s$ be a singularity of $\ms{M}$ in $\Aone$. Take a geometric point
 $\mf{s}\in\mb{A}^1_k(\overline{k})$ sitting over $s$. We define the
 Frobenius structure on
 $\ms{F}^{(s,\infty')}(\ms{M})$ by using the canonical isomorphism of
 Lemma \ref{locFourrel}
 \begin{equation*}
  \ms{F}^{(s,\infty')}(\ms{M})\cong\mr{Res}^{K_s}_K\bigl(\Phi
   ^{(\mf{s},\infty')}({\tau}_{s*}\ms{M}|_{S_{s}})\bigr).
 \end{equation*}
 The Frobenius structure is well-defined since it does not depend on the
 choice of $\mf{s}$ by Proposition \ref{frobstruccalcabsw}.
\end{dfn}

\begin{thm}
 \label{stationaryfrobcom}
 The regular stationary phase isomorphism
 {\normalfont(\ref{stationaryhom})} is compatible with Frobenius
 structures.
\end{thm}
\begin{proof}
 To show this, it suffices to show that the following diagram is
 commutative.
\begin{equation*}
 \xymatrix@C=80pt{
  F^*\nF{\ms{M}}|_{\eta_\infty}\ar[r]^<>(.5){\sim}
  \ar[d]_{\sim}&
  \bigoplus_{s\in S}\tau'^*F^*\ms{F}^{(s,\infty)}(\ms{M})
  \ar[d]^{\sim}\\
 \nF{F^*\ms{M}}|_{\eta_\infty}\ar[r]_<>(.5){\sim}&
  \bigoplus_{s\in S}\tau'^*\ms{F}^{(s,\infty)}(F^*\ms{M})
  }
\end{equation*}
 The left vertical arrow is defined by the Frobenius structure of
 geometric Fourier transform, and the right vertical arrow by Definition
 \ref{dfnlocFourFrob}. To show that it is commutative, it suffices to
 show the commutativity for $1\otimes\widehat{m}\in F^*\nF{\ms{M}}$
 for any $m\in\ms{M}$. This follows from the description of the vertical
 isomorphisms in terms of the operator $\Upsilon$ given in Lemma
 \ref{globalfourier} and Proposition \ref{frobstruccalcabsw}.
\end{proof}

\subsection{Explicit calculations of the Frobenius structures on Fourier
 transforms}
\label{explicitcalcfrob}
To calculate the Frobenius structure of Fourier transforms
concretely, the results of the last subsection imply that all we need
to do is to determine the differential operator $\Upsilon$. To calculate
this, it suffices to calculate the isomorphism
\begin{equation*}
 \Phi:=\epsilon_{\DdagQ{\Pone}(\infty)}\colon
  F^*(\nF{\DdagQ{\Pone^{(1)}}(\infty)})\xrightarrow{\sim}
  \nF{F^*\DdagQ{\Pone^{(1)}}(\infty)}
\end{equation*} 
concretely, which is the goal of this subsection.

\subsubsection{}\label{par_def_Dwork_op} 
Recall the notation of \ref{setupFour}, and consider the following
diagram.
\begin{equation*}
 \xy
 \xymatrix@C=40pt@R=10pt{
  &{\Poneoned}\ar[dr]\ar[dl]\ar[dd]|{F_{\Poneoned}}&\\
 \Pone\ar[dd]_{F_{\Pone}}&&\Poned\ar[dd]^{F_{\Poned}}\\
 &\Poneoned^{(1)}\ar[dr]\ar[dl]&\\
 \Pone^{(1)}&&\Poned^{(1)}}
  \def\objectstyle{\scriptstyle}
  \def\labelstyle{\scriptstyle}
  \xymatrix@C=40pt@R=10pt{
  &&\\
 *!(-1,4){x}&&*!(-7,4){x'}\\
 &&\\
 *!(-1,-1){y}&&*!(-7,-1){y'}
  }
  \endxy
\end{equation*}
Here, $^{(1)}$ means $\otimes_{R,\sigma}R$, and small letters $x$, $y$,
$x'$, $y'$ denote the coordinates. The middle vertical morphism
$F_{\Poneoned}$ is defined by sending $y$ and $y'$ to $x^q$ and
$x'^q$ respectively. By the definition of the morphisms, we note that the diagram is
commutative.

To proceed, we will review the construction of the fundamental
isomorphism
\begin{equation*}
 \ms{H}^{1}\ms{F}_{\pi}(\DdagQ{\Pone}(\infty))\xrightarrow{\sim}
  \nF{\DdagQ{\Pone}(\infty)}
\end{equation*}
of Noot-Huyghe (\ref{compgeomnaive}). First, 
we recall that $\DdagQ{\Poneoned\rightarrow\Pone}(Z)$, cf.\ \cite[\S1.4]{NH}, allows to compute direct images and we consider consider the Spencer
resolution \cite[4.2.1]{NH}:
\begin{equation*}
 \DdagQ{\Poneoned}(Z)\otimes\bigwedge^{\bullet}\ms{T}_{\Poneoned/\Pone}
  \rightarrow\DdagQ{\Poneoned\rightarrow\Pone}(Z).
\end{equation*}
Here $\ms{T}_{\Poneoned/\Pone}$ denotes the relative tangent bundle of
$\Poneoned$ over $\Pone$. She showed that we can calculate
$\ms{F}(\DdagQ{\Pone}(\infty))$ using this resolution, namely, there
exists an isomorphism
\begin{equation*}
 \ms{F}(\DdagQ{\Pone}(\infty))\cong p'_{*}
  (\DdagQ{\Poned\leftarrow\Poneoned}(Z)
  \otimes_{\DdagQ{\Poneoned}(Z)}(\ms{L}_{\pi,\mu}\otimes_{\mc{O}
  _{\Poneoned,\mb{Q}}(Z)}\DdagQ{\Poneoned}(Z)
  \otimes\bigwedge^{\bullet}\ms{T}_{\Poneoned/\Pone})).
\end{equation*}
Then, she defined a homomorphism
\begin{equation}
 \label{fourierlast}
  p'_*(\DdagQ{\Poned\leftarrow\Poneoned}(Z)
  \otimes_{\DdagQ{\Poneoned}(Z)}(\ms{L}_{\pi,\mu}
  \otimes_{\mc{O}_{\Poneoned,\mb{Q}}(Z)}
  \DdagQ{\Poneoned}(Z)))\rightarrow\nF{\DdagQ{\Pone}(\infty)}
\end{equation}
and showed that this factors through the geometric Fourier
transform. Let us recall how (\ref{fourierlast}) is defined.
We identify
\begin{equation}
 \label{identify}
 \DdagQ{\Poned\leftarrow\Poneoned}(Z)=\omega^{-1}_{\Poned}
  \otimes(\mc{O}_{\Poneoned}\otimes\DdagQ{\Poneoned}(Z)\otimes
  \omega_{\Poneoned}).
\end{equation}
The homomorphism sends $\bigl((dx')^\vee\otimes1\otimes1\otimes(dx\wedge
dx')\bigr)\otimes(e\otimes P)$ to $\widehat{P}$ where $e$ is
the canonical section of $\ms{L}_{\pi,\mu}$ (cf.\ \ref{defoflofdw}). To
verify that this defines a homomorphism, see \cite[5.2.1 {\it
etc.}]{NH}.

Before starting the calculation, we introduce the Dwork operator. Let
$\ms{X}$ be a smooth formal scheme possessing a system of local
coordinates $\{x_1,\dots,x_d\}$. Let $\{x'_1,\dots,x'_d\}$ be the system
of local coordinates of $\ms{X}^{(1)}$ induced by pulling-back
$\{x_1,\dots,x_d\}$.
We uniquely have a lifting of the relative Frobenius morphism
$\ms{X}\rightarrow\ms{X}^{(1)}$ sending $x'_i$ to $x_i^q$. Then we put
\begin{equation*}
 H_{\ms{X}}:=\frac{1}{q^d}\prod_{1\leq i\leq d}\sum_{\zeta^q=1}\sum_{k\geq 0}
  (\zeta-1)^kx_i^k\partial_i^{[k]} %\index{.@miscellaneous II!$H_{\ms{X}}$}
\end{equation*}
in $\Gamma(\ms{X},\DcompQ{h}{\ms{X}})$.
If there is nothing to be confused, we denote $H_{\ms{X}}$ by $H$.
We note that even if $\zeta\not\in K$, the operator is defined over $K$,
and do not need to extend $K$ to define this operator. For
the details, we refer to \cite{Gar3}.

By applying $F^*_{\Poned}$ to (\ref{fourierlast}), we get the
homomorphism
\begin{equation*}
 F^*_{\Poned}p'_*\bigl(\DdagQ{\Poned^{(1)}\leftarrow\Poneoned^{(1)}}(Z)
  \otimes_{\DdagQ{\Poneoned^{(1)}}(Z)}(\ms{L}^{(1)}_{\pi,\mu}\otimes
  _{\mc{O}_{\Poneoned^{(1)},\mb{Q}}(Z)}\DdagQ{\Poneoned^{(1)}})\bigr)
  \rightarrow F^*_{\Poned}\nF{\DdagQ{\Pone^{(1)}}},
\end{equation*}
where $\ms{L}^{(1)}_{\pi,\mu}$ on $\Poneoned^{(1)}$ denotes the base
change of $\ms{L}_{\pi,\mu}$.
From the next paragraph, we will calculate the Frobenius commutation
homomorphism on the source of the homomorphism. For this, we
always use the identification (\ref{identify}).

\subsubsection{}
In this paragraph, we calculate the canonical homomorphism of
sheaves on $\Poned$
\begin{align*}
 \phi\colon&F^*_{\Poned}p'_*\bigl(\DdagQ{\Poned^{(1)}\leftarrow
 \Poneoned^{(1)}}(Z)\otimes_{\DdagQ{\Poneoned^{(1)}}(Z)}(\ms{L}^{(1)}
 _{\pi,\mu}\otimes_{\mc{O}_{\Poneoned^{(1)},\mb{Q}}(Z)}
 \DdagQ{\Poneoned^{(1)}}(Z))\bigr)\\
  &\qquad\rightarrow p'_*\DdagQ{\Poned\leftarrow\Poneoned}(Z)
 \otimes_{\DdagQ{\Poneoned}(Z)}\bigl(\ms{L}_{\pi,\mu}\otimes
  _{\mc{O}_{\Poneoned,\mb{Q}}(Z)}\DdagQ{\Poneoned\rightarrow\Pone}(Z)
 \otimes F^*_{\Pone}\DdagQ{\Pone}(\infty)\bigr).
\end{align*}
In this paragraph, for simplicity, we denote
$\DdagQ{\Pone}(Z)$, $\DdagQ{\Pone\leftarrow\Poneoned}(Z)$ {\it etc.\
}by $\ms{D}_{\Pone}$, $\ms{D}_{\Pone\leftarrow\Poneoned}$ {\it etc.},
and we identify sheaves and its global sections.

In the following, we will compute $\phi(\alpha_0)$ where
$\alpha_0:=1\otimes\bigl((dy')^{\vee}\otimes1\otimes1
\otimes(dy\wedge dy')\bigr)\otimes(e\otimes1)$ is a global section of
the source of $\phi$. First, we get an isomorphism
\begin{equation*}
 F^*_{\Poned}\bigl(\ms{D}_{\Poned^{(1)}\leftarrow\Poneoned^{(1)}}
  \otimes_{\ms{D}_{\Poneoned^{(1)}}}(\ms{L}^{(1)}_{\pi,\mu}\otimes
  _{\mc{O}_{\Poneoned^{(1)}}}\ms{D}_{\Poneoned^{(1)}})\bigr)\xrightarrow{\sim}
  \ms{D}_{\Poned\leftarrow\Poneoned}\otimes_{
  \ms{D}_{\Poneoned}}F^*_{\Poneoned}\bigl(\ms{L}^{(1)}_{\pi,\mu}\otimes
  _{\mc{O}_{\Poneoned^{(1)}}}\ms{D}_{\Poneoned^{(1)}}\bigr).
\end{equation*}
By \cite[Proposition 2.5]{Abe3}, the element $\alpha_0$ is sent to
\begin{equation*}
 \alpha_1:=\bigl((dx')^{\vee}\otimes1\otimes
  Hx'^{-(q-1)}\otimes(dx\wedge dx')\bigr)\otimes
  x^{q-1}x'^{q-1}(e\otimes1).
\end{equation*}
Now, we get an isomorphism
\begin{equation*}
 \ms{D}_{\Poned\leftarrow\Poneoned}\otimes_{
  \ms{D}_{\Poneoned}}F^*_{\Poneoned}\bigl(\ms{L}^{(1)}_{\pi,\mu}\otimes
  _{\mc{O}_{\Poneoned^{(1)}}}\ms{D}_{\Poneoned^{(1)}}\bigr)\xrightarrow{\sim}
  \ms{D}_{\Poned\leftarrow\Poneoned}\otimes_{
  \ms{D}_{\Poneoned}}\bigl(\ms{L}_{\pi,\mu}\otimes
  _{\mc{O}_{\Poneoned}}F^*_{\Poneoned}\ms{D}_{\Poneoned^{(1)}}\bigr)
\end{equation*}
using the Frobenius structure of $\ms{L}_{\pi,\mu}$. This Frobenius
structure $F^*\ms{L}_{\pi,\mu}\xrightarrow{\sim}\ms{L}_{\pi,\mu}$ sends
$1\otimes e$ to $\exp(\pi((xx')-(xx')^q))\cdot e$ as written in
\cite[(2.12.1)]{BB}\footnote{Note that in {\it loc.\ cit.}, our $\pi$ is
equal to $-\pi$ in their notation, and this is why we get
$(xx')-(xx')^q$ instead of $(xx')^q-(xx')$.}. Using this, $\alpha_1$ is
sent by this isomorphism to
\begin{equation*}
 \alpha_2:=\bigl(\underbrace{(dx'_{\bullet})^{\vee}\otimes1\otimes
  Hx'^{-(q-1)}\otimes(dx_{\bullet}\wedge dx'_\bullet)}_{=:A}\bigr)
  \otimes x^{q-1}x'^{q-1}\bigl(
  \underbrace{\exp(\pi((xx')-(xx')^q))\cdot e}_{=:E}
  \otimes(1\otimes1)\bigr).
\end{equation*}
Then we get a homomorphism
\begin{equation*}
 \ms{D}_{\Poned\leftarrow\Poneoned}\otimes_{
  \ms{D}_{\Poneoned}}\bigl(\ms{L}_{\pi,\mu}\otimes
  _{\mc{O}_{\Poneoned}}F^*_{\Poneoned}\ms{D}_{\Poneoned^{(1)}}\bigr)
  \rightarrow\ms{D}_{\Poned\leftarrow\Poneoned}\otimes_{
  \ms{D}_{\Poneoned}}\bigl(\ms{L}_{\pi,\mu}\otimes
  _{\mc{O}_{\Poneoned}}F^*_{\Poneoned}\ms{D}_{\Poneoned^{(1)}\rightarrow
  \Pone^{(1)}}\bigr),
\end{equation*}
which sends $\alpha_2$ to $\alpha_3:=A\otimes x^{q-1}x'^{q-1}(E\otimes
(1\otimes1\otimes1))$. Then we have an isomorphism
\begin{equation*}
 \ms{D}_{\Poned\leftarrow\Poneoned}\otimes_{
  \ms{D}_{\Poneoned}}\bigl(\ms{L}_{\pi,\mu}\otimes
  _{\mc{O}_{\Poneoned}}F^*_{\Poneoned}\ms{D}_{\Poneoned^{(1)}\rightarrow
  \Pone^{(1)}}\bigr)
  \xrightarrow{\sim}\ms{D}_{\Poned\leftarrow\Poneoned}\otimes_{
  \ms{D}_{\Poneoned}}\bigl(\ms{L}_{\pi,\mu}\otimes
  _{\mc{O}_{\Poneoned}}\ms{D}_{\Poneoned\rightarrow\Pone}\otimes
  F^*_{\Pone}\ms{D}_{\Poned}\bigr),
\end{equation*}
which sends $\alpha_3$ to $\alpha_4:=A\otimes x^{q-1}x'^{q-1}
(E\otimes(1\otimes1)\otimes(1\otimes1))$. Summing up, we have that the
homomorphism $\phi$ sends
\begin{align*}
 1\otimes\bigl((dx'_{\bullet})^{\vee}\otimes1\otimes1\otimes
 (dx_{\bullet}\wedge dx'_\bullet)\bigr)\otimes(e\otimes1)\mapsto
 A\otimes x^{q-1}x'^{q-1}(E\otimes(1\otimes1)\otimes(1\otimes1)).
\end{align*}

\subsubsection{}
Let us finish the calculation of $\Phi$. Consider the following
commutative diagram
\begin{equation*}
 \xymatrix{
  p'_*\Ddag{\Poned\leftarrow\Poneoned}\otimes
  _{\Ddag{\Poneoned}}\bigl(\ms{L}_{\pi,\mu}\otimes
  _{\mc{O}_{\Poneoned}}\Ddag{\Poneoned\rightarrow\Pone}\bigr)
  \ar[r]^<>(.5){\beta}\ar[d]_{\alpha}
  &\nF{\Ddag{\Pone}(\infty)}\ar[d]^{\delta}\\
 p'_*\Ddag{\Poned\leftarrow\Poneoned}
  \otimes_{\Ddag{\Poneoned}}\bigl(\ms{L}_{\pi,\mu}\otimes
  _{\mc{O}_{\Poneoned}}\Ddag{\Poneoned\rightarrow\Pone}\otimes
  F^*_{\Pone}\Ddag{\Pone^{(1)}}(\infty)\bigr)\ar[r]_<>(.5){\gamma}&
  \nF{F^*_{\Pone}\Ddag{\Pone^{(1)}}(\infty)},
  }
\end{equation*}
where we have omitted $\mb{Q}$ in subscripts and $(Z)$ to save
space. The homomorphisms $\alpha$ and $\delta$ are the canonical
homomorphisms induced by the homomorphism
$\DdagQ{\Pone}(\infty)\rightarrow
F^*_{\Pone}\DdagQ{\Pone^{(1)}}(\infty)$, $\beta$ is nothing but
(\ref{fourierlast}), and $\gamma$ is the homomorphism of Huyghe
(\ref{compgeomnaive}). By the computation of the last paragraph, we have
\begin{equation*}
 \Phi(1\otimes\widehat{1})=\gamma(\alpha_4).
\end{equation*}
Since these sheaves are $\DdagQ{\Poned}(\infty')$-modules, we identify
the sheaves with their global sections.
Since $\alpha\bigl(A\otimes x^{q-1}x'^{q-1}(E\otimes(1\otimes1))
\bigr)=\alpha_4$, all we have to calculate is
$(\delta\circ\beta)\bigl(A\otimes x^{q-1}x'^{q-1}(E\otimes
(1\otimes1))\bigr)$. Let
\begin{equation*}
 \exp(\pi(t-t^q))=\sum_{n\geq 0}\alpha_nt^n.
\end{equation*}
Then we get
\begin{align*}
 &\beta\Bigl(\bigl((dx')^{\vee}\otimes1\otimes
 Hx'^{-(q-1)}\otimes(dx\wedge dx')\bigr)\otimes
 x^{q-1}x'^{q-1}\bigl(\exp(\pi((xx')-(xx')^q))\cdot e\otimes
 (1\otimes1)\bigr)\Bigr)\\
 &\qquad=\beta\Bigl(\bigl((dx')^{\vee}\otimes1\otimes
 Hx'^{-(q-1)}\otimes(dx\wedge dx')\bigr)\otimes
 \bigl(\sum_{n\geq 0}\alpha_n(xx')^nx^{q-1}x'^{q-1}\bigr)\bigl(e\otimes
 (1\otimes1)\bigr)\Bigr)\\
 &\qquad=\beta\Bigl(\sum_{n\geq 0}\bigl((dx')^{\vee}
 \otimes1\otimes(x'^{q-1}x'^n\alpha_nHx'^{-(q-1)})\otimes(dx\wedge
 dx')\bigr)\otimes\bigl(e\otimes(1\otimes
 x^nx^{q-1})\bigr)\Bigr)\\
 &\qquad=\Bigl(\bigl(x'^{q-1}Hx'^{-(q-1)}\bigr)^t\cdot\sum_{n\geq 0}
 \alpha_nx'^n\Bigr)\cdot\beta\Bigl(\bigl((dx')^{\vee}
 \otimes1\otimes1\otimes(dx\wedge
 dx')\bigr)\otimes\bigl(e\otimes(1\otimes
 x^nx^{q-1})\bigr)\Bigr)\\
 &\qquad=\Bigl(\bigl(x'^{q-1}Hx'^{-(q-1)}\bigr)^t\cdot\sum_{n\geq 0}
 \alpha_nx'^n\Bigr)\cdot(x^nx^{q-1})^{\wedge}\\
 &\qquad=\bigl(x'^{q-1}Hx'^{-(q-1)}\bigr)^t\cdot\sum_{n\geq 0}
 \alpha_nx'^n\Bigl(\frac{-\partial'}{\pi}\Bigr)^n\Bigl(\frac{-\partial'}
 {\pi}\Bigr)^{q-1}\cdot\widehat{1}.
\end{align*}

Summing up, we get the following theorem.
\begin{thm}
 Let $\ms{M}$ be a coherent $\DdagQ{\Pone^{(1)}}(\infty)$-module. We
 write
 \begin{equation*}
  \exp(\pi(t-t^q))=\sum_{n\geq 0}\alpha_nt^n
 \end{equation*}
 with $\alpha_n\in K$. The canonical isomorphism
 $\Phi\colon F^*_{\Poned}(\nF{\ms{M}})\xrightarrow{\sim}
 \nF{F^*_{\Pone}\ms{M}}$ can be described as follows: for any
 $m\in\Gamma(\Pone^{(1)},\ms{M})$, we have 
 \begin{equation*}
  \Phi(1\otimes\widehat{m})=\bigl(x'^{q-1}\cdot H_{\Poned}\cdot
   x'^{-(q-1)}\bigr)^t\cdot\sum_{n\geq 0}\alpha_nx'^n\Bigl(
   \frac{-\partial'}{\pi}\Bigr)^n\Bigl(\frac{-\partial'}
   {\pi}\Bigr)^{q-1}\cdot(1\otimes m)^\wedge.
 \end{equation*}
\end{thm}

\section{A key exact sequence}
\label{section6}
To show Laumon's formulas, one of the key point was to use the exact
sequence appearing in the proof of \cite[3.4.2]{Lau}. This exact
sequence was deduced from an exact sequence connecting nearby cycles and
vanishing cycles. Since our definition of local Fourier transforms
does not use vanishing cycles, we need some arguments to acquire an
analogous exact sequence, which is the main goal of this section.

\subsection{Commutation of Frobenius}
\label{secommfrob}
In this subsection, we show a commutativity result of Frobenius
pull-back and microlocalization. This result is used to define a
Frobenius structure on microlocalizations defined in the next
subsection.

\subsubsection{}\label{F_cent}
Let $\ms{X}$ be an affine formal curve over $R$. We consider the
situation in paragraph \ref{setupFrob}. We moreover suppose that
$\ms{X}$ and $\ms{X}'$ possess a global 
coordinate denoted by
$x$ and $x'$ respectively, and that the relative Frobenius homomorphism
$\ms{X}\rightarrow\ms{X}'$ sends $x'$ to $x^q$. For any smooth formal
scheme, we can choose such $x$ and $x'$ locally, with the same property. We denote by $\partial$
and $\partial'$ the differential operators corresponding to $x$ and $x'$
respectively. We denote by $X_i$ and $X'_i$ the reduction of $\ms{X}$
and $\ms{X}'$ over $R_i$ as usual.

\begin{lem*}
 \label{frobcenter}
 Let $i$ and $m$ be  non-negative integers. Let
 $\Phi\colon\Dmod{m+\fr}{X_i}\rightarrow F^*\Dmod{m}{X'_i}$
 %\index{.@miscellaneous II!$\Phi$}
 be the canonical homomorphism of Berthelot
 {\normalfont\cite[2.5.2.3]{Ber2}}. Then there exists a positive integer
 $N_{i}$ such that $(\partial'^{\angles{m}{p^m}})^{N_{i}}$ is in the
 center of $\Dmod{m}{X'_i}$ and, for any positive integer $l$,
 \begin{equation}
  \label{genber2.2.4}
   \Phi((\partial^{\angles{m+\fr}{p^{m+\fr}}})^{lN_{i}})=1\otimes
   (\partial'^{\angles{m}{p^m}})^{lN_{i}}.
 \end{equation}
\end{lem*}
\begin{proof}   
 To have lighter notation we will put, for every $N$ and $m$,
 $\partial^{\angles{m}{p^m}^{N}}:=(\partial^{\angles{m}{p^m}})^{N}$.
 The case $i=0$ of the first statement is proven in \cite[2.2.6]{Ber1}
 (we have $N_0=p$), and the general case goes similarly ($N_i=p^{i+1}$
 is enough); see also \cite[2.6 Remark]{Abe}.

 To finish the proof we proceed by induction on $i$. Suppose that the
 statement is true for $i$ and let us prove it for $i+1$. By induction
 hypothesis, there exists $N_i$, such that for every $l$, we may write
 \begin{equation}\label{eq_ipotesi_induttiva}
  \Phi(\partial^{\angles{m+\fr}{p^{m+\fr}}^{lN_i}})=1\otimes
   \partial'^{\angles{m}{p^m}^{lN_i}}+\varpi^{i}\sum_rf_r^{(l)}\otimes
   \partial'^{\angles{m}{r}}
 \end{equation}
 in $\Dmod{m}{X'_{i+1}}$, for some $f_r^{(l)}\in\mc{O}_{X'_{i+1}}$. 
 By changing $N_i$ by a multiple, we can assume it
 divisible by $p^{i+1}$, so that
 $\partial^{\angles{m+\fr}{p^{m+\fr}}^{N_i}}$ is contained in the center of
 $\Dmod{m+\fr}{X_{i+1}}$. 
 For any integer $j>0$, we get
 \begin{align}
  \label{calberfrob}
  &\partial^{\angles{m+\fr}{p^{m+\fr}}^{lN_i}}\cdot\bigl(1\otimes 
  \partial'^{\angles{m}{p^m}^{lN_ij}}+j\varpi^{i}\sum_rf_r^{(l)}\otimes
  \partial'^{\angles{m}{r}}\partial'^{\angles{m}{p^m}^{lN_i(j-1)}}\bigr)\\
  \notag
  &\qquad=\bigl(1\otimes
   \partial'^{\angles{m}{p^m}^{lN_i}}+\varpi^{i}\sum_rf_r^{(l)}\otimes
   \partial'^{\angles{m}{r}}\bigr) 
	\bigl(1\otimes 
  \partial'^{\angles{m}{p^m}^{lN_ij}}+j\varpi^{i}\sum_rf_r^{(l)}\otimes
  \partial'^{\angles{m}{r}}\partial'^{\angles{m}{p^m}^{lN_i(j-1)}}\bigr)\\
  \notag
  %&\qquad=1\otimes
  %\partial'^{\angles{m}{p^m}^{N(j+1)}}+\varpi^{i}\sum_rf_r\otimes
  %\partial'^{\angles{m}{r}}\partial'^{\angles{m}{p^m}^{Nj}}+ %\\ \notag &\hspace{8cm} 
	%j\varpi^{i}\sum_rf_r\otimes\partial'^{\angles{m}{r}}
  %\partial'^{\angles{m}{p^m}^{Nj}}\\
  %\notag
  &\qquad=1\otimes \partial'^{\angles{m}{p^m}^{lN_i(j+1)}}
	+(j+1)
  \varpi^{i}\sum_rf_r^{(l)}\otimes\partial'^{\angles{m}{r}}
  \partial'^{\angles{m}{p^m}^{lN_ij}},
 \end{align}
 where $\cdot$ in the first line stand for the left action of
 $\Dmod{m+\fr}{X_{i+1}}$ on $F^*\Dmod{m}{X'_{i+1}}$ and in the second
 equality we are using that $\partial^{\angles{m+\fr}{p^{m+\fr}}^{N_i}}$
 is central. Thus, we get
 \begin{align*}
  \Phi(\partial^{\angles{m+\fr}{p^{m+\fr}}^{plN_i}}) &=
  \partial^{\angles{m+\fr}{p^{m+\fr}}^{(p-1)lN_i}}
  \cdot(1\otimes\partial'^{\angles{m}{p^m}^{lN_i}} +\varpi^{i}
  \sum_rf_r^{(l)}\otimes\partial'^{\angles{m}{r}})\\
  &=\bigl(1\otimes\partial'^{\angles{m}{p^m}^{plN_i}}
  +p\varpi^{i}\sum_rf_r^{(l)}\otimes\partial'^{\angles{m}{r}}
  \partial'^{\angles{m}{p^m}^{(p-1)lN_i}}\bigr)\\
  &=1\otimes\partial'^{\angles{m}{p^m}^{plN_i}},
 \end{align*}
 where for the second
 equality, we used (\ref{calberfrob}) $(p-1)$-times. By taking $N_{i+1}$
 to be $pN_i$, the equality (\ref{genber2.2.4}) holds for $i+1$, and
 the lemma follows.
\end{proof}

\subsubsection{}\label{FrobE}
We keep previous notation. Let $i$ and $m$ be  non-negative integers.
First, let us construct a canonical homomorphism
\begin{equation}
 \label{frobforE}
 \Emod{m+\fr}{X_i}\rightarrow F^*\Emod{m}{X'_i}
\end{equation}
compatible with $\Phi$, where $F^*$ denotes
$\pi^{-1}\mc{O}_{X_i}\otimes_{\pi^{-1}\mc{O}_{X'_i}}$. Let $N$ be the
integer $N_i$ in Lemma \ref{frobcenter}. Let $S_{m+\fr}$ be the
multiplicative system of $\Dmod{m+\fr}{X_i}$ generated by
$(\partial^{\angles{m+\fr}{p^{m+\fr}}})^N$, and $S_m$ be
that of $\Dmod{m}{X'_i}$ generated by
$(\partial'^{\angles{m}{p^{m}}})^N$, which is also contained in the center
of $\Dmod{m}{X'_i}$.
For a (non-commutative) ring $A$, we denote by $A[\zeta]$ the ring of
polynomials in the variable $\zeta$ such that $\zeta$ is in the center.
We know that $S_{m+\fr}^{-1}\Dmod{m+\fr}{X_i}\cong
\Dmod{m+\fr}{X_i}[\zeta]/((\partial^{\angles{m+\fr}{p^{m+\fr}}})^N\zeta-1)$.
Define a $\Dmod{m+\fr}{X_i}[\zeta]$-module structure
on $F^*S_m^{-1}\Dmod{m}{X'_i}$ by putting
\begin{equation*}
 \zeta(f\otimes\partial'^{\angles{m}{i}}):=f\otimes
  \partial'^{\angles{m}{i}}(\partial'^{\angles{m}{p^m}})^{-N}.
\end{equation*}
This structure defines a homomorphism $\Dmod{m+\fr}{X_i}[\zeta]
\rightarrow F^*S_m^{-1}\Dmod{m}{X_i}$ of left
$\Dmod{m+\fr}{X_i}[\zeta]$-modules by sending $1$ to $1\otimes1$.
Since $(\partial^{\angles{m+\fr}{p^{m+\fr}}})^N\cdot(1\otimes
(\partial'^{\angles{m}{p^m}})^{-N})=1\otimes1$, the homomorphism
factors through $\Dmod{m+\fr}{X_i}[\zeta]\rightarrow S_{m+\fr}^{-1}
\Dmod{m+\fr}{X_i}$, and defines a well-defined homomorphism
\begin{equation*}
 \alpha\colon S^{-1}_{m+\fr}\Dmod{m+\fr}{X_i}\rightarrow
  F^*S_m^{-1}\Dmod{m}{X'_i}
\end{equation*}
compatible with $\Phi$ and sending
$(\partial^{\angles{m+\fr}{p^{m+\fr}}})^{-N}$ to
$(\partial'^{\angles{m}{p^m}})^{-N}$. For an integer $k$, we denote by
$(F^*S_m^{-1}\Dmod{m}{X'_i})_k:=\pi^{-1}\mc{O}_{X_i}\otimes_{\pi^{-1}
\mc{O}_{X'_i}}(S_m^{-1}\Dmod{m}{X'_i})_k$. By the choice of local
coordinates (cf.\ \ref{frobcenter}),
$\alpha(\partial^{\angles{m+\fr}{1}})=qx^{q-1}\otimes
\partial'^{\angles{m}{1}}$. This implies that, for any integer $k$, we
get
\begin{equation*}
  \alpha\bigl((S^{-1}_{m+\fr}\Dmod{m+\fr}{X_i})_{k}\bigr)\subset
   (F^*S_m^{-1}\Dmod{m}{X'_i})_{[kp^{-\fr}]+Np^{m+\fr}}\,.
\end{equation*}
Thus by taking the completion with respect to the filtration by order,
$\alpha$ induces (\ref{frobforE}).

Let $m'\geq m$ be an integer.
The homomorphism (\ref{frobforE}) induces a canonical homomorphism
\begin{equation*}
 \Emod{m'+\fr}{X_i}\otimes_{\Dmod{m+\fr}{X_i}}
  F^*\Dmod{m}{X'_i}\rightarrow F^*\Emod{m'}{X'_i}.
\end{equation*}
This is an isomorphism. Indeed, since $F^*\Dmod{m}{X'_i}$ is
flat over $\Dmod{m+\fr}{X_i}$, cf.\ \cite[Cor.\ 2.5.7-(i)]{Ber2}, we get that the canonical homomorphism
induced by the injective homomorphism $\varpi\Emod{m'+\fr}{X_i}
\rightarrow\Emod{m'+\fr}{X_i}$,
\begin{equation*}
 \varpi\Emod{m'+\fr}{X_i}\otimes_{\Dmod{m+\fr}{X_i}}
  F^*\Dmod{m}{X'_i}\rightarrow\Emod{m'+\fr}{X_i}
  \otimes_{\Dmod{m+\fr}{X_i}}F^*\Dmod{m}{X'_i}
\end{equation*}
is injective. Since
$\varpi\Emod{m'+\fr}{X_i}\otimes_{\Dmod{m+\fr}{X_{i}}}F^*\Dmod{m}{X'_{i}}
\cong\Emod{m'+\fr}{X_{i-1}}\otimes_{\Dmod{m+\fr}{X_{i-1}}}
F^*\Dmod{m}{X'_{i-1}}$ is an isomorphism, by a finite induction on $i$, it remains to prove the
claim for $i=0$. In this case the verification is
straightforward, and left to the reader.

Since $\mc{O}_{X_i}$ is free of rank $q$ over $\mc{O}_{X'_i}$, $F^*$
commutes with taking inverse limit over $i$. Thus by taking the inverse
limit, we get a homomorphism
$\Psi_m\colon\Ecomp{m+\fr}{\ms{X}}\rightarrow
F^*\Ecomp{m}{\ms{X}'}$. For an integer $m'\geq m$, this induces a
canonical homomorphism
\begin{equation*}
 \beta_{m'}\colon\Ecomp{m'+\fr}{\ms{X}}\otimes_{\Dcomp{m+\fr}
  {\ms{X}}}F^*\Dcomp{m}{\ms{X}}\rightarrow F^*\Ecomp{m'}{\ms{X}}.
\end{equation*}
This is an isomorphism since both sides are complete with respect to
the $p$-adic filtrations and its reduction over $R_i$ is an
isomorphism for any $i$.

\subsubsection{}
We keep using notation of \ref{F_cent} and \ref{FrobE}.
Let $m'\geq m$ be integers. We will first define a homomorphism
$\Ecomp{m+\fr,m'+\fr}{\ms{X}}\rightarrow F^*\Ecomp{m,m'}{\ms{X}'}$
compatible with $\Psi_{m}$ and $\Psi_{m'}$. Let $\ms{E}$ be either
$\Dcomp{m+\fr}{\ms{X}}$ or $\Ecomp{m+\fr,m'+\fr}{\ms{X},0}$ considered
as a subgroup of $\Ecomp{m+\fr,m'+\fr}{\ms{X}}$. For non-negative
integers $b\geq a$, we consider $\Ecomp{a,b}{\ms{X}}$ as a subring of
$\Ecomp{a}{\ms{X}}$ or $\Ecomp{b}{\ms{X}}$ using the canonical
inclusions. Then we claim that
$\Psi_\bullet(\ms{E})\subset F^*\Ecomp{m,m'}{\ms{X}'}$ where $\bullet$
is $m$ or $m'$, and $\Psi_m|_{\ms{E}}=\Psi_{m'}|_{\ms{E}}$. For
$\ms{E}=\Dcomp{m+\fr}{\ms{X}}$, the claim is nothing but the
compatibility of $\Psi_{\bullet}$ and $\Phi$. Let us see the
claim for $\ms{E}=\Ecomp{m+\fr,m'+\fr}{\ms{X},0}$. By definition, we get
$\Psi_m(\partial^{\angles{m'+\fr}{-kp^{m'}}})=\Psi_{m'}
(\partial^{\angles{m'+\fr}{-kp^{m'}}})$ for large enough integer $k$,
and thus $\Psi_m(\partial^{\angles{m'+\fr}{-n}})=\Psi_{m'}(\partial
^{\angles{m'+\fr}{-n}})$ holds for any positive integer $n$. By a
standard continuity argument, we get the claim.

Since $\Ecomp{m+\fr,m'+\fr}{\ms{X}}$ is equal to
$\Dcomp{m'+\fr}{\ms{X}}+\Ecomp{m+\fr,m'+\fr}{\ms{X},0}$ in
$\Ecomp{m+\fr}{\ms{X}}$, we get the desired homomorphism.
Now we get the following.

\begin{lem*}
 \label{canhomm,m'}
 The canonical homomorphism
 \begin{equation}
  \label{intfrobisod}
   \Ecomp{m'+\fr,m''+\fr}{\ms{X}}\otimes_{\Dcomp{m+\fr}
   {\ms{X}}}F^*\Dcomp{m}{\ms{X}'}\rightarrow F^*\Ecomp{m',m''}
   {\ms{X}'}
 \end{equation}
 is an isomorphism.
\end{lem*}
\begin{proof}
 Since $F^*\Dcomp{m}{\ms{X}'}$ is locally projective over
 $\Dcomp{m+\fr}{\ms{X}}$, the canonical homomorphism
 \begin{equation*}
  \Ecomp{m'+\fr,m''+\fr}{\ms{X}}\otimes_{\Dcomp{m+\fr}{\ms{X}}}
   F^*\Dcomp{m}{\ms{X}'}\rightarrow\Ecomp{m'+\fr}{\ms{X}}
   \otimes_{\Dcomp{m+\fr}{\ms{X}}}F^*\Dcomp{m}{\ms{X}'}
   \xrightarrow{\sim}F^*\EcompQ{m'}{\ms{X}'}
 \end{equation*}
 is injective. Thus we get the injectivity of (\ref{intfrobisod}). Let
 us see the surjectivity. Since $\Dcomp{m'}{\ms{X}'}\xrightarrow{\sim}
 \Ecomp{m',m''}{\ms{X}'}/(\Ecomp{m',m''}{\ms{X}'})_{-1}$, it suffices to
 show that the image of (\ref{intfrobisod}) contains
 $(F^*\EcompQ{m',m''}{\ms{X}'})_{0}$. This follows from the
 fact that $(F^*\EcompQ{m',m''}{\ms{X}'})_{0}=(F^*\EcompQ{m''}
 {\ms{X}'})_{0}$, and the surjectivity of the homomorphism
 $\beta_{m''}$.
\end{proof}

\begin{lem}
 \label{EcommFrob}
 Let $\ms{M}$ be a coherent $\DcompQ{m}{\ms{X}'}$-module. Then there
 exists a canonical isomorphism
 \begin{equation*}
  \EdagQ{\ms{X}}\otimes F^*\ms{M}\rightarrow F^*(\EdagQ{\ms{X}'}
   \otimes\ms{M}).
 \end{equation*}
\end{lem}
\begin{proof}
 Since tensor product is right exact and $F^*$ is exact, it suffices to
 prove the lemma for $\ms{M}=\DcompQ{m}{\ms{X}'}$ by the coherence of
 $\DcompQ{m}{\ms{X}'}$. Since inductive limit
 commutes with tensor product, it suffices to see that the homomorphism
 $\Emod{m'+\fr,\dag}{\ms{X},\mb{Q}}\otimes F^*\DcompQ{m}{\ms{X}'}
 \rightarrow F^*\Emod{m',\dag}{\ms{X}',\mb{Q}}$ is
 an isomorphism. Since $F^*\DcompQ{m}{\ms{X}'}$ is locally projective
 over $\DcompQ{m+\fr}{\ms{X}}$, this claim follows from Lemma
 \ref{canhomm,m'} above by taking the inverse limit over $m''$, and we
 conclude the proof.
\end{proof}

\begin{rem}
 The construction does not depend on the choice of $x$ and $x'$, and
 Lemma \ref{EcommFrob} holds for any smooth formal curve
 $\ms{X}$. Since $x'$ is determined uniquely when $x$ is determined, it
 suffices to see that the construction only depends on $x$. This
 verification is left to the reader.
\end{rem}

\subsection{An exact sequence}
\label{secanexse}
In this subsection, we construct a key exact sequence. We
consider the situation in paragraph \ref{setupFour}. First, we use
a result of the previous subsection in the following definition.

\begin{dfn}
 \label{dfnfrobmu}
 Let $\ms{M}$ be a holonomic $F$-$\DdagQ{\Pone}(\infty)$-module,
 and consider
 $\Edag{0,\mb{Q}}\otimes\ms{M}:=(\EdagQ{\ms{X}}\otimes\ms{M})
 _{\xi_0}$ (cf.\ Notation \ref{not3} in \ref{setup}). This module is naturally a
 $\Dan{\ms{S}_0,\mb{Q}}$-module by
 Corollary \ref{ananddag}. Now, we get the following isomorphism
 \begin{equation*}
  \Edag{0,\mb{Q}}\otimes\ms{M}\xrightarrow[1\otimes\Phi]{\sim}
   \Edag{0,\mb{Q}}\otimes F^*\ms{M}\xrightarrow{\sim}
   F^*(\Edag{0,\mb{Q}}\otimes\ms{M}),
 \end{equation*}
 where the first isomorphism is induced by the Frobenius structure
 $\Phi\colon\ms{M}\xrightarrow{\sim}F^*\ms{M}$, and the second by Lemma
 \ref{EcommFrob}. This defines a Frobenius structure on
 $\Edag{0,\mb{Q}}\otimes\ms{M}$. We denote this
 $F$-$\Dan{\ms{S}_0,\mb{Q}}$-module by $\mu(\ms{M})$. \index{functors!.mu@$\mu(-)$}
\end{dfn}

\subsubsection{}\label{LFTinfty0'}
Let us recall that in \ref{setup} we put
$E^{\dag}_{\Aone,\mb{Q}}:=\Gamma(\mathring{T}^*\Aone,\ms{E}^{\dag}
_{\Aone,\mb{Q}})$, $\Emodb{m,\dag}{\Aone,\mb{Q}}:=\Gamma(
\mathring{T}^*\Aone,\Emod{m,\dag}{\Aone,\mb{Q}})$,
$\EcompQb{m,m'}{\Aone}:=\Gamma(\mathring{T}^*\Aone,\EcompQ{m,m'}{\Aone})$,
{\itshape etc.}. We claim that the ring isomorphism of the naive Fourier
transform
$\iota\colon\Gamma(\Pone,\DdagQ{\Pone}(\infty))\xrightarrow{\sim}
\Gamma(\Poned,\DdagQ{\Poned}(\infty'))$, $x\mapsto
{\pi}^{-1}\partial',\partial\mapsto -\pi x'$ (cf.\ \ref{NFT}), extends
to a ring homomorphism
$\iota'\colon\ms{D}^{\mr{an}}_{\ms{S}_\infty,\mb{Q}}(0)\rightarrow
E^{\dag}_{\Aoned,\mb{Q}}$, which fits into the following commutative
diagram
\begin{equation*}
 \xymatrix@C=50pt{
  \Gamma(\Pone,\DdagQ{\Pone}(\infty))\ar@{^{(}->}[d]\ar[r]^-{\iota}_-{\sim}&
  \Gamma(\Poned,\DdagQ{\Poned}(\infty'))\ar@{^{(}->}[d] \\
 \ms{D}^{\mr{an}}_{\ms{S}_\infty,\mb{Q}}(0)\ar@{-->}[r]^{\iota'} &
  E^{\dag}_{\Aoned,\mb{Q}} }
\end{equation*}
where  the right vertical injection is
\begin{equation*}
 \Gamma(\Poned,\DdagQ{\Poned}(\infty'))\xrightarrow{\mathrm{restriction}}
  \Gamma(\Aoned,\DdagQ{\Aoned})
  \xrightarrow{\sim}\Gamma(\mathring{T}^*\Aoned,\pi^{-1}\DdagQ{\Aoned})
  \hookrightarrow \Gamma(\mathring{T}^*\Aoned,\ms{E}^{\dag}
  _{\Aoned,\mb{Q}})= E^{\dag}_{\Aoned,\mb{Q}}\,;
\end{equation*}
and the left one is the natural injection, sending $x$ to $u^{-1}$ and
$\partial$ to $-u^2\partial_u$, with $u:=x^{-1}$ local parameter of
$\ms{S}_{\infty}$. Such a morphism $\iota'$ should send $u$ to
$\pi{\partial'}^{-1}$ and $\partial_u$ to ${\pi}^{-1}{\partial'}^2 x'$.
 
Let us construct $\iota'$. We recall that, for every $n,m\geq 0$, the
integers $q^{(m)}_n\geq 0$ and $0\leq r_n^{(m)} < p^m$  are defined by
$n=q^{(m)}_n p^m + r_n^{(m)}$.
We begin by defining a homomorphism, for every $m'\geq m \geq 0$, 
\begin{equation*}
  \iota'_{m'}\colon R\cc{u} \rightarrow
   \EcompQb{m,m'}{\Aoned} 
\end{equation*}
sending $u$ to $\pi{\partial'}^{-1}$.
To see this is well defined, it is enough to note that the sequence
$\pi^{p^m\cdot n}/(p^m!)^n$ goes to $0$ when $n$ goes to infinity.
%%%%%%%%%%%%%%%%%%%%%%%%
%To see that is well defined, it is enough to check that, for $n$ big enough, $(\pi{\partial'}^{-1})^{n}$ belongs to the neighborhood $U_{n,0}$
%of the $\ms{T}_0$-topology on $\Ecompb{m,m'}{\Aoned}$, cf.~\ref{topologyonring}. 
%By the definition of ${\partial'}^{\angles{m'}{-n}}$ given in 
%\ref{microrecall} 
%we have ${\partial'}^{-n} = ( l_n^{(m')}! j_n^{(m')}!)/ (j_n^{(m')}p^{m'})! {\partial'}^{\angles{m'}{-n}}$,
%where $j_n^{(m')}$ is the smallest integer such that $n\leq j_n^{(m')}p^{m'}$ and $l_n^{(m')}:= j_n^{(m')}p^{m'} - n$. 
%Put $u_n^{(m')}:= ( l_n^{(m')}! j_n^{(m')}!)/ (j_n^{(m')}p^{m'})!$.
%For $m'$ fixed and $n$ big enough with respect to $m'$, the element $\pi^n u_n^{(m')}$ belongs to $R$, so that 
%$\pi^n{\partial'}^{-n}= \pi^n u_n^{(m')}  {\partial'}^{\angles{m'}{-n}}$ is in $U_{n,0}$.
%%%%
%%%%%%%%%%%%%%%%%%%%%%%%
By construction, $\iota'_{m'}$ is independent of $m$, compatible with
natural morphisms changing the level $m'$,
and continuous for the $(\varpi,u)$-adic topology of $R\cc{u}$ and
$\ms{T}_n$-topology (cf.\ \ref{topologyonring}) for $n\gg0$ on
$\EcompQb{m,m'}{\Aoned}$. Now we are going to extend
this morphism to three different kind of rings. First, for any integer
$r=cp^{m'+1}$, with $c\in\mb{N}\backslash\{0\}$, the morphism
$\iota'_{m'}$ extends by continuity to a morphism
\begin{equation}\label{morph1}
 \mc{O}_r:= R\cc{u}\{T\}/(pT-u^r) \rightarrow \EcompQb{m,m'}{\Aoned}
\end{equation} 
sending $T$ to ${\pi^r}{p^{-1}{\partial'}^{-r}}=
(\pi^r/r!)\,(pc)!\,p^{-1}{\partial'}^{\angles{m'}{-r}}$.
Second, for $s:=p^{m+1}$, the morphism $\iota'_{m'}$ extends by
continuity to a morphism
\begin{equation}\label{morph2}
\mc{B}_s:= R\cc{u}\{Z\}/(u^sZ-p) \rightarrow \EcompQb{m,m'}{\Aoned}
\end{equation}
sending $Z$ to $p\pi^{-s}{\partial'}^s=s!\,\pi^{-s}\,
(p/p!){\partial'}^{\angles{m}{s}}$.
Third, for $m''\leq m-2$ and $m\geq 2$, the morphism $\iota'_{m'}$
extends to a continuous morphisms
\begin{equation}\label{morph3}
 \Dcomp{m''}{\ms{S}_{\infty}} \rightarrow \EcompQb{m,m'}{\Aoned}
\end{equation}
sending $\partial_u$ to ${\pi}^{-1}{{\partial'}^2x'}$. To verify this, it is enough to prove
that the sequence $\{\partial_u^{\angles{m''}{n}}\}_{n\geq 0}$ is
sent to a sequence converging to zero.  
We have 
\begin{equation*}
 \partial_u^{\angles{m''}{n}} = \frac{q_n^{(m'')}!}{n!} \partial_u^n \longmapsto Q_n^{(m'')}:=\frac{q_n^{(m'')}!}{n!} 
  \frac{({\partial'}^2x')^n}{\pi^n} \in\DcompQ{m}{\Aoned} \subset \EcompQb{m,m'}{\Aoned}.
\end{equation*}
We consider the spectral norm $\left\| - \right\|^{'(m)}$
of $\DcompQ{m}{\Aoned}$, cf.\ \cite[Rem. 2.1.4-(ii)]{Gar}, we want to compute 
$\| Q_n^{(m'')} \|^{'(m)}$.
 Recall 
that $\left\| x'\right\|^{'(m)}=1$, 
$\left\|\partial'\right\|^{'(m)}=\omega_0/\omega_m$, where $\omega_m=p^{-1/p^m(p-1)}$; and $|\pi|=\omega_0$.  
% short proof
For every integer $n\geq 0$, $({\partial'}^2 x')^n$ is divisible by $(n!)^2$ in
$\Ddag{\Aoned}$, more precisely we can prove by induction the relation 
\begin{equation}\label{eq_claim} ({\partial'}^2x')^n=(n!)^2\sum_{j=n}^{j=2n} \tbinom{j}{n}\tbinom{n+1}{j-n+1}(x')^{j-n} {\partial'}^{[j]}, 
 \end{equation}
where ${\partial'}^{[j]}= {\partial'}^{j}/j!$.
%% LONG Proof
%We need the following fact.
%\begin{cl}
%Put $P:={\partial'}^2x'\in \Ddag{\Aoned}$. For every integer $n\geq 0$, $P^n$ is divisible by $(n!)^2$ in
%$\Ddag{\Aoned}$. More precisely we have 
%\begin{equation}\label{eq_claim} P^n=\sum_{j=n}^{j=2n} a_j(x') {\partial'}^{[j]},\qquad 
%\text{with }{\partial'}^{[j]}= {\partial'}^{j}/j!\text{ and  }a_j(x')\in (n!)^2 R[x']. \end{equation}
%\end{cl}
%\begin{proof} %[Proof of the Claim]
%We have $P=x'{\partial'}^2+ 2\partial'$, therefore when we develop $P^n= \sum_{j\geq 0} a_j(x') {\partial'}^{[j]}$ we get $a_j(x')=0$ if $j\not\in[n,2n]$.  
%By induction we prove that for every $i, n\geq 0$, we have 
%$$P^n\cdot (x')^i=  \begin{cases} (n!)^2 \tbinom{i+1}{n}\tbinom{i}{n} (x')^{i-n} & \text{if }i\geq n; \\
%0 & \text{if } i< n;\end{cases}$$ where $\cdot$ denotes the action of $P^n$ on $R[x']$. 
%Therefore for $i\geq n$, we get the relations
%\begin{equation}\label{inv_rel}
    %\sum_{j=n}^{i}a_j(x')\tbinom{i}{j}{(x')}^{i-j}   =P^n\cdot (x')^i= (n!)^2 \tbinom{i+1}{n}\tbinom{i}{n} (x')^{i-n}
%\end{equation}
%We proceed now by induction on $j$: for $j=n$, taking $i=n$ in \eqref{inv_rel} we get 
%$a_n(x') = n!(n+1)!$; if we assume that for every $j\leq j_0 < 2n$, $(n!)^2$ divides $a_j(x')$, taking 
%$i=j_0+1$ in \eqref{inv_rel} we deduce that $(n!)^2$ divides also $a_{j_0+1}(x')$ and the claim is proven.  
%\end{proof}
Using \eqref{eq_claim}, and denoting by $\sigma(-)$ the sum of $p$-adic digits, we get  
$$\| ({\partial'}^2 x')^n \|^{'(m)}\leq |n!|^2 \sup_{n\leq j \leq 2n} \Big\{\tfrac{1}{|j!|} \bigl(\tfrac{\omega_0}{\omega_m}\bigr)^j\Bigr\}  
= \omega_0^{ 2n-2\sigma(n)+\inf_{n\leq j \leq 2n}\bigl\{\sigma(j) -\tfrac{j}{p^m} \bigr\}}\leq 
\omega_0^{ 2n-2\sigma(n)-\tfrac{2n}{p^m} },$$
and finally we have
\begin{equation*}
\begin{split}
\left\| Q_n^{(m'')} \right\|^{'(m)}=\left\|\frac{q_n^{(m'')}!}{n!} 
  \frac{({\partial'}^2 x')^n}{\pi^n}\right\|^{'(m)} & \leq 
	\omega_0^{ q_n^{(m'')}-\sigma(q_n^{(m'')}) -\sigma(n)-\tfrac{2n}{p^m}  }. 
	%\\ 
	%& = \omega_0^{n\left(\tfrac{1}{p^{m''}}-\tfrac{2}{p^m}\right) -\tfrac{r_n^{(m'')}}{p^{m''}}-\sigma(q_n^{(m'')}) -\sigma(n)}. 
\end{split}
\end{equation*}
For $m\geq 2$ and $m''\leq m-2$, this is converging to zero when $n$
goes to $+\infty$. 

Now, since $\EcompQb{m,m'}{\Aoned}$ is complete, combining morphisms
\eqref{morph1}, \eqref{morph2}, and \eqref{morph3}, we get, for any
$m'\geq m\geq 2$,
\begin{equation*}
  \mc{O}_{p^{m'+1}}\widehat{\otimes}_{R\cc{u}}
   \mc{B}_{p^{m+1}}\widehat{\otimes}_{R\cc{u}}
   \Dcomp{m-2}{\ms{S}_{\infty}} \rightarrow \EcompQb{m,m'}{\Aoned}.
\end{equation*}
Taking the inverse limit on $m'$ and tensoring with $\mb{Q}$, we have
\begin{equation*}
 (\DcompQ{m-2}{\ms{S}_{\infty}}(0))^{\mr{an}}:=\Bigl(\invlim_{m'}
  \bigl(\mc{O}_{p^{m'+1}}\widehat{\otimes}_{R\cc{u}}
  \mc{B}_{p^{m+1}}\widehat{\otimes}_{R\cc{u}}
  \Dcomp{m-2}{\ms{S}_{\infty}}\bigr)\Bigr)_{\mb{Q}} \rightarrow
  \Emodb{m,\dag}{\Aoned,\mb{Q}},
\end{equation*}
and finally by a direct limit on $m$ we get the homomorphism $\iota'$ we
wanted.
%%$\iota'\colon\ms{D}^{\mr{an}}_{\ms{S}_\infty,\mb{Q}}(0)\rightarrow
%E^{\dag}_{\Aoned,\mb{Q}}$ we wanted.

%%%%%%%%%%%%%%%%%%%%%%%%%%%%%%%%%%%%%%%%%

\begin{lem*}
 Let $\ms{M}$ be a holonomic $F$-$\DdagQ{\Pone}(\infty)$-module. Then
 the $F$-$\ms{D}^{\mr{an}}$-module $\mu(\nF{\ms{M}})$ only depends
 on $\ms{M}|_{\eta_\infty}$.
\end{lem*}
\begin{proof}
 First, we note here that as modules, we have
  \begin{equation*}
   \ms{M}|_{\eta_\infty}\cong\ms{D}^{\mr{an}}_{\ms{S}_\infty,\mb{Q}}(0)
    \otimes_{\DdagQ{\Pone}(\infty)}\ms{M}.
  \end{equation*}
 Since $\ms{M}$ possesses a Frobenius structure, we know that $\ms{M}$
 is a coherent $\DdagQ{\Pone}$-module. Thus, we get
 \begin{equation*}
  \ms{D}_{\ms{S}_\infty,\mb{Q}}^{\mr{an}}\otimes_{\DdagQ{\Pone}}
   \ms{M}\xrightarrow{\sim}\ms{D}_{\ms{S}_\infty,\mb{Q}}^{\mr{an}}(0)
   \otimes_{\DdagQ{\Pone}(\infty)}\ms{M}
 \end{equation*}
 by \cite[Theorem 4.1.2]{Cr}. 
 We denote still by $\iota'$ 
the composition 
 \begin{equation*}
  \ms{D}^{\mr{an}}_{\ms{S}_\infty,\mb{Q}}(0) \xrightarrow{\iota'} E^{\dag}_{\Aone,\mb{Q}}
   \hookrightarrow\ms{E}^{\mr{an}}_{\ms{S}_{0'},\mb{Q}}.
 \end{equation*}
of the homomorphism $\iota'$ constructed above and the natural injection.
We have
 \begin{align*}
  &\mu(\nF{\ms{M}})\cong\ms{E}^{\mr{an}}_{\ms{S}_{0'},\mb{Q}}
  \otimes_{\DdagQ{\Poned}(\infty)}\nF{\ms{M}}\cong
  \ms{E}_{\ms{S}_{0'},\mb{Q}}^{\mr{an}}\otimes_{\DdagQ{\Poned}(\infty)}
  (\DdagQ{\Poned}(\infty)\otimes_{\iota,\DdagQ{\Pone}(\infty)}\ms{M})\\
  &\qquad\cong\ms{E}^{\mr{an}}_{\ms{S}_{0'},\mb{Q}}\otimes_{\iota',
  \ms{D}_{\ms{S}_\infty,\mb{Q}}^{\mr{an}}(0)}\ms{M}|
  _{\eta_{\infty}}.
 \end{align*}
 To see that it is compatible with Frobenius isomorphisms, it suffices
 to apply Proposition \ref{frobstruccalcabsw}.
\end{proof}

\subsubsection{}
\label{LFspacetoppartcase}
We use the notation of paragraph \ref{setupFour}.
Let ${M}$ be a holonomic $F$-$\Dan{\ms{S},\mb{Q}}$-module. Let us
define an LF-topology on $\ms{E}^{\mr{an}}\otimes {M}$. Let $\ms{M}$
be the canonical extension of ${M}$ at $0$. By Corollary
\ref{ananddag}, we know that
\begin{equation*}
 \ms{E}^{\mr{an}}\otimes {M}\cong E^\dag_{\Aone,\mb{Q}}
  \otimes\ms{M}
%  \tau_{0*}(\EdagQ{\Aone}\otimes\ms{M})|_{\eta_0}.
\end{equation*}
Let $\ms{M}^{(m)}$ be a stable coherent $\DcompQ{m}{\Aone}$-module such
that $\DdagQ{\Aone}\otimes\ms{M}^{(m)}\cong\ms{M}$. For $m'\geq m$,
since $\Emodb{m',\dag}{\Aone,\mb{Q}}$ is a Fr\'{e}chet-Stein algebra,
any finitely presented module becomes a Fr\'{e}chet space.
By taking $m'$ to be sufficiently large, we may suppose that
$\Emodb{m',\dag}{\Aone,\mb{Q}}\otimes\ms{M}^{(m)}$ is finite free over
$\CK{\Aone}\{\partial\}^{(m',\dag)}$ by Corollary \ref{coinccrew}.
There exist two topologies on $\Emodb{m',\dag}{\Aone,\mb{Q}}\otimes
\ms{M}^{(m)}$; the Fr\'{e}chet topology induced from
$\Emodb{m',\dag}{\Aone,\mb{Q}}$-module structure denoted by $\ms{T}$,
and topology induced from the finite free
$\CK{\Aone}\{\partial\}^{(m',\dag)}$-module structure denoted by
$\ms{T}'$. Since $(\ms{M}^{(m)},\ms{T})$ becomes a topological
$\CK{\Aone}\{\partial\}^{(m',\dag)}$-module as well, we get that
$\ms{T}$ and $\ms{T}'$ are equivalent by the open mapping theorem. Now,
we put on ${E}^\dag_{\Aone,\mb{Q}}\otimes\ms{M}$ the inductive limit
topology. By the observation above, the inductive limit topology
coincides with the $\mc{R}\cong\indlim_{m'}\CK{\Aone}\{\partial\}
^{(m',\dag)}$-module topology (cf.\ Lemma \ref{fieldconstmicdif}). Since
$\mc{R}$ is an LF-space and the module is finite free over $\mc{R}$, we
get that the inductive limit-topology is separated and
${E}^\dag_{\Aone,\mb{Q}}\otimes\ms{M}$ becomes an LF-space. Thus this
defines an LF-space structure on $\ms{E}^{\mr{an}}\otimes {M}$.
By using Remark \ref{anyelem} (ii), we note here that this is also a
topological $\ms{D}^{\mr{an}}$-module. Thus, if
$\ms{E}^{\mr{an}}\otimes {M}$ is finite over
$\ms{D}^{\mr{an}}$, this LF-space topology coincides with the
$\ms{D}^{\mr{an}}$-module topology by the open mapping theorem.

\begin{lem}
 \label{reductiontomu}
 Recall the notation of {\normalfont\ref{defvanishcycle}}.
 Assume that $\mu(\ms{M})$ is a {\em holonomic}
 $F$-$\Dan{\ms{S},\mb{Q}}$-module. Then there exists a canonical
 isomorphism
 \begin{equation*}
  \Phi(\ms{M}|_{S_0})\xrightarrow{\sim}\Phi(\mu(\ms{M})).
 \end{equation*}
\end{lem}
\begin{proof}
 Let ${N}$ be a holonomic $F$-$\ms{D}^{\mr{an}}$-module such that
 ${M}:=\ms{E}^{\mr{an}}\otimes {N}$ is also a holonomic
 $F$-$\ms{D}^{\mr{an}}$-module. First, we will show that
 $\ms{E}^{\mr{an}}\otimes_{\ms{D}^{\mr{an}}}{M}\cong {M}$. Let
 \begin{equation*}
  \alpha\colon {M}\rightarrow\ms{E}^{\mr{an}}\otimes
   _{\ms{D}^{\mr{an}}} {M},\qquad
   \beta\colon\ms{E}^{\mr{an}}\otimes_{\ms{D}^{\mr{an}}}
   {M}\rightarrow {M},
 \end{equation*}
 be the homomorphisms mapping $m\mapsto 1\otimes m$ and $P\otimes m\mapsto
 Pm$ respectively. Obviously, we get $\beta\circ\alpha=\mr{id}$. It
 suffices to show
 that $\alpha$ is surjective. Let $\ms{N}_0$ and $\ms{M}_0$ be the
 canonical extensions of ${N}$ and ${M}$ respectively. By Corollary
\ref{ananddag}, we get
 \begin{equation*}
  {M}\cong{E}^\dag_{\Aone,\mb{Q}}\otimes\ms{N}_0,\qquad
   \ms{E}^{\mr{an}}\otimes {M}\cong{E}^\dag_{\Aone,\mb{Q}}
   \otimes\ms{M}_0.
 \end{equation*}
 Using paragraph \ref{LFspacetoppartcase}, we endow these with the
 LF-space topologies (or equivalently the $\ms{D}^{\mr{an}}$-module
 topologies).
 Since these are equipped with $\ms{D}^{\mr{an}}$-module topologies and
 $\alpha$ is $\ms{D}^{\mr{an}}$-linear, $\alpha$ is continuous with
 these topologies (cf.\ \cite[3.7.3/2]{BGR}). To see the
 surjectivity of $\alpha$, it suffices to show that it is a homomorphism
 of ${E}^\dag_{\Aone,\mb{Q}}$-modules. Since 
 $\partial$ is invertible in ${E}^\dag_{\Aone,\mb{Q}}$, we
 get that $\alpha(\partial^{-1}m)=\partial^{-1}\otimes m$. This shows
 that $\alpha$ is linear with respect to the subring $E$ generated by
 $D^\dag_{\Aone,\mb{Q}}$ and $\partial^{-1}$. Note that $E$ is dense in
 ${E}^\dag_{\Aone,\mb{Q}}$. Since the target of $\alpha$ is an
 LF-space, it is separated, and we get the claim.

 Now, since the category $F$-$\mr{Hol}(\ms{S})$ is abelian by
 \cite[Th.~7.1.1]{Cr}, we define the $F$-$\ms{D}^{\mr{an}}$-modules $\mc{K}$,
 $\mc{C}$ by the following exact sequence of
 $F$-$\ms{D}^{\mr{an}}$-modules:
 \begin{equation*}
  0\rightarrow\mc{K}\rightarrow\ms{M}|_{S_0}\rightarrow\mu(\ms{M})
   \rightarrow\mc{C}\rightarrow 0.
 \end{equation*}
 Here the middle homomorphism is induced by the scalar extension of the
 identity map on $\ms{M}$. By the definition of Frobenius structures
 (cf.\ \ref{dfnfrobmu}), this homomorphism is a homomorphism of modules
 with Frobenius structures.
 Take $\ms{E}^{\mr{an}}\otimes_{\ms{D}^{\mr{an}}}$,
 and we get an exact sequence by Remark \ref{locFordfnrem} and the claim
 above:
 \begin{equation*}
  0\rightarrow\ms{E}^{\mr{an}}\otimes\mc{K}\rightarrow\mu(\ms{M})
   \rightarrow\mu(\ms{M})\rightarrow\ms{E}^{\mr{an}}\otimes\mc{C}
   \rightarrow 0.
 \end{equation*}
 This shows that
 $\ms{E}^{\mr{an}}\otimes\mc{K}=\ms{E}^{\mr{an}}\otimes\mc{C}=0$. Let
 ${N}$ be a holonomic $F$-$\ms{D}^{\mr{an}}$-module such that
 $\ms{E}^{\mr{an}}\otimes {N}=0$. Let ${N}^{\mr{can}}$ be the
 canonical extension of $N$, we have
 \begin{equation*}
  \EdagQ{0}\otimes {N}^{\mr{can}}\cong\ms{E}^{\mr{an}}
   \otimes {N}=0.
 \end{equation*}
 This shows that there are no singularities for ${N}$ at $0$, so that
 ${N}$ is a convergent isocrystal around $0$. In particular,
 we get an isomorphism ${N}\cong(\mc{O}^{\mr{an}})^{\oplus n}$ of
 differential modules for some integer $n$
 (note that in the isomorphism, we are forgetting the Frobenius
 structures).  Applying this observation to $\mc{K}$ and $\mc{C}$, we
 get that these are direct sums of trivial modules. Since
 $\Phi(\mc{O}^{\mr{an}})=0$ and the functor $\Phi$ is exact, we finish
 the proof.
\end{proof}

\subsubsection{}
\label{conccalclocfouunr}
Let $\ms{M}$ be a holonomic ($F$-)$\DdagQ{\Pone}(\infty)$-module. We
define $\ms{M}|_{s_\infty}:=\mr{Hom}_{\partial}(\mc{R}_{\ms{S}_\infty},
\ms{M}|_{\eta_\infty})$. When $\ms{M}$ possesses a Frobenius structure,
this is a $K$-vector space with Frobenius structure. There exists a
canonical homomorphism
\begin{equation}
 \label{solisom}
 \mc{R}_{\ms{S}_\infty}\otimes_K\ms{M}|_{s_\infty}\rightarrow
 \ms{M}|_{\eta_\infty}.
\end{equation}
This homomorphism is injective, and compatible with Frobenius structures
if they exist. When this injection is an isomorphism,
$\ms{M}$ is said to be {\em unramified at infinity}. %\index{unramified at infinity} 
This is equivalent to saying that $\ms{M}|_{\eta_\infty}$ is a trivial
differential module.

\begin{lem*}
 Let $\ms{M}$ be a holonomic $F$-$\DdagQ{\Pone}(\infty)$-module
 unramified at infinity. 
Then 
 %\linebreak   $\mu(\nF{\ms{M}})$ is a holonomic
% $F$-$\ms{D}^{\mr{an}}$-module.
 the\linebreak  $F$-$\ms{D}^{\mr{an}}$-module $\mu(\nF{\ms{M}})$ is  holonomic.
\end{lem*}
\begin{proof}
 Since $\mu(\nF{\ms{M}})$ depends only on $\ms{M}|_{\eta_\infty}$
 and the claim does not depend on the Frobenius structure by Lemma
 \ref{characterisation_of_Dan_hol_mod}, we may suppose that $\ms{M}$ is isomorphic to
 $\mc{O}_{\Pone,\mb{Q}}(\infty)^{\oplus n}$ where
 $n=\mr{rk}(\ms{M})$. In this case, we know that $\nF{\ms{M}}$ is
 equal to $i_{0'+}(K^{\oplus n})$ by Proposition
 \ref{calcfouriereasy}. Now,
 $i_{0'+}K|_{\Aoned}\cong\DdagQ{\Aoned}/(x')$, and we may
 check directly that $\ms{E}^{\mr{an}}_{0',\mb{Q}}/
 \ms{E}^{\mr{an}}_{0',\mb{Q}}\cdot x'$ is generated
 over $\ms{D}_{0',\mb{Q}}^{\mr{an}}$ by $\partial'^{-1}$. Thus the claim
 follows using Lemma \ref{characterisation_of_Dan_hol_mod} once again.
\end{proof}

\begin{lem}
 \label{kernelofexact}
 Let $\ms{M}$ be a holonomic $F$-$\DdagQ{\Pone}(\infty)$-module which
 is unramified at infinity. Then we have
 \begin{equation*}
  \Phi(\nF{\ms{M}}|_{S_{0'}})=(K^{\mr{ur}}\otimes_K
   \ms{M}|_{s_\infty})(1).
 \end{equation*}
\end{lem}
\begin{proof}
 By Lemma \ref{conccalclocfouunr}, we get that $\mu(\nF{\ms{M}})$
 is finite
 over $\ms{D}^{\mr{an}}_{0',\mb{Q}}$. Thus by Lemma \ref{reductiontomu},
 it suffices to calculate $\Phi(\mu(\nF{\ms{M}}))$. Since
 $\mu(\nF{\ms{M}})$ depends only on
 $\ms{M}|_{\eta_\infty}$, we may assume that
 $\ms{M}=\ms{M}|_{s_\infty}\otimes_K\mc{O}_{\Pone,\mb{Q}}(\infty)$.
 Indeed, its $\mc{R}$-module around $\infty$ is isomorphic to
 $\ms{M}|_{\eta_\infty}$ as differential $\mc{R}$-modules with Frobenius
 structures using the isomorphism (\ref{solisom}). In
 this case, we may use Proposition \ref{calcfouriereasy} to get that
 \begin{equation*}
  \nF{\ms{M}}\cong i_{0'+}(\ms{M}|_{s_\infty})(1).
 \end{equation*}
 Now, using Proposition \ref{reductiontomu} again, it suffices to
 calculate $\Phi(\nF{\ms{M}}|_{S_{0'}})$, which is nothing but
 what we stated, and concludes the proof.
\end{proof}

\begin{rem}
 We believe that $\mu(\nF{\ms{M}})$ is a holonomic
 $F$-$\ms{D}^{\mr{an}}$-module even when $\ms{M}$ is not unramified at
 infinity, and $\mu(\nF{\ms{M}})$ is closely related to the
 $(\infty,0')$ local Fourier transform of Laumon. However, we do not
 need this generality in this paper, and we do not go into this problem
 further.
\end{rem}

\begin{prop}
 \label{exactseq}
 Let $\ms{M}$ be a holonomic $F$-$\DdagQ{\Pone}(\infty)$-module which is
 an overconvergent isocrystal on a dense open subscheme $U$ of $\mb{A}^1_k$,
 and assume that it is unramified at infinity. We denote
 by $j\colon(\Pone,\mb{P}_k^1\setminus U)\rightarrow(\Pone,\{\infty\})$
 the canonical morphism of couples.
 Then there exists the following exact sequence of Deligne modules:
 \begin{equation}
  0\rightarrow H^1_{\mr{rig},c}(U_{K^{\mr{ur}}},\ms{M})
  \rightarrow\Psi(\nFnp{\ms{M}'}|_{S_{0'}})(-2)\rightarrow
  (K^{\mr{ur}}\otimes_K\ms{M}|_{s_\infty})(-1)\rightarrow
  H^2_{\mr{rig},c}(U_{K^{\mr{ur}}},\ms{M})
  \rightarrow 0.
 \end{equation}
 Here, by abuse of language, we denoted
 $H^i_{\mr{rig},c}(U,\mr{sp}^*(\ms{M}))\otimes K^{\mr{ur}}$ by
 $H^i_{\mr{rig},c}(U_{K^{\mr{ur}}},\ms{M})$, and $\ms{M}':=j_!j^+\ms{M}$
 using the cohomological functors of {\normalfont\ref{defpullbackpush}}.
\end{prop}

\begin{proof}
 By (\ref{vanishexact}), there is the following exact sequence.
\begin{equation*} 
\begin{split}    
0 & \rightarrow\mr{Hom}_{\ms{D}^\dag}(\mb{D}\nFnp{\ms{M}'}|_{S_{0'}},
   \mc{O}_{K^\mr{ur}}^{\mr{an}})\rightarrow\mb{V}(\mb{D}
   \nFnp{\ms{M}'}|_{S_{0'}}) 
   \rightarrow\mb{W}(\mb{D}\nFnp{\ms{M}'}|_{S_{0'}})\rightarrow \\ 
   & \rightarrow\mr{Ext}^1_{\ms{D}^\dag}(\mb{D}   \nFnp{\ms{M}'}|_{S_{0'}},\mc{O}_{K^\mr{ur}}^{\mr{an}})\rightarrow 0.
 \end{split} \end{equation*}
 Consider the following cartesian diagram.
 \begin{equation*}
  \xymatrix{(\Pone,\{\infty\})\ar[d]_{q_0}\ar[r]^<>(.5){\iota_0}
   \ar@{}[rd]|{\square}&(\Poneoned,Z)\ar[d]^{p'}\\
  \{0'\}\ar[r]_<>(.5){i_{0'}}&(\Poned,\{\infty'\})}
 \end{equation*}
 By \cite[Theorem 2.2]{Cr3}, we get
 \begin{equation*}
  \mr{Ext}^i_{\ms{D}^\dag}(\mb{D}\nFnp{\ms{M}'}|_{S_{0'}},
   \mc{O}_{K^\mr{ur}}^{\mr{an}})\cong \ms{H}^{i-1}(i_{0'}^+
   (\mb{D}\mb{D}\nFnp{\ms{M}'}))\otimes
   K^\mr{ur}\cong \ms{H}^{i-1}(i_{0'}^+(\nFnp{\ms{M}'}))\otimes K^{\mr{ur}}
 \end{equation*}
 for any $i$. We get the following calculation:
 \begin{align*}
  &\ms{H}^i(i_{0'}^+(\nFnp{\ms{M}'}))\xrightarrow[\sim]{\ccirc{1}}
  \ms{H}^i(i_{0'}^+(p'_!(\mu^!\ms{L}_{\pi}\otimes p^!\ms{M}')))[1][-3]
  \xrightarrow[\sim]{\ccirc{2}} \ms{H}^i(q_{0!}~\iota_0^+(\mu^!\ms{L}_{\pi}
  \otimes p^!\ms{M}'))[-2]\\
  &\qquad\xrightarrow[\sim]{\ccirc{3}}
  \ms{H}^i(q_{0!}\iota_0^+(\mu^!\ms{L}_{\pi}\widetilde{\otimes}
  p^!\ms{M}'))(-2)[-2]\xrightarrow[\sim]{\ccirc{4}}
  \ms{H}^i(q_{0!}\iota_0^+(\mu^+\ms{L}_{\pi}\widetilde{\otimes}p^+\ms{M}'))[2]\\
  &\qquad\xrightarrow[\sim]{\ccirc{5}}
  \ms{H}^i(q_{0!}(\iota^+_0\mu^+\ms{L}_{\pi}\widetilde{\otimes}
  \iota_0^+p^+\ms{M}'))[-1][2]\xrightarrow[\sim]{\ccirc{6}}
  \ms{H}^i(q_{0!}(\mc{O}_{\Pone,\mb{Q}}(\infty)\widetilde{\otimes}
  \ms{M}'))[1]\\
  &\qquad\xrightarrow[\sim]{\ccirc{7}}\ms{H}^i(q_{0!}(\ms{M}'))(1)[1]
  \xrightarrow[\sim]{\ccirc{8}} H^{i+2}_{\mr{rig},c}(U,\ms{M})(2).
 \end{align*}
 Here $\ccirc{1}$ follows from Huyghe's theorems (\ref{KatzLaumonisom})
 and (\ref{compgeomnaive}), $\ccirc{2}$ from the
 proper base change theorem \ref{properbasech}, $\ccirc{3}$ from
 \ref{prop6func}.\ref{twisttensor}, $\ccirc{4}$ from
 \ref{prop6func}.\ref{smoothpoin} and Lemma
 \ref{Tsuzukicalc}, $\ccirc{5}$ from
 \ref{prop6func}.\ref{raisondetretwist}, $\ccirc{6}$
 from $p\circ\iota_0=\mathrm{id}_{(\Pone,\{\infty\})}$ and $\iota^+_0\mu^+\ms{L}_{\pi}\cong\mc{O}_{\Pone,\mb{Q}}(\infty)$ (by construction of $\mu^+\ms{L}_{\pi}$), $\ccirc{7}$ from
 \ref{prop6func}.\ref{twisttensor} again,
 and $\ccirc{8}$ from (\ref{rigDmodcompact}).
 For the calculation of $\Phi(\nF{\ms{M}'}|_{S_{0'}})$, it suffices
 to apply Lemma \ref{kernelofexact}, and the fact that
 $\ms{M}'|_{s_\infty}\cong\ms{M}|_{s_\infty}$.
\end{proof}

\section{The $p$-adic epsilon factors and product formula}
\label{section7}
In this section, we prove the  product formula for
$p$-adic epsilon factors  
and the determinant formula relating the local
epsilon factor to the local Fourier transform.
We start in \S\ref{sb:Loc_con} by defining the epsilon factors 
for holonomic modules over a formal disk, then we state the main theorem
(the product formula) in \S\ref{sb:main}. Its proof takes the rest of
this paper.
We begin in \S\ref{proof:geo_mod} by proving it for  $F$-isocrystals
with geometrically (globally) finite mo\-no\-dromy: this proof is
slightly more technical than the $\ell$-adic one and we need some
generalities on scalars extension in Tannakian categories, which we
collect in \S\ref{sb:scal_ext}. We finish the proof of the
main theorem in \S\ref{sb:proof_gen_case} where we give also the
determinant formula.

\subsection{Local constants for holonomic $\protect\ms{D}$-modules}
\label{sb:Loc_con}

\subsubsection{}
\label{susubsection:k_finite}
Let us fix some assumptions for this section.
Let $\FF$ \index{.@miscellaneous!F@$\FF$, $\bCt$, $\bFt$|(} be a finite subfield of $k$, $p^h$ \index{.@miscellaneous!hq@$h$, $q$|(} be the number of
elements of $\mb{F}$, so that $\FF$ is the subfield of $k$ fixed by the
$h$-th absolute Frobenius automorphism $\sigma$. Let $\Ctf$ \index{.@miscellaneous!Lambda@$\Ctf$}
be a finite extension of $\Qp$ with residue field $\FF$ and absolute
ramification index $e$. \index{.@miscellaneous!efa@$e,f,a$|(} We put $\Ct:=\Ctf\otimes_{\Wt{\FF}}\Wt{k}$. \index{.@miscellaneous!K@$K$}
It is a complete discrete valuation field with residue field $k$ and
ramification index $e$. We denote by $R$ its ring of integers. We endow
$\Ct$ with the endomorphism $\sigma_{\Ct}=\id[{\Ctf}]\otimes
{\sigma}$. \index{.@miscellaneous!sigma@$\sigma_{\Ct}$}
The subfield of $\Ct$ fixed by $\sigma_{\Ct}$ is $\Ctf$. We
say that $\sigma_{\Ct}$ is  a Frobenius of  $\Ct$ of {\em order}
$h$. Let $\val[K]$ %\index{.@miscellaneous!$\val[K]$}
 denote the valuation of $\Ct$ normalized by
$\val[K](\Ct^*)= \ZZ$.

We assume (except for \S\ref{sb:scal_ext}) that $k$ is finite
with $q=p^f$ elements, \index{.@miscellaneous!hq@$h$, $q$|)} and that the order $h$ of the Frobenius
$\sigma_{\Ct}$ divides $f$, so that $f=ah$ \index{.@miscellaneous!efa@$e,f,a$|)} for an integer $a$. We will
see later (cf.\ Remark \ref{remark:semplification})
that, in the proofs, it is not restrictive to assume $h=f$ and so
$\FF=k$ and $\Ct=\Ctf$.

We choose an algebraic closure $\bCt$ of $\Ct$, 
and we denote by $\bFt$ its residue field. \index{.@miscellaneous!F@$\FF$, $\bCt$, $\bFt$|)}
We choose also a root $\pi$ of the polynomial $X^{p-1}+p$
and we assume $\pi\in\Ctf$. We recall (cf.\ \ref{setupFour}) that
the choice of $\pi$ determines a non-trivial additive character
$\psi_{\pi}\colon \Fp \rightarrow \Ct^*$.
By composing $\psi_{\pi}$ with the trace map $\tr_{k/\Fp}\colon
k\rightarrow \Fp$, we get a non-trivial additive character $k
\rightarrow \Ct^*$ that we denote also by $\psi_{\pi}$. \index{.@miscellaneous!psi@$\psi_{\pi}$}

\subsubsection{}
\label{WD}
To fix notation, we review some results collected in
\cite{Marmora:Facteurs_epsilon}. See {\it loc.\ cit.\ }for more
details. We follow the notation and the assumptions of
\ref{existconcdesc} and \ref{susubsection:k_finite}.
In particular $\ms{S}:=\Spf(A)$ denotes a formal disk, $\eta_{\ms{S}}$
its geometric point, $S=\Spec (A/\varpi A)$\index{.@miscellaneous!S@$S$} its special fiber,
$s=\Spec(k_A)$ the closed point, and $\eta=\Spec(\mc{K})$ the generic
point of $S$. 
We recall (cf. Lemma \ref{existconcdesc}) that 
$\mc{K}\cong k_A\pp{u}$, for any $u$ in $A$ lifting an uniformizer of $A/\varpi A$. 
We put $\bar{s}:=\Spec(\bFt)$, and we choose a geometric
algebraic generic point $\bar{\eta}$. \index{.@miscellaneous!eta@$\bar{\eta}$, $\bar{s}$}
We identify $\pi_1(s,\bar{s})$ with $\widehat{\ZZ}$, by sending any
$n\in\widehat{\ZZ}$ to $\F^n$, where $\F\colon x\mapsto x^{-q}$   is the
geometric Frobenius in $\pi_1(s,\bar{s})$.
Let us denote by $\nu\colon \pi_1(\eta,\bar{\eta})\rightarrow
\widehat{\ZZ}$ the specialization homomorphism and by
$W(\eta,\bar{\eta}):=\nu^{-1}(\ZZ)$ (resp.\ $I_{\eta}:=\ker\nu$) the \index{.@miscellaneous!W@$W(\eta,\bar{\eta})$, $I_{\eta}$}
Weil (resp.\ inertia) subgroup.

Let $\Rep{\Ct\nr}{W\!D(\eta,\bar{\eta})}$ \index{categories!Rep@$\Rep{\Ct\nr}{WD(\eta,\bar{\eta})}$}
 denote the category of
Weil-Deligne representations: {\itshape i.e.\ }the category of finite
dimensional $K\nr$-vector spaces $V$ endowed with an action $\rho\colon
W(\eta,\bar{\eta})\rightarrow \Aut{\Ct\nr}{V}$ and  a nilpotent
endomorphism $N\colon V\rightarrow V$, satisfying $\rho(g)N
\rho^{-1}(g)=q^{\nu(g)}N$, for any $g\in W(\eta,\bar{\eta})$.

On the other hand, $p$-adic mo\-no\-dromy theorem gives an equivalence
of Tannakian categories (nearby cycles)\footnote{In [{\it loc.\ cit.},
(3.2.18)] the functor $\Psi(-1)$ was denoted by $S$,  the field of
constants $\Ct$ by $C$, the category of free differential modules over
$\Rob_{\ms{S}}$ by $\Fisosur[an](\eta|C)$ or $\Phi\!M(\Rob_{\ms{S}})$,
and the field $\mc{K}=k(\eta)$ by $K$.}
\begin{equation*}
 \Psi(-1)\colon F\text{-}\mr{Hol}(\eta_{\ms{S}})\rightarrow
 \Del{\Ct\nr}{\pi_1(\eta,\overline{\eta})}.
\end{equation*}
For the twist $(-1)$, see \ref{defvanishcycle}.
Let $(V,\Fro,N)\in\Del{\Ct\nr}{\pi_1(\eta,\overline{\eta})}$ be a
Deligne module. We can endow $V$ with a linear action $\rho\colon
W(\eta,\bar{\eta}) \rightarrow \Aut{\Ct\nr}{V}$ by putting
$\rho(g)(m):= g(\Fro^{a\nu(g)}(m))$, for all $m\in V$ and $g\in
W(\eta,\bar{\eta})$. In this way we obtain a Weil-Deligne representation
$(V,\rho,N)$ and we denote by
\begin{equation}
 \label{linearizationfunc}
 {L}_k\colon\Del{\Ct\nr}{\pi_1(\eta,\overline{\eta})}\rightarrow
  \Rep{\Ct\nr}{W\!D(\eta,\bar{\eta})}  %\index{functors!.Lk@$L_k$}
\end{equation}
this functor of ``Frobenius linearization''.
Since Weil-Deligne representations are linear, we will often implicitly
extend the scalars from $\Ct\nr$ to  $\bCt$.
Composing the functors $\Psi(-1)$ and ${L}_k$ and extending the scalars
to $\bCt$, we obtain a faithful exact $\otimes$-functor
$\WD\colon F\text{-}\mr{Hol}(\eta_{\ms{S}}) \rightarrow
\Rep{\bCt}{W\!D(\eta,\bar{\eta})}$ (cf.\
\cite[3.4.4]{Marmora:Facteurs_epsilon}). \index{functors!.WD@$\WD$}

\subsubsection{}
\label{defepsfac}
Now, let us introduce local epsilon factors.
Langlands \cite{La} defined local epsilon factors extending Tate's
definition for
rank one case, and in \cite{Del:const_loc}, Deligne simplified the
construction of the epsilon factors for Weil-Deligne representations.
Deligne's definition translates well to free
differential $\Rob_{\ms{S}}$-modules with Frobenius structure
$F\text{-}\mr{Hol}(\eta_{\ms{S}})$, via the functor $\WD$ recalled
above.
Here, we extend it from  $F\text{-}\mr{Hol}(\eta_{\ms{S}})$ to
$F\text{-}\mr{Hol}(\ms{S})$ by {\itshape d\'evissage}.

\bigskip

We follow the notation and the assumptions of \ref{existconcdesc}, \ref{defsFhol},
\ref{defvanishcycle}, \ref{susubsection:k_finite} and \ref{WD}. Let
${M}$ be a holonomic
$F$-$\Dan{\ms{S},\mb{Q}}$-module,  $\omega\in\Omega^{1}_{\mc{K}/k}$ a
non-zero meromorphic $1$-form, and $\mu$ a Haar measure on the additive
group of $\mc{K}$ with values in $\Ct$. We denote by
$\psi_\pi(\omega)\colon \mathcal{K}\rightarrow{\Ct\nr}^*$ the
additive character given by
$\alpha\mapsto\psi_{\pi}\bigl(\Tr_{k/\Fp}(\Res(\alpha\omega))\bigr)$
(cf.\ \cite[Remarque 3.1.3.6]{Lau}). The next proposition allows us to
define the (local) epsilon factor of the triple $({M},\omega,\mu)$.

\begin{prop*}
 There exists a unique map
 \begin{equation*}
  \varepsilon_{\pi}
   \colon({M},\omega,\mu)\mapsto\varepsilon_\pi
   ({M},\omega,\mu)\in\bCt^*, \index{epsilon factors!local!$\varepsilon_\pi
   ({M},\omega,\mu)$, $\varepsilon({M},\omega)$|(}
 \end{equation*}
 satisfying the following properties.
 \begin{enumerate}
  \item For every exact sequence $0\rightarrow {M}'\rightarrow
	{M}\rightarrow {M}''\rightarrow 0$ in
	$F$-$\mr{Hol}(\ms{S})$, we have 
	\begin{equation*}
	 \varepsilon_\pi({M},\omega,\mu)=
	  \varepsilon_\pi({M}',\omega,\mu)\cdot
	  \varepsilon_\pi({M}'',\omega,\mu).  
	\end{equation*}

  \item If ${M}$ is punctual, {\itshape i.e.\
	}${M}=i_+V$ for some $\Fro$-$\Ct$-module $V$,
	then
	\begin{equation*}
	 \varepsilon_\pi({M},\omega,\mu)=\det_{\Ct}(-F,V)^{-1},
	\end{equation*}
	where $F=\Fro^a$ is the smallest linear power of the Frobenius
	$\Fro$ of $V$ (cf.\ see {\normalfont\ref{susubsection:k_finite}}
	for the definition of the integer $a$).

  \item If the canonical homomorphism $j_!\,j^+{M}\rightarrow {M}$
	is an isomorphism, then
	\begin{equation*}
	 \varepsilon_\pi({M},\omega,\mu)=\varepsilon_0
	  (j^+{M}(1),\psi_\pi(\omega),\mu)^{-1},
	\end{equation*}
	where $\varepsilon_0$ is the local epsilon factor defined
	in {\normalfont\cite[3.4.4]{Marmora:Facteurs_epsilon}}.
 \end{enumerate}
\end{prop*}
\begin{proof}
 Immediate by applying the distinguished triangle
 (\ref{anotherloctriag}) and \cite[2.19 (2)]{Marmora:Facteurs_epsilon}.
\end{proof}

In the following, we will always assume $\mu(A/\varpi A)=1$; 
to lighten the notation, we put $\varepsilon({M},\omega):=
\varepsilon_\pi({M},\omega,\mu)$. \index{epsilon factors!local!$\varepsilon_\pi
   ({M},\omega,\mu)$, $\varepsilon({M},\omega)$|)}
 Moreover, for a free differential
$\mc{R}$-module ${M}$ with Frobenius structure, we define
$\emar({M},\omega):=\varepsilon_0({M},\psi_\pi(\omega),\mu)$ and
$\emars({M},\omega):=\varepsilon({M},\psi_\pi(\omega),\mu)$ using
the notation of \cite[3.4.4]{Marmora:Facteurs_epsilon}. \index{epsilon factors!local!$\emars({M},\omega)$,
$\emar({M},\omega)$}
For a complex $\mc{C}$ of $F$-$\Dan{\ms{S},\QQ}$-modules with bounded
holonomic cohomology, we put
\begin{equation*}
 \varepsilon(\mc{C},\omega):=\prod_{i\in\ZZ}\varepsilon
  (\ms{H}^i\mc{C},\omega)^{(-1)^i}.
\end{equation*}

\begin{rem}
 Let ${M}$ be an object of $F$-$\mr{Hol}(\eta)$. We define
 $j_{!+}({M}):=\mr{Im}(j_!{M}\rightarrow j_+{M})$. \index{functors!six operations!j@$j_{"!+}$} Then we have
 \begin{equation*}
  \emars({M},\omega)=\varepsilon_\pi
   (j_{!+}({M}(1)),\omega)^{-1}.
 \end{equation*}
 This follows from Lemma \ref{calcofdiff} and
 \cite[(3.4.5.4)]{Marmora:Facteurs_epsilon}.
For intermediate extensions in a wider context see \cite[\S\S 1.4 and 4.3.12]{AbeCaro}.
\end{rem}

\subsection{Statement of the main result}
\label{sb:main}

\subsubsection{}
Let us begin by fixing notation and definitions of global
objects. We follow the notation and assumptions of
\S\ref{susubsection:k_finite}.
Let $X$ be a (smooth) curve over $k$. We denote by $C$ %\index{.@miscellaneous II!$C$} 
the number of
connected components of $X\otimes_k \bFt$
and by $g$ %\index{.@miscellaneous II! genus $g$} 
the genus of any of them. Let $\eta_X$ be the generic point
of $X$, and we choose a geometric point $\bar{\eta}_X$ over $\eta_X$. \index{.@miscellaneous!eta@$\bar{\eta}_X$, $\eta_X$, $\eta_x$, $\overline{\eta}_x$}
We denote by $|X|$ %\index{.@miscellaneous II!$"|X"|$}
 the set of closed points of $X$.
For any $x\in|X|$, let $\mf{m}_x$ \index{.@miscellaneous!Ox3@$\mc{O}_{X,x}$, $\mf{m}_x$, $k_x$, $i_x$, $\Ct_x$,
$\widehat{\mc{O}}_{X,x}$, $\mathcal{K}_x$} be the maximal ideal of
$\mc{O}_{X,x}$, $k_x$ its residue field,
$i_x\colon\Spec k_x\rightarrow X$ the canonical morphism and
$\Ct_x:=\Ct\otimes_{\Wt{k}} \Wt{k_x}$.
Let $\widehat{\mc{O}}_{X,x}$ be the completion of $\mc{O}_{X,x}$ for
the $\mf{m}_x$-adic topology, $\mathcal{K}_x$ the field of fractions of
$\widehat{\mc{O}}_{X,x}$,
$\eta_x=\Spec(\mathcal{K}_x)$ the generic point of
$S_x:=\Spec\widehat{\mc{O}}_{X,x}$, $\overline{\eta}_x$
(resp.\ $\bar{x}$) a geometric point over $\eta_x$ (resp.\ $x$). \index{.@miscellaneous!S@$S_x$} 
 Let us denote by  $k(X)$ the field of functions of $X$ and %\index{.@miscellaneous II!$k(X)$}
by $\Omega^1_{k(X)/k}$ the module of meromorphic differential $1$-forms 
%\index{.@miscellaneous II!$\Omega^1_{k(X)/k}$, $\val[x](\omega)$}
on $X$. For every non-zero $\omega \in \Omega^1_{k(X)/k}$
and $x\in|X|$, we denote by $\omega_x\in \Omega^{1}_{\mc{K}_x/k_x}$
the germ of $\omega$ at $x$, $\val[x](\omega)$ the order of $\omega$ at
$x$.

\subsubsection{}
Let $\overline{X}$ %\index{.@miscellaneous!$\overline{X}$}
 be the smooth compactification of $X$ over $k$, and
$Z:=\overline{X}\setminus X$. There exists a smooth
formal scheme $\ms{X}$ over $\mr{Spf}(R)$ such that
$\ms{X}\otimes_Rk\cong\overline{X}$ by [SGA I, Exp.\ III, 7.4]. The
category $F^{(h)}$-$D^b_{\mr{hol}}(\DdagQ{\ms{X}}(Z))$ (cf.\
\ref{setupFrob}) does not depend on the
choice of $\ms{X}$ up to canonical equivalence of categories using
\cite[2.2.1]{Ber:rigD}. We denote this category by
$F^{(h)}$-$D^b_{\mr{hol}}(X)$ \index{categories!FhDbh@$F^{(h)}$-$D^b_{\mr{hol}}(X)$} and call it the category of {\em bounded
holonomic $F^{(h)}$-$\DdagQ{X}$-complexes}.

Now let $f\colon X\rightarrow\mr{Spec}(k)$ and $\widetilde{f}\colon
(\ms{X},Z)\rightarrow(\mr{Spf}(R),\emptyset)$ be the
structural morphisms. The functors $\widetilde{f}_+$ and 
$\mb{D}_{\ms{X},Z}$ (cf.\ \ref{defpullbackpush}) do not depend on the
choice of $\ms{X}$ up to equivalences. By
abuse of language, these are denoted by $f_+$ and \index{functors!six operations!D@$\mb{D}_X$}\index{functors!six operations!f@$f_+$, $f_"!$} 
$\mb{D}_X$ respectively. We note that $f_!$ can also be used, and in the
same way, we can consider the functor $j_!\colon F^{(h)}\mbox{-}
D^b_{\mr{hol}}(U)\rightarrow F^{(h)}\mbox{-}D^b_{\mr{hol}}(X)$ for an
open immersion $j\colon U\hookrightarrow X$. \index{functors!six operations!j@$j_"!$}

Let $R'$ be a discrete valuation ring finite \'{e}tale over $R$, and let
$\ms{C}$ be an object of $F^{(h)}$-$D^b_{\mr{hol}}(\mr{Spf}(R'))$. The
associated $\sigma_K$-semi-linear automorphism
$\ms{C}\xrightarrow{\sim}\ms{C}$ (cf. \ref{frobstrcor}) is denoted by
$\varphi_{\ms{C}}$. From now on we denote the {\em $K$-linear}
automorphism $\varphi_{\ms{C}}^a$ by $F$ (cf.\
\ref{susubsection:k_finite} for the definition of $a$).
\index{.@miscellaneous!F@$F$, $\varphi$}

\subsubsection{}
In \cite{Caro:courbes}, Caro defines the $L$-function of a complex
$\ms{C}$ in $F^{(h)}$-$D^b_{\mr{hol}}(X)$. Let us recall the
definition. We set\footnote{
The definition of the $L$-function is slightly different from that of
Caro. We have chosen a different  convention in order that $L(X,f^+K,t)$ coincides
with the $L$-function of $X$.}
\begin{align*}
 L(X,\ms{C},t)&:=\!\prod_{x\in |X|}\det_{\Ct_x}(1-t^{\deg(x)}
  F^{\deg(x)};~i_x^+\ms{C})^{-1}\\
 &:=\!\prod_{x\in |X|}\prod_{r\in\ZZ}
  \det_{\Ct_x}(1-t^{\deg(x)}F^{\deg(x)};~\ms{H}^r(i_x^+\ms{C}))^{(-1)^{r+1}}
 \!\!. \index{L@$L$-function $L(X,\ms{C},t)$}
\end{align*} 
Recall that $F^{\deg(x)}:=\Fro^{\deg(x)\,a}$ is the smallest linear power
of the Frobenius.  
 Using a result of Etesse-LeStum, Caro gave the
following cohomological interpretation of his $L$-function (cf.\
\cite[3.4.1]{Caro:courbes}):
\begin{equation*}
 L(X,\ms{C},t) = \prod_{r\in\ZZ} \det_{\Ct}(1-tF;\ms{H}^r(f_!\ms{C}))
  ^{(-1)^{r+1}}.
\end{equation*}
For careful readers, we remind that in {\it loc.\ cit.}, the definition
of Frobenius structure of push-forward is re-defined so that it is
compatible with adjoints (cf.\ [{\it loc.\ cit.}, 1.2.11]). However,
this coincides with the usual definition (see \cite[Remark
3.12]{Abe3}).

Now, assume that $f$ is proper. Then $f_!$ can be replaced by $f_+$.
By Poincar\'{e} duality (\ref{poincaredual}),
we get the following functional equation:
\begin{equation*}
 L(X,\ms{C},t)=\varepsilon(\ms{C})\cdot t^{-\chi(f_!\ms{C})} 
  \cdot L(X,\mathbb{D}_X(\ms{C}),t^{-1}),
\end{equation*}
where
\begin{equation*}
 \varepsilon(\ms{C}):= \det(-F;f_+\ms{C})^{-1}=
  \prod_{r\in\mb{Z}}\det(-F;\ms{H}^{r}(f_+\ms{C}))^{(-1)^{r+1}} \index{epsilon factors!global!$\varepsilon(\ms{C})$}
\end{equation*}
and $F$ is the smallest linear power of the Frobenius. This invariant is
called the {\em (global) epsilon factor of $\ms{C}$}.
Finally, for $\ms{C}$ in $F$-$D^b_{\mr{hol}}(X)$, we put
$r(\ms{C}):=\sum_{i\in\mb{Z}}(-1)^i r(\ms{H}^{i}\ms{C})$ using the
notation of \ref{geomnotlau}. We call it the {\it generic rank of
$\ms{C}$}. \index{local constants!$r(\ms{C})$}

\begin{rem}
 \label{ocisoccaseinv}
 Let $U$ be a non-empty open subscheme of a proper curve $X$, and $M$ be
 an overconvergent $F$-isocrystal on $U$ over $\Ct$. Etesse-LeStum
 \cite{LE} defined the $L$-function for $M$ by
 \begin{equation*}
  L_{\mr{EL}}(U,M,t):=\prod_{x\in |U|}
   \det_{\Ct_x}(1-t^{\deg(x)}F^{\deg(x)};i_x^*M )^{-1}.
 \end{equation*}
 We are able to interpret this global invariant in terms of the global
 invariant we have just defined in the following way. Let $j\colon
 U\hookrightarrow X$ be the open immersion and $f_U\colon U\rightarrow
 \mr{Spec}(k)$ be the structural morphism. Put
 $\ms{C}:=j_!(\mr{sp}_*M)[-1](-1)$. Then the $L$-function coincides
 with that given by Etesse-LeStum; namely, we get
 $L(X,\ms{C},t)=L_{\mr{EL}}(U,M,t)$. The easiest way to
 see this might be to use the cohomological interpretation of the two
 $L$-functions, and the fact that
 \begin{equation*}
  \ms{H}^i(f_+\ms{C})\cong \ms{H}^i(f_{U!}(\mr{sp}_*M[-1](-1)))\cong
   H^i_{\mr{rig},c}(U,M),
 \end{equation*}
 where the second isomorphism holds by (\ref{rigDmodcompact}). See also
 \cite[Remark 3.12]{Abe3} for some account.
\end{rem}

\begin{thm}[Product formula]
 \label{Product formula}
 Let $X$ be a proper (smooth) curve over $k$, $\ms{C}$ a complex in
 $F^{(h)}$-$D^b_{\mr{hol}}(X)$ and $\omega\in\Omega^1_{k(X)/k}$ a
 non-zero meromorphic form on $X$. We have the following relation
 between the global and local factors:
 \begin{equation}
 \label{PF}\tag{PF}
  \varepsilon(\ms{C})=q^{C(1-g)r(\ms{C})}\prod_{x\in|X|}
   \varepsilon(\ms{C}|_{S_x},\omega_{x}),
 \end{equation}
 where $q$ denotes the number of elements of $k$,  $r(\ms{C})$ is the
 generic rank of $\ms{C}$, $C$ denotes the number of geometrically
 connected components of $X$, and $g$ is the genus of any of them. We
 recall that $\ms{C}|_{S_x}$ denotes the complex of  modules
 defined by restriction (cf.\ {\normalfont\ref{preisouadsf}}) from $X$
 to its complete {\it trait} $S_x$ at $x$.
\end{thm}

The proof of the product formula will be given in
\S\ref{sb:proof_gen_case}.

\begin{cor}[{\cite[Conjecture 4.3.5]{Marmora:Facteurs_epsilon}}]
 Let $U$ be an non-empty open subscheme of a proper curve $X$, $M$ an
 overconvergent $F$-isocrystal on $U$ over $\Ct$, and
 $\omega\in\Omega^1_{k(X)/k}$  a non-zero meromorphic form on $X$. For any $x\in |X|$, let
 us denote by $M|_{\eta_x}$ the free differential
 $\mc{R}_{K_x}$-module associated to $M$ at $\eta_x$. We have
 \begin{equation}
  \label{PF-iso}\tag{PF*}
  \prod_{i=0}^{2}\det_{\Ct}(-F;H^i_{\rig,c}(U,M))^{{(-1)}^{i+1}} \!\!= 
   q^{C(1-g)\mathrm{rk}(M)} \!\!
   \prod_{x\in |U|} {q}_x^{\val[x](\omega)\mathrm{rk}(M)}
   {{\det_{M}}(x)}^{\val[x](\omega)}\!\!\prod_{x\in X\backslash U}
   \emar({M|}_{\eta_x},\omega_x),
 \end{equation}
 where $F:=\Fro^a$, $q_x:=q^{\deg(x)}$, 
 ${\det_{M}}(x):=\det_{\Ct_x}(\Fro[\eta_x]^{a\deg(x)}
 ;i_x^*{M})$ and $\val[x](\omega)$ denotes the
 order of $\omega$ at $x$.
\end{cor}
\begin{proof}
 Let us prove that (\ref{PF}) implies (\ref{PF-iso}):
 we need only to specialize all the factors. Let $j\colon
 U\hookrightarrow X$ be the open immersion. We replace
 $\ms{C}:=j_!(\mathrm{sp}_*M)[-1](-1)$ in (\ref{PF}). Then
 for $x\in |U|$, we get
 \begin{equation*}
 \begin{split}
  \varepsilon(\ms{C}|_{S_x},\omega_x)=
  \emar({M|}_{\eta_x},\omega_x)\cdot\det_{\Ct_x}(-\Fro^
  {a\deg(x)};i^*_xM)^{-1}
  & = \emars({M|}_{\eta_x},\omega_x) \\
  & ={q}_x^{\val[x](\omega)\mathrm{rk}(M)}{{\det_{M}}(x)}^{\val[x](\omega)},
 \end{split}
 \end{equation*}
 where the first equality follows from the localization triangle
 (\ref{anotherloctriag}) and \ref{prop6func}.\ref{smoothpoin},
 and the second (resp.\ third) from
 \cite[(3.4.5.4)]{Marmora:Facteurs_epsilon} (resp.\
 \cite[(2.19-2)]{Marmora:Facteurs_epsilon}).
 Finally, for every $x\in X\backslash U$, by definition,
 we have $\varepsilon(\ms{C}|_{S_x},\omega_x)=
 \emar({M|}_{\eta_x},\omega_x)$. Considering Remark
 \ref{ocisoccaseinv}, we get the corollary.
\end{proof}

\begin{rem} 
 (i) Note that in {\it loc.\ cit.}, the curve $X$ was  assumed to be
 geometrically connected for simplicity, so that $C=1$. Since the
 product formula (\ref{PF}) is immediate for punctual  arithmetic
 $\ms{D}$-modules, the two statements (\ref{PF-iso}) and (\ref{PF}) are
 equivalent by {\itshape d\'evissage}.

 (ii) \label{remark:semplification}
 By definition, the global factor $\varepsilon(\ms{C})$
 appearing in (\ref{PF}) and (\ref{PF-iso}) does not change if we
 replace the Frobenius $\Fro[\ms{C}]$ of $\ms{C}$ by its smallest linear
 power $\Fro[\ms{C}]^a$. The same is true for the Weil-Deligne
 representation $\WD(\ms{C}|_{\eta_x})$ and {\it a fortiori} for the
 local factors. Replacing $\Fro[\ms{C}]$ by $\Fro[\ms{C}]^a$  is
 equivalent to assuming $h=f$, thus $\FF=k$
 and $\Ct=\Ctf$ in \ref{susubsection:k_finite}.
\end{rem}

\subsection{Interlude on scalar extensions}
\label{sb:scal_ext}
In this subsection, we recall a formal  way to extend scalars (cf.\
\cite[p.155]{DeligneMilne}) in a general Tannakian category, with
the intention of use in \S\ref{proof:geo_mod}.
Most of the proofs are formal and they are only sketched.
To shorten the exposition, we may implicitly assume that the objects of
our categories have elements, so that they are $\Lambda$-vector spaces
endowed with some extra structures.
For a general treatment, confer to {\it loc.\ cit}. In this subsection
the field $k$ is only assumed to be perfect, whereas in the rest of
\S\ref{section7} it is finite.

Let $\mathcal{A}$ be a Tannakian category over $\Ctf$. Assume that the
objects of $\mathcal{A}$ have finite length.
%Let us denote by $\one$ the unit object of $\mc{A}$.

\begin{dfn}
 Let $\Ctf'$ be a field extension of $\Ctf$ and $M$ an object of
 $\mc{A}$. A \emph{$\Ctf'$-structure} on $M$ is a homomorphism \index{Lambda@$\Ctf'$-structure}
 of $\Ctf$-algebras  $\lambda_M\colon\Ctf'\rightarrow \End{\mc{A}}{M}$.

 Let $\lambda_M$ (resp.\ $\lambda_N$) be a $\Ctf'$-structure on $M$
 (resp.\ $N$). A morphism $f\colon M\rightarrow N$ in $\mc{A}$ is said to
 be \emph{compatible} with the $\Ctf'$-structures if for every $\alpha$
 in $\Ctf'$ we have $\lambda_N(\alpha)f = f \lambda_M(\alpha)$.
 The couples $(M,\lambda_M)$, where $\lambda_M$ is a $\Ctf'$-structure
 on an object $M$ in $\mc{A}$, form a category, whose morphisms are the
 morphisms in $\mc{A}$ compatible with the $\Ctf'$-structures. We denote
 this category by $\mc{A}_{\Ctf'}$. Sometimes, we denote simply by $M$
 an object $(M,\lambda_M)$ in $\mc{A}_{\Ctf'}$.
\end{dfn}

\subsubsection{}
\label{sub:tens_prod} 
We define an internal tensor product in $\mc{A}_{\Ctf'}$ as follows.
Let $(M_1,\lambda_1)$ and $(M_2, \lambda_2)$ be two objects in
$\mc{A}_{\Ctf'}$. Since $M_1\otimes M_2$ has finite length,
there exists a smallest sub-object $\iota\colon I \hookrightarrow
M_1\otimes M_2$ such that, for all $a$ in $\Ctf'$, the image of
$\lambda_1(a)\otimes \id[M_2]- \id[M_1]\otimes \lambda_2(a)$
factors through $\iota$.
We put $M_1\otimes' M_2=\Coker(\iota)$. There are two natural
$\Ctf'$-structures on $M_1\otimes M_2$, given respectively by the
endomorphisms $\lambda_1(a)\otimes \id[M_2]$ and $\id[M_1]\otimes
\lambda_2(a)$. By construction they induce the same  $\Ctf'$-structure
$\lambda_{M_1\otimes' M_2}$ on $M_1\otimes' M_2$. 
The couple $(M_1\otimes' M_2, \lambda_{M_1\otimes' M_2})$ defines
an object of $\mc{A}_{\Ctf'}$ denoted by $(M_1,\lambda_1)\otimes(M_2,
\lambda_2)$. It satisfies the usual universal property of the tensor
product in the category $\mc{A}_{\Ctf'}$, which makes $\mc{A}_{\Ctf'}$
a Tannakian category.

\begin{ex}
\label{example:rep'}
 Let $U$ be a non-empty open subscheme of a proper curve $X$,
 $\overline{\eta}$ a geometric point of $X$. The category
 $\mathrm{Rep}_{\Ctf}^{\mathrm{fg}}(\pi_1(U,\overline{\eta}))_{\Ctf'}$
 of representations with local finite geometric mo\-no\-dromy is
 equivalent, and even isomorphic, to
 $\mathrm{Rep}_{\Ctf'}^{\mathrm{fg}}(\pi_1(U,\overline{\eta}))$ as
 Tannakian category.

 Let $\Ct$ be a field as in \ref{susubsection:k_finite}, and
 $\Ctf'/\Ct$ be a finite Galois extension. By construction, the category
 $\mathrm{Isoc^{\dagger}}(U,X/\Ct)_{\Ctf'}$ of overconvergent
 isocrystals with ${\Ctf'}$-structure is equivalent, as Tannakian
 category, to $\mathrm{Isoc^{\dagger}}(U,X/ \Ctf'\otimes_{\Ctf} \Ct)$.

 Now assume $k$ to be a finite field with $q=p^f$ elements and that the
 order of the Frobenius is $f$, so that $\sigma_{\Ct}= \id[\Ct]$.
 If $\Ctf'/\Ct$ is totally ramified, then $\Fisosur(U,X/\Ct)_{\Ctf'}$ is
 equivalent to $\Fisosur(U,X/\Ctf')$ as Tannakian category, where
 $\sigma_{\Ctf'}:= \id[\Ctf']$. If $\Ctf'$ is not totally ramified, an
 overconvergent $F$-isocrystal with $\Ctf'$-structure $M'$ on $U$ over
 $\Ct$ can be identified with an overconvergent isocrystal on $U/\Ctf'$,
 endowed with a ``Frobenius'' $\Fro[M']$ which is $\Ctf'$-linear although only of order $f$.
\end{ex}

\subsubsection{}
\label{scalar_ext}
Let $V$ be a finite dimensional $\Ctf$-vector space and $M$ an object of
$\mc{A}$. The tensor product $V\otimes_{\Ctf} M$ is defined canonically
in \cite[p.156 and p.131]{DeligneMilne} as an essentially constant
ind-object. In particular, if $\Ctf'/\Ctf$ is a finite field extension,
the product $\Ctf'\otimes_{\Ctf} M$ can be endowed with the
$\Ctf'$-structure induced by the multiplication of $\Ctf'$, so it
belongs to $\mc{A}_{\Ctf'}$: we have a functor of {\em extension of
scalars} $\Ctf'\otimes_{\Ctf}-\colon\mc{A}\rightarrow \mc{A}_{\Ctf'}$.
If $a_1,\ldots, a_n$ is a base of $\Ctf'$ over $\Ctf$, then
$\Ctf'\otimes_{\Ctf} M$ is non-canonically isomorphic to
$\oplus_{i=1}^{n} a_i\Ctf \otimes_{\Ctf} M$, with an obvious meaning of
the latter.

\subsubsection{}
\label{ss:Ctf'_ext}
Let $S\colon\mc{A}\rightarrow \mc{B}$ be an additive functor between
two categories $\mc{A}$ and $\mc{B}$  as in
\S\ref{sb:scal_ext}. It extends  to a functor
$S_{\Ctf'}\colon\mc{A}_{\Ctf'}\rightarrow \mc{B}_{\Ctf'}$ defined by
functoriality as
\begin{equation*}
 S_{\Ctf'}(M,\lambda) := (S(M), S(\lambda)\colon a\mapsto S(\lambda(a)) ). 
\end{equation*}
By the additivity of $S$, it is clear that $S_{\Ctf'}$ commutes with the
extensions of scalars ${\Ctf'}\otimes_{\Ctf} -$. Moreover, if $S$ is
compatible with $\otimes$ and  right exact, then  $S_{\Ctf'}$ is
compatible with the inner  product $\otimes'$ defined in
\ref{sub:tens_prod}: this is a consequence of the construction of
$\otimes'$, considered that $\Ctf'/\Ctf$ is a finite extension and $S$
commutes with finite direct limits and $\otimes$. Finally, if $S$ is an
equivalence of Tannakian categories so is $S_{\Ctf'}$.

\subsection{Proof for finite geometric monodromy}
\label{proof:geo_mod}
The goal of this subsection is to prove the product formula in the case
of overconvergent $F$-isocrystals with geometrically finite
mo\-no\-dromy, in particular for $F$-isocrystals which are canonical
extensions: see Proposition \ref{PF:fgm} and Corollary
\ref{cor:cong_Fiso_spec}.
Although the proof for  overconvergent $F$-isocrystals with finite
mo\-no\-dromy is analogous to that of \cite[3.2.1.7]{Lau} and it is
given in \cite[4.3.15]{Marmora:Facteurs_epsilon}, there are some
technical difficulties to show that for {\em geometrically} finite case,
and we treat this case by using the formal scalar extension reviewed in
the previous subsection.\\

In \ref{subsubsec:loc_const_Ctf'} and \ref{ss:glob_const_Ctf'}, we
define local  and global ``constants'' for generalized isocrystals
in $\Fisosur(U/\Ct)_{\Ctf'}$. They will be seen as elements of the
$\bCt$-algebra of functions $\Spec{(\Ctf'\otimes_{\Ctf} \bCt)}
\rightarrow \bCt$. To avoid any confusion with constant functions, we
employ the term \emph{factors} instead of ``constants''. For simplicity,
we will often assume that the order of Frobenius is $f$, so that
$\Ctf=\Ct$ and $\sigma_{\Ct}=\id[\Ct]$, cf.\
\ref{susubsection:k_finite}; we might state the definitions and prove
the lemmas in the general case, but this is not needed for proving
Proposition \ref{PF:fgm}.

\subsubsection{}
\label{subsubsec:loc_const_Ctf'}
Assume $h=f$, and so  $\Ctf=\Ct$ (cf.\ \ref{susubsection:k_finite}). Let
$\Ctf'/\Ct$ be a finite Galois
extension, $U$ be a non-empty open subscheme of a proper curve $X$, and
$(M,\lambda)\in\Fisosur(U,X/\Ct)_{\Ctf'}$. The Weil-Deligne
representation $\WD(M|_{\eta_x})_{\Ctf'}$ is a $(\Ctf'\otimes_{\Ct}
\bCt)$-module with a linear action of  $\rho_{\eta_x}$ and
$N_{\eta_x}$. For any $\mathfrak{p}\in \Spec{(\Ctf'\otimes_{\Ct}\bCt)}$,
we denote by ${(\WD(M|_{\eta_x})_{\Ctf'})}_{\mathfrak{p}}$ the
localization of $\WD(M|_{\eta_x})_{\Ctf'}$ at ${\mathfrak{p}}$; it is
stable under $\rho_{\eta_x}$ and $N_{\eta_x}$. Let us define the
\emph{local factors} as functions
$\Spec{(\Ctf'\otimes_{\Ct}\bCt)}\rightarrow \bCt$.
For any $x\in |X|$, we set:
\begin{enumerate}
 \item $\mathrm{rk}(M,\lambda)\colon\mathfrak{p}\mapsto
       \dim_{\bCt}({(\WD(M|_{\eta_x})_{\Ctf'})}_{\mathfrak{p}})$, which
       does not depend on $x$.

 \item $\det_{(M,\lambda)}(x)  \colon\mathfrak{p}\mapsto \det_{\bCt }(
       \rho_{\eta_x}(\F_x);({(\Ker N_{\eta_x})}^{I_{\eta_x}})
       _{\mathfrak{p}})=\det_{\bCt}( \Fro[\eta_x]^{a\deg (x)};
       (\bCt\otimes_{\Ct_x} M^{\nabla_x}_{\eta_x})_{\mathfrak{p}})$\\
       (cf.\ \cite[(3.4.5.3)]{Marmora:Facteurs_epsilon} for the
       equality). %\index{local factors!$\mathrm{rk}(M,\lambda)$, $\det_{(M,\lambda)}(x)$}

 \item Let $\omega\not = 0$ be in $\Omega^1_{k(X)/k}$, and $\mu_x$ be
       the Haar measure on $\mathcal{K}_x$ with values in
       $\bCt$ normalized by $\mu_x(\widehat{\mathcal{O}}_{X,x})=1$ as
       usual, and $\psi(\omega_x)\colon \mathcal{K}_x\rightarrow
       \bCt^{*}$ is also the additive character associated to $\omega_x$
       (cf.\ \ref{defepsfac}). As already
       appeared in \cite[6.4]{Del:const_loc}, the epsilon factors
       $\emar({(M,\lambda)|}_{\eta_x},\psi(\omega_x),\mu_x)$ are
       defined by
       \begin{equation*}
        \mathfrak{p}\mapsto\emar(
	 {(\WD(M|_{\eta_x})_{\Ctf'})}_{\mathfrak{p}},\psi(\omega_x),\mu_x).
       \end{equation*}
       Let us write simply $\emar({(M,\lambda)|}_{\eta_x},\omega_x)$,
       instead of
       $\emar({(M,\lambda)|}_{\eta_x},\psi(\omega_x),\mu_x)$.
\end{enumerate}

\subsubsection{}
\label{ss:glob_const_Ctf'}
As in \ref{subsubsec:loc_const_Ctf'}, assume $h=f$, and so
$\Ctf=\Ct$. Let $(M,\lambda)$ be in $\Fisosur(U,X/\Ct)_{\Ctf'}$.
The rigid cohomology groups (with and without supports) of $M$ inherit a
$\Ctf'$-structure,  so that they are $F$-isocrystals with
$\Ctf'$-structure on $\Spec(k)$ over $\Ct$; we denote them respectively
by $H_{\mathrm{rig}}^i(U,M)_{\Ctf'}$ and
$H_{\mathrm{rig},c}^i(U,M)_{\Ctf'}$ (cf.\ \ref{ss:Ctf'_ext}).  \index{functors!.Hrig@$H_{\mathrm{rig}}^i(U,-)_{\Ctf'}$,
$H_{\mathrm{rig},c}^i(U,-)_{\Ctf'}$}
They are
$\Ctf'$-vector spaces endowed with linear Frobenius isomorphisms
$\Fro$. To shorten the notation, let $H$ denote
$H_{\mathrm{rig},c}^i(U,M)_{\Ctf'}$ or $H_{\mathrm{rig}}^i(U,M)_{\Ctf'}$.

We define $\det(-F; H)$ as the constant function
$\Spec{(\Ctf'\otimes_{\Ct}\bCt)}\rightarrow \bCt$, $\mathfrak{p}\mapsto
\det_{\Ctf'}(-F;H)$. We denote by $\det(-F^*;
H_{\mathrm{rig},c}^*(U,M)_{\Ctf'})$ the product $\prod_{i=0}^{2}\det(-F;
H_{\mathrm{rig},c}^i(U,M)_{\Ctf'})$. 
%\index{local factors!$\det(-F^*;H_{\mathrm{rig},c}^*(U,M)_{\Ctf'})$}
From the long exact sequence of
rigid cohomology, it follows that $\det(-F^*;
H_{\mathrm{rig},c}^*(U,M)_{\Ctf'})$ is multiplicative for short exact
sequences and so it is defined on the Grothendieck group of
$\Fisosur(U/\Ct)_{\Ctf'}$.

\begin{rem*}
 In the general case, where $h|f$, the group $H$ is a module over the
 semi-local ring $\Ctf'\otimes_{\Ctf} \Ct$ and we endow it with the
 $(\Ctf'\otimes_{\Ctf} \Ct)$-linear endomorphism $F=\Fro^a$, where
 $a=hf^{-1}$. The module $H$ decompose as $\oplus_{\mathfrak{p}\in\Spec
 (\Ctf'\otimes_{\Ctf} K)} H_{\mathfrak{p}}$. For every
 $\mathfrak{p}\in\Spec (\Ctf'\otimes_{\Ctf} K)$, the localization
 $H_{\mathfrak{p}}$ is a vector space over the field
 $\Ct':=\Ctf'\otimes_{\Ctf''} K$ (where $\Ctf'':=\Ctf'\cap\Ct\subset
 \bCt$ is a finite unramified extension of $\Ctf$), and $F$ induces a
 $\Ct'$-linear  endomorphism of $H_{\mathfrak{p}}$. We define $\det(-F;
 H)$ as the composition of the canonical map
 $\Spec{(\Ctf'\otimes_{\Ctf} \bCt)}\rightarrow
 \Spec{(\Ctf'\otimes_{\Ctf} K)}$, induced by the inclusion $K\subset
 \bCt$, and the function $\Spec{(\Ctf'\otimes_{\Ctf} K)}\rightarrow\bCt$
 sending $\mathfrak{p}$ to $\det_{\Ct'}(-F; H_{\mathfrak{p}})$.
 We are not using this remark in the following.
\end{rem*}

\subsubsection{}
Let us follow the notation of \ref{subsubsec:loc_const_Ctf'} and
\ref{ss:glob_const_Ctf'}. Let us state a variant of the  product
formula (\ref{PF-iso}) for overconvergent $F$-isocrystals in
$\Fisosur(U,X/\Ct)_{\Ctf'}$.
\label{conj:gen_PF_scalar_ext} 
Let $U$ be a non-empty open subscheme of $X$, $M'=(M,\lambda)$ be in
$\Fisosur(U,X/\Ct)_{\Ctf'}$ and $\omega$ a non-zero element of
$\Omega^1_{k(X)/k}$. The product formula for $M'$ is the following
relation
\begin{equation}
 \label{eq:gen_PF_scalar_ext}
  \det\bigl(-F^*;H^*_{\rig,c}(U,M')_{\Ctf'}\bigr)^{-1}=
  q^{C(1-g)\mathrm{rk}(M')} 
  \prod_{x\in |U|} {q}_x^{\val[x](\omega)\mathrm{rk}(M')}
  {{\det_{M'}}(x)}^{\val[x](\omega)}\prod_{x\in X\backslash U}
  \emar({M}'_{\eta_x},\omega_x)
\end{equation}
between global and local factors associated to $M'$.

\subsubsection{}
\label{lemma:PF_scalar_invariance}
Assume $h=f$ and so $\Ctf=\Ct$. Let $U$ be a non-empty open subscheme of
$X$, and $M$ an overconvergent $F$-isocrystal on $U$ over $\Ct$. For any
finite Galois extension $\Ctf'/\Ct$, we can define an overconvergent
$F$-isocrystal $\Ctf'\otimes_{\Ct} M$ with $\Ctf'$-structure
(cf.\ \ref{scalar_ext}).

\begin{lem*}
 The $F$-isocrystal $M$ satisfies the product formula {\normalfont
 (\ref{PF-iso})} if and only if $\Ctf'\otimes_{\Ct} M$ satisfies the
 product formula with $\Ctf'$-structure {\normalfont
 (\ref{eq:gen_PF_scalar_ext})}.
\end{lem*}
\begin{proof}
 For an abelian category $\mc{A}$, we denote by $\Gr{\mc{A}}$ its
 Grothendieck group. In this proof we put
 $\mc{A}=\Fisosur(U,X/\Ct)$. The formula (\ref{PF-iso}) (resp.\
 (\ref{eq:gen_PF_scalar_ext})) is a relation on the Grothendieck group
 of $\mc{A}$ (resp.\ $\mc{A}_{\Ctf'}$) with values in $\bCt^*$ (resp.\
 in the group of units of the $\bCt$-algebra
 ${\bCt}^{\,\Spec(\Ctf'\otimes_{\Ct} \bCt)}$). Each factors $\nu$
 appearing in the equality (\ref{PF-iso}) (resp.\
 (\ref{eq:gen_PF_scalar_ext})) are homomorphisms
 $\nu\colon\Gr{\mc{A}}\rightarrow \bCt^*$ (resp.\
 $\nu_{\Ctf'}\colon\Gr{\mc{A}_{\Ctf'}}\rightarrow
 \bigl({\bCt}^{\,\Spec(\Ctf'\otimes_{\Ct} \bCt)}\bigr)^* $). By the
 definitions of these factors, cf.\
 (\ref{subsubsec:loc_const_Ctf'}--\ref{ss:glob_const_Ctf'}), it follows
 the commutativity of the diagram 
\begin{equation*}
 \xymatrix{
  \Gr{\mc{A}}\ar[r]^-{\nu}\ar[d]^-{\Ctf'\otimes_{\Ct} -} & \bCt^*\ar[d] \\  
 \Gr{\mc{A}_{\Ctf'}}\ar[r]^-{\nu_{\Ctf'}} &
  \bigl({\bCt}^{\,\Spec(\Ctf'\otimes_{\Ct} \bCt)}\bigr)^*,}
\end{equation*}
 where the right vertical homomorphism  maps each element $c$ of
 $\bCt^*$ to the constant function $\Spec(\Ctf'\otimes_{\Ct} \bCt)
 \rightarrow \bCt$ of value $c$. Since this homomorphism is injective we
 conclude.
\end{proof}

\subsubsection{}
\label{Tsuzuki-Katz}
We recall that the theorem of Tsuzuki \cite[(7.2.2), Theorem
7.2.3]{Tsuzuki:FLM} gives an equivalence
$G^{\dagger}\colon\Fisosur(U,X/\Ct)\ur\rightarrow
\mathrm{Rep}_{\Ctf}^{\mathrm{fg}}(\pi_1(U,\overline{\eta}_X))$, \index{functors!.G@$G^{\dagger}$}
\index{categories!Rep@$\mathrm{Rep}_{\Ctf}^{\mathrm{fg}}(\pi_1(U,\overline{\eta}_X))$, $\Fisosur(U,X/\Ct)\ur$} 
between
the categories of unit-root overconvergent $F$-isocrystals on $U$ over
$\Ct$ and continuous $\Ctf$-representations of
$\pi_1(U,\overline{\eta}_X)$ with local geometrically finite
mo\-no\-dromy. We say that a unit-root
overconvergent $F$-isocrystal $M\in\Fisosur(U,X/\Ct)\ur$ has
\emph{global finite mo\-no\-dromy} if the associated representation
$G^{\dagger}(M)$ factors through a finite quotient of
$\pi_1(U,\overline{\eta}_X)$; we say that $M$ has \emph{global
geometrically finite mo\-no\-dromy} if the restriction of
$G^{\dagger}(M)$ to $\pi_1(U \otimes_k \bFt,\overline{\eta}_X)$ factors
through a finite quotient. \index{global (geometrically) finite monodromy}

The following lemma extends \cite[Theorem
4.3.15]{Marmora:Facteurs_epsilon} to isocrystals with
$\Ctf'$-structure.

\begin{lem}
 \label{lemma:extention_to_lambda_str}
 Assume $h=f$, so that $\Ctf=\Ct$ and $\sigma_{\Ct}=\id[\Ct]$. Let
 $\Ctf'/\Ct$ be a finite Galois extension and $U$ be a non-empty open
 subscheme of $X$. Let $M'=(M,\lambda)$ be an overconvergent
 $F$-isocrystal with $\Ctf'$-structure on $U$ over $\Ct$. Assume $M$ is
 unit-root with global finite mo\-no\-dromy, then the formula
 {\normalfont(\ref{eq:gen_PF_scalar_ext})} is satisfied.
\end{lem}
\begin{proof}
 By a base change to the algebraic closure $k'$ of $k$ in $k(X)$, 
 we may assume that $X$ is geometrically connected, {\itshape i.e.\
 }$C=1$. If $\Ctf'/\Ct$ is totally ramified, we may put
 $\sigma_{\Ctf'}:=\id[\Ctf']$; then the category
 $\Fisosur(U,X/\Ct)_{\Ctf'}$ identifies to $\Fisosur(U,X/\Ctf')$, cf.\
 Example \ref{example:rep'}, and so we finish by \cite[Theorem
 4.3.15]{Marmora:Facteurs_epsilon}.

 Let us treat the general case. For any representation
 $\rho\colon\pi_1(U,\bar{\eta}_X)\rightarrow \Aut{\Ct}{W}$ and  any
 closed point $x\in |X|$, we denote by $W_{\eta_x}$ the representation
 ${\pi_1(\eta_x,\bar{\eta}_x)}\rightarrow\pi_1(U,\bar{\eta}_X)
 \xrightarrow{\rho}\Aut{\Ct}{W}$. We put
 $V_{\eta_x}:=G^{\dag}(M)_{\eta_x}$, and
 $V'_{\eta_x}:=G^{\dag}_{\Ctf'}(M')_{\eta_x}$ (cf.\ \ref{scalar_ext})
 which will be treated as a $\Ctf'$-vector space with a linear action of
 $W(\eta_x,\bar{\eta}_x)$. Let us start by proving the following
 statement.

 \begin{cl}
  The $(\bCt \otimes_{\Ct} \Ctf')$-module $\WD_{\Ctf'}(M'_{\eta_x})$ is
  free and  for any $\mf{p}\in\Spec(\bCt \otimes_{\Ct} \Ctf')$, we have
  $\WD_{\Ctf'}(M'_{\eta_x})_{\mf{p}} = \bCt \otimes _{\Ctf'}
  V'_{\eta_x}$.
 \end{cl}
 \renewcommand{\qedsymbol}{$\square$}
 \begin{proof}[Proof of the claim]
  Let us compute $\Psi(-1)_{\Ctf'}(M'_{\eta_x})$ and
  $\WD_{\Ctf'}(M'_{\eta_x})$. Their mo\-no\-dromy operators $N$ are zero,
  because $M$ is unit-root. Let us denote by
  $\underline{\text{D\'eco}}(\widehat{\Ct}\nr \otimes_{\Ct} V_{\eta_x})$
  the sub-$\Ct\nr$-vector space of $\widehat{\Ct}\nr
  \otimes_{\Ct}V_{\eta_x}$, spanned by the finite orbits under the action
  of $\pi_1(\eta_x,\bar{\eta}_x)$. By
  \cite[3.3.6]{Marmora:Facteurs_epsilon}, we have $\Psi(-1)(M_{\eta_x}) =
  \underline{\text{D\'eco}}(\widehat{\Ct}\nr \otimes_{\Ct}
  V_{\eta_x})$. Since $M$ has finite mo\-no\-dromy, we get
  $\Psi(-1)(M_{\eta_x})=\Ct\nr \otimes_{\Ct} V_{\eta_x}$ endowed with the
  diagonal action of $\pi_1(\eta_x,\bar{\eta}_x)$ (it acts on $\Ct\nr$
  via the residual action). Hence $\WD(M_{\eta_x})= \bCt \otimes_{\Ct}
  V_{\eta_x}$, where the action of $W(\eta_x,\bar{\eta}_x)$ is nothing
  else than the extension by linearity of the action of
  $W(\eta_x,\bar{\eta}_x)$. We finish by the equality
  $\WD_{\Ctf'}(M'_{\eta_x})= \bCt \otimes_{\Ct} V'_{\eta_x}   = \bCt
  \otimes_{\Ct} \Ctf' \otimes _{\Ctf'} V'_{\eta_x}.$
 \end{proof}
 \renewcommand{\qedsymbol}{$\blacksquare$}
 To establish the product formula for $M'$, it remains to prove the following
 relation: 
 \begin{multline*}
  \det_{\Ctf'}\bigl(-F^*;H^*_{\rig,c}(U,M')_{\Ctf'}\bigr)^{-1} = \\
  = q^{\frac{(1-g)\mathrm{rk}(M)}{|\Ctf':\Ct|}}\prod\limits_{x\in |U|}
  {q}_x^{\frac{\val[x](\omega)\mathrm{rk}(M)}{|\Ctf':\Ct|}}
  {{\det_{\Ctf'}}(\rho_{\eta_x}(\F_x);
  {(V_{\eta_x}')}^{I_{\eta_x}})}^{\val[x](\omega)}\prod\limits_{x\in
  X\backslash U} \emar(V'_{\eta_x},\omega_x).
 \end{multline*}
 The proof of this equation works in the same way as the proofs of
 \cite[Theorem 4.3.11 and 4.3.15]{Marmora:Facteurs_epsilon}, by
 replacing $\Ctf$ (resp.\ $\mathrm{rk}(M)$) with $\Ctf'$
 (resp.\ $\frac{\mathrm{rk}(M)}{|\Ctf':\Ct|}$).
\end{proof}

\begin{prop}
\label{PF:fgm}
 Let $X$ be a proper curve over $k$, $U$ be a non-empty open
 subscheme of $X$, and $M$ be an overconvergent $F$-isocrystal on
 $U$ over $\Ct$. Assume that $M$ is unit-root with global geometrically
 finite mo\-no\-dromy. Then $M$ satisfies the product formula
 {\normalfont(\ref{PF-iso})}.
\end{prop}
\begin{proof}
 In this proof we put $\overline{\eta}:=\overline{\eta}_X$ for
 brevity. Let $\rho\colon \pi_1(U,\overline{\eta}) \rightarrow
 \Aut{\Ctf}{V}$ be the representation associated to $M$ by
 $G^{\dagger}$, cf.\ \ref{Tsuzuki-Katz}. By assumptions, $M$ has global
 geometrically finite mo\-no\-dromy: {\it i.e.\ }the restriction of
 $\rho$ to $\pi_1(U\otimes_k \bFt,\overline{\eta})$ factors through a
 finite quotient $I$. In particular, the representation $\rho$ factors
 through a quotient $Q$ of $\pi_1(U,\overline{\eta})$ which is an
 extension of $\widehat{\ZZ}$ by the finite group
 $I$. By Remark \ref{remark:semplification}, we may assume that the
 order of the Frobenius $\Fro[M]$ of $M$ is $f$, so that $\Ct=\Ctf$ and
 $\sigma_{\Ct}=\id[\Ct]$.

 Let us show how we can reduce to the case of global finite
 mo\-no\-dromy which is treated in Lemma
 \ref{lemma:extention_to_lambda_str}. The equation (\ref{PF-iso}) that
 we have to prove is a relation in the  Grothendieck group of
 $\Fisosur(U,X/\Ct)$. By Lemma \ref{lemma:PF_scalar_invariance}, we
 can extend scalars to any finite Galois extension $\Ctf'/\Ct$. The
 equivalence $G^{\dagger}$ above extends to an equivalence of Tannakian
 categories
 $G^{\dagger}_{\Ctf'}\colon\Fisosur(U,X/\Ct)_{\Ctf'}\ur\rightarrow
 \mathrm{Rep}_{\Ct}^{\mathrm{fg}}(\pi_1(U,\overline{\eta}))_{\Ctf'}$
 (cf.\ \ref{ss:Ctf'_ext}). We identify
 $\mathrm{Rep}_{\Ct}^{\mathrm{fg}}(\pi_1(U,\overline{\eta}))_{\Ctf'}$
 with $\mathrm{Rep}_{\Ctf'}^{\mathrm{fg}}(\pi_1(U,\overline{\eta}))$
 (cf.\ \ref{example:rep'}). The product formula is a relation in the
 Grothendieck group of
 $\mathrm{Rep}_{\Ctf'}^{\mathrm{fg}}(\pi_1(U,\overline{\eta}))$; we may
 assume that the representation is absolutely irreducible.
 By a classical argument using Schur's lemma (cf.\ for example
 \cite[Proof of 3.2.1.7]{Lau} or \cite[Variant 4.10.3]{Del:const_loc}),
 the representation $(V',\rho')$ is isomorphic to
 $(\tilde{V},\tilde{\rho})\otimes_{\Ctf'}(\Ctf',\chi)$,
 where $\tilde{\rho}\colon \pi_1(U,\overline{\eta}) \rightarrow
 \Aut{\Ctf'}{\tilde{V}}$ factors through a finite quotient and
 $\chi\colon\pi_1(U,\overline{\eta}) \rightarrow {(\Ctf')}^*$ is an
 unramified character. Let $D^{\dagger}_{\Ctf'}\colon
 \mathrm{Rep}_{\Ct}^{\mathrm{fg}}(\pi_1(U,\overline{\eta}))_{\Ctf'}\rightarrow
 \Fisosur(U,X/\Ct)_{\Ctf'}\ur$ be a quasi-inverse of
 $G^{\dagger}_{\Ctf'}$. Let us put
 $M':=D^{\dagger}_{\Ctf'}((V',\rho'))$,
 $M_1:=D^{\dagger}_{\Ctf'}((\tilde{V},\tilde{\rho}))$ and
 $M_2:=D^{\dagger}_{\Ctf'}((\Ctf',\chi))$. We have $M'\cong M_1 \otimes'
 M_2$ in $\Fisosur(U,X/\Ct)_{\Ctf'}\ur$ (cf.\ \ref{sub:tens_prod}). By
 construction, $M_1$ has global finite mo\-no\-dromy and $M_2$ is
 constant as isocrystal, {\it i.e.\ }$M_2 = \iota^*N$, for
 $N\in\Fiso(\Spec(k)/\Ct)_{\Ctf'}\ur$ and $\iota\colon U \rightarrow
 \Spec(k)$. Since $\iota^*N$ is a constant isocrystal, we have
 $H^i_{\rig,c}(U,M_1 \otimes' \iota^*N )_{\Ctf'} \cong
 H^i_{\rig,c}(U,M_1)_{\Ctf'} \otimes' N$. By a direct calculation
 analogous to that of the proof of
 \cite[4.3.6]{Marmora:Facteurs_epsilon},  we reduce to the case of
 global finite mo\-no\-dromy, which is proven in
 \ref{lemma:extention_to_lambda_str}. 
\end{proof}

\begin{cor}
 \label{cor:cong_Fiso_spec}
 Let ${M}$ be in $F$-$\mr{Hol}(\eta_{\ms{S}_0})$. Then the canonical
 extension ${M}^{\mathrm{can}}$ satisfies the product formula
 {\normalfont(\ref{PF-iso})}.
\end{cor}
\begin{proof} 
 By Kedlaya's filtration theorem \cite[7.1.6]{Kedlaya:Slope_Rev}, there
 exists a filtration 
 \begin{equation*}
  {M}={M}_0\supset {M}_1\supset\ldots\supset {M}_s=0 
 \end{equation*}
 such that the quotient ${M}_i/{M}_{i+1}$ is isoclinic for every
 $i$. By applying the canonical extension functor \ref{intcanext}, we
 get an analogous filtration on ${M}^{\mathrm{can}}$. Considering
 that the equation (\ref{PF-iso})
 we have to prove is a relation in the  Grothendieck group of
 $\Fisosur(\mathbb{G}_m,\mathbb{P}^1_k/\Ct)$, we may assume ${M}$ to
 be isoclinic of Dieudonn\'{e}-Manin slope $\lambda\in\QQ$. By the definition of Dieudonn\'{e}-Manin slopes,
 $\lambda$ belongs to the discrete subgroup
 $(\rg({M})eh)^{-1}\ZZ$. Taking a finite {\em totally ramified}
 extension of $\Ctf$ does not affect the local factors; therefore, by
 extending $\Ctf$ to such an extension of degree $\rg({M})$, we may
 assume that $\lambda$ belongs to $(eh)^{-1}\ZZ$. So the isocrystal
 $K^{(-\lambda)}$ is of rank $1$ by construction. 
We put $\widetilde{M}:= {M} \otimes K^{(-\lambda)}$. We have 
${M}=\widetilde{M}^{(\lambda)}$, with $\widetilde{M}$  unit-root.
 %By tensoring
% ${M}$ with  $K^{(-\lambda)}$, we are reduced to the case
% ${M}=\widetilde{M}^{(\lambda)}$, where $\widetilde{M}$ is unit-root. 
 By applying
 \cite[Lemme 4.3.6]{Marmora:Facteurs_epsilon} to ${M}^{\mathrm{can}}=
 (\widetilde{M}^{\mathrm{can}})^{(\lambda)}$, we may assume $\lambda=0$. Since
 $\widetilde{M}$ is unit-root, $\widetilde{M}^{\mathrm{can}}$ has global geometrically
 finite  mo\-no\-dromy by the very construction of the canonical
 extension (cf.\ \cite[2.6 and 2.7]{Crew:irr_swan}), and we finish by
 Proposition \ref{PF:fgm}.
\end{proof}

\subsection{Proof of the main result}
\label{sb:proof_gen_case}
We use the notation of \ref{setupFour}.
\begin{lem}
 \label{Hurwitz}
 Let $E$ be an overconvergent $F$-isocrystal on $\mb{A}^1_k$. Suppose
 that it is regular at infinity. Then $E$ is a constant overconvergent
 $F$-isocrystal.
\end{lem}
\begin{proof}
 Let $\iota$ denote the structural morphism of $\mb{A}^1_k$.
 By construction of the rigid cohomology \cite[(8.1.1)]{Crew:ens}, we
 have $H^2_{\rig}(\mb{A}^1_k, E) =0$ and, by the GOS formula for rigid
 cohomology, we get $\dim_{\Ct} H^0_{\rig}(\mb{A}^1_k,E)-\dim_{\Ct}
 H^1_{\rig}(\mb{A}^1_k, E) = \chi(\mb{A}^1_k,E)=\rg(E)$.
 In particular $h_0:= \dim_{\Ct} H^0_{\rig}(\mb{A}^1_k, E)\geq\rg(E)$.
 By \cite[2.1.2]{Chiar-LeStum:F-iso_unipot}, there is an injection of
 $F$-isocrystals $\iota^*H^0_{\rig}(\mb{A}^1_k, E)\hookrightarrow E$, so
 that $h_0=\rg(E)$ and  $E$ is isomorphic to the constant isocrystal
 $\iota^* H^0_{\rig}(\mb{A}^1_k, E)$.
\end{proof}

\subsubsection{}
\label{intnotrecipmap}
Let $\ms{S}:=\mr{Spf}(k\cc{u})$, $\ms{S}':=\mr{Spf}(k\cc{u'})$. We put
$\mc{K}:=k\pp{u}$, $\mc{K}':=k\pp{u'}$. \index{.@miscellaneous!S@$\ms{S}$} \index{.@miscellaneous!S@$\ms{S}'$}
\index{.@miscellaneous!K@$\mc{K}$, $\mc{K}'$}
Let
$\pi_{\infty'}:\ms{S}_{\infty'}\xrightarrow{\sim}\ms{S}'$ sending $u'$
to $1/x'$. For a free differential module ${M}$ with Frobenius
structure on $\eta'$, using the canonical extension, there exists a
canonical overconvergent $F$-isocrystal on $\mb{G}_m$ denoted by
$\ms{M}$ such that it is tamely ramified at $0$ and
$\ms{M}|_{\eta_{\infty'}}\cong\pi^*_{\infty'}({M})$. We denote
$\Psi(\ms{M}|_{S_1})$ by $\Psi_1({M})$. \index{functors!.Psi1@$\Psi_1$}

On the other hand, suppose moreover that ${M}$ is of rank $1$. Using
the linearization functor (\ref{linearizationfunc}), we get a character
\begin{equation*}
 \chi:=(L_k\circ\Psi)({M})\colon G_{\mc{K}'}\rightarrow
  K^{\mr{ur}*}.
\end{equation*}
For $f'\in\mc{K}'^*$, we put
\begin{equation*}
 {M}(f'):=(\chi\circ\mr{rec})(f') \in K^{\mr{ur}*},
\end{equation*}
where $\mr{rec}\colon\mc{K}'^*\rightarrow G_{\mc{K}'}^{\mr{ab}}$ \index{.@miscellaneous!rec@$\mr{rec}$, 
$M(f')$} is the
reciprocity map normalized {\it \`{a} la} Deligne ({\it i.e.}\ it sends
uniformizers of $\mc{K}'$ to elements of $G^{\mr{ab}}_{\mc{K}'}$ whose
image in $G^{\mr{ab}}_k$ is the geometric Frobenius $F$).

When ${M}$ is regular of rank $1$, we get
\begin{equation}
 \label{calcrecipsp}
 \mr{tr}(F^*,\Psi_1({M}))={M}(-u).
\end{equation}
This can be seen in exactly the same way as \cite[3.5.2.1]{Lau}, and
we leave the details to the reader.

\begin{rem*}
 We need to be careful for the multiplicativity. Namely, given
 rank $1$ free differential modules ${M}$, ${M}'$, ${M}''$ such that
 ${M}={M}'\otimes{M}''$, we have
 \begin{equation*}
  {M}(-1)(f')={M}'(-1)(f')\cdot{M}''(-1)(f').
 \end{equation*}
 See \ref{WD} for an explanation.
\end{rem*}

\begin{prop}[{\cite[Th\'{e}or\`{e}me 3.4.2]{Lau}}]
 \label{Laumondetformula}
 Let $U\subset\mb{A}^1$ be an open subscheme, and we put
 $S:=\mb{A}^1\setminus U$. Let $\ms{M}$ be an overconvergent
 $F$-isocrystal of rank $r$ on $U$ which is unramified at infinity. Then
 we get
 \begin{equation*}
  \det(R\Gamma_c(U_{K^{\mr{ur}}},\ms{M})[1])\otimes\det(K^{\mr{ur}}
   \otimes_K\ms{M}|_{s_\infty})(-r)\cong
   \bigotimes_{s\in S}\Psi_1\bigl(\det\bigl(\Phi^{(0,\infty')}(
   j_!\tau_{s*}\ms{M}|_{\eta_s})\bigr)(-\gamma_s-1)\bigr)
 \end{equation*}
 as Deligne modules, where
 $\gamma_s=\mr{rk}(\ms{M}|_{\eta_s})+\mr{irr}(\ms{M}|_{\eta_s})$, and we used
 the notation of {\normalfont\ref{defvanishcycle}}.
\end{prop}

\begin{proof}
 Using the notation of Proposition \ref{exactseq}, we get
 \begin{equation*}
  \det(\Psi(\nFnp{\ms{M}'}|_{S_0})(-2))\cong
   \det(R\Gamma_c(U_{K^{\mr{ur}}},\ms{M})[1])\otimes
   \det(K^{\mr{ur}}\otimes_K\ms{M}|_{s_\infty}(-1))
 \end{equation*}
 as Deligne modules by the same proposition. Since the
 $G_{\mc{K}_0}$-action on the right hand side is unramified ({\it
 i.e. }the action of the inertia subgroup is trivial), the left hand
 side is also unramified. Since $\mb{V}$ is an exact functor and
 commutes with tensor product, we get an equivalence of functors
 $\det\circ\mb{V}\cong\mb{V}\circ\det$. On the other hand, for a
 free differential module ${M}$, we get
 $\mb{D}_{\eta}(\det({M})(-1))\cong\det(\mb{D}_{\eta}({M}(-1)))$. Thus,
 we get
 \begin{equation}
  \label{detpsicom}
  \det(\Psi(\nFnp{\ms{M}'}|_{S_0})(-2))\cong
   \Psi(-1)(\det(\nFnp{\ms{M}'}|_{S_0}(-1))).
 \end{equation}
 We note that the singularity of $\nFnp{\ms{M}'}$ in $\Aone$ is
 only at $0'$ by Corollary \ref{calcofirr}, and we showed that
 $\det(\nFnp{\ms{M}'})$ is unramified at $0'$. Thus there exists an
 overconvergent $F$-isocrystal $\ms{N}$ on $\Aone$ such that
 $\ms{N}|_{\Aoned-\{0'\}}\cong\det(\nFnp{\ms{M}'})|
 _{\Aoned-\{0'\}}$.

 Now, we get
 \begin{align*}
  \ms{N}|_{\eta_\infty}=\det\bigl(\nFnp{\ms{M}'}|_{\eta_\infty}
  (-1)\bigr)&\cong\bigotimes_{s\in S}\tau'^*\det\bigl
  (\ms{F}^{(s,\infty')}(\ms{M}')(-1)\bigr)\\&
  \cong\bigotimes_{s\in S}\det\bigl(\Phi^{(0,\infty')}
   (j_!\tau_{s*}\ms{M}|_{\eta_s})\bigr)(-\gamma_s)
  \otimes\DwL(\delta),
 \end{align*}
 where $\delta:=\sum_{s\in S}\gamma_s\cdot\mr{tr}(s)$,
 by the stationary phase formula \ref{stationaryfrobcom}, Lemma
 \ref{locFourrel}, and Corollary \ref{locfourstricless}. We note
 that the differential slope of $\Phi^{(0,\infty')}(j_!\tau_{s*}\ms{M}|_{\eta_s})$ is
 strictly less than $1$ by Corollary \ref{locfourstricless} for any
 $s\in S$. Thus, the differential slope of
 $\det\bigl(\Phi^{(0,\infty')}(j_!\tau_{s*}\ms{M}|_{\eta_s})\bigr)$ is
 also strictly less than $1$. Since the rank is $1$, the differential slope of
 $\det\bigl(\Phi^{(0,\infty')}(j_!\tau_{s*}\ms{M}|_{\eta_s})\bigr)$ is
 $0$ for any $s\in S$ by the Hasse-Arf
 theorem \cite[14.12]{CM}. Thus
 \begin{equation*}
  \ms{N}':=
   \ms{N}\otimes\ms{L}(-\delta\cdot x')
 \end{equation*}
 is a convergent $F$-isocrystal on $\Aoned$, and regular at
 infinity. By Lemma \ref{Hurwitz}, $\ms{N}'$ is in fact a constant
 overconvergent $F$-isocrystal. By (\ref{detpsicom}) and the fact that
 $\Psi(-1)$ commutes with $\otimes$, the proposition follows.
\end{proof}

\begin{thm}[$p$-adic determinant formula]
 For any free differential $\mc{R}_{\ms{S}}$-module with Frobenius
 structure ${M}$, we get
 \begin{equation*}
  \emar({M},du)=(-1)^{\gamma}\det(\Phi^{(0,\infty')}
   (j_!{M}))(-\gamma-1)(u'),
 \end{equation*}
 where $\gamma:=\mr{rk}({M})+\mr{irr}({M})$.
\end{thm}

\begin{rem*}
 (i) Before proving the theorem, we remark that the right hand side of
 the equality is multiplicative with respect to short exact sequences by
 Remark \ref{intnotrecipmap}.

 (ii) Although the idea of the proof is exactly the same as {\it loc.\
 cit.}, we include the complete proof of the proposition since there are
 a lot of minor differences in the quantities appearing (especially the
 Tate twist $-\gamma-1$ in the statement of the theorem), and we think that
 it might help the reader to understand the differences with the
 $\ell$-adic case.
\end{rem*}

\begin{proof}
 Let $V$ be a $\sigma_K$-$K$-vector space of dimension $1$. Then the
 theorem holds for a free differential $\mc{R}$-module ${M}$ if and
 only if it holds for ${M}\otimes_KV$. Indeed, by \cite[2.19
 (2)]{Marmora:Facteurs_epsilon}, we have
 \begin{equation*}
  \emar({M}\otimes_KV,du)=\emar({M},du)\cdot
   \det(F;V)^\gamma.
 \end{equation*}
 On the other hand, we see that
 \begin{equation*}
  \Phi^{(0,\infty')}(j_!{M}\otimes_KV)\cong
   \Phi^{(0,\infty')}(j_!{M})\otimes_KV
 \end{equation*}
 by using Proposition \ref{frobstruccalcabsw}. Thus, the claim follows.

 First we will treat the case where ${M}$ is regular. Since both
 sides are multiplicative, we may assume that
 ${M}$ is irreducible. By Kedlaya's slope filtration theorem, we get
 that ${M}$ is isoclinic (for Dieudonn\'{e}-Manin slopes). Since both sides of the equality are stable
 under base change by a totally ramified extension of $\Lambda$,
 we may assume that the Dieudonn\'{e}-Manin slope $\lambda$ of ${M}$ is in
 $(eh)^{-1}\mb{Z}$. Since the equality is stable under twisting and
 $({M}^{(-\lambda)})^{(\lambda)}\cong {M}$, it suffices to show the
 proposition for ${M}^{(-\lambda)}$, and we may suppose that ${M}$
 is unit-root. Thus, it corresponds to a geometrically finite
 representation of $G_{\mc{K}}$ denoted by $\rho$. We know that there
 exists a finite extension $k'/k$ such that $\rho$
 is the induced representation of $G_{\mc{L}}$ of rank $1$ where
 $\mc{L}:=k'\otimes_k\mc{K}$. This shows
 that there exists a finite unramified extension $L$ of $K$, and a
 free differential $\mc{R}_L$-module with
 Frobenius structure ${M}_L$ such that ${M}\cong f_*{M}_L$
 where $f\colon\mc{R}\hookrightarrow\mc{R}_L$ is
 the canonical finite \'{e}tale homomorphism. By \ref{baseextcom}, we
 get
 \begin{equation*}
  \Phi^{(0,\infty')}(j_!{M})\cong f_{*}
   (\Phi^{(0,\infty')}(j_!{M}_L)).
 \end{equation*}
 For the calculation of $\emar$, use \cite[2.14
 (2)]{Marmora:Facteurs_epsilon}. It remains to show the theorem
 for ${M}_L$, and thus, we may assume that ${M}$ is of rank $1$.

 In this case we can write
 ${M}\cong\ms{K}_\alpha|_{\eta_0}\otimes\mc{W}$ such that $\mc{W}$ is
 a trivial differential module if we forget the Frobenius
 structure, and $\ms{K}_\alpha$ is the Kummer isocrystal (cf.\
 \ref{calcKummergeom}). Then we may suppose that $\mc{W}$ is trivial by
 the observation at the beginning of this proof. It remains
 to show the theorem in the
 ${M}=\ms{K}_\alpha|_{\eta_0}=:\ku{\alpha}$ case.

 Now, first, assume that $\alpha\not\in\mb{Z}$. Then
 $j_!j^+\ku{\alpha}\cong\ku{\alpha}$ by Proposition
 \ref{calcKummergeom}. By the stationary phase formula and the same
 proposition, we get
 \begin{equation*}
  \Phi^{(0,\infty')}(\ku{\alpha})\cong\tau'^*\bigl(\ms{K}_{1-\alpha}
   \otimes G(\alpha,\pi)(1)\bigr)|_{\eta_\infty'}\cong\ku{\alpha}
   \otimes G(\alpha,\pi)(1).
 \end{equation*}
 Furthermore, we have
 \begin{equation*}
  \Psi(\ku{\alpha})\cong\mb{V}((\ku{\alpha})^\vee(-1))\cong
   (\mc{B}\otimes\ku{\alpha})^{\partial=0}(1)
 \end{equation*}
 by definition (cf.\ \ref{defvanishcycle}).
 Combining these, we get $\Psi\bigl(\Phi^{(0,\infty')}(\ku{\alpha})
 \bigr)(-2)\cong(\mc{B}\otimes\ku{\alpha})^{\partial=0}
 \otimes G(\alpha,\pi)$.
 The space $(\mc{B}\otimes\ku{\alpha})^{\partial=0}$ is the
 sub-$K^{\mr{ur}}$-vector space spanned by $x^{-\alpha}e$ where $e$ is
 the canonical base of $\ku{\alpha}$. Let
 $\alpha=i\cdot(q-1)^{-1}$, and
 $\chi_\alpha$ be the $i$-th power of Teichm\"{u}ller character
 $k^*\rightarrow K^{\mr{ur}}$. Then we get
 \begin{equation*}
  \mr{tr}(F^*,G(\alpha,\pi))=-\sum_{x\in k^*}\chi_\alpha(x)\cdot
   \psi_k(x)
 \end{equation*}
 by \cite[6.5]{LE} using the fact that
 $H^i_{\mr{rig}}(\mb{G}_m,\ms{K}_\alpha\otimes\ms{L}_\pi)=0$ for
 $i\neq1$ (which can be proved by GOS-formula for example). Now, let us
 treat the case where $\alpha=0$, and thus ${M}=\mc{R}$. In this
 case, we get an exact sequence
 \begin{equation*}
  0\rightarrow\delta(1)\rightarrow j_!\mc{R}\rightarrow j_+\mc{R}
   \rightarrow\delta\rightarrow0.
 \end{equation*}
 We get $\Phi^{(0,\infty')}(j_+\mc{R})\cong j_+\mc{R}$ by using
 Proposition \ref{calcfouriereasy} and Theorem
 \ref{involForge}. Thus, we obtain $\Phi^{(0,\infty')}(j_!\mc{R})\cong
 j_+\mc{R}(1)$. Combining all of these, we get
 \begin{equation*}
  \det(\Phi^{(0,\infty')}(j_!\ku{\alpha}))(-2)(u')=
   \begin{cases}
    1&\quad\mbox{if $\alpha=0$}\\
    -\chi(-1)\cdot\sum_{x\in k^*}\chi_\alpha(x)\cdot\psi_k(x)
    &\quad\mbox{if $\alpha\neq0$}.
   \end{cases}
 \end{equation*}

 On the other hand, let
 $\widetilde{\chi}:=(\mr{rec}\circ\WD)({M})$. By using \cite[XIV \S
 3, Prop 8]{Se},
 we get
 \begin{equation*}
  \widetilde{\chi}(x)=
   \begin{cases}
    \chi_\alpha^{-1}(x)&\quad x\in k^*\subset\mc{K}^*\\
    \chi_\alpha(-1)&\quad x=u\\
    1&\quad x\in 1+\mf{m}
   \end{cases}
 \end{equation*}
 (note that the image of the geometric Frobenius is different, so we
 need to calculate $(x^{-i},a^{-1})$ with $n=q-1$ in the notation of
 {\it loc.\ cit}.).
 This shows that
 \begin{equation*}
  \emar(\ku{\alpha},du)=
   \begin{cases}
    -1&\quad\mbox{if $\alpha=0$}\\
    \chi(-1)\cdot\sum_{x\in k^*}\chi_\alpha(x)\cdot\psi_k(x)
    &\quad\mbox{if $\alpha\neq0$}.
   \end{cases}
 \end{equation*}
 Thus the theorem follows in this case.

 We denote by $r$ the rank of ${M}$.
 By taking the canonical extension, there exists an overconvergent
 $F$-isocrystal $\ms{M}'$ on $\widehat{\mb{P}}^1\setminus\{0,1\}$ {\it
 regular at $1$}, and $\ms{M}'|_{\eta_0}\cong\pi_{0}^*({M})$. We put
 $\ms{M}:=\DdagQ{\Pone}(\infty)\otimes_{\DdagQ{\Pone}}\ms{M}'$. By
 Corollary \ref{cor:cong_Fiso_spec}, we get
 \begin{equation*}
  \det(-F;H^*_c(\mb{A}^1\setminus\{0,1\},\ms{M}))^{-1}\cdot q^{r}\cdot
   \det(-F_\infty;\ms{M}|_{s_\infty})=\emar({M},-du)\cdot
   \emar(\ms{M}|_{\eta_1},-dx|_{\eta_1}).
 \end{equation*}
 On the other hand, by Proposition \ref{Laumondetformula}, we get
 \begin{align*}
  &\det(-F;H^*_c(\mb{A}^1\setminus\{0,1\},\ms{M}))^{-1}\cdot q^{r}
  \cdot\det(-F_\infty;\ms{M}|_{s_\infty})=\\
  &\qquad(-1)^{\gamma}\det(\Phi^{(0,\infty')}(j_!{M}))
  (-\gamma-1)(-u')\cdot(-1)^{r}\det(\Phi^{(0,\infty')}
  (j_!\ms{M}|_{\eta_1}))(-r-1)(-u')
 \end{align*}
 taking (\ref{calcrecipsp}) into account
 (which is applicable  because the differential modules
 $\det(\Phi^{(0,\infty')}(\dots))$ are regular of rank $1$). By
 considering the regular case proven above, we get the theorem.
\end{proof}

\begin{cor}
 \label{lastcor}
 Let $U\subset\mb{A}^1_k$ be an open subscheme, and $M$ be an
 overconvergent $F$-isocrystal of rank $r$ on $U$ which is unramified at
 $\infty$. Let $S:=\mb{A}^1\setminus U$. Then we get
 \begin{equation*}
  \det(-F;H^*_{\mr{rig},c}(U,M))^{-1}\cdot q^r\cdot
   \det(-F_{\infty};M_{\infty})=\prod_{s\in S}\emar
   (M|_{\eta_s},-dx|_{\eta_s}).
 \end{equation*}
\end{cor}
\begin{proof}
 The proof is exactly the same as \cite[3.5.2]{Lau} using Proposition
 \ref{Laumondetformula}, and we leave the details to the reader.
\end{proof}

\subsubsection{Proof of Theorem $\ref{Product formula}$}\mbox{}
\\
 The proof is essentially the same as \cite[3.3.2]{Lau}: it is a
 reduction to Corollary \ref{lastcor}. We point out a difference from
 {\it loc.\ cit.}: to prove that the right hand side of the product
 formula does not depend on the choice of the differential form
 $\omega$, we proceed by \emph{d\'evissage} and we use \cite[Proposition
 4.3.9]{Marmora:Facteurs_epsilon}. We leave the details to the reader.
 \hspace{\fill}$\blacksquare$

\clearpage
\addcontentsline{toc}{section}{Index}\printindex

Tomoyuki Abe:\\
Institute for the Physics and Mathematics of the Universe (WPI)\\
The University of Tokyo\\
5-1-5 Kashiwanoha,  
Kashiwa, Chiba, 277-8583, Japan\\
e-mail: {\tt tomoyuki.abe@ipmu.jp}

\bigskip\noindent
Adriano Marmora:\\
Institut de Recherche Math\'ematique Avanc\'ee\\
UMR 7501, Universit\'e de Strasbourg et CNRS\\
7 rue Ren\'e Descartes, 67000 Strasbourg, France\\
e-mail: {\tt marmora@math.unistra.fr}

\end{document}